\documentclass{amsart}
\usepackage{graphicx} 
\usepackage{amsmath,amssymb,mathtools,mathrsfs}
\usepackage{enumitem}
\usepackage{microtype}
\usepackage{tikz-cd}
\usepackage{hyperref}
\hypersetup{hidelinks}
\usepackage{aliascnt}
\usepackage{cleveref}
\usepackage{color}
\usepackage{soul}
\usepackage{quiver}
\usepackage{comment}
\usepackage{booktabs,array,placeins,xurl}

\numberwithin{equation}{section}
\newtheorem{theorem}{Theorem}[section]
\newaliascnt{lemma}{theorem}
\newtheorem{lemma}[lemma]{Lemma}
\aliascntresetthe{lemma}
\newaliascnt{proposition}{theorem}
\newtheorem{proposition}[proposition]{Proposition}
\aliascntresetthe{proposition}
\newaliascnt{corollary}{theorem}
\newtheorem{corollary}[corollary]{Corollary}
\aliascntresetthe{corollary}

\theoremstyle{definition}
\newaliascnt{definition}{theorem}
\newtheorem{definition}[definition]{Definition}
\aliascntresetthe{definition}
\newaliascnt{notation}{theorem}
\newtheorem{notation}[notation]{Notation}
\aliascntresetthe{notation}
\newaliascnt{remark}{theorem}
\newtheorem{remark}[remark]{Remark}
\aliascntresetthe{remark}
\newaliascnt{convention}{theorem}

\aliascntresetthe{convention}
\newaliascnt{example}{theorem}
\newtheorem{example}[example]{Example}
\aliascntresetthe{example}

\newcommand{\qrep}[1][]{quasi-represen\-tations}

\newcommand{\Ad}{\operatorname{Ad}}
\newcommand{\supp}{\operatorname{supp}}

\newcommand{\dist}{\operatorname{dist}}
\newcommand{\cK}{\mathcal K}
\newcommand{\cL}{\mathcal L}
\newcommand{\C}{\mathbb C}

\newcommand{\EG}{\underline{E}G}

\newcommand{\A}{\mathcal A}
\newcommand{\E}{\mathcal E}
\newcommand{\K}{\mathcal K}
\newcommand{\whot}{\mathbin{\widehat{\otimes}}}
\newcommand{\QR}{\mathrm{QR}}
\newcommand{\FQR}{\mathrm{FQR}}
\newcommand{\QRucpstage}[2]{\mathrm{QR}^{\mathrm{ucp}}(#1;#2)}
\newcommand{\FQRucpstage}[2]{\mathrm{FQR}^{\mathrm{ucp}}(#1;#2)}

\newcommand{\ucp}{\mathrm{ucp}}
\newcommand{\Klim}{\mathscr K_G^0}
\newcommand{\LL}{\mathcal L}
\newcommand{\HH}{\mathcal H}
\newcommand{\QQ}{\mathsf Q}
\newcommand{\FF}{\mathsf F}
\newcommand{\RR}{\mathcal R}

\newcommand{\Z}{\mathbb Z}

\DeclareMathOperator{\diag}{diag}

\DeclareMathOperator{\defect}{def}

\title{Classifying quasi-representations of groups via their characters}
\author{Shmuel Weinberger}
\author{Jianchao Wu}
\author{Guoliang Yu}

\begin{document}

\thanks{S.~ Weinberger is partially supported by NSF~2505738.  J.~Wu is partially supported by National Key R\&D Program of China~2022YFA100700 and NSFC Key Program No.~12231005.   G.~Yu is partially supported by NSF~2247313. }

\begin{abstract}
For every countable discrete group, we construct a  character
for  finite-dimensional unitary quasi-representations, with
values in the inverse limit of the equivariant $K$-theory of its Rips
complexes. The character is a complete invariant for quasi-representations up to stable equivalence for amenable groups and,
more generally, for groups satisfying the Haagerup property and
QR-quasidiagonality. For Haagerup groups satisfying quantitative
QR-quasidiagonality, we also obtain a quantiative version of the theorems. These conclusions apply when the maximal group
$C^*$-algebra is quasidiagonal and has the lifting property, including
amenable groups, free groups, and closed orientable surface groups.
For surface groups, rank and winding number give complete invariants
for the classification. Finally, composing with the equivariant Chern
character gives a delocalized cohomological obstruction to stable
equivalence of a quasi-representation with genuine representations.
\end{abstract}

\maketitle

\tableofcontents 

\section{Introduction}

Two unitary matrices may almost commute without being close to any commuting pair. Voiculescu constructed examples in which this happens even as the operator norm of the commutator tends to zero \cite{voiculescu_1983}, and Exel and Loring explained this phenomenon topologically through an integer-valued invariant, a winding number, which obstructs approximation by commuting pairs \cite{exel_loring_1991}.
These examples raise a more general question: can approximate representations of a group be classified by the topological information they carry?

Let $G$ be a countable discrete group. A finite-dimensional unitary
representation is a map $\pi:G\to U(n)$ satisfying
$\pi(g)\pi(h)=\pi(gh)$, where $U(n)$ is the group of $n\times n$ unitary
matrices. A \emph{quasi-representation} satisfies these relations only
approximately. Given a finite set $F\subseteq G$ containing the identity
and $\varepsilon>0$, an
$(F,\varepsilon)$-quasi-representation is a map $q:FF\to U(n)$ such that
\[
 \|q(g)q(h)-q(gh)\|<\varepsilon\qquad(g,h\in F).
\]
Here $FF=\{gh:g,h\in F\}$, so $F\subseteq FF$.
The set $F$ specifies the relations being tested, and $\varepsilon$ their allowed error. For $G=\mathbb Z^2$, the basic relation is commutation, so
the example above is the simplest instance of this problem.

Quasi-representations also arise naturally in theoretical computer science, in the theory of almost flat bundles, in index theory, and in the study of approximation properties of groups and $C^*$-algebras. They can be abundant even when genuine finite-dimensional representations are scarce. A central question is the following:
\begin{quote}
\emph{When is a quasi-representation equivalent to a genuine representation?}
\end{quote}

The goal of this article is to classify quasi-representations up to stable equivalence by developing a theory of characters for them.
Two quasi-representations are regarded
as stably equivalent when, after adding the same auxiliary quasi-representation
to both, one is close to (or equivalently, can be deformed into) the other. 

Direct sum of quasi-representations makes the stable classes into an
abelian semigroup. For each pair $(F,\varepsilon)$, formal differences of equivalence classes of $(F,\varepsilon)$-quasi-representations form an abelian group, and $\mathrm{QR}(G)$ is the inverse limit of these groups as $F$ grows and $\varepsilon$ decreases.
To each such class of quasi-representations, we associate a \emph{character}. On a proper $G$-space with compact quotient, a
sufficiently accurate quasi-representation determines an equivariant
vector bundle. For a genuine representation, this is the vector bundle
obtained from its action on the representation space. For a
quasi-representation, the same formula gives only an approximate action;
through an averaging procedure inside a canonical infinite-dimensional equivariant bundle, we introduce a method to still produce an equivariant vector bundle. Its stable $K$-theory class is the character of the quasi-representation.

To place these vector bundles in a common geometric setting, we use the
universal space for proper actions, denoted $\underline EG$
\cite{baum_connes_higson_1994}. The Rips complexes $P_d(G)$ give a
concrete realization of this space: their vertices are the elements of
$G$, and a finite set spans a simplex when its diameter is at most $d$.
As $d$ increases, these complexes encode topological information of  the group at larger scales. The
character construction  gives a
homomorphism
\[
 \chi:\mathrm{QR}(G)\longrightarrow K_G^0(\underline EG).
\]
Here $K_G^0$ is equivariant topological $K$-theory, and we use
$K_G^0(\underline EG)$ to denote the inverse limit of the groups
$K_G^0(P_d(G))$ as $d\rightarrow \infty$. 

Our first main theorem states that, for amenable groups, the character
contains all the information needed for the classification of quasi-representations.

\begin{theorem}
\label{cor:amenable-character}
Let $G$ be an amenable countable discrete group.  The character induces an
isomorphism
\begin{equation}\label{eq:main-nonquantitative}
 \chi:\mathrm{QR}(G)\xrightarrow{\cong}
 \varprojlim_{d\to\infty}K_G^0(P_d(G))
 =K_G^0(\underline EG).
\end{equation}
\end{theorem}

This theorem is, in effect, an existence and uniqueness statement. Every class in $K_G(\underline{E}G)$, for instance the class of an equivariant vector bundle over $\underline{E}G$, is the character of some class of quasi-representations, and this class is unique, since two classes of quasi-representations are equal precisely when their characters agree. The theorem therefore turns the classification of quasi-representations into a topological classification on $\underline{E}G$, which is a much more  computable problem.

The proof also yields an effective, quantitative version of this result. Write $\chi_d$ for the character at Rips scale $d$, with values in $K_G^0(P_d(G))$, and $r_{D,d}\colon K_G^0(P_D(G))\to K_G^0(P_d(G))$ for the restriction map when $D>d$.

\begin{theorem}
\label{thm:main-quantitative}
Let $G$ be an amenable countable discrete group.
\begin{enumerate}[label=(\roman*)]
\item For every finite set $F_0\subseteq G$ and every $\varepsilon_0>0$ there is $d_0>0$ such that, for every $d\geq d_0$, there are a finite set $F_1\supseteq F_0$ and $0<\varepsilon_1\leq\varepsilon_0$ with the following property: if $q_1$ and $q_2$ are $(F_1,\varepsilon_1)$-quasi-representations with $\chi_d([q_1])=\chi_d([q_2])$ in $K_G^0(P_d(G))$, then $q_1$ and $q_2$ determine the same class in $\mathrm{QR}(F_0,\varepsilon_0)$.
\item For every $d>0$ there are $D>d$, a finite set $F\subseteq G$ and $\varepsilon>0$ such that for every $z\in K_G^0(P_D(G))$ there is $x\in\mathrm{QR}(F,\varepsilon)$ with $\chi_d(x)=r_{D,d}(z)$.
\end{enumerate}
\end{theorem}

Part~(i) says that, once a finite set and an error tolerance for comparing quasi-representations are fixed, two sufficiently accurate quasi-representations with the same character at a large enough Rips scale are already equivalent at the prescribed level. Part~(ii) says that every topological class on a sufficiently large Rips complex is, after restriction, the character of a class of quasi-representations.

Our classification result extends beyond amenable groups. The relevant
property concerns the natural comparison
\[
 \iota:\mathrm{QR}(G)\longrightarrow\mathrm{FQR}(G;\mathbb C)
\]
with the group of \emph{Fredholm quasi-representations}. We call $G$
\emph{QR-quasidiagonal} when this map is an isomorphism. Its quantitative
version requires compatible inverse homomorphisms after strengthening
the finite sets and tolerances; both notions are defined in
Section~\ref{sec:quantitative-KQD}.

\begin{theorem}
\label{thm:main-beyond}
Let $G$ be a countable discrete group with the Haagerup property.
If $G$ is QR-quasidiagonal, then the character in \eqref{eq:main-nonquantitative} is an isomorphism. 
Moreover, if $G$ is quantitatively QR-quasidiagonal, the two finite-control
conclusions of \Cref{thm:main-quantitative} also hold for $G$.
\end{theorem}

The qualitative and quantitative character isomormorphism theorems are stated separately in
\Cref{thm:qual,thm:QKQD-Haagerup}. The first concerns only the comparison
of inverse-limit groups; the second also provides uniform controls over the finite sets and the error tolerances. 

To see the link with \emph{quasidiagonality}, a prominent property for $C^*$-algebras, we need to mention that there is a parallel comparison for (finite-dimensional and Fredholm) quasi-representations that give rise to unital completely positive
(u.c.p.) maps on the maximal group $C^*$-algebra $C^*_{\max}(G)$. 
This extra u.c.p.~requirement becomes automatic when $C^*_{\max}(G)$ has the \emph{lifting property (LP)}, i.e., for every unital
$C^*$-algebra $B$ and closed ideal $I\subset B$, each u.c.p. map
$C^*_{\max}(G) \to B/I$ admits a u.c.p. lift $C^*_{\max}(G) \to B$. Quasidiagonality of
$C^*_{\max}(G)$ implies a u.c.p.~version of QR-quasidiagonality of $G$ (even at the quantitative level), and LP transfers it to the original QR-quasidiagonality.

\begin{theorem}[{see \Cref{cor:character} and Subsection~\ref{sec:applications}}]
\label{thm:main-QD-LP}
Let $G$ be a countable discrete group with the Haagerup property.
Suppose that $C^*_{\max}(G)$ is quasidiagonal and has the lifting
property. Then the character in \eqref{eq:main-nonquantitative} is an isomorphism,  
and the quantitative conclusions of \Cref{thm:main-quantitative}
holds for $G$. In particular, these conclusions apply to countable
amenable groups, countable free groups, and fundamental groups of
closed orientable surfaces.
\end{theorem}

For surface groups, the isomorphism has the following explicit form.

\begin{theorem}[{see \Cref{cor:surface-character-classification} and \Cref{thm:surface-QKQD}}]\label{thm:main-surface}
For every $g\geq2$, the character for the surface group $\Gamma_g$ is an
isomorphism. 
As a consequence, their quasi-representations are determined, up to stable equivalence, by a pair of integers: rank and winding number. 
\end{theorem}

In general, we raise the question for which classes of groups the character map is an isomorphism.  

Even without knowing if the character gives an isomorphism, it still provides an obstruction to a quasi-representation
being equivalent to a genuine representation. Composing it with the
Chern character gives a \emph{rational character} in the delocalized
cohomology of the group. A genuine representation gives a flat bundle;
its rational character has no component of positive cohomological degree.
A nonzero component of positive degree is therefore an obstruction to
stable equivalence with a genuine representation. For torsion-free
amenable groups, the converse holds after rationalization and
stabilization. 

The questions studied in this article belong to the broader study of stability of approximate
algebraic relations. Work of de~la Harpe and Karoubi on representations in
Banach algebras and Kazhdan's stability theorem for amenable groups provide
foundational perspectives
\cite{de_la_harpe_karoubi_1977,kazhdan_epsilon_1982}. In the operator norm,
Dadarlat has developed the connections with almost flat bundles, quasidiagonality, index
theory, and group homology
\cite{dadarlat_index_2012,dadarlat_bundles_2014,dadarlat_2021,
dadarlat_two_homology_2022}. Eilers, Shulman, and S{\o}rensen study
$C^*$-stability and matricial stability
\cite{eilers_shulman_sorensen_2020}, while Willett investigates when the
vanishing of topological obstructions permits stabilized approximation
by genuine representations \cite{willett_conditional_2026}.

Different choices of metric on the unitary groups lead to different stability questions. 
For example, one may consider the normalized Hamming distance (on permutation groups) \cite{becker_lubotzky_thom_2019}, the Frobenius norm \cite{de_chiffre_glebsky_lubotzky_thom_2020}, and the Schatten $p$-norms \cite{lubotzky_oppenheim_2020}. 
Related ideas enter
\emph{property testing} in computer science, where one asks how a small
sample of relations can distinguish a solution from an assignment far
from every solution; see \cite{goldreich_property_testing_2017}
and \cite{becker_lubotzky_mosheiff_2023}.

The
proof of our main result for amenable groups uses the work of Higson and Kasparov on the Baum-Connes conjecture for groups acting properly on
Hilbert space \cite{higson_kasparov_2001}. For surface groups, it also
draws on Lubotzky and Shalom's property FD
\cite{lubotzky_shalom_2004}, Carri\'on and Dadarlat's winding-number
formula \cite{carrion_dadarlat_surface_2013}, the quantitative
$K$-homology of Willett and Yu
\cite{willett_yu_controlled_2026,willett_yu_uct_2024}, Willett's
stable-uniqueness theorem \cite{willett_conditional_2026}, and the
lifting-property results of Fournier-Facio and Willett
\cite{fournier_facio_willett_2026}.

\subsection*{Organization of the article.}
Section~\ref{sec:preliminaries} fixes the conventions for proper
actions, cutoff functions, Rips complexes, almost projections, and
equivariant topological $K$-theory.
Section~\ref{sec:qrep} defines quasi-representations, stable equivalence,
and the group $\mathrm{QR}(G)$.
Section~\ref{sec:character_map} constructs the geometric character by
averaging and spectral projection, and proves that it depends only on
the stable class. Section~\ref{sec:Chern-character} introduces the
rational character and its obstruction to genuine representations.

Section~\ref{sec:analytic-tools} collects the analytic tools and the
finite-control conventions. Section~\ref{sec:FQR-cycle} introduces
Fredholm quasi-representations and extends the character to them.
Section~\ref{sec:qd} proves the finite-dimensional--Fredholm comparison
for amenable groups. Section~\ref{sec:amenable-proof} establishes Bott
cancellation and the controlled inverse character, completing the
proofs of \Cref{cor:amenable-character,thm:main-quantitative}.
Finally, Section~\ref{sec:quantitative-KQD} proves the qualitative
and quantitative character theorems under QR-quasidiagonality. It
then establishes quantitative comparison from quasidiagonality of
the maximal group algebra, treats the u.c.p. comparison and its LP
transfer, and verifies the hypotheses for amenable, free, and surface
groups. The surface-group discussion gives the rank-and-winding
coordinates and a direct finite-control classification.

\subsection*{Acknowledgments}
The authors thank Lily Zhang for a careful reading of an earlier
version and for helpful suggestions to improve the exposition.
Rufus Willett has kindly informed us that he has independently proved similar results for torsion free groups.

\subsection*{AI disclosure.}
ChatGPT was used solely for language editing and grammatical refinement during the preparation of this manuscript. All mathematical proofs and core results were developed by the authors.

\section{Preliminaries}\label{sec:preliminaries}

The purpose of this section is to fix the conventions that will be used throughout the paper and to recall the standard facts about proper actions, cutoff functions, crossed products, almost idempotents,  vector bundles,  and topological K-theory.

Throughout the paper $G$ denotes a countable discrete group.  Unless explicitly stated otherwise, actions on spaces are right actions and actions on $C^*$-algebras are left actions.  If $X$ is a right $G$-space, we write $xg$ or $x \cdot g$ for the action of $g \in G$ at $x \in X$ and define the associated left action on $C_0(X)$ by
\[
   (\beta_g f)(x)=f(xg),\qquad \text{ for any } f\in C_0(X),\ g\in G.
\]
Given subsets $E, F$ of $G$, we write $EF=\{gh:g \in E, h\in F\}$. 
In particular, $FF := \{gh:g,h\in F\}$. 
Whenever $F$ is used as a finite control set, we require $e\in F$.
An arbitrary prescribed finite target set is interpreted after adjoining
$e$. This cofinal convention ensures that $F\subseteq FF$ and leaves
the inverse-limit theory unchanged.

Since oftentimes a symmetric finite set containing the identity is needed, we define, recursively,
\begin{equation}\label{eq:set-product}
    F^{(0)} := \{e\}\quad \text{ and } \quad F^{(m+1)}:=F^{(m)} (F\cup \{e\}\cup F^{-1}) \; .
\end{equation}
Similarly, given $Y \subseteq X$ and $E \subseteq G$, we may write $YE := \{ xg \colon x \in Y, g \in E \}$.

\subsection{Proper actions and cutoff functions}

\begin{definition}\label{def:proper-action}
A right action $X\curvearrowleft G$ on a locally compact Hausdorff space is \emph{proper} if for any compact subset $K\subseteq X$ the set
\[
   \{g\in G:Kg\cap K\neq\varnothing\}
\]
is finite.  It is \emph{cocompact} or \emph{$G$-compact} if $X/G$ is compact, or equivalently, there is a compact $K \subseteq X$ such that $KG = X$.  A \emph{proper $G$-space} means a locally compact Hausdorff $G$-space with a proper action.
\end{definition}

Properness can also be characterized in a more flexible way: given subsets $K,L\subseteq X$, define 
\begin{equation}\label{eq:finite-overlap}
    E(K,L):=\{g\in G:Kg\cap L\neq\varnothing\} \; .
\end{equation}
Then it follows from the monotonicity of $E(-,-)$ in both arguments that the action $X\curvearrowleft G$ is proper if and only if $E(K,L)$ is finite for any compact subsets $K,L\subseteq X$. Moreover, when $K = L$, we abbreviate
\[
    E(K)  := E(K, K) \; .
\]
It is straightforward to check that the set $E(K)$ is symmetric and contains the identity of $G$.

\begin{lemma}\label{lem:finite-overlap-fatten}
    For any compact subsets $K_1, K_2$ in a proper $G$-space $X$, there are open subsets $U_1, U_2$ in $X$ such that $K_i \subseteq U_i$ for $i = 1,2$ and $E(K_1, K_2) = E(U_1, U_2)$. 
\end{lemma}

\begin{proof}
Choose compact neighborhoods $L_i$ of $K_i$. The set $E(L_1,L_2)$
is finite. For each
$g\in E(L_1,L_2)\setminus E(K_1,K_2)$, the compact sets $K_1g$
and $K_2$ are disjoint, so choose disjoint open neighborhoods
$V_{g,1}\supset K_1g$ and $V_{g,2}\supset K_2$. Set
\[
 U_1=\operatorname{int}(L_1)\cap\bigcap_g V_{g,1}g^{-1},\qquad
 U_2=\operatorname{int}(L_2)\cap\bigcap_g V_{g,2},
\]
where the intersections run over this finite set of excluded elements.
Then $K_i\subset U_i\subset L_i$ and $U_1g\cap U_2=\varnothing$
for every excluded $g$. Elements outside $E(L_1,L_2)$ cannot create
an overlap either. Monotonicity now gives
$E(K_1,K_2)=E(U_1,U_2)$.
\end{proof}

\begin{definition}\label{def:cutoff}
Let $X\curvearrowleft G$ be proper and cocompact.  A \emph{cutoff function} is a continuous compactly supported function $\theta:X\to[0,\infty)$ such that
\[
       \sum_{g\in G}\theta(xg)=1
       \quad \text{ for any }x\in X.
\]
Note that the sum is a finite one thanks to the properness of $X\curvearrowleft G$.
\end{definition}

\begin{lemma}\label{lem:cutoff-exists}
For any open set $U$ in a proper cocompact $G$-space $X$ satisfying $U G = X$, there is a cutoff function supported in $U$.
\end{lemma}

\begin{proof}
Choose a compact set $K_0\subset X$ with $K_0G=X$. Since $UG=X$,
each $x\in K_0$ has a precompact open neighborhood $V_x$ and an
element $g_x\in G$ such that $\overline{V_x}g_x\subset U$.
Choose a finite subcover $V_1,\ldots,V_n$ of $K_0$ and put
$K=\bigcup_i\overline{V_i}g_i\subset U$. Then $K$ is compact
and $KG=X$. Choose $a\in C_c(X)$ with $a\geq0$, $a>0$ on $K$,
and $\supp(a)\subset U$. Properness makes
$A(x)=\sum_{g\in G}a(xg)$ locally finite and continuous; it is
positive because $KG=X$, and reindexing shows that it is
$G$-invariant. Thus $\theta=a/A$ is a cutoff function supported in $U$.
\end{proof}

For a cutoff function $\theta$ and $x\in X$, define
\begin{equation}\label{eq:cutoff-vector}
    \theta_x^{1/2}(g):=\sqrt{\theta(xg)}\quad \text{ for each } g\in G.
\end{equation}
Then $\theta_x^{1/2}$ is a unit vector in $\ell^2(G)$.  Let $p_\theta(x)$ be the rank-one orthogonal projection onto $\C\theta_x^{1/2}$.  Let $\lambda$ denote the left regular representation, $\lambda_g\delta_h=\delta_{gh}$.

\begin{lemma}\label{lem:cutoff-projection-equivariance}
For every $g\in G$ and $x\in X$,
\begin{equation}\label{eq:p-Ad-tilde-invariant-prelim}
     \Ad_{\lambda_g}(p_\theta(xg))=p_\theta(x).
\end{equation}
Equivalently, if $\beta_g(T)(x)=T(xg)$ on $C_b(X,\cK(\ell^2G))$, then
$\Ad_{\lambda_g}\circ\beta_g(p_\theta)=p_\theta$.
\end{lemma}

\begin{proof}
For $h\in G$,
\[
 (\lambda_g\theta_{xg}^{1/2})(h)
   =\theta_{xg}^{1/2}(g^{-1}h)
   =\sqrt{\theta(xgg^{-1}h)}
   =\sqrt{\theta(xh)}
   =\theta_x^{1/2}(h).
\]
Thus $\lambda_g$ carries the line $\C\theta_{xg}^{1/2}$ onto
$\C\theta_x^{1/2}$, proving the formula for the corresponding orthogonal projections.
\end{proof}

\subsection{Rips complexes and the universal space for proper actions}

Any countable discrete group admits a proper length function $\ell:G\to[0,\infty)$, unique up to coarse equivalence; that is, $\ell(g)=0$ if and only if $g=e$, $\ell(g^{-1})=\ell(g)$, $\ell(gh)\leq\ell(g)+\ell(h)$, and $\{g\in G:\ell(g)\leq R\}$ is finite for every $R\geq0$. Proper length functions correspond bijectively to proper right-invariant metrics via $\dist(g,h)=\ell(gh^{-1})$. We fix a proper length function $\ell$ and its associated metric, and write $B(e,R)=\{g\in G:\ell(g)\leq R\}$.

For each $d \geq 0$, the \emph{Rips complex} $P_d(G)$ is the simplicial complex with vertex set $G$ in which a finite subset spans a simplex whenever its diameter is at most $d$. Right translation on the vertex set extends simplicially to a proper cocompact action on $P_d(G)$. For $d<d'$ there is a canonical inclusion $P_d(G)\hookrightarrow P_{d'}(G)$.

A \emph{universal proper $G$-space} $\EG$ is a proper $G$-space such that for every proper $G$-space $Y$ there is a $G$-map $Y\to\EG$, unique up to $G$-homotopy. Such a space is unique up to $G$-homotopy equivalence. The directed union of the Rips complexes is a convenient model of $\EG$; more precisely, the family $(P_d(G))_{d>0}$ is cofinal among $G$-compact subspaces of a universal proper space. We therefore use the convention
\begin{equation}\label{eq:RK-inverse-limit}
   K_G^0(\EG):=\varprojlim_{d\to\infty}K_G^0(P_d(G)),
\end{equation}
where, for $d'>d$, the bonding map $K_G^0(P_{d'}(G))\to K_G^0(P_d(G))$ is restriction along the inclusion $P_d(G)\hookrightarrow P_{d'}(G)$.

Finally, we record how the Rips scale controls the accuracy needed to construct the character in \Cref{sec:character_map}. Let $X$ be a cocompact proper $G$-space. For a compact subset $K\subseteq X$, write $E(K):=E(K,K)=\{g\in G: Kg\cap K\neq\varnothing\}$ as in the previous subsection, which is finite by \eqref{eq:finite-overlap}. If $KG=X$, the translates $Kg$ cover $X$, and $E(K)$ records which of these translates overlap $K$. To construct the character on $X$, the quasi-representation must be nearly multiplicative on a finite subset of $G$ determined by $E(K)$, so it is useful to choose $K$ with $E(K)$ small. For the Rips complex $X=P_d(G)$ there is a convenient choice: let $K$ be the union of all simplices spanned by finite subsets of $G$ that contain the identity $e$ and have diameter at most $d$. Then $KG=P_d(G)$ and
\begin{equation}\label{eq:Rips-complexity}
    E(K) = \{ab: a,b\in B(e,d)\} \subseteq B(e,2d).
\end{equation}
Thus the accuracy needed to construct the character on $P_d(G)$ is controlled by the ball $B(e,2d)$.

\subsection{Perturbations of projections and unitaries}\label{sec:almost-idempotent}

In our constructions, we shall naturally encounter \emph{almost projections} (or ``quasi-projections''), which are an approximate version of orthogonal projections, just like what quasi-representations are to genuine representations. 

Given a Hilbert space $H$, we shall write 
\begin{align*}
    B_{\mathrm{fin,sa}} (H) &:= \left\{ T \in B(H) \colon T = T^*  \text{ and } T \text{ has finite rank} \right\} \, , \\
    \operatorname{Proj}_{\mathrm{fin}}(H) &:= \left\{ T \in B_{\mathrm{fin,sa}}(H) \colon T  = T^2  \right\} \; , \\
    \operatorname{QProj}_{\mathrm{fin}}(H) &:= \left\{ T \in B_{\mathrm{fin,sa}}(H) \colon  \ \| T -  T^2\| < \frac{1}{4} \right\} \; , \quad \text{ and } \\
    \operatorname{NProj}_{\mathrm{fin}}(H) &:= \left\{ T \in B_{\mathrm{fin,sa}}(H) \colon  \ \sigma(T) \subseteq (-\frac{1}{2}, \frac{1}{2}) \cup (\frac{1}{2}, \frac{3}{2}) \right\} \; . 
\end{align*}
There is a canonical ``spectral correction'' map (a special case of a functional calculus)
\begin{equation}\label{eq:almost-idempotent-correction}
    \kappa \colon B_{\mathrm{fin,sa}}(H) \longrightarrow \operatorname{Proj}_{\mathrm{fin}}(H) \; , \quad 
    T \longmapsto \chi_{[1/2, \infty)} (T) \; ,
\end{equation}
where $\chi_{[1/2, \infty)} $ is the characteristic function on $[1/2, \infty)$. 
More concretely, 
standard facts from linear algebra tell us that there is a spectral decomposition $T = \sum_{\lambda \in \sigma(T)} \lambda  E_\lambda$, where $\sigma(T)$ equals the (finite) set of eigenvalues of $T$ and, for each eigenvalue $\lambda$, $E_\lambda$ is the spectral projection associated to $\lambda$ (note that the $E_\lambda$'s are mutually orthogonal). Then we define a projection 
\[
    \kappa (T) = \sum_{\lambda \in \sigma(T)} \chi_{[1/2, \infty)} (\lambda)  E_\lambda = \sum_{\lambda \in \sigma(T) \cap [1/2, \infty)}   E_\lambda \; .
\]

We shall need the following basic facts.

\begin{lemma}\label{lem:almost-idempotent-perturb}
For any Hilbert space $H$, the following statements hold: 
\begin{enumerate}

    \item \label{lem:almost-idempotent-perturb:Q_in_N} We have $\operatorname{QProj}_{\mathrm{fin}}(H)  \subseteq \operatorname{NProj}_{\mathrm{fin}}(H)$. 
    
    \item \label{lem:almost-idempotent-perturb:deviation} For any $T \in \operatorname{NProj}_{\mathrm{fin}}(H)$, we have $\| T - \kappa(T) \| < \frac{1}{2}$; in particular, we have  $\operatorname{NProj}_{\mathrm{fin}}(H) = \left\{ T \in B_{\mathrm{fin,sa}}(H) \colon  \ \| T -  T'\| < \frac{1}{2} \text{ for some } T' \in \operatorname{Proj}_{\mathrm{fin}}(H) \right\}$.

    \item \label{lem:almost-idempotent-perturb:continuous} Composing $\kappa$ with any continuous map from a locally compact Hausdorff space $X$ to $\operatorname{NProj}_{\mathrm{fin}}(H)$ yields a continuous map $X \to \operatorname{Proj}_{\mathrm{fin}}(H)$. 

    \item \label{lem:almost-idempotent-perturb:unitary} For any unitary $U$ on $H$, the map $\kappa$ commutes with conjugation by $U$. 
\end{enumerate}
\end{lemma}

\begin{proof}
    The requirement that $\| T -  T^2\| < \frac{1}{4}$ implies that $\sigma(T) \subseteq (\frac{1 - \sqrt{2}}{2}, \frac{1}{2}) \cup (\frac{1}{2}, \frac{1 + \sqrt{2}}{2})$, which proves \eqref{lem:almost-idempotent-perturb:Q_in_N}. For \eqref{lem:almost-idempotent-perturb:deviation}, we see that $\| T - \kappa(T) \| = \max_{\lambda \in \sigma(T)} | \lambda - \chi_{[1/2, \infty)} (\lambda) | < \frac{1}{2}$ thus verifying the first statement, which then proves ``$\subseteq$'' in the second statement, with ``$\supseteq$'' following from Weyl's inequality. 
    For \eqref{lem:almost-idempotent-perturb:continuous}, first note that since continuity can be checked in a small neighborhood around each point, we may as well assume that $X$ is compact. Now it is convenient to use a bit of $C^*$-algebra theory: a continuous function $X \to \operatorname{NProj}_{\mathrm{fin}}(H)$ defines a self-adjoint element $a \in C(X, B(H))$ whose spectrum is contained in $(-\frac{1}{2}, \frac{1}{2}) \cup (\frac{1}{2}, \frac{3}{2})$, on which $\chi_{[1/2, \infty)} $ is a continuous function, whence applying $\kappa$ to $a$ yields another element in $C(X, B(H))$ by continuous functional calculus, which amounts to a continuous map $X \to \operatorname{Proj}_{\mathrm{fin}}(H)$. 
    Lastly, \eqref{lem:almost-idempotent-perturb:unitary}, follows directly from the unitary equivariance of spectral decomposition. 
\end{proof}

It is also well known that almost unitaries can be perturbed to genuine unitaries via polar decomposition. We state a more precise statement in this spirit. 

\begin{lemma}\label{lem:polar-correction}
For any $\delta \in (0, 1/2)$, if a square complex matrix $v$ (or more generally, an element in a $C^*$-algebra) is a \emph{$\delta$-almost unitary} in that sense that $\|v^*v-1\|<\delta$ and $\|vv^*-1\|<\delta$, then it is invertible and $u:=v(v^*v)^{-1/2}$ is unitary, with $\|u-v\|\le 2\delta$. 
\end{lemma}

\begin{proof}
Both $v^*v$ and $vv^*$ have spectrum in $(1-\delta,1+\delta)$
and are invertible. Hence $(v^*v)^{-1}v^*$ is a left inverse of $v$
and $v^*(vv^*)^{-1}$ is a right inverse; these agree, so $v$ is
invertible. Its polar part $u=v(v^*v)^{-1/2}$ is therefore unitary.
Since $v=u(v^*v)^{1/2}$, functional calculus gives
\[
 \begin{aligned}
 \|v-u\|&=\|(v^*v)^{1/2}-1\|\\
 &\leq\max\{\sqrt{1+\delta}-1,\,1-\sqrt{1-\delta}\}\\
 &=1-\sqrt{1-\delta}
   =\frac{\delta}{1+\sqrt{1-\delta}}<\delta\leq2\delta.
 \end{aligned}
\]
\end{proof}

\subsection{Vector bundles and topological \texorpdfstring{$K$}{K}-theory}
\label{subsec:topological-K}

We recall the geometric description of equivariant $K$-theory for proper
actions that will be used throughout the paper.  Let $X$ be a proper
$G$-space.  A finite-dimensional $G$-vector bundle over $X$ is a
finite-dimensional complex vector bundle $V\to X$, equipped with a
continuous action of $G$ by vector-bundle automorphisms covering the
action on $X$.  After choosing a $G$-invariant Hermitian metric, which
is possible for proper actions by the usual averaging argument, we may
and shall regard these automorphisms as unitary on the fibers.

Suppose first that $X$ is proper and $G$-compact.  We define
$K_G^0(X)$ to be the Grothendieck group of the abelian semigroup of
isomorphism classes of finite-dimensional $G$-vector bundles over $X$,
with addition given by direct sum.  Thus
\[
[V_0\oplus V_1]=[V_0]+[V_1].
\]
Equivalently, two finite-dimensional $G$-vector bundles $V_0$ and $V_1$
determine the same class if, after adding suitable equivariant vector
bundles, they become  isomorphic.

This finite-dimensional vector-bundle description of equivariant
$K$-theory for proper actions is due to L\"uck--Oliver
\cite{lueck_oliver_2001}.  Phillips's original construction of equivariant
$K$-theory for proper actions \cite{PhillipsProper} is formulated more
generally in terms of equivariant Hilbert bundles and Fredholm-type
cocycles.  For the proper $G$-compact spaces considered here, the
finite-dimensional bundle picture suffices and is particularly
convenient for our character construction.  See also
\cite{BaumHigsonSchick,EmersonMeyer} for related geometric and analytic
pictures of equivariant $K$-theory and $K$-homology.

The construction is contravariantly functorial.  If $f:X\to Y$ is a
continuous $G$-equivariant map between proper $G$-compact spaces, then
pullback of equivariant vector bundles induces
\[
f^*:K_G^0(Y)\longrightarrow K_G^0(X).
\]
Equivariantly homotopic maps induce the same homomorphism.

When the action of $G$ on $X$ is free, passage to the quotient identifies
finite-dimensional $G$-vector bundles over $X$ with finite-dimensional
vector bundles over $X/G$.  Consequently,
\[
K_G^0(X)\cong K^0(X/G).
\]
For a general proper action, the stabilizers are finite and an
equivariant vector bundle retains, at a point $x\in X$, the
finite-dimensional representation of the stabilizer $G_x$ on the fiber
$V_x$.  Thus ordinary vector bundles on $X/G$ do not in general contain
all the information carried by $K_G^0(X)$.

We apply this construction to the Rips complexes introduced in
Section~2.2.  Each $P_d(G)$ is a proper $G$-compact space, and for
$d<d'$ the inclusion
\[
P_d(G)\longrightarrow P_{d'}(G)
\]
induces by pullback a homomorphism
\[
K_G^0(P_{d'}(G))\longrightarrow K_G^0(P_d(G)).
\]
We use the notation
\[
K_G^0(\underline EG)
   :=
\varprojlim_{d\to\infty}K_G^0(P_d(G)).
\]
Thus an element of $K_G^0(\underline EG)$ is a compatible family
$(z_d)_d$, with
\[
z_d\in K_G^0(P_d(G)),
\qquad
z_{d'}|_{P_d(G)}=z_d
\quad\text{for }d<d'.
\]
This definition does not require a single finite-dimensional
$G$-vector bundle on the whole space $\underline EG$.  It is precisely
this inverse-limit group that occurs as the target of the character
constructed in Section~4.

\begin{remark}
\label{rem:representable-K-theory}
The notation $K_G^0(\underline EG)$ above should be distinguished from
the full equivariant representable $K$-theory of the generally
non-$G$-compact space $\underline EG$.

More generally, let
\[
X_1\subseteq X_2\subseteq\cdots\subseteq X
\]
be a $G$-compact exhaustion of a proper $G$-space $X$.  Phillips's
representable $K$-theory for $\sigma$-$C^*$-algebras gives a Milnor
exact sequence, and in the equivariant setting Emerson--Meyer
\cite[Equation~(4.6)]{EmersonMeyer} obtain
\[
0\longrightarrow
\varprojlim\nolimits^1_n K_G^{*-1}(X_n)
\longrightarrow
RK_G^*(X)
\longrightarrow
\varprojlim_n K_G^*(X_n)
\longrightarrow 0.
\]
See also Phillips \cite[Proposition~3.6]{PhillipsRepresentable} for the
corresponding $\sigma$-$C^*$-algebraic result.

Accordingly, our group
\[
K_G^0(\underline EG)
=
\varprojlim_d K_G^0(P_d(G))
\]
is the quotient term in the Milnor $\varprojlim^1$ sequence for
equivariant representable $K$-theory.  We use this quotient rather than
the full representable group because the character constructed from a
quasi-representation is naturally a compatible family of
finite-dimensional equivariant vector-bundle classes on the
$G$-compact Rips complexes.  No vanishing of the
$\varprojlim^1$ term is assumed.

The relation between this geometric equivariant $K$-theory and
Kasparov's equivariant $KK$-theory is recalled in Section~\ref{subsec:representable-KK},
following Emerson--Meyer.  We keep that analytic identification
separate from the finite-dimensional vector-bundle definition used
here.
\end{remark}

\subsection{Fell absorption}\label{subsec:Fell}

Fell's absorption principle will play an important role in the construction of our character map. In its basic form, it states that for any unitary representation $\pi \colon G \to U(H)$, we have a unitary equivalence between two unitary representations: 
\[
    \lambda  \otimes \pi \cong \lambda  \otimes \mathrm{id}_{H} \colon G \to U(\ell^2(G) \otimes H) 
\]
where $\lambda\colon G \to U(\ell^2(G))$ is the left regular representation defined by the equations $\lambda(g)(\delta_h) = \delta_{gh}$ for any $g \in G$ and any Dirac delta function $\delta_h \in \ell^2(G)$; more precisely, the unitary 
\begin{equation}\label{eq:Vpi-prelim}
    V_\pi \in  U(\ell^2(G)  \otimes H) \quad \text{ given by } V_\pi (\delta_h \otimes \xi) = \delta_h \otimes \pi(h)\xi
\end{equation}
for any Dirac delta function $\delta_h \in \ell^2(G)$ and any $\xi \in H$, satisfies 
\begin{equation}
\label{eq:Fell_absorption}
    V_\pi \, (\lambda \otimes \mathrm{id}_{H}) (g) \, V_\pi^* = (\lambda \otimes \pi) (g) \in U(\ell^2(G) \otimes H) \, \quad \text{for any } g \in G \; .
\end{equation}

We shall need to generalize this to non-multiplicative maps $\pi \colon G \to U(H)$. 
Note that \eqref{eq:Vpi-prelim} still defines a unitary operator  $V_\pi$  when $\pi$ is not required to be multiplicative, though a direct calculation reveals that equation~\eqref{eq:Fell_absorption} is reduced to 
\begin{equation}\label{eq:Fell-general-prelim}
    V_\pi \, (\lambda \otimes \mathrm{id}_{H}) (g) \, V_\pi^* (\delta_h \otimes \xi) = \delta_{gh} \otimes \pi(gh)\pi(h)^{-1}\xi
\end{equation}
for any $g \in G$, any Dirac delta function $\delta_h \in \ell^2(G)$ and any $\xi \in H$.

\section{Quasi-representations and stable equivalence}\label{sec:qrep}

In this section, we define when two quasi-representations are stably equivalent, and construct an abelian group $\mathrm{QR}(G)$ from the resulting equivalence classes.

\begin{notation}
    Given complex matrices (or more generally, bounded operators, or even more generally, elements in a $C^*$-algebra) $a,b$ and $\varepsilon \geq 0$, we write: 
    \begin{itemize}
        \item $a \approx_\varepsilon b$ for the condition $\|a - b\| \in [0, \varepsilon) \cup \{0\}$; 
        \item $\mathrm{Ad}_a (b)$ for the element $a b a^*$ (usually when $a$ is a unitary). 
    \end{itemize}
For each $n \in \mathbb{N}$, we let $U(n)$ be the group of all $n \times n$ unitary matrices (with $U(0)$ being the trivial group). 
\end{notation}

\begin{definition}
Given $\varepsilon \geq 0$, a finite subset $F$ of $G$ containing $e$,
and $n\in\mathbb N$, an
\emph{$n$-dimensional $(F,\varepsilon)$-quasi-representation} is a map
$q:FF\to U(n)$ that is \emph{$(F,\varepsilon)$-multiplicative} in the sense that
\[
 q(g)q(h) \approx_\varepsilon q(gh) \qquad \text{ for any }g,h\in F.
\]
We denote the set of all such maps by $\mathrm{QRep}(F,\varepsilon;n)$ and
endow it with the subspace topology inherited from $U(n)^{FF}$.  We then set
\[
 \mathrm{QRep}(F,\varepsilon)
 :=\bigsqcup_{n=0}^{\infty}\mathrm{QRep}(F,\varepsilon;n).
\]
An element of this disjoint union will simply be called an
\emph{$(F,\varepsilon)$-quasi-representation}.\footnote{More explicitly, all
quasi-representations in this paper are finite-dimensional unless the word
``Fredholm'' is present.  
}

We also call a map $q:G\to U(n)$ an
$n$-dimensional $(F,\varepsilon)$-quasi-representation when its restriction to
$FF$ has the above property.  
\end{definition}

\begin{remark}
\label{remark:QRep-open}
    If $\varepsilon>0$, then
    \[
        \mathrm{QRep}(F, \varepsilon;n) = \bigcup_{\eta \in (0,\varepsilon)} \mathrm{QRep}(F, \varepsilon - \eta;n) = \bigcup_{k = 1}^\infty \mathrm{QRep}(F, \varepsilon - \eta_k;n) \; ,
    \]
    where $(\eta_k)_k$ is any sequence in $(0,\varepsilon)$ converging to
    $0$.  
	These two unions yield open covers of $\mathrm{QRep}(F, \varepsilon;n)$. 
\end{remark}

\begin{remark}\label{remark:polar-correction}
    One may relax the requirement that $q$ takes value in the unitary group. In fact, by \Cref{lem:polar-correction}, for any $\delta \in (0,1/2)$, a $\delta$-almost-unitary-valued map that is $(F,\varepsilon)$-multiplicative is close, up to a pointwise norm error of $2\delta$, to an $(F, 8 \delta + \varepsilon)$-quasi-representation. 
\end{remark}

\begin{definition}
We say two $(F,\varepsilon)$-quasi-representations $q_0 , q_1 \in \mathrm{QRep}(F, \varepsilon)$ are 
\begin{enumerate}
    \item  \emph{homotopic} and write  $q_0 \overset{\mathrm{h}}{\sim} q_1$ 
    if 
    there is $n \in \mathbb{N}$ such that $q_0 , q_1 \in \mathrm{QRep}(F, \varepsilon; n)$ and
    there exists a continuous map $H \colon [0,1] \to \mathrm{QRep}(F, \varepsilon;n)$ such that $H(t) = q_t$ for $t \in \{0,1\}$;  
    \item \emph{$\eta$-unitarily equivalent} (for some $\eta \geq 0$) and write  $q_0 \overset{\mathrm{u}}{\sim}_{\eta} q_1$ 
    if 
    there is $n \in \mathbb{N}$ such that $q_0 , q_1 \in \mathrm{QRep}(F, \varepsilon; n)$ and 
    there exists a unitary $V \in U(n)$ such that $\| q_1 (g) - V q_0 (g) V^* \| \in [0, \eta) \cup \{0\}$ for any $g \in FF$. 
\end{enumerate}

It is worth noting that according to these definitions, quasi-representations of different dimensions can never be homotopic or $\eta$-unitarily equivalent to each other. 

Note that $\overset{\mathrm{h}}{\sim}$ is an equivalence relation on $\mathrm{QRep}(F, \varepsilon;n)$ (respectively, $\mathrm{QRep}(F, \varepsilon)$). On the other hand, $\overset{\mathrm{u}}{\sim}_\eta$ is reflexive and symmetric and satisfies $q_0 \overset{\mathrm{u}}{\sim}_{\eta +\eta'} q_2$ whenever $q_0 \overset{\mathrm{u}}{\sim}_\eta q_1$ and $q_1 \overset{\mathrm{u}}{\sim}_{\eta'} q_2$. Hence $\overset{\mathrm{u}}{\sim}_0$ is also an equivalence relation on $\mathrm{QRep}(F, \varepsilon;n)$ (respectively, $\mathrm{QRep}(F, \varepsilon)$), and the formula 
\[
    d_{\mathrm{u}} (q_0, q_1) = \inf \left\{ \eta \geq 0 \colon q_0 \overset{\mathrm{u}}{\sim}_{\eta} q_1 \text{ in } \mathrm{QRep}(F, \varepsilon) \right\}
\]
defines a pseudometric on $\mathrm{QRep}(F, \varepsilon)$. 
\end{definition}

\begin{remark}\label{rmk:direct_sum_operation}
    For $n,n'\in\mathbb N$, direct sum gives a map
    \[
        \oplus \colon \mathrm{QRep}(F, \varepsilon; n)
        \times \mathrm{QRep}(F, \varepsilon; n')
        \longrightarrow \mathrm{QRep}(F, \varepsilon; n+n'),
        \qquad (q,q') \longmapsto q \oplus q'.
    \]
    These combine to give a binary operation 
    \[
        \oplus \colon \mathrm{QRep}(F, \varepsilon) \times \mathrm{QRep}(F, \varepsilon) \to \mathrm{QRep}(F, \varepsilon) \, , \quad (q,q') \mapsto q \oplus q' \; .
    \]
    The operation respects homotopy because the direct sum of two paths of
    $(F, \varepsilon)$-quasi-representations is again a path of
    $(F, \varepsilon)$-quasi-representations.
    It respects $\eta$-unitary equivalence because, if $V$ implements the
    equivalence between $q_0$ and $q_1$, then $V\oplus 1$ implements the same
    estimate after adding a common summand.  Finally, if $q$ has dimension $n$
    and $q'$ has dimension $m$, let
    $S:\mathbb C^n\oplus\mathbb C^m\to
    \mathbb C^m\oplus\mathbb C^n$ be the unitary defined by
    $(\xi,\eta)\mapsto (\eta,\xi)$.  Then
    $S(q\oplus q')S^*=q'\oplus q$.  Thus direct sum is commutative up to
    $0$-unitary equivalence; since $U(n+m)$ is path connected, it is also
    commutative up to homotopy. 
\end{remark}

\begin{lemma} \label{lemma:unitary-implies-homotopy}
    For any $\varepsilon, \eta \geq 0$, any finite subset $F$ of $G$, and any
    $q_0,q_1\in \mathrm{QRep}(F,\varepsilon)$, if
    $q_0\overset{\mathrm u}{\sim}_{\eta}q_1$, then
    $q_0\overset{\mathrm h}{\sim}q_1$ in
    $\mathrm{QRep}(F,\varepsilon+3\eta)$.
\end{lemma}

\begin{proof}
When $\eta=0$, exact unitary equivalence gives the required homotopy
by a path of conjugating unitaries. Assume henceforth that $\eta>0$.
By our definition, there is $n$ such that $q_0,q_1$ take values in
$U(n)$ and that there is $V\in U(n)$ satisfying
\[
 \|q_1(g)-Vq_0(g)V^*\|<\eta \qquad \text{ for any } g\in FF.
\]
Choose a path $(V_t)_{t\in[0,1]}$ in $U(n)$ from $1$ to $V$.  The path
$t\mapsto \operatorname{Ad}_{V_t}(q_0)$ consists of
$(F,\varepsilon)$-quasi-representations, so it remains to connect
$q_1' := \operatorname{Ad}_V(q_0)$ to $q_1$.

If $\varepsilon+3\eta>2$, arbitrary coordinatewise unitary paths
connect $q_1'$ to $q_1$, since a unitary-valued map has multiplication
defect at most $2$. Otherwise $\varepsilon+3\eta\leq2$, so
$\eta<2$ and the principal logarithms below are defined.
For each $g\in FF$ set $w_g=q_1(g)^*q_1'(g)$. 
Then $w_g$ is unitary and, since $\|w_g-1\|<\eta<2$, its spectrum is contained in $U(1) \setminus \{-1\}$, on which the principal logarithm $\log$ is a well-defined continuous function.  
Using continuous functional calculus, we define
\[
 H_t(g)=q_1(g)\exp\big((1-t)\log w_g\big),\qquad \text{ for any } t \in [0,1] \text{ and } g \in FF.
\]
Thus $H_0=q_1'$ and $H_1=q_1$.  
In view of the elementary inequality $|e^{(1-t) \log u} - 1 | \leq |e^{\log u} - 1 |$ from planar geometry, where $t \in [0,1]$ and $u \in U(1) \setminus \{-1\}$, we conclude that  
\[
 \|H_t(g)-q_1(g)\| \leq  \|\exp\big((1-t)\log w_g\big)-1\| \leq \| w_g - 1 \| <\eta\qquad \text{ for any }  g\in FF.
\]
For any $g,h\in F$, we have 
\begin{align*}
 \|H_t(g)H_t(h)-H_t(gh)\|
 &\leq \|H_t(g)-q_1(g)\|\,\|H_t(h)\|
 +\|q_1(g)\|\,\|H_t(h)-q_1(h)\|\\
 &\quad+\|q_1(g)q_1(h)-q_1(gh)\|
       +\|q_1(gh)-H_t(gh)\|\\
 &<\varepsilon+3\eta.
\end{align*}
Thus $H_t\in\mathrm{QRep}(F,\varepsilon+3\eta;n)$ for every $t$.
Concatenating this path with $t\mapsto\operatorname{Ad}_{V_t}(q_0)$ proves the claim.
\end{proof}

The following result essentially says that homotopy
$\overset{\mathrm h}{\sim}$ and approximate unitary equivalence
$\overset{\mathrm u}{\sim}_{\eta}$ agree \emph{stably}, i.e., after adding the same quasi-representation to both sides.  This is an analogue of the equivalence between stable homotopy and stable unitary equivalence in operator $K$-theory.

\begin{proposition} \label{prop:homotopy-vs-unitary}
    For any $\varepsilon>0$, any finite subset $F$ of $G$, and any
    $q_0,q_1\in\mathrm{QRep}(F,\varepsilon)$, the following are equivalent:
    \begin{enumerate}
        \item \label{prop:homotopy-vs-unitary::homotopy}
        there exists $r\in\mathrm{QRep}(F,\varepsilon)$ such that
        $q_0\oplus r\overset{\mathrm h}{\sim}q_1\oplus r$;
        \item \label{prop:homotopy-vs-unitary::unitary}
        there is $\delta\in(0,\varepsilon)$ such that for every $\eta>0$
        there exists $r\in\mathrm{QRep}(F,\delta)$ such that
        $q_0\oplus r\overset{\mathrm u}{\sim}_{\eta}q_1\oplus r$.
    \end{enumerate}
\end{proposition}

\begin{proof}
Assume first \eqref{prop:homotopy-vs-unitary::homotopy}, that is, there is $r\in\mathrm{QRep}(F,\varepsilon; n)$ and a path $H:[0,1]\longrightarrow\mathrm{QRep}(F,\varepsilon;N)$ 
for some $n, N \in \mathbb{N}$ satisfying $H(0)=q_0\oplus r$ and $H(1)=q_1\oplus r$.
By the compactness of $H([0,1])$ and \Cref{remark:QRep-open}, there is
$\delta\in(0,\varepsilon)$ such that $H([0,1]) \subseteq \mathrm{QRep}(F,\delta;N)$.

Fix $\eta>0$.  The uniform continuity of $H$ gives a subdivision
$0=t_0<t_1<\cdots<t_m=1$ such that
\[
 \max_{g\in FF}\|H(t_i)(g)-H(t_{i+1})(g)\|<\eta
 \qquad \text{ for any } i \in \{0, \ldots, m-1\}.
\]
Set
\[
 r_\eta=r\oplus H(t_1)\oplus\cdots\oplus H(t_{m-1}) \in \mathrm{QRep}(F,\delta; n + (m-1)N).
\]
Observe that 
\[
 q_0\oplus r_\eta
   =H(t_0)\oplus H(t_1)\oplus\cdots\oplus H(t_{m-1}),
\]
whereas, by choosing a suitable block permutation unitary, we have 
\[
 q_1\oplus r_\eta
 \overset{\mathrm u}{\sim}_{0}
 H(t_1)\oplus H(t_2)\oplus\cdots\oplus H(t_m).
\]
The difference of the two block diagonal maps has norm less than $\eta$; hence $q_0\oplus r_\eta\overset{\mathrm u}{\sim}_{\eta}q_1\oplus r_\eta$, which is 
\eqref{prop:homotopy-vs-unitary::unitary}.

Conversely assume \eqref{prop:homotopy-vs-unitary::unitary} holds for some $\delta \in (0, \varepsilon)$.  
It follows from \Cref{remark:QRep-open} that there is $\delta_0 \in (0, \varepsilon)$ with
$q_0,q_1\in\mathrm{QRep}(F,\delta_0)$.  Put
$\mu=\max\{\delta,\delta_0\}<\varepsilon$ and 
set $\eta = (\varepsilon - \mu) / 3 > 0$. 
By our assumption there is
$r\in\mathrm{QRep}(F,\delta)$ with $ q_0\oplus r\overset{\mathrm u}{\sim}_{\eta}q_1\oplus r$. 
Both direct sums lie in $\mathrm{QRep}(F,\mu)$, so
\Cref{lemma:unitary-implies-homotopy} gives a homotopy between them through
$\mathrm{QRep}(F,\mu+3\eta) = \mathrm{QRep}(F,\varepsilon)$.
This proves \eqref{prop:homotopy-vs-unitary::homotopy}.
\end{proof}

\begin{definition}
    \label{def:stable-equivalence}
    Let $\varepsilon>0$ and let $F\subseteq G$ be finite.  For
    $q_0,q_1\in\mathrm{QRep}(F,\varepsilon)$, write
    $q_0\overset{\mathrm{s}}{\sim}q_1$ and say that they are
    \emph{stably equivalent} when the equivalent conditions of
    \Cref{prop:homotopy-vs-unitary} hold.

    In view of \Cref{rmk:direct_sum_operation}, the quotient
    $\mathrm{QR}^+(F,\varepsilon)
      :=\mathrm{QRep}(F,\varepsilon)/\overset{\mathrm{s}}{\sim}$
    is an abelian semigroup under direct sum.  Its Grothendieck group is denoted
    $\mathrm{QR}(F,\varepsilon)$.

    Let $\mathcal C_G$ be the set of \emph{control parameters}, i.e., pairs $(F,\varepsilon)$ with
    $F\subseteq G$ finite, $e\in F$, and $\varepsilon>0$, ordered by
    \[
      (F,\varepsilon)\preccurlyeq(F',\varepsilon')
      \quad\Longleftrightarrow\quad
      F\subseteq F'\ \text{and}\ \varepsilon \geq \varepsilon'.
    \]
    Given $(F,\varepsilon)\preccurlyeq(F',\varepsilon')$, we have containments
\[
    \mathrm{QRep}(F, \varepsilon'; n) \subseteq \mathrm{QRep}(F, \varepsilon; n) \text{ and } \mathrm{QRep}(F, \varepsilon') \subseteq \mathrm{QRep}(F, \varepsilon)
\]
and restriction of the domain from $F'F'$ to $FF$ gives maps
\[
    \mathrm{QRep}(F', \varepsilon'; n) \to \mathrm{QRep}(F, \varepsilon'; n) \text{ and } \mathrm{QRep}(F', \varepsilon') \to \mathrm{QRep}(F, \varepsilon') \; .
\]
Composing them yields a map $r_{(F',\varepsilon'),(F,\varepsilon)} \colon \mathrm{QRep}(F', \varepsilon') \to \mathrm{QRep}(F, \varepsilon)$, which clearly respects the direct sum operation and stable equivalence, and thus induces maps $\mathrm{QR}^+(F',\varepsilon') \to \mathrm{QR}^+(F,\varepsilon) $ and  $\mathrm{QR}(F',\varepsilon') \to \mathrm{QR}(F,\varepsilon) $, which we again denote by $r_{(F',\varepsilon'),(F,\varepsilon)}$. 
    
    Using these as bonding maps, we define
    \[
        \mathrm{QR}(G)
        =\varprojlim_{(F,\varepsilon)\in\mathcal C_G}
          \mathrm{QR}(F,\varepsilon).
    \]
    In other words, an element in $ \mathrm{QR}(G)$ is a family $(x_{F,\varepsilon})_{(F,\varepsilon) \in \mathcal C_G} \in \prod_{(F,\varepsilon) \in \mathcal C_G} \mathrm{QR}(F,\varepsilon)$ satisfying
    $r_{(F',\varepsilon'),(F,\varepsilon)}(x_{F',\varepsilon'})
      =x_{F,\varepsilon}$ whenever
    $(F,\varepsilon)\preccurlyeq(F',\varepsilon')$.
\end{definition}

\begin{remark}
\label{prop:three-QR-pictures}\label{prop:QR-two-descriptions}
For any $(F,\varepsilon)\in\mathcal C_G$, the group
$\mathrm{QR}(F,\varepsilon)$ of \Cref{def:stable-equivalence} is canonically
the Grothendieck group of the abelian semigroup of homotopy classes of
$(F,\varepsilon)$-quasi-representations.  In particular, two
quasi-representations $q_0$ and $q_1$ determine the same element of
$\mathrm{QR}(F,\varepsilon)$ if and only if there is an
$(F,\varepsilon)$-quasi-representation $r$ such that
\[
 q_0\oplus r\overset{\mathrm h}{\sim}q_1\oplus r.
\]
Equivalently, equality in $\mathrm{QR}(F,\varepsilon)$ may be characterized by stable approximate unitary equivalence in the precise sense of \Cref{prop:homotopy-vs-unitary}\eqref{prop:homotopy-vs-unitary::unitary}.
\end{remark}

\section{The character: a geometric construction}
\label{sec:character_map}\label{sec:geometric-character}

In this section, we construct the character map for quasi-representations, an invariant taking values in equivariant topological $K$-theory. 
We shall present a construction using equivariant vector bundles. The only
analytic operations in this section are averaging, matrix norm estimates,
and spectral projection.

\begin{example}\label{ex:genuine-rep-character}
    To motivate our construction, let us consider the special case where we have a genuine representation $\pi: G \rightarrow U(n)$ and the group $G$ is torsion-free (so that each proper action is automatically free). In this case, for any fixed $G$-compact proper (right) $G$-space $X$, the character of $\pi$
is the equivariant $K$-theory class associated to the flat Hermitian $G$-vector bundle 
\begin{equation}\label{eq:flat_equivariant_bundle}
    E_\pi = (X \times \mathbb{C}^n, \pi^X)  \quad \text{defined as follows: }
\end{equation}
\begin{itemize}
    \item $X \times \mathbb{C}^n$ is the trivial Hermitian vector bundle over $X$ with the standard inner product on $\mathbb{C}^n$, and 
    \item $\pi^X \colon G \curvearrowright X \times \mathbb{C}^n$ is given by the prescription $\pi^X_g (x, v) := (x \cdot g^{-1}, \pi(g) v)$ for any $g\in G, x\in X$ and $v\in \mathbb{C}^n$.
\end{itemize}
For a free proper action, we have an isomorphism $K^0_G(X)\cong K^0(X/G)$ that takes the class of an equivariant vector bundle to that of its quotient bundle \cite[Proposition~2.1]{segal_1968}; see \cite{lueck_oliver_2001} for proper actions of discrete groups. 
The character is then the class of the
flat Hermitian vector bundle $E_\pi /G := X\times_G\mathbb{C}^n$ over $X/G$.  

When $X$ is a universal covering space with $G$ acting by deck transformation, the above construction in fact provides a one-to-one correspondence between unitary equivalence classes of finite-dimensional unitary representations and isomorphism classes of flat Hermitian vector bundles:  
a representation $\pi$ gives the quotient bundle $X\times_G\mathbb C^n$ with its flat connection, while 
given a flat Hermitian vector bundle over $X / G$, parallel transport around loops in $X/G$ based at a point gives a unitary monodromy representation. 
\end{example}

For a quasi-representation, \eqref{eq:flat_equivariant_bundle} does not define a genuine
$G$-action on the bundle.  
However, there is a conceptually simple adaptation of this construction, with the help of Fell's absorption principle (\Cref{subsec:Fell}).  
This method also works without the torsion-free assumption on $G$.

\begin{notation}\label{notation:equivariant-bundles}
    Since Fell's absorption principle involves the left regular representation $\lambda \colon G \curvearrowright \ell^2(G)$, which is typically infinite-dimensional, we shall need to work with a very special kind of $G$-equivariant Hilbert space bundles: namely, given a unitary representation $\rho \colon G \to U(H)$ on a possibly infinite-dimensional Hilbert space $H$, we note that \eqref{eq:flat_equivariant_bundle} still makes sense and defines a flat Hilbert space $G$-vector bundle $E_{\rho} = (X \times H, \rho^X)$.  

    The notation $\rho^X_g$ for bundle automorphisms also applies to a greater generality:
    any pair $(g, u) \in G \times U(H)$, where $H$ is a complex Hilbert space, gives rise to an automorphism of the trivial Hermitian vector bundle $X \times H$ by the formula $u_g^X(x ,\xi) = (x \cdot g^{-1}, u \cdot \xi)$. We also abbreviate $u^X_e$ to $u^X$.  
\end{notation}

In our construction, we shall use cutoff functions on $X$ (see \Cref{def:cutoff} and \Cref{lem:cutoff-exists}) twice: one time to embed, in a way that encodes a quasi-representation $\pi$ on $\mathbb{C}^n$, the trivial bundle $X \times \mathbb{C}^n$ into the (typically infinite-dimensional) Hilbert space bundle $E_{\lambda^{\oplus n}}$ as an ``approximately invariant subbundle'', and a second time to implement an averaging procedure to produce an invariant subbundle of dimension $n$, whose equivariant $K$-theory class is the character we are looking for.

\subsection{Step 1: From a quasi-representation to an approximately invariant subbundle}

Fix a cutoff function $\theta$ on $X$. 
It follows from \Cref{lem:cutoff-projection-equivariance} that $\theta$ gives rise to a $G$-equivariant isometric embedding
\[
    \iota_\theta \colon E_{\mathrm{id}_{\mathbb{C}}} \longrightarrow E_\lambda \; ,
\]
simply by mapping $(x, 1) \in X \times \mathbb{C}$ to $(x, \theta_x^{1/2})$ and extending linearly.

\begin{example}\label{ex:genuine-rep-Fell}
    As a guiding example, we explain how, in the classical setting of a genuine representation $\pi$ of a torsion-free group $G$ discussed in \Cref{ex:genuine-rep-character}, one can embed $E_\pi$ as a $G$-invariant subbundle of $E_{\lambda \otimes \mathrm{id}_{_{\mathbb{C}^n}}} \cong E_{\lambda^{\oplus n}}$. 
    Indeed, there is a $G$-equivariant isometric embedding
\[
    \iota_{\theta, \pi} \colon E_\pi \hookrightarrow E_{\lambda \otimes \mathrm{id}_{_{\mathbb{C}^n}}}
\]
obtained by composing the $G$-equivariant isometric embedding $$\iota_\theta \otimes \mathrm{id}_{E_\pi}  \colon E_\pi \hookrightarrow E_{\lambda \otimes \pi}$$ with the $G$-equivariant Hilbert space bundle isomorphism $$(V_\pi^X)^{-1} \colon E_{\lambda \otimes \pi} \xrightarrow{\cong} E_{\lambda \otimes \mathrm{id}_{\mathbb{C}^n}}$$ induced as in \Cref{notation:equivariant-bundles} from the unitary $V_\pi$ from \eqref{eq:Vpi-prelim}. 
We compute that for any $(x,v) \in E_\pi$, we have 
\begin{align}\label{eq:iota_embedding}
    \iota_{\theta, \pi} (x, v) = (x, I_{\theta, \pi} (x) \, v) \quad \text{ with } I_{\theta, \pi} (x) \colon \mathbb{C}^n &\longrightarrow \ell^2(G) \otimes \mathbb{C}^n \, , \\ v &\longmapsto    \sum_{g \in G} \theta(x \cdot g)^{\tfrac{1}{2}} \ \delta_g \otimes \pi(g)^* v \notag
\end{align}
where the sum is actually a finite one for $\theta$ is compactly supported. Hence the character $[E_\pi]$ of the representation $\pi$ given in \Cref{ex:genuine-rep-character} equals the equivariant $K$-theory class of the image of $\iota_{\theta, \pi}$ as a $G$-invariant subbundle of $E_{\lambda \otimes \mathrm{id}_{\mathbb{C}^n}}$. 

The advantage of applying Fell's absorption principle is that now $\pi$ appears not in the construction of the background $G$-bundle $E_{\lambda \otimes \mathrm{id}_{\mathbb{C}^n}}$, but only in the embedding. 
\end{example}

For an arbitrary map $\pi:G\to U(n)$, formula~\eqref{eq:iota_embedding}
still defines an isometric vector bundle embedding
\[
 \iota_{\theta,\pi}:X\times\mathbb C^n
 \longrightarrow E_{\lambda\otimes\mathrm{id}_{\mathbb C^n}},
\]
because $V_\pi$ remains unitary even when $\pi$ is not multiplicative.  
The following observation, which uses a notation from \eqref{eq:finite-overlap}, follows immediately from the formula defining $I_{\theta, \pi} $ in \eqref{eq:iota_embedding}. 

\begin{lemma}
 \label{lemma:projection-dependence}
    For each $x \in X$, the images $I_{\theta, \pi} (x)$ and $p_{\theta,\pi}(x)$ only depend on $\theta$ and $\pi|_{E(\{x\}, \operatorname{supp} \theta)} \colon E(\{x\}, \operatorname{supp} \theta) \to U(n)$. 
\end{lemma}
Hence in the following, we may slightly abuse notation and use functions $\pi \colon F \to U(n)$ for some $F \subseteq G$ to define functions $I_{\theta, \pi}$ and $p_{\theta,\pi}$ on some $K \subseteq X$, provided that $E(K, \operatorname{supp} \theta) \subseteq F$.

A key observation for us is that when $\pi$ is a quasi-representation, the image of $\iota_{\theta, \pi}$ is an ``approximately $G$-invariant subbundle'' in a sense we make precise below. 
Note that the image of $\iota_{\theta, \pi}$, as a subbundle of $E_{\lambda^{\oplus n}}$, can be characterized by the continuous family 
\begin{equation}\label{eq:cutoff-projection-map}
 p_{\theta,\pi}:X\longrightarrow
\operatorname{Proj}_{\mathrm{fin}}(\ell^2(G)^{\oplus n}) \,  , \quad x \longmapsto I_{\theta,\pi}(x)I_{\theta,\pi}(x)^* \; ,
\end{equation}
of rank-$n$ orthogonal projections.  

\begin{lemma} \label{lemma:approx-invariant}
Let $K\subseteq X$ be compact and let $S\subseteq G$ be finite. 
Then for any $\varepsilon > 0$ and any $\pi \in \mathrm{QRep}(F, \varepsilon)$, where $F := S\cup S^{-1} E(K, \operatorname{supp}\theta)$ as in \eqref{eq:finite-overlap}, 
we have 
\[
 \big\|(\lambda_g\otimes1)I_{\theta,\pi}(x)
       -I_{\theta,\pi}(xg^{-1})\pi(g)\big\|<\varepsilon \quad \text{ for any } g\in S \text{ and } x\in K \; ;
\]
in particular, 
\[
 \big\|\Ad_{\lambda_g\otimes1}\big(p_{\theta,\pi}(x)\big)
       -p_{\theta,\pi}(xg^{-1})\big\|<2\varepsilon \quad \text{ for any } g\in S \text{ and } x\in K \; .
\]
\end{lemma}

\begin{proof}
To prove the first inequality, we fix a unit vector $\xi \in \mathbb{C}^n$,  
where $n \in \mathbb{N}$ is such that $\pi \in \mathrm{QRep}(F, \varepsilon; n)$. 
Plugging in \eqref{eq:iota_embedding} and re-indexing the first sum with the bijection $k \mapsto g^{-1}k$, 
we see that 
\begin{align*}
    &\big((\lambda_g\otimes1)I_{\theta,\pi}(x)
       -I_{\theta,\pi}(xg^{-1})\pi(g)\big) \,\xi \\
       &= \left( \sum_{k\in G}\theta(xg^{-1}k)^{1/2}\delta_k\otimes
       \pi(g^{-1}k)^*\xi \right) - \left( \sum_{k\in G}\theta(xg^{-1}k)^{1/2}\delta_k\otimes
       \pi(k)^*\pi(g)\xi \right) \\
       &= \sum_{k\in G}\theta(xg^{-1}k)^{1/2}\delta_k\otimes
\left(  \pi(g) \pi(g^{-1}k) - \pi(k) \right)^* \pi(g)\xi \; .       
\end{align*}
Using the identity $\| \xi' \|^2 = \langle \xi', \xi' \rangle$, we conclude that 
\begin{align*}
    &\left\| \big((\lambda_g\otimes1)I_{\theta,\pi}(x)
       -I_{\theta,\pi}(xg^{-1})\pi(g)\big) \,\xi  \right\|^2 \\
       &= \sum_{k\in G}\theta(xg^{-1}k)
       \left\| \left(  \pi(g) \pi(g^{-1}k) - \pi(k) \right)^* \pi(g)\xi \right\|^2 \\
       &\leq \max_{\substack{k \in G \\ \theta(xg^{-1}k) \not= 0}} \left\|   \pi(g) \pi(g^{-1}k) - \pi(k)  \right\|^2 \; ,
\end{align*}
which is less than $\varepsilon^2$, since  $\pi \in \mathrm{QRep}(F, \varepsilon; n)$ while the condition $\theta(xg^{-1}k) \not= 0$ implies $g^{-1}k \in E(K, \operatorname{supp}\theta)$ and thus $g, g^{-1}k \in F$. This proves the first inequality. 

Finally, observe that since $\Ad_{\lambda_g\otimes1}\big(p_{\theta,\pi}(x)\big) = \big((\lambda_g\otimes1)I_{\theta,\pi}(x)\big) \big((\lambda_g\otimes1)I_{\theta,\pi}(x)\big)^*$, the second inequality follows from the first one. 
\end{proof}

\newcommand{\thetap}{\theta} 
\newcommand{\thetatheta}{\theta} 

\subsection{Step 2: Averaging to an invariant subbundle}

We shall replace the approximately equivariant projection-valued map
$p_{\theta,\pi}$ by a genuinely equivariant one.  We do this by averaging and then applying the ``spectral correction'' from \Cref{sec:almost-idempotent}.
More precisely, define a family $a_{\thetatheta,\pi}$ of operators on
$\ell^2(G)\otimes\mathbb C^n \cong \ell^2(G)^{\oplus n}$ by the following averaging procedure:
\begin{equation}\label{eq:averaging}
    a_{\thetatheta, \pi}(x) = \sum_{g \in G} \thetap(x \cdot g) \mathrm{Ad}_{\lambda_g \otimes \mathrm{id}_{\mathbb{C}^n}}( p_{\theta, \pi}(x \cdot g) ) \; .
\end{equation}
Again thanks to properness, the sum is a finite one, whence each $a_{\thetatheta,\pi} (x)$ is a finite-rank, (semi-definite) positive operator; 
it also follows from \Cref{lemma:projection-dependence} that the map $a_{\thetatheta, \pi}$ only depends on $\thetatheta$ and a function $\pi \colon E(\operatorname{supp} \thetap ) \to U(n)$.

\begin{proposition}\label{prop:average-vector-bundle}
Write $F = E(\supp \thetap)^{(2)}$ as in \eqref{eq:set-product}. 
Then for any $\pi \in \mathrm{QRep}(F, \frac{1}{4})$, 
the map $a_{\thetatheta,\pi}$ defined in \eqref{eq:averaging} is continuous in operator norm and takes value in $\operatorname{NProj}_{\mathrm{fin}}(\ell^2(G)^{\oplus n})$ as defined in \Cref{sec:almost-idempotent}, and is equivariant in the sense that 
\[
    \Ad_{\lambda_h^{\oplus n}}(a_{\thetatheta,\pi}(xh))
=a_{\thetatheta,\pi}(x) \quad \text{ for any } x\in X \text{ and } h\in G \; .
\]
In this case, composing with $\kappa$ from \eqref{eq:almost-idempotent-correction} yields a continuous map 
\[
    \kappa \circ a_{\thetatheta,\pi} \colon X \longrightarrow \operatorname{Proj}_{\mathrm{fin}}(\ell^2(G)^{\oplus n})
\]
which is also equivariant in the sense above and takes value in rank-$n$ projections; in particular, the ranges of these projections form a $G$-invariant $n$-dimensional Hermitian vector subbundle of $E_{\lambda\otimes\mathrm{id}_{\mathbb C^n}}$. 
\end{proposition}

\begin{proof}
Thanks to the properness of the action and the compactness of $\supp\thetap$, in each precompact neighborhood around an arbitrary $x \in X$, the sum in \eqref{eq:averaging} is a finite one, whence $a_{\thetatheta,\pi}$ is continuous and takes value in $B_{\mathrm{fin,sa}} (\ell^2(G)^{\oplus n})$. 
To check equivariance, for each $h\in G$, we compute 
\begin{align*}
 \Ad_{\lambda_h\otimes1}\big(a_{\thetatheta,\pi}(xh)\big)
 &=\sum_{g\in G}\thetap(xhg)\,
   \Ad_{\lambda_{hg}\otimes1}\big(p_{\theta,\pi}(xhg)\big)\\
 &=\sum_{k\in G}\thetap(xk)\,
   \Ad_{\lambda_k\otimes1}\big(p_{\theta,\pi}(xk)\big)
   =a_{\thetatheta,\pi}(x),
\end{align*}
where in the second line we changed variables $k=hg$.

Now we check that $a_{\thetatheta,\pi}(x) \in \operatorname{NProj}_{\mathrm{fin}}(\ell^2(G)^{\oplus n})$ for any $x \in X$. By the equivariance of $a$ and the unitary invariance of the set $\operatorname{NProj}_{\mathrm{fin}}(\ell^2(G)^{\oplus n})$, it suffices to consider $x \in \operatorname{supp}\thetap$. 
Note that if we write $K= \operatorname{supp}\thetap$ and  $S=E(K)$, then $F = S\cup S^{-1} E(K, \supp \theta)$. 
For any $x \in K$ and $g \in G$, if
$\thetap(xg)\ne0$, then we have $xg \in K$ and $g \in S$, whence it follows from  \Cref{lemma:approx-invariant} that 
\[
 \|\Ad_{\lambda_g\otimes1}p_{\theta,\pi}(xg)
       -p_{\theta,\pi}(x)\|< 2 \cdot \frac{1}{4} = \frac{1}{2}
\]
Since $\sum_{g \in G}\thetap(xg) = 1$, it follows that 
\[
 \|a(x)-p_{\theta,\pi}(x)\|
 \leq\sum_{g \in S}\thetap(xg)
 \|\Ad_{\lambda_g\otimes1}p_{\theta,\pi}(xg)-p_{\theta,\pi}(x)\|
 < \frac{1}{2}
 \quad \text{ for any } x \in K \; .
\]
This verifies $a_{\thetatheta,\pi}(x) \in \operatorname{NProj}_{\mathrm{fin}}(\ell^2(G)^{\oplus n})$. 

Finally, the facts that $\kappa \circ a_{\thetatheta,\pi}$ is continuous and equivariant follow directly from \Cref{lem:almost-idempotent-perturb}\eqref{lem:almost-idempotent-perturb:continuous} and~\eqref{lem:almost-idempotent-perturb:unitary}.  
\end{proof}

\subsection{The character map as group homomorphisms}

We are now ready to give a provisional definition of the character for a quasi-representation. 

\begin{definition}
    The \emph{character} of an $(F,\varepsilon)$-quasi-representation $\pi$ is the class $[\kappa \circ a_{\thetatheta, \pi}]$ in the equivariant $K$-theory $K^0_G(X)$,  
    as constructed in \Cref{prop:average-vector-bundle}. 
\end{definition}

It remains to verify the definition does not depend on the choices we made along the way. Meanwhile, we shall see that it factors through the Grothendieck construction in \Cref{def:stable-equivalence}. 

\begin{proposition}\label{prop:character-well-defined}
Let $X$ be a cocompact proper $G$-space and let $K \subseteq X$ be a compact set satisfying $KG = X$. Write $F_K := E(K)^{(2)}$ as in \eqref{eq:set-product}. Then there is a cutoff function $\thetatheta$ on $X$ such that the assignment in \Cref{prop:average-vector-bundle} defines a homomorphism
\[
   \chi_X \colon \mathrm{QR}(F_K,\frac{1}{4})\longrightarrow K_G^0(X).
\]

Moreover, suppose $X'$ is another $G$-compact proper $G$-space and $\varphi:X\to X'$ is $G$-equivariant. Let $K' \subseteq X'$ be a compact subset satisfying $K'G = X'$ and write $F_{K'} = E(K')^{(2)}$. Then there is a finite subset $F_{K, K'} \subseteq G$ containing $F_K \cup F_{K'}$ such that for any finite subset $F \subseteq G$ containing $F_{K, K'}$ and any $\varepsilon \in (0, \frac{1}{4})$, the diagram
\[
\begin{tikzcd}
 {\mathrm{QR}(F,\varepsilon)} && {\mathrm{QR}(F_{K'},\frac{1}{4})} & {K_G^0(X')} \\
 && {\mathrm{QR}(F_K,\frac{1}{4})} & {K_G^0(X)}
 \arrow["\chi_{X'}", from=1-3, to=1-4]
 \arrow["r_{(F, \varepsilon),(F_{K'}, 1/4)}", from=1-1, to=1-3]
 \arrow["r_{(F, \varepsilon),(F_{K}, 1/4)}"', from=1-1, to=2-3]
 \arrow["{\varphi^*}", from=1-4, to=2-4]
 \arrow["\chi_{X}", from=2-3, to=2-4]
\end{tikzcd}
\]
commutes. 
Hence these homomorphisms induce
\[
   \chi:\mathrm{QR}(G)\longrightarrow K_G^0(\underline EG).
\]

\end{proposition}

\begin{proof}
It follows from \Cref{lem:finite-overlap-fatten} and \Cref{lem:cutoff-exists} that $X$ admits a cutoff functions $\theta$ such that $E(K \cup \supp \theta) = E(K)$. Hence \Cref{prop:average-vector-bundle} provides a map 
\begin{equation}\label{prop:character-well-defined:eq:Qrep-map}
    \mathrm{QRep}(F_K, \frac{1}{4}) \longrightarrow K^0_G(X) \, , \quad \pi \longmapsto [\kappa \circ a_{\thetatheta, \pi}] \; .
\end{equation}
A direct sum $\pi \oplus \pi'$ of quasi-representations produces a direct sum of the isometries $I_{\theta,\pi \oplus \pi'} = I_{\theta,\pi} \oplus I_{\theta,\pi'}$ as in \eqref{eq:iota_embedding}, whence $p_{\theta,\pi \oplus \pi'} = p_{\theta,\pi} \oplus p_{\theta,\pi'}$, $a_{\theta,\pi \oplus \pi'} = a_{\theta,\pi} \oplus a_{\theta,\pi'}$, and $[\kappa \circ a_{\theta,\pi \oplus \pi'}] = [\kappa \circ  a_{\theta,\pi}] + [\kappa \circ a_{\theta,\pi'}]$. 
Therefore the map above is additive. 

If $(\pi_t)_{t\in[0,1]}$ is a homotopy in $\mathrm{QRep}(F_K, \frac{1}{4})$, then by the local finiteness of the sum in \eqref{eq:iota_embedding}, $(I_{\theta,\pi_t})_{t\in[0,1]}$ is a continuous path in $C(\supp \theta , B(\mathbb{C}^n , \ell(G)^{\oplus n}))$, whence $(p_{\theta,\pi_t})_{t\in[0,1]}$ is a continuous path in $C(\supp \theta , \operatorname{Proj}_{\mathrm{fin}}(\ell(G)^{\oplus n}))$. It then follows from the local finiteness of the sum in \eqref{eq:averaging} as well as the equivariance of each $a_{\thetatheta, \pi_t}$ that  $(a_{\theta,\pi_t})_{t\in[0,1]}$ is a continuous path in $C(X , \operatorname{NProj}_{\mathrm{fin}}(\ell(G)^{\oplus n}))^G$. Hence by \Cref{lem:almost-idempotent-perturb}\eqref{lem:almost-idempotent-perturb:continuous},  $(p_{\theta,\pi_t})_{t\in[0,1]}$ is a continuous path in $C(X , \operatorname{Proj}_{\mathrm{fin}}(\ell(G)^{\oplus n}))^G$, whence $[\kappa \circ  a_{\theta,\pi_0}] = [\kappa \circ a_{\theta,\pi_1}]$. Consequently, the map in \eqref{prop:character-well-defined:eq:Qrep-map} descends to a homomorphism from $\mathrm{QR}^+(F_K,\frac{1}{4})$ and then to a homomorphism from its Grothendieck group $\mathrm{QR}(F_K,\frac{1}{4})$.

We claim that given any two cutoff functions $\theta_0,\theta_1$ on $X$ such that $E(K \cup {\supp \theta_0} \cup {\supp \theta_1}) = E(K)$, we have $[\kappa \circ  a_{\theta_0,\pi}] = [\kappa \circ a_{\theta_1,\pi}]$ for any $\pi \in \mathrm{QRep}(F_K, \frac{1}{4}) $. Indeed, we have a continuous path $(\theta_t)_{t \in [0,1]}$ of cutoff functions defined by setting $\theta_t = (1-t)\theta_0+t\theta_1$ for each $t \in [0,1]$. Clearly $\supp \theta_t \subseteq {\supp \theta_0} \cup {\supp \theta_1}$ for any $t \in [0,1]$. It follows that $(I_{\theta_t,\pi})_{t\in[0,1]}$ is a continuous path in $C({\supp \theta_0} \cup {\supp \theta_1} , B(\mathbb{C}^n , \ell(G)^{\oplus n}))$. The claim is then proved by the same argument as in the last paragraph. 

Now let $X'$, $K'$, and $\varphi$ be as in the statement.
The map $\varphi$ is proper: for a compact $L\subset X'$, every
$x\in\varphi^{-1}(L)$ can be written $x=kg$ with $k\in K$, and
$g\in E(\varphi(K),L)$, a finite set. Thus $\varphi^{-1}(L)$ is
a closed subset of a finite union of compact translates of $K$.

Choose cutoffs $\theta,\theta'$ defining $\chi_X,\chi_{X'}$ as
above, and set
\[
 L=K\cup\supp\theta\cup\varphi^{-1}(K'\cup\supp\theta'),
 \qquad F_{K,K'}=F_K\cup F_{K'}\cup E(L)^{(2)}.
\]
This is a finite set containing both original controls. The pullback
$\theta'\circ\varphi$ is a cutoff on $X$, and
$\theta_t=(1-t)\theta+t(\theta'\circ\varphi)$ is a path of cutoffs
supported in $L$. Every index needed for the cutoff columns,
averaging, and spectral correction lies in the stated common
control. Applying the preceding homotopy argument at this control
therefore identifies the characters obtained from its two endpoints.
The construction at $t=1$ is exactly the pullback of the construction
on $X'$, so
\[
 \varphi^*\chi_{X'}\bigl(r_{(F,\varepsilon),(F_{K'},1/4)}[\pi]\bigr)
 =\chi_X\bigl(r_{(F,\varepsilon),(F_K,1/4)}[\pi]\bigr).
\]
This proves the diagram's commutativity without requiring $\varphi$
to be surjective.
\end{proof}

In the case of Rips complexes, we may make the parameters in \Cref{prop:character-well-defined} more transparent. 

\begin{corollary}\label{cor:character-well-defined-Rips}
The assignment in \Cref{prop:average-vector-bundle} defines a family $(\chi_{P_d(G)})_{d \geq 0}$ of homomorphisms 
\[
   \chi_{P_d(G)} \colon \mathrm{QR}(B(e, 4d),\frac{1}{4})\longrightarrow K_G^0(P_d(G))
\]
that fit into the following commuting diagram whenever $d' \geq d \geq 0$: 
\[
\begin{tikzcd}
 {\mathrm{QR}(B(e, 4d'),\frac{1}{4})} & {K_G^0(P_{d'}(G))} \\
 {\mathrm{QR}(B(e, 4d),\frac{1}{4})} & {K_G^0(P_{d}(G))}
 \arrow["\chi_{P_{d'}(G)}", from=1-1, to=1-2]
 \arrow["r_{(B(e, 4d'), 1/4),(B(e, 4d), 1/4)}"', from=1-1, to=2-1]
 \arrow["{\iota^*}", from=1-2, to=2-2]
 \arrow["\chi_{P_{d}(G)}", from=2-1, to=2-2]
\end{tikzcd}
\]
where $\iota \colon P_{d}(G) \to P_{d'}(G)$ is the canonical embedding. 
\end{corollary}

\begin{proof}
Use the barycentric cutoff $\theta_d(y)=y(e)$ on $P_d(G)$.
Its support is the closed star of $e$, and \eqref{eq:Rips-complexity}
gives the stated control $B(e,4d)$. These cutoffs agree under the
Rips inclusions. The cutoff columns, averaging, and spectral
projections therefore commute with restriction at the larger control
$B(e,4d')$, giving the displayed diagram.
\end{proof}

\section{The rational character and the obstruction to genuine representations}
\label{sec:Chern-character}

In this section, we construct a cohomological version of the character
map by composing the character of Section~4 with the equivariant Chern
character. The resulting Chern character of a quasi-representation
takes values in delocalized cohomology with complex coefficients.
A genuine finite-dimensional unitary representation has no
positive-degree Chern character, so a nonzero positive-degree
component obstructs stable equivalence with a genuine representation.

For finitely generated torsion-free groups satisfying the quantitative
character theorem, we also prove a converse: if all positive-degree
components vanish, then, at each prescribed finite set and error
tolerance, some positive integer multiple of the quasi-representation
class is stably equivalent to a difference of genuine
finite-dimensional unitary representations.

We first recall Baum and Connes' Chern character for proper actions.
Let $X$ be a finite proper $G$-CW complex, meaning that it has
finitely many equivariant cells and finite cell stabilizers. For the
formulas in this subsection write the associated left action as
$gx=xg^{-1}$. Define its \emph{inertia space}
\[
 \widehat X=\{(g,x)\in G\times X:gx=x\}.
\]
If $gx=x$, then $g$ belongs to the stabilizer $G_x$, which is finite because the
action is proper; hence $g$ has finite order.  The group $G$ acts on
$\widehat X$ by
\[
 h\cdot(g,x)=(hgh^{-1},hx)\quad \text{ for any } h \in G \text{ and } (g,x) \in \widehat X.
\]
This action is again proper.  For the precise statement in the presence of torsion it is simplest first to
use complex coefficients.  Define
\[
 H^0(X,G;\mathbb C)=\bigoplus_{j\ge0}H^{2j}(\widehat X/G;\mathbb C) \quad \text{ and } \quad 
 H^1(X,G;\mathbb C)=\bigoplus_{j\ge0}H^{2j+1}(\widehat X/G;\mathbb C).
\]
Equivalently, if $[g]$ ranges over conjugacy classes of finite-order elements,
this is the direct sum of the cohomologies of
$X^g/C_G(g)$.

To keep the notation light, below $H^*(X,G)$ denotes this delocalized
cohomology with complex coefficient.  Its
cohomological degree-zero summand is
\[
 H^0_{\mathrm{deg}=0}(X,G)=H^0(\widehat X/G)
 \subseteq H^0(X,G).
\]

\begin{theorem}[{\cite{baum_connes_chern_1988}}]\label{thm:BC-chern}
For a finite proper $G$-CW complex $X$, the equivariant Chern
character is a natural isomorphism
\[
 \mathrm{Ch}:K_G^*(X)\otimes\mathbb C
 \xrightarrow{\cong} H^*(X,G;\mathbb C).
\]
It respects the $\mathbb Z/2$ parity, and its components are graded by ordinary
cohomological degree. For free actions it is the ordinary rational
Chern character after extension of scalars to $\mathbb C$. With finite
stabilizers, the sector traces involve roots of unity and take values
in cyclotomic extensions of $\mathbb Q$.
\end{theorem}

After barycentric subdivision, each $P_d(G)$ is a finite proper
$G$-CW complex, so the theorem applies at every Rips scale used below.
For $x\in\mathrm{QR}(G)$, vanishing of its cohomological character
means $\mathrm{Ch}(\chi_d(x))=0$ at every scale $d$; complexification
is performed at those finite stages.

The following theorem provides a computable obstruction for a quasi-representation to come, stably, from a genuine representation.

\begin{theorem}\label{thm:genuine-character-degree-zero}
Let $d>0$, and let $(F_d,\varepsilon_d)$ be a control supplied by
\Cref{prop:character-well-defined} for $X=P_d(G)$.  Let
$F\supseteq F_d$ and $0<\varepsilon\leq\varepsilon_d$.  Suppose an
$(F,\varepsilon)$-quasi-representation $\phi$ and the restriction of a
genuine finite-dimensional unitary representation
$\psi:G\to U(n)$ determine the same class in
$\mathrm{QR}(F,\varepsilon)$.  Then
\[
 \mathrm{Ch}(\chi(\phi))\in
 H^0_{\mathrm{deg}=0}(P_d(G),G)
 =H^0(\widehat{P_d(G)}/G;\mathbb C).
\]
Equivalently, all positive cohomological-degree components of
$\mathrm{Ch}(\chi(\phi))$ vanish.
\end{theorem}

\begin{proof}
By homotopy and stable invariance of the character map, it is enough to treat
the genuine representation $\psi$.  The class $\chi(\psi)$ is represented by
the equivariant flat bundle
\[
 E_\psi=P_d(G)\times\mathbb C^n,
 \qquad g(x,v)=(gx,\psi(g)v),
\]
as proved in \Cref{sec:character_map}.

We recall the local form of the delocalized equivariant Chern character.  For a
finite-order element $g$, restrict $E_\psi$ to the fixed-point set
$P_d(G)^g$.  The action of $g$ on the fiber is the constant unitary
$\psi(g)$.  Decompose the fiber into its $g$-eigenspaces.  On each eigenspace
the underlying connection is the trivial flat connection, so its curvature is
zero.  Consequently the equivariant Chern--Weil representative on that sector is
\[
 \operatorname{Tr}\!\left(\psi(g)\exp\left(\frac{i}{2\pi}F_\nabla\right)\right)
 =\operatorname{Tr}(\psi(g)),
\]
a locally constant function on $P_d(G)^g$.  There are no terms of positive
cohomological degree.  These sectorwise functions are invariant under
conjugation and therefore descend to $\widehat{P_d(G)}/G$.  This proves that
the full delocalized Chern character lies in cohomological degree zero.

This also explains the point at which flatness is used: the Chern character is
not zero; its degree-zero part is the rank on the identity sector and, on a
nontrivial finite-order sector, the trace $\operatorname{Tr}\psi(g)$.  Only the
positive-degree terms vanish.
\end{proof}

\begin{corollary}[Torsion-free case]
\label{cor:torsion-free-chern} Let $G$ be a finitely generated torsion-free group, equipped with a word
metric.  Suppose that the character is a complete
invariant of quasi-representations in the quantitative sense  of
Theorem~1.2. Let the Rips scale $d\geq 1$ and $x\in QR(G)$. Then the following are equivalent:

\begin{enumerate}
\item all positive-degree components of $\operatorname{Ch}(\chi(x))$
vanish;

\item for every finite set $F_0\subseteq G$ and every $\varepsilon_0>0$,
there is an integer $N\geq 1$ such that, at the control
$(F_0,\varepsilon_0)$, $Nx$ is stably equivalent to a virtual genuine
finite-dimensional unitary representation.
\end{enumerate}

Equivalently, in (2) there are genuine finite-dimensional unitary
representations $\pi_+$ and $\pi_-$ such that
\[
Nx=[\pi_+]-[\pi_-]
\]
in $QR(F_0,\varepsilon_0)$.

In particular, this conclusion holds for torsion-free amenable groups.
\end{corollary}

\begin{proof} Note that the Rips complex $P_d(G)$ is connected under the assumption that $d\geq 1$ and $G$ is given the word metric.

Suppose first that (2) holds. Fix a sufficiently large Rips scale $d$.
Choose a control $(F_0,\varepsilon_0)$ strong enough that the character
at scale $d$ is defined. By assumption, there are an integer $N\geq 1$
and genuine finite-dimensional unitary representations $\pi_+$ and
$\pi_-$ such that
\[
Nx=[\pi_+]-[\pi_-]
\]
at this control. Applying the character and then the Chern character
gives
\[
N\operatorname{Ch}(\chi_d(x))
 =
\operatorname{Ch}(\chi_d([\pi_+]))
 -
\operatorname{Ch}(\chi_d([\pi_-])).
\]
By Theorem~5.2, the two terms on the right have no components of
positive cohomological degree. Since $N\neq 0$, the same is true of
$\operatorname{Ch}(\chi_d(x))$. As this holds at every sufficiently
large Rips scale, (1) follows.

Conversely, assume that all positive-degree components of
$\operatorname{Ch}(\chi(x))$ vanish. Since $G$ is torsion-free, its
action on every Rips complex $P_d(G)$ is free. Hence the delocalized
equivariant Chern character reduces to the ordinary Chern character
\[
\operatorname{Ch}:
K_G^0(P_d(G))\otimes\mathbb C
\longrightarrow
\bigoplus_{j\geq 0}H^{2j}(P_d(G)/G;\mathbb C).
\]
For all $d \geq 1$, the Rips complex $P_d(G)$ is connected,
and therefore $P_d(G)/G$ is connected. It follows that
\[
H^0(P_d(G)/G;\mathbb C)\cong\mathbb C.
\]

Let $r$ be the virtual rank of $\chi_d(x)$. The trivial one-dimensional
representation $1_G$ determines the trivial line bundle, so
\[
\operatorname{Ch}(\chi_d([1_G]))=1.
\]
Since all positive-degree components of $\operatorname{Ch}(\chi_d(x))$
vanish, we therefore have
\[
\operatorname{Ch}(\chi_d(x))
 =
r\,\operatorname{Ch}(\chi_d([1_G])).
\]
The Chern character is an isomorphism after complexification, and hence
\[
\chi_d(x-r[1_G])=0
\]
in $K_G^0(P_d(G))\otimes\mathbb C$. Thus
$\chi_d(x-r[1_G])$ is a torsion element of $K_G^0(P_d(G))$.

Now fix a finite set $F_0\subseteq G$ and $\varepsilon_0>0$. By the
quantitative completeness of the character, choose a sufficiently large
Rips scale $d$ and a stronger control $(F_1,\varepsilon_1)$ such that
equality of character classes at scale $d$ implies equality after
restriction to $(F_0,\varepsilon_0)$. At this fixed scale, choose
$N\geq 1$ annihilating the torsion class
$\chi_d(x-r[1_G])$. Then
\[
N\chi_d(x)=Nr\,\chi_d([1_G])
\]
in $K_G^0(P_d(G))$.

The quantitative property of the character now implies
\[
Nx=Nr[1_G]
\]
after restriction to $(F_0,\varepsilon_0)$. If $r\geq 0$, the right-hand
side is the class of the genuine representation $1_G^{\oplus Nr}$; if
$r<0$, it is the negative of the class of
$1_G^{\oplus(-Nr)}$. Thus in either case $Nx$ is the difference of
classes of genuine finite-dimensional unitary representations.

Finally, equality in $QR(F_0,\varepsilon_0)$ is precisely stable
equivalence: after adding a common $(F_0,\varepsilon_0)$-quasi-representation,
the two sides are homotopic through controlled quasi-representations.
This proves (2).

For amenable groups, the required completeness and quantitative
property of the character are supplied by Theorems~1.1 and~1.2.
\end{proof}

\section{Analytic tools for the proof}
\label{sec:analytic-tools}

In this section, we collect the analytic tools used to prove that the character is a complete invariant for amenable groups and for a class of non-amenable  groups including surface groups. We recall only the parts of the theory needed in the proof: Hilbert modules, equivariant Kasparov cycles, and the Kasparov product, which provides the compositions needed to construct an inverse to the character.

\subsection{Hilbert modules and equivariant Kasparov cycles}

Let $A$ be a $C^*$-algebra and let $\mathcal E$ be a countably generated right Hilbert $A$-module.  We write
$\cL_A(\mathcal E)$ for the $C^*$-algebra of adjointable $A$-linear operators and
$\cK_A(\mathcal E)$ for the closed linear span of the rank-one operators
$\theta_{\xi,\eta}(\zeta)=\xi\langle\eta,\zeta\rangle$.
If the coefficient algebra is clear, we abbreviate these to
$\cL(\mathcal E)$ and $\cK(\mathcal E)$.  A Hilbert module is \emph{graded} if
$\mathcal E=\mathcal E^+\oplus\mathcal E^-$; an operator is even or odd according to whether it preserves or reverses the grading.

Let $(A,\alpha)$ be a $G$-$C^*$-algebra.  A \emph{$G$-Hilbert $A$-module} is a Hilbert $A$-module $\mathcal E$ together with a map
$g\mapsto U_g$ from $G$ to the invertible isometries of $\mathcal E$ satisfying
\begin{equation} \label{eq:alpha_g-unitary}
  U_g(\xi a)=U_g(\xi)\alpha_g(a),
  \qquad
  \langle U_g\xi,U_g\eta\rangle
      =\alpha_g(\langle\xi,\eta\rangle),
\end{equation}
and $U_gU_h=U_{gh}$.  
We call a linear surjective isometry $u:\mathcal E\to\mathcal E$ satisfying \eqref{eq:alpha_g-unitary} an \emph{$\alpha_g$-unitary}. Note that the inverse of an  $\alpha_{g}$-unitary is an  $\alpha_{g^{-1}}$-unitary, and the product of an
$\alpha_g$-unitary and an $\alpha_h$-unitary is an
$\alpha_{gh}$-unitary.  
Conjugation by $U_g$ takes $A$-linear adjointable operators to $A$-linear adjointable operators.

We recall the building blocks of equivariant $KK$-theory.  An even Kasparov $G$-cycle from $A$ to $B$ is a triple
$(\mathcal E,\varphi,T)$ consisting of a graded countably generated $G$-Hilbert $B$-module, a graded equivariant $*$-homomorphism
$\varphi:A\to\cL_B(\mathcal E)$, and an odd operator $T\in\cL_B(\mathcal E)$ such that for all $a\in A$ and $g\in G$,
\[
 (T-T^*)\varphi(a),\quad
 (T^2-1)\varphi(a),\quad
 [T,\varphi(a)],\quad
 (U_gTU_g^{-1}-T)\varphi(a)
 \in\cK_B(\mathcal E).
\]
Homotopy classes modulo addition of degenerate cycles form $KK^G(A,B)$.

We recall stabilization and the Kasparov product, with the group-action
convention in the stabilization statement made explicit; see
\cite{Kas80,kasparov_1988}. For background on these constructions
and their role in higher index theory, see \cite{willett_higher_2020}.

\begin{theorem}[stabilization of Hilbert modules]
\label{thm:equivariant-stabilization}
Let $B$ be a separable $C^*$-algebra and let $\mathcal E$ be a
countably generated Hilbert $B$-module. For
$\mathcal H_B=\ell^2(\mathbb N)\otimes B$, there is a $B$-linear
unitary
\[
 \mathcal E\oplus\mathcal H_B\cong\mathcal H_B.
\]
If $B$ and $\mathcal E$ carry compatible $G$-actions, equip the
left-hand side with the direct-sum action and transport that action
through the unitary. The resulting action on the standard underlying
module may depend on $\mathcal E$. Thus equivariant cycles can be
stabilized on standard underlying modules with their transported
actions.
\end{theorem}

\begin{theorem}[Kasparov product and connection theorem]\label{thm:kasparov-product}
Let $A,B,C$ be separable $G$-$C^*$-algebras.  There is a bilinear associative product
\[
  KK^G(A,B)\times KK^G(B,C)\longrightarrow KK^G(A,C),
  \qquad (x,y)\longmapsto x\otimes_B y,
\]
with identity $[\mathrm{id}_A]$.  If cycles for $x$ and $y$ are fixed, a product cycle can be constructed after stabilization by choosing a Kasparov connection; different choices give homotopic cycles.  The construction is natural under equivariant $*$-homomorphisms and compatible with exterior products.
\end{theorem}

We do not reprove these foundational theorems.  They are the only parts of
the formalism of equivariant $KK$-theory that are invoked without proof;
every additional quantitative estimate specific to the constructions in this
paper is supplied below.  The other external theorems---nuclear
CP-stability, quasidiagonality, stable uniqueness, and the
Higson--Kasparov theorem---are each stated, with hypotheses and references,
at the point where they are used.

\subsection{Representable equivariant \texorpdfstring{$K$}{K}-theory and the proper pullback}\label{subsec:representable-KK}

For a locally compact $G$-space $X$ and $G$-$C^*$-algebras $A,B$, we use Kasparov's representable group
\[
 RKK^G(X;A,B):=KK^G(C_0(X)\otimes A,C_0(X)\otimes B),
\]
where the left representation is $C_0(X)$-linear.  Explicitly, on elementary
tensors it satisfies
\[
 \varphi(ff'\otimes a)\xi=f\cdot\varphi(f'\otimes a)\xi
 \quad(f,f'\in C_0(X),\ a\in A),
\]
where $f$ acts through the central $C_0(X)$-action on the Hilbert module.  For proper $G$-compact $X$, in particular,
\[
  K_G^0(X)=RKK^G(X;\C,\C).
\]
For non-$G$-compact universal spaces, the symbol $K_G^0(\EG)$ in the
main statements retains the ordinary inverse-limit convention
\eqref{eq:RK-inverse-limit}. That convention must be distinguished from a
single full representable Kasparov group; no identification of the two is
being made merely by changing the order of the defining limits.
If $p:X\to *$ is the map to a point, exterior product with
$[\mathrm{id}_X]$ defines the exterior pullback (usually called the proper
pullback when $X$ is proper)
\begin{equation}\label{eq:proper-pullback}
 p_X^*:KK^G(A,B)\longrightarrow RKK^G(X;A,B).
\end{equation}
The map is compatible with Kasparov products:
\begin{equation}\label{eq:proper-pullback-product}
 p_X^*(x\otimes_B y)=p_X^*(x)\otimes_{C_0(X)\otimes B}p_X^*(y).
\end{equation}

We will later use the following external proper-algebra localization theorem.

\begin{theorem}[proper first-variable pullback]\label{thm:proper-coefficient-pullback}
Let $X=\EG$ be a universal proper $G$-space.  If a $G$-$C^*$-algebra $A$ is $KK^G$-equivalent to a proper $G$-$C^*$-algebra, then for every separable $G$-$C^*$-algebra $B$ the map
\[
 p_X^*:KK^G(A,B)\longrightarrow RKK^G(X;A,B)
\]
is an isomorphism.  The inverse is natural in $B$ and is compatible with Kasparov products.
\end{theorem}

Here a \emph{proper $G$-$C^*$-algebra} means a $G$-$C^*$-algebra $D$
equipped with a nondegenerate equivariant $*$-homomorphism
\[
 C_0(Z)\longrightarrow ZM(D)
\]
for some locally compact proper $G$-space $Z$, where $ZM(D)$ denotes the
center of the multiplier algebra $M(D)$.  Equivalently, $D$ is a
$G$-equivariant $C_0(Z)$-algebra.  This is the proper-algebra localization
theorem in Kasparov theory; a duality formulation for a universal proper
space is given by Emerson--Meyer
\cite{emerson_meyer_dualities_2010}.
The letter $A$ in this theorem is the \emph{first} Kasparov variable; the
statement must not be confused with a claim that properness of the second
coefficient algebra alone makes an arbitrary pullback invertible.

\subsection{Finite controls and compatible comparison maps}\label{subsec:control-comparisons}

The pairs $(F,\varepsilon)$ are ordered by saying that
$(F',\varepsilon')$ is \emph{stronger} than $(F,\varepsilon)$ if
$F\subseteq F'$ and $0<\varepsilon'\le\varepsilon$.  Restriction of domains
and relaxation of the error produce bonding maps from stronger to weaker
control.  Thus all groups built from quasi-representations form inverse
systems.

For the quantitative arguments below, let
$\phi_\bullet:\{A_c\}\to\{B_c\}$ be a natural transformation between
inverse systems with the same control set and bonding maps $r^A_{c',c}$ and
$r^B_{c',c}$.  We use the following explicitly quantified conditions:
\begin{enumerate}[label=(\roman*)]
\item for every $c$ there is $c'\succeq c$ such that
\[
 a\in A_{c'},\quad \phi_{c'}(a)=0
 \quad\Longrightarrow\quad r^A_{c',c}(a)=0;
\]
\item for every $c$ there is $c'\succeq c$ such that
\[
 r^B_{c',c}(B_{c'})\subseteq\phi_c(A_c).
\]
\end{enumerate}
Condition (i) says that a kernel element becomes zero after restriction;
condition (ii) says that a class can be lifted after restriction. In both
conditions, the stronger control is chosen before the class.

\begin{definition}[equivalence of inverse systems]
\label{def:inverse-system-equivalence}
We define explicitly what it means below for two inverse systems with
possibly different index sets to be equivalent.  Let
$A=(A_i,r^A_{i',i})_{i\in I}$ and
$B=(B_j,r^B_{j',j})_{j\in J}$.  An \emph{inverse-system equivalence}
consists of order-preserving maps
\[
 \sigma:J\to I,\qquad \tau:I\to J,
\]
homomorphisms
\[
 f_j:A_{\sigma(j)}\to B_j,\qquad
 g_i:B_{\tau(i)}\to A_i,
\]
and the following identities.  First,
$\sigma(\tau(i))\succeq i$ and $\tau(\sigma(j))\succeq j$.  Second, the
component maps commute with bonding maps:
\begin{align*}
 r^B_{j',j}f_{j'}&=f_jr^A_{\sigma(j'),\sigma(j)}&& (j\preccurlyeq j'),\\
 r^A_{i',i}g_{i'}&=g_ir^B_{\tau(i'),\tau(i)}&& (i\preccurlyeq i').
\end{align*}
Finally, the two componentwise composites are the indicated bonding maps:
\begin{align}
 g_i f_{\tau(i)}
   &=r^A_{\sigma(\tau(i)),i},\label{eq:inverse-system-equivalence-A}\\
 f_j g_{\sigma(j)}
   &=r^B_{\tau(\sigma(j)),j}.
   \label{eq:inverse-system-equivalence-B}
\end{align}
This definition uses only homomorphisms between the displayed groups.  It also
gives inverse homomorphisms of ordinary inverse limits: for
$a=(a_i)_i$ and $b=(b_j)_j$, set
\[
 f_*(a)_j=f_j(a_{\sigma(j)}),\qquad
 g_*(b)_i=g_i(b_{\tau(i)}).
\]
Compatibility with the bonding maps makes these families elements of the
respective inverse limits, and
\eqref{eq:inverse-system-equivalence-A}--
\eqref{eq:inverse-system-equivalence-B} give $g_*f_*=1$ and $f_*g_*=1$.
\end{definition}

\begin{lemma}[composition of inverse-system equivalences]
\label{lem:inverse-system-equivalence-composition}
Suppose that $A=(A_i)_{i\in I}$ and $B=(B_j)_{j\in J}$ are related by the
maps $(\sigma_1,\tau_1,f^1,g^1)$ of
\Cref{def:inverse-system-equivalence}, and that $B$ and
$C=(C_k)_{k\in K}$ are related by the maps
$(\sigma_2,\tau_2,f^2,g^2)$. Then $A$ and $C$ are related by maps satisfying the same conditions, with
index maps
\[
 \sigma=\sigma_1\sigma_2:K\to I,
 \qquad
 \tau=\tau_2\tau_1:I\to K,
\]
and component homomorphisms
\[
 F_k=f^2_k f^1_{\sigma_2(k)}:
       A_{\sigma_1(\sigma_2(k))}\longrightarrow C_k,
 \qquad
 G_i=g^1_i g^2_{\tau_1(i)}:
       C_{\tau_2(\tau_1(i))}\longrightarrow A_i.
\]
\end{lemma}

\begin{proof}
The index inequalities follow by applying the order-preserving maps to the
two inequalities for the given equivalences.  The two naturality identities
for $F$ and $G$ follow by substituting the naturality identity for the first
equivalence and then the one for the second.  For example, the composite on
$A_i$ is
\[
 A_{\sigma_1\sigma_2\tau_2\tau_1(i)}
 \xrightarrow{f^1_{\sigma_2\tau_2\tau_1(i)}}
 B_{\sigma_2\tau_2\tau_1(i)}
 \xrightarrow{f^2_{\tau_2\tau_1(i)}}
 C_{\tau_2\tau_1(i)}
 \xrightarrow{g^2_{\tau_1(i)}}B_{\tau_1(i)}
 \xrightarrow{g^1_i}A_i.
\]
The second inverse-system equivalence gives
\[
 g^2_{\tau_1(i)}f^2_{\tau_2\tau_1(i)}
 =r^B_{\sigma_2\tau_2\tau_1(i),\tau_1(i)}.
\]
Naturality of $f^1$ then gives
\[
 r^B_{\sigma_2\tau_2\tau_1(i),\tau_1(i)}
 f^1_{\sigma_2\tau_2\tau_1(i)}
 =f^1_{\tau_1(i)}
 r^A_{\sigma_1\sigma_2\tau_2\tau_1(i),
      \sigma_1\tau_1(i)}.
\]
Finally,
$g^1_if^1_{\tau_1(i)}=r^A_{\sigma_1\tau_1(i),i}$, so the entire composite
is $r^A_{\sigma_1\sigma_2\tau_2\tau_1(i),i}$.
This gives \eqref{eq:inverse-system-equivalence-A} for $(F,G)$.  The same
calculation with $A$ and $C$ interchanged gives
\eqref{eq:inverse-system-equivalence-B}.
\end{proof}

\section{Fredholm quasi-representations}\label{sec:FQR-cycle}

In this section, we introduce a theory of Fredholm quasi-representations and their characters. Roughly speaking, a Fredholm quasi-representation is an approximately multiplicative action of the group on a graded Hilbert space, or more generally a Hilbert module, together with an odd Fredholm operator that commutes with the action up to compact operators.  A finite-dimensional quasi-representation can be regarded as a Fredholm quasi-representation whose operator is zero. Unlike finite-dimensional quasi-representations, Fredholm quasi-representations allow compact perturbations and infinite-dimensional stabilizations,  while keeping the multiplicative error under control.

Throughout the proof, $\chi$ denotes the geometric character defined in
\Cref{sec:character_map}. The Fredholm character $\chi^F$ constructed below
uses the same cutoff functions, averaging, and spectral projections. On
finite-dimensional quasi-representations, regarded as Fredholm
quasi-representations with operator zero, it recovers $\chi$; this is proved
in \Cref{lem:character-compatibility}. The proof of the classification theorem
uses this extension of the geometric construction.
Thus the geometric and crossed-product constructions describe one character,
while $\chi^F$ extends it to the Fredholm quasi-representations used in the
proof.

Finite-dimensional quasi-representations are the objects to be classified,
but finite-dimensionality makes several of the operations needed in the
proof difficult to perform.  In particular, coefficient changes, boundary
maps, and cutting-and-pasting arguments naturally introduce compact
perturbations and infinite-dimensional stabilizations.  Fredholm
quasi-representations provide a setting in which these operations can be
carried out while the multiplicative error of the approximate group action
remains under explicit control.

A Fredholm quasi-representation combines an approximately multiplicative
action with an odd Fredholm operator. The action records the multiplication
relations being tested and their error, while the Fredholm operator records
a $K$-theory class. We use Kasparov products and exterior products with
spaces to compare these classes. Cutoff functions on proper $G$-spaces
then give the character construction developed below. Fredholm
quasi-representations thus connect the finite-dimensional objects to be
classified with the equivariant $K$-theory used in the proof.

Throughout this section, $G$ is a countable discrete group, $A$ is a
separable $G$-$C^*$-algebra, and $X$ is a locally compact, second countable,
Hausdorff right $G$-space.  Until the paragraph following
\Cref{def:FQR-final}, we additionally assume that $X/G$ is compact and call
such an $X$ \emph{$G$-compact}.  The definition for an arbitrary locally
compact, second countable, Hausdorff right $G$-space $X$ is then made by
restriction to all closed invariant $G$-compact subspaces, with the inverse
limit written out explicitly.  The functoriality statement below also writes
out the exact condition needed for a continuous equivariant map to act on
these inverse limits; that condition is automatic for maps whose target is
proper or $G$-compact.

We now give the definition of Fredholm quasi-representations used in the rest
of the paper.

\begin{definition}\label{def:FQR-cycle}
Let $X$ be a cocompact $G$-space (not necessarily proper) and let $A$ be a $G$-$C^*$-algebra.  Given $\varepsilon\ge0$ and a finite
$F\subseteq G$ containing $e$, an
\emph{$X$-Fredholm $(F,\varepsilon)$-quasi-representation with coefficients in $A$} (or, in anticipation of our upcoming definition, an \emph{$\mathrm{FQR}(X;F,\varepsilon;A)$-cycle}) is a
triple $(\mathcal{E},q,T)$ where
\begin{itemize}
\item  $\mathcal E=\mathcal E^+\oplus\mathcal E^-$ is a graded countably generated
Hilbert $A_X$-module, where $A_X := C_0(X)\otimes A$;
\item $q = (q(g))_{g \in FF}$ where each $q(g)$ is an even $\alpha_g$-unitary (see \eqref{eq:alpha_g-unitary}) on $\mathcal E$;
\item $T\in\mathcal L(\mathcal E)$ is odd;
\end{itemize}
such that, when we define  
\begin{equation}\label{eq:local-compact-ideal}
 \mathcal K_X(\mathcal E) :=
 \{S\in\mathcal L(\mathcal E):Sf\in\mathcal K(\mathcal E)
       \text{ for every }f\in C_0(X)\}, 
\end{equation}
which is a closed two-sided ideal of $\mathcal L(\mathcal E)$ invariant under conjugations by $\alpha_g$-unitaries for all $g \in G$, then the following conditions hold:
\begin{align}
 q(g)q(h) &\approx_\varepsilon q(gh) &&   \text{ for any } g,h\in F,\label{eq:FQR-cycle-mult}\\
 T-T^*,\ T^2-1&\in\mathcal K_X(\mathcal E),
      \label{eq:FQR-cycle-Fredholm}\\
 \operatorname{Ad}_{q(g)}(T)-T&\in\mathcal K_X(\mathcal E)
      &&  \text{ for any } g\in FF.\label{eq:FQR-cycle-equiv}
\end{align}

The set of all $\mathrm{FQR}(X;F,\varepsilon;A)$-cycles is denoted $\mathrm{FQRep}(X;F,\varepsilon;A)$. 
An $\mathrm{FQR}(X;F,\varepsilon;A)$-cycle is \emph{degenerate} when in
\eqref{eq:FQR-cycle-Fredholm} and \eqref{eq:FQR-cycle-equiv} the ideal $K_X(\mathcal E)$ is replaced by $\{0\}$.  
Note that given any closed invariant subset $Y \subseteq X$, the induced ideal $C_0(X \setminus Y) \otimes A \mathrel{\triangleleft} C_0(X) \otimes A$ is invariant under all the operators above, and thus any $\mathrm{FQR}(X;F,\varepsilon;A)$-cycle $(\mathcal{E}, q, T)$ restricts to a $\mathrm{FQR}(Y;F,\varepsilon;A)$-cycle $(\mathcal{E}|_Y, q|_Y, T|_Y)$ by taking a quotient module. 
A \emph{homotopy} between two $\mathrm{FQR}(X;F,\varepsilon;A)$-cycles is an $\mathrm{FQR}([0,1] \times X;F,\varepsilon;A)$-cycle such that its restrictions to $\{0\} \times X$ and $\{1\} \times X$ canonically identifies with the original cycles. 
\emph{Unitary equivalence} between two $\mathrm{FQR}(X;F,\varepsilon;A)$-cycles $(\mathcal{E}, q, T)$ and $(\mathcal{E}', q', T')$ means there is an even $A_X$-linear unitary $U \colon \mathcal{E} \to \mathcal{E}'$ such that $q' = \mathrm{Ad}_U \circ q$ and $T' = \mathrm{Ad}_U(T)$
\emph{Stable homotopy} is generated by homotopy, unitary equivalence, and addition of degenerate cycles.
\end{definition}

\begin{remark}
If $X$ is compact, then $1_X\in C(X)$ and
$\mathcal K_X(\mathcal E)=\mathcal K(\mathcal E)$.  Consequently the cycles
for which $q$ is defined on all of $G$ and is a genuine group action give precisely
the central $C(X)$-linear cycle picture of $RKK^G(X;\mathbb C,A)$ described
in Section~\ref{sec:preliminaries} (for $X=\{*\}$ this is
$KK^G(\mathbb C,A)$).  For noncompact $X$, the ideal
$\mathcal K_X(\mathcal E)$ records local compactness.  Thus ``Fredholm'' means
that $T$ is self-adjoint and unitary modulo this ideal and equivariant modulo
the same ideal.
\end{remark}

\begin{definition}\label{def:FQR-final}
Use the directed set $\mathcal C_G$ and its order from
\Cref{def:stable-equivalence}.  Thus \(c\preccurlyeq c'\) means that
\(c'\) imposes at least as many multiplication relations and permits no
larger error than \(c\).  If
\(c=(F,\varepsilon)\), abbreviate
\[
 \mathrm{FQR}(X;c;A)=\mathrm{FQR}(X;F,\varepsilon;A).
\]

Let $\mathrm{FQR}(X;F,\varepsilon;A)$ be the set of stable-homotopy
classes of the cycles in \Cref{def:FQR-cycle}, with direct sum as addition
and a degenerate cycle as zero.  This set is already an abelian group.  Indeed,
if $z=(\mathcal E,q,T)$ and $\mathcal E^{\mathrm{op}}$ denotes the same module
with the grading reversed, put
\[
 z^{\mathrm{op}}=(\mathcal E^{\mathrm{op}},q,-T).
\]
Then $z^{\mathrm{op}}$ is a cycle with the same control.  On
$\mathcal E\oplus\mathcal E^{\mathrm{op}}$ let
\[
 S=\begin{pmatrix}0&1\\1&0\end{pmatrix},\qquad
 T_s=\cos(\pi s/2)\begin{pmatrix}T&0\\0&-T\end{pmatrix}
     +\sin(\pi s/2)S\quad(0\leq s\leq1).
\]
The operator $S$ is odd, self-adjoint, unitary, commutes with $q\oplus q$,
and anticommutes with $T\oplus(-T)$.  Hence
\begin{align*}
 T_s-T_s^*&=\cos(\pi s/2)
       \big((T-T^*)\oplus(-T+T^*)\big),\\
 T_s^2-1&=\cos^2(\pi s/2)
       \big((T^2-1)\oplus(T^2-1)\big),
\end{align*}
and the equivariance defect of $T_s$ is the same cosine times the direct sum
of the equivariance defects of $T$ and $-T$.  Thus $(q\oplus q,T_s)$ is a
homotopy from $z\oplus z^{\mathrm{op}}$ to the degenerate cycle
$(q\oplus q,S)$.  It follows that $[z^{\mathrm{op}}]=-[z]$.  The block-flip
unitary proves commutativity.  Define
\[
 \mathrm{FQR}(X;G;A)=
 \varprojlim_{c\in\mathcal C_G}\mathrm{FQR}(X;c;A).
\]
More explicitly, whenever
\(c=(F,\varepsilon)\preccurlyeq
c'=(F',\varepsilon')\), restriction from \(F'F'\) to
\(FF\), followed by allowing the larger error \(\varepsilon\),
defines a homomorphism
\[
 r^A_{c',c}:\mathrm{FQR}(X;c';A)
       \longrightarrow\mathrm{FQR}(X;c;A).
\]
These maps satisfy \(r^A_{c,c}=\operatorname{id}\) and
\(r^A_{c'',c}=r^A_{c',c}r^A_{c'',c'}\) for
\(c\preccurlyeq c'\preccurlyeq c''\), and the displayed inverse limit is
taken with respect to these maps.  Hence its elements are precisely the
families \((z_c)_{c\in\mathcal C_G}\) satisfying
\(r^A_{c',c}(z_{c'})=z_c\) whenever \(c\preccurlyeq c'\).
The groups with \(\varepsilon=0\) are defined above and may be used
separately, but \(\varepsilon=0\) is not included in \(\mathcal C_G\); the
inverse limit records approximation through positive error tolerances tending
to zero.
When $X$ is a point we write $\mathrm{FQR}(G;A)$.  In view of
\Cref{prop:FQR-asymptotic-equivariance}, an auxiliary family $(T_t)$ satisfying (b) can be attached whenever it is needed.
No such family is part of the definition of a Fredholm quasi-representation or
of its stable-homotopy class.  For a homotopy, the proposition is applied with
$P=[0,1]$, so that one value of $t$ works simultaneously at every point of the
homotopy.
\end{definition}

A bounded family $(S_t)_{t\geq1}$ that is uniformly continuous in operator
norm is called
\emph{asymptotically $F$-equivariant with respect to $q$} if
\[
 \max_{g\in F}\|\operatorname{Ad}_{q(g)}(S_t)-S_t\|
 \longrightarrow0.
\]

We also fix the meaning of a parameterized family.  Give a compact metrizable
space $P$ the trivial $G$-action.  A \emph{$P$-parameterized family} of
$X$-Fredholm $(F,\varepsilon)$-quasi-representations is, by definition, a
$(P\times X)$-Fredholm $(F,\varepsilon)$-quasi-representation with
coefficients in $A$.  Its module is therefore a module over
$C(P)\otimes C_0(X)\otimes A$, and
$\mathcal K_{P\times X}$ is used in all Fredholm and equivariance conditions.
This formulation makes every norm estimate uniform in $p\in P$.  The case
$P=[0,1]$ is precisely a homotopy in the sense of
\Cref{def:FQR-cycle}.

\begin{proposition}[asymptotically equivariant normalization]
\label{prop:FQR-asymptotic-equivariance}
Let $(\mathcal{E},q,T)$ be an $X$-Fredholm
$(F,\varepsilon)$-quasi-representations with coefficients in $A$.  Then there is an even family
$(h_t)_{t\geq1}$ of positive contractions in
$\mathcal K_{ X}(\mathcal E)$ such that, for
\begin{equation}\label{eq:FQR-normalized-operator}
 T_t=(1-h_t)T(1-h_t),
\end{equation}
both $t\mapsto h_t$ and $t\mapsto T_t$ are uniformly continuous in operator
norm, and the following assertions hold:
\begin{enumerate}[label=(\alph*)]
\item $T_t-T\in\mathcal K_{ X}(\mathcal E)$ for every $t$; hence $T_t$ has the
same image as $T$ in
$\mathcal L(\mathcal E)/\mathcal K_{ X}(\mathcal E)$ and satisfies the same
Fredholm relations modulo $\mathcal K_{ X}(\mathcal E)$;
\item
\begin{equation*}
 \lim_{t\to\infty}\max_{k\in FF}
 \big\|\operatorname{Ad}_{q(k)}(T_t)-T_t\big\|=0;
\end{equation*}
\end{enumerate}
\end{proposition}

\begin{proof}
Write $I=\mathcal K_{ X}(\mathcal E)$ and $K=FF$.  Let $\Lambda_K$ be the
free group on symbols $s_k$, $k\in K$, and let
$\rho:\Lambda_K\to G$ be the homomorphism determined by $\rho(s_k)=k$.
Declare that $s_k$ acts on $\mathcal E$ by the $\alpha_k$-unitary $q(k)$ and
extend multiplicatively to $\Lambda_K$.  This is a genuine action of the free
group, even though the original map $q$ is only approximately multiplicative
as a map on $G$.  Conjugation therefore gives an action
$\beta:\Lambda_K\curvearrowright I$, and
\[
 \Delta_k:=\beta_{s_k}(T)-T\in I,\qquad k\in K.
\]

Let $J$ be the closed graded $C^*$-subalgebra of $I$ generated by
$T-T^*$, $T^2-1$, the $\Delta_k$, their adjoints, all their
$\Lambda_K$-translates, and the elements obtained by repeatedly multiplying
these elements on either side by $T$ and $T^*$.  This description uses only
countably many operations, so $J$ is separable.  It is invariant under
$\Lambda_K$ and the grading, and
\[
 TJ+JT+T^*J+JT^*\subseteq J.
\]
Consequently $T\in M(J)$, $J$ is an invariant ideal in the separable algebra
$D=C^*(1,T,J)$, and $\beta_w(D)=D$ for $w\in\Lambda_K$.  Notice that every
operator which will be cut down below belongs to $J$.

We next prove, in the special form needed here, the approximate-unit fact used
in the argument.  Suppose a countable discrete group $\Lambda$ acts on a
separable $C^*$-algebra $D$ and leaves a $\sigma$-unital ideal
$J$ invariant.  There is a sequence $(u_n)$ of positive contractions in $J$
such that
\begin{equation}\label{eq:equivariant-qcau}
 \begin{aligned}
  u_nj&\longrightarrow j,\qquad ju_n\longrightarrow j &&(j\in J),\\
  \|[u_n,d]\|&\longrightarrow0 &&(d\in D),\\
  \|\beta_w(u_n)-u_n\|&\longrightarrow0 &&(w\in\Lambda).
 \end{aligned}
\end{equation}
Choose a faithful nondegenerate
representation $\pi_0:D\to\mathcal L(H_0)$ and form the regular covariant
representation on $\ell^2(\Lambda)\otimes H_0$:
\[
 (\pi(d)\xi)(w)=\pi_0(\beta_{w^{-1}}(d))\xi(w),\qquad
 (U_v\xi)(w)=\xi(v^{-1}w).
\]
Then $U_v\pi(d)U_v^*=\pi(\beta_v(d))$.  Let $\mathcal C$ be the separable
$C^*$-algebra generated by $\pi(D)$ and the unitaries $U_v$.  If $(e_m)$ is
any countable positive contractive approximate identity for $J$, then
$(\pi(e_m))$ is an approximate identity for the ideal of $\mathcal C$ generated by
$\pi(J)$.  Indeed, it is enough to check this on the dense linear span of
monomials $\pi(j)U_v\pi(d)$; on the right one uses that
$(\beta_v(e_m))$ is again an approximate identity for $J$.

In the Banach bidual $\mathcal C^{**}=(\mathcal C^*)^*$, the approximate
identity $(\pi(e_m))$ converges in the weak$^*$ topology to the central
support projection of the ideal. Since this projection is central,
$[\pi(e_m),c]$ converges weakly to zero for every $c\in\mathcal C$.
Choose dense sequences $(c_i)$ in the unit ball of $\mathcal C$ and $(j_i)$ in the
unit ball of $J$.  For each $n$, choose $M_n$ so that, whenever $m\geq M_n$
and $1\leq i\leq n$,
\[
 \|e_mj_i-j_i\|+\|j_ie_m-j_i\|<1/n.
\]
Apply the weak-convergence observation to the vectors
\[
 ([\pi(e_m),c_1],\ldots,[\pi(e_m),c_n])\in\mathcal C^n,\qquad m\geq M_n,
\]
where $\mathcal C^n$ has the maximum norm.  If zero were not in the norm closure of the
convex hull of this tail, Hahn--Banach
separation would give a bounded functional for which the same commutators do
not converge weakly to zero, a contradiction.  Hence there is a finite convex
combination $u_n$ of elements $e_m$ with $m\geq M_n$ such that
\[
 \max_{1\leq i\leq n}\|[\pi(u_n),c_i]\|<1/n.
\]
The defining choice of $M_n$ also gives
\[
 \max_{1\leq i\leq n}
 \big(\|u_nj_i-j_i\|+\|j_iu_n-j_i\|\big)<1/n.
\]
Density now shows that $(u_n)$ is an approximate identity for $J$ and is
quasicentral in $\mathcal C$.
In particular,
\[
 \|\beta_v(u_n)-u_n\|
 =\|U_v\pi(u_n)U_v^*-\pi(u_n)\|
 =\|[U_v,\pi(u_n)]\|\longrightarrow0.
\]
This proves \eqref{eq:equivariant-qcau}.  The grading automorphism
$\delta(a)=\Gamma a\Gamma$ leaves $J$ and $D$ invariant and commutes with the
$\Lambda_K$-action because every $q(k)$ is even.  Replacing $u_n$ by
$\frac12(u_n+\delta(u_n))$ preserves all three limits in
\eqref{eq:equivariant-qcau} and makes every $u_n$ even.

Apply this construction to $J\triangleleft D$ and $\Lambda_K$.  After passing
to a subsequence we may arrange, for all $k\in K$,
\begin{equation}\label{eq:FQR-qcau-estimates}
 \|\beta_{s_k}(u_n)-u_n\|<\frac1n,
 \qquad
 \|(1-u_n)\Delta_k\|+\|\Delta_k(1-u_n)\|<\frac1n.
\end{equation}
For $n\leq t\leq n+1$, set
\[
 h_t=(n+1-t)u_n+(t-n)u_{n+1}.
\]
This selection-and-interpolation device is the one used in
\cite[Lemma~4.1]{dadarlat_willett_wu_2018}; the preceding paragraphs include
the extra equivariance argument needed here.  The family $h_t$ is even,
positive, contractive, and $1$-Lipschitz, and the two estimates in
\eqref{eq:FQR-qcau-estimates} still tend to zero along the interpolated family.

Define $T_t$ by \eqref{eq:FQR-normalized-operator}.  It is odd and bounded;
since $h_t$ is $1$-Lipschitz, $T_t$ is Lipschitz with constant at most
$2\|T\|$.  Moreover,
\[
 T_t-T=-h_tT-Th_t+h_tTh_t\in J\subseteq I,
\]
which proves (a).  For $k\in K$, add and subtract
$(1-h_t)\beta_{s_k}(T)(1-h_t)$ to obtain
\begin{align*}
 \|\beta_{s_k}(T_t)-T_t\|
 &\leq 2\|T\|\,\|\beta_{s_k}(h_t)-h_t\|
       +\|(1-h_t)\Delta_k(1-h_t)\|.
\end{align*}
Both terms tend to zero, uniformly over the finite set $K$.  Since
$\beta_{s_k}=\operatorname{Ad}_{q(k)}$, this is exactly (b).
\end{proof}

We record the diagonal observation that will be used below.  Let
$F_1\subseteq F_2\subseteq\cdots$ be finite symmetric subsets
with identity and union $G$, and for each $n$ let $(q_n,T_n)$ be an
$X$-Fredholm $(F_n,\varepsilon_n)$-quasi-representation.  Apply the proposition
to the $n$th pair and choose $t_n$ so that
\[
 \max_{k\in F_nF_n}
 \|\operatorname{Ad}_{q_n(k)}((T_n)_{t_n})-(T_n)_{t_n}\|<\frac1n.
\]
Then for every fixed $g\in G$ the displayed defect with $k=g$ tends to zero.
This numerical assertion is the entire content of the diagonal choice.  It is
not part of the definition of the groups below, and it introduces no new
equivalence relation.

\subsection{Arbitrary spaces, functoriality, and homotopy invariance}
\label{subsec:FQR-space-functoriality}

We now remove the temporary $G$-compactness hypothesis.  Let
$\mathscr Z_G(X)$ be the directed set of all closed, invariant,
$G$-compact subspaces of $X$, ordered by inclusion.  It is directed: if
$Z_1,Z_2\in\mathscr Z_G(X)$, then $Z_1\cup Z_2$ is again closed,
invariant, and $G$-compact.  No assertion that these subspaces cover an
arbitrary nonproper $G$-space is needed for the inverse-limit definition.
They do cover every proper $G$-space: if $K\subseteq X$ is compact, then
$GK$ is closed and $G$-compact because $X/G$ is Hausdorff and the image of
$K$ in $X/G$ is compact.

If $Z\subseteq Z'$ are in $\mathscr Z_G(X)$, restriction
$C_0(Z')\to C_0(Z)$ is an equivariant quotient homomorphism.  A cycle
$z=(\mathcal E,q,T)$ over $C_0(Z')\otimes A$ restricts to
\begin{equation}\label{eq:FQR-space-restriction-cycle}
 z|_Z=
 \bigl(\mathcal E\widehat\otimes_{C_0(Z')\otimes A}
       (C_0(Z)\otimes A),\ q|_Z,\ T\widehat\otimes1\bigr).
\end{equation}
The formula for $q|_Z$ is
$q|_Z(g)(\xi\otimes b)=q(g)\xi\otimes\alpha_g(b)$.
It is well defined because the kernel of the restriction homomorphism is
$G$-invariant.  If $f\in C_0(Z)$, choose a lift
$\widetilde f\in C_0(Z')$.  The product of $f$ with any Fredholm defect in
\eqref{eq:FQR-space-restriction-cycle} is the image of the compact operator
obtained by multiplying the original defect by $\widetilde f$.  Thus
\eqref{eq:FQR-space-restriction-cycle} is a cycle and defines a homomorphism
\[
 \operatorname{res}^{Z'}_Z:
 \mathrm{FQR}(Z';c;A)\longrightarrow\mathrm{FQR}(Z;c;A).
\]
These homomorphisms commute with the group-control bonding maps.

\begin{definition}[FQR for an arbitrary $G$-space]
\label{def:FQR-arbitrary-X}
For a locally compact second countable right $G$-space $X$, not necessarily
$G$-compact, define
\begin{align}
 \mathrm{FQR}(X;c;A)
   &:=\varprojlim_{Z\in\mathscr Z_G(X)}
       \mathrm{FQR}(Z;c;A),\label{eq:FQR-space-limit}\\
 \mathrm{FQR}(X;G;A)
   &:=\varprojlim_{c\in\mathcal C_G}\mathrm{FQR}(X;c;A).
       \label{eq:FQR-two-limits}
\end{align}
Thus an element of the first group is a family $(z_Z)_Z$ with
$\operatorname{res}^{Z'}_Z(z_{Z'})=z_Z$ whenever $Z\subseteq Z'$.  An
element of the second is a family $(z_{Z,c})_{Z,c}$ compatible both with
restriction in $Z$ and with restriction and relaxation in $c$.  If $X$ is
$G$-compact, $X$ is a largest element of $\mathscr Z_G(X)$, so
\eqref{eq:FQR-space-limit} agrees canonically with the cycle group already
defined in \Cref{def:FQR-final}.

More explicitly, for $Z\subseteq Z'$ and $c\preccurlyeq c'$, the joint
bonding map is
\[
 u_{(Z',c'),(Z,c)}
 =\operatorname{res}^{Z'}_Z\,r^{Z',A}_{c',c}
 =r^{Z,A}_{c',c}\,\operatorname{res}^{Z'}_Z.
\]
Thus the group in \eqref{eq:FQR-two-limits} is also
\[
 \varprojlim_{(Z,c)\in\mathscr Z_G(X)\times\mathcal C_G}
       \mathrm{FQR}(Z;c;A).
\]
In particular, taking the two inverse limits in the opposite order gives the same group. 
\end{definition}

\begin{proposition}[contravariance in $X$ and $G$-homotopy invariance]
\label{prop:FQR-space-functoriality}
Let $f:X\to Y$ be a continuous $G$-equivariant map between locally compact
second countable right $G$-spaces, and assume that $f$ is \emph{bornological} in the sense that 
\begin{equation}\label{eq:FQR-map-Gcompact-condition}
 \text{for every }Z\in\mathscr Z_G(X)\text{ there is }
 W\in\mathscr Z_G(Y)\text{ such that }f(Z)\subseteq W.
\end{equation}
For every control $c$ there is a
homomorphism
\[
 f^*:\mathrm{FQR}(Y;c;A)\longrightarrow
       \mathrm{FQR}(X;c;A)
\]
which commutes with the bonding maps in $c$.  These maps satisfy
$(\operatorname{id}_X)^*=\operatorname{id}$ and
$(f_2f_1)^*=f_1^*f_2^*$.  If $f_0$ and $f_1$ are joined by a continuous
$G$-equivariant homotopy $H:[0,1]\times X\to Y$ of bornological maps, 
then $f_0^*=f_1^*$.  The same statements hold
after taking either inverse limit in
\eqref{eq:FQR-space-limit}--\eqref{eq:FQR-two-limits}.

Note that bornologicality of a continuous map $f : X \to Y$ holds automatically if $Y$ is proper or $G$-compact.  In particular, its holds for all continuous maps when the acting group is finite.
\end{proposition}

\begin{proof}
First suppose that $X$ and $Y$ are $G$-compact.  The map $f$ gives the
nondegenerate equivariant homomorphism
\[
 \Phi_f:C_0(Y)\otimes A\longrightarrow
 M(C_0(X)\otimes A),\qquad
 \Phi_f(a)(x)=a(f(x)).
\]
Nondegeneracy does not require $f$ to be proper: for
$h\in C_c(X)$, a function in $C_c(Y)$ which equals one on the compact set
$f(\operatorname{supp}h)$ acts as the identity on $h$.  Pull back a cycle
$z=(\mathcal E,q,T)$ by putting
\[
 f^*\mathcal E=\mathcal E\widehat\otimes_{\Phi_f}
       (C_0(X)\otimes A),\quad
 q_f(g)(\xi\otimes b)=q(g)\xi\otimes\alpha_g(b),\quad
 T_f=T\widehat\otimes1.
\]
The multiplication defect of $q_f$ has no larger norm than that of $q$.
To check local compactness, let $h\in C_c(X)$ and choose
$k\in C_c(Y)$ equal to one on $f(\operatorname{supp}h)$.  For any Fredholm
or equivariance defect $D$ of $z$,
\[
 h(D\widehat\otimes1)=h((kD)\widehat\otimes1),
 \qquad
 (D\widehat\otimes1)h=((Dk)\widehat\otimes1)h,
\]
and the right-hand operators are compact.  Density of $C_c(X)$ in $C_0(X)$
proves that all pulled-back defects belong to $\mathcal K_X(f^*\mathcal E)$.
The construction plainly preserves direct sums, degenerate cycles, and
parameterized homotopies, so it descends to the stated homomorphism.

The canonical associativity unitary for interior tensor products proves
functoriality under composition.  If
$H:[0,1]\times X\to Y$ is a $G$-homotopy, pullback along $H$ is one
$( [0,1]\times X)$-Fredholm cycle; its evaluations at the endpoints are
$f_0^*z$ and $f_1^*z$.  Hence the endpoint maps agree on stable-homotopy
classes.

For arbitrary spaces, fix $Z\in\mathscr Z_G(X)$ and choose
$W\in\mathscr Z_G(Y)$ as in
\eqref{eq:FQR-map-Gcompact-condition}.  Restrict a compatible family over
$Y$ to its $W$-coordinate and apply the preceding construction to
$f|_Z:Z\to W$.  If $W_1$ and $W_2$ are two choices, restriction from the
$W_1\cup W_2$ coordinate shows that the resulting class over $Z$ is the
same.  The classes are compatible when $Z$ is enlarged and therefore define
the map on the inverse limit.  The identity and composition formulas hold in
every $Z$-coordinate.  For a homotopy, choose one $W$ containing
$H([0,1]\times Z)$ and apply the $G$-compact case to
$H|_{[0,1]\times Z}$; its endpoint evaluations give the two pullbacks.

Finally suppose that $Y$ is proper.  Its orbit space $Y/G$ is Hausdorff.  If
$Z/G$ is compact, the induced map $Z/G\to Y/G$ has compact, hence closed,
image.  The inverse image of that compact subset under $Y\to Y/G$ is a
closed invariant $G$-compact subspace containing $f(Z)$.  Applying the same
argument to $[0,1]\times Z$ proves the homotopy condition.  If $Y$ is
$G$-compact, take $W=Y$.  This proves the final assertion.
\end{proof}

\begin{corollary}[the finite-subgroup case of ``times $\underline EG$'']
\label{cor:FQR-times-EG-finite}
Let $H\leq G$ be finite and regard $\underline EG$ as an $H$-space by
restriction.  For every locally compact second countable $H$-space $X$ and
every separable $H$-$C^*$-algebra $B$, pullback by the projection is an
isomorphism
\[
 \operatorname{pr}_X^*:\mathrm{FQR}(X;H;B)
 \xrightarrow{\cong}
 \mathrm{FQR}(\underline EG\times X;H;B).
\]
\end{corollary}

\begin{proof}
For a finite group $H$, the defining fixed-point property of the universal
proper space says that $(\underline EG)^H$ is nonempty and that
$\underline EG$ is $H$-equivariantly contractible.  Choose
$y_0\in(\underline EG)^H$ and put $s(x)=(y_0,x)$.  Then
$\operatorname{pr}_Xs=\operatorname{id}_X$, while
$s\operatorname{pr}_X$ is $H$-homotopic to the identity of
$\underline EG\times X$.  Contravariance and homotopy invariance in
\Cref{prop:FQR-space-functoriality} show that $s^*$ is the inverse of
$\operatorname{pr}_X^*$.
\end{proof}

\subsection{\texorpdfstring{$KK^G$}{KKG} and coefficient functoriality}
\label{subsec:FQR-KK-cycle}

For a locally compact second countable right $G$-space $X$ and a separable
$G$-$C^*$-algebra $A$, let
\[
 \mathbf{FQR}_{X,A}:\mathcal C_G^{\mathrm{op}}\longrightarrow\mathbf{Ab},
 \qquad
 (F,\varepsilon)\longmapsto\mathrm{FQR}(X;F,\varepsilon;A),
\]
denote the inverse system whose bonding maps were defined in
\Cref{def:FQR-final}.  We now define the action of equivariant $KK$-theory on
this entire inverse system.  The construction takes the exterior product
with $C_0(X)$ and uses the ideal $\mathcal K_X$.

\begin{proposition}[Kasparov product for Fredholm quasi-representations]
\label{prop:FQR-KK-product}
Let $X$ be a locally compact second countable right $G$-space. Then for any $(F, \varepsilon) \in \mathcal{C}_G$, there is an additive functor from the equivariant $KK$-category $\mathbf{KK}^G$ to the category $\mathbf{Ab}$ of abelian groups that takes each separable $G$-$C^*$-algebra $A$ to $\mathrm{FQR}(X;F,\varepsilon;A)$. More concretely, any pair $(A,B)$ of separable $G$-$C^*$-algebras induces a homomorphism 
\[
    KK^G(A,B) \longrightarrow \mathrm{Hom}(\mathrm{FQR}(X;F,\varepsilon;A)
 , \mathrm{FQR}(X;F,\varepsilon;B)) \, , \quad  x \longmapsto x_{F,\varepsilon}
\]
in a way that is compatible with taking Kasparov products, that is, for any $x\in KK^G_0(A,B)$ and $y\in KK^G_0(B,C)$, we have 
\[
 (1_A)_{F,\varepsilon}=\operatorname{id}_{\mathrm{FQR}(X;F,\varepsilon;A)} \quad \text{ and } \quad
 (x\otimes_B y)_{F,\varepsilon}=y_{F,\varepsilon}\circ x_{F,\varepsilon} \; .
\]

Moreover, when we vary $(F,\varepsilon)$, these functors assemble into an additive functor from  $\mathbf{KK}^G$ to the functor category $\operatorname{Fun}(\mathcal C_G^{\mathrm{op}},\mathbf{Ab})$. 
More concretely, each $x\in KK^G_0(A,B)$ induces a natural transformation
\[ x_\bullet:\mathbf{FQR}_{X,A}\Longrightarrow\mathbf{FQR}_{X,B} \; ,
\]
where naturality means whenever
$c=(F,\varepsilon)\preccurlyeq c'=(F',\varepsilon')$, then
\begin{equation}\label{eq:FQR-KK-bonding}
 r^B_{c',c}\,x_{F',\varepsilon'}
 =x_{F,\varepsilon}\,r^A_{c',c}.
\end{equation}

As a consequence, by taking inverse limits, these natural transformations induce an additive functor from $\mathbf{KK}^G$ to $\mathbf{Ab}$ that takes each separable $G$-$C^*$-algebra $A$ to $\mathrm{FQR}(X;G;A)$. 
\end{proposition}

\begin{proof}
It is enough first to treat the case in which $X$ is $G$-compact.  Indeed,
for a general $X$, \Cref{def:FQR-arbitrary-X} defines each controlled group
as the inverse limit over the $G$-compact invariant subspaces $Z\subseteq X$.
The construction below applied to every $Z$ commutes with restriction from a
larger such subspace to a smaller one.  It therefore gives, and uniquely
determines, the asserted homomorphism for arbitrary $X$.

Put
\[
 A_X=C_0(X)\otimes A,\qquad B_X=C_0(X)\otimes B.
\]
Represent $x$ by a graded equivariant Kasparov $A$--$B$ cycle
$(\mathcal F,\pi,S)$.  Its exterior product with $C_0(X)$ is the equivariant
Kasparov $A_X$--$B_X$ cycle
\[
 \mathcal F_X=C_0(X)\widehat\otimes\mathcal F,\qquad
 \pi_X=1_{C_0(X)}\widehat\otimes\pi,
 \qquad S_X=1\widehat\otimes S.
\]
If $V_g$ is the action on $\mathcal F$, the action $V^X_g$ on
$\mathcal F_X$ is
\[
 V^X_g(f\otimes\eta)=\alpha_g(f)\otimes V_g\eta.
\]
This cycle represents the exterior pullback $p_X^*(x)$ from
\eqref{eq:proper-pullback}.

Let $(q,T)$ be an $X$-Fredholm $(F,\varepsilon)$-quasi-representation on a
graded Hilbert $A_X$-module $\mathcal E$.  We first turn the finitely many
operators $q(k)$ into a genuine group action without imposing any false
relations on them.  Let $\Lambda_F$ be the free group on symbols $s_k$ indexed
by $k\in FF$, and let
$\rho_F:\Lambda_F\to G$ be determined by $\rho_F(s_k)=k$.  Define
\[
 U_{s_k}=q(k)
\]
and extend multiplicatively to $\Lambda_F$.  By the defining module rule for
an $\alpha_k$-unitary, $U_w$ is an
$\alpha_{\rho_F(w)}$-unitary.  Thus
$w\mapsto U_w$ is a genuine $\Lambda_F$-action on $\mathcal E$, where
$\Lambda_F$ acts on $A_X$ through $\rho_F$.

Let $m:C_0(X)\to\mathcal L_{A_X}(\mathcal E)$ be multiplication.  Since $T$
is $A_X$-linear, $[T,m(f)]=0$.  The definition of
$\mathcal K_X(\mathcal E)$ gives
\[
 (T-T^*)m(f),\quad (T^2-1)m(f)\in\mathcal K(\mathcal E)
 \qquad(f\in C_0(X)).
\]
For a generator $s_k$,
\[
 \operatorname{Ad}_{U_{s_k}}(T)-T
 =\operatorname{Ad}_{q(k)}(T)-T\in\mathcal K_X(\mathcal E).
\]
The identity
\begin{equation}\label{eq:free-word-equivariance}
 \operatorname{Ad}_{U_vU_w}(T)-T
 =\operatorname{Ad}_{U_v}
    \big(\operatorname{Ad}_{U_w}(T)-T\big)
   +\big(\operatorname{Ad}_{U_v}(T)-T\big)
\end{equation}
and the corresponding formula for $U_v^{-1}$ show by induction on word length
that the same assertion holds for every $w\in\Lambda_F$.  It follows that
$(\mathcal E,m,T)$ is a $\Lambda_F$-equivariant Kasparov
$C_0(X)$--$A_X$ cycle.

Restrict the $G$-action on $(\mathcal F_X,\pi_X,S_X)$ along $\rho_F$ and form
the equivariant Kasparov product
\begin{equation}\label{eq:FQR-product-cycle}
 [\mathcal E,m,T]\otimes_{A_X}p_X^*(x)
 \in KK^{\Lambda_F}_0(C_0(X),B_X).
\end{equation}
The product theorem quoted in \Cref{thm:kasparov-product} gives an odd product
operator $R$ on
\[
 \mathcal H=\mathcal E\widehat\otimes_{A_X}\mathcal F_X.
\]
Concretely, $R$ is obtained from an $S_X$-connection and Kasparov's technical
theorem.  The defining properties of the product cycle say, for every
$f\in C_0(X)$ and $w\in\Lambda_F$, that
\begin{align*}
 &(R-R^*)m(f),\quad (R^2-1)m(f)\in\mathcal K(\mathcal H),\\
 &\big(\operatorname{Ad}_{U^{\mathcal H}_w}(R)-R\big)m(f)
       \in\mathcal K(\mathcal H),
\end{align*}
where
\[
 U^{\mathcal H}_w(\xi\otimes\eta)
 =U_w\xi\otimes V^X_{\rho_F(w)}\eta.
\]
The operators involved are $B_X$-linear and hence commute with $m(f)$.
Therefore the displayed relations are equivalent to
\[
 R-R^*,\quad R^2-1,\quad
 \operatorname{Ad}_{U^{\mathcal H}_w}(R)-R
 \in\mathcal K_X(\mathcal H).
\]

For $k\in FF$ define
\begin{equation}\label{eq:FQR-product-q}
 q_x(k)=U^{\mathcal H}_{s_k},\qquad
 q_x(k)(\xi\otimes\eta)=q(k)\xi\otimes V^X_k\eta.
\end{equation}
The formula is well defined on the balanced tensor product.  Indeed, for
$a\in A_X$,
\begin{align*}
 q_x(k)((\xi a)\otimes\eta)
 &=q(k)\xi\,\alpha_k(a)\otimes V^X_k\eta\\
 &=q(k)\xi\otimes\pi_X(\alpha_k(a))V^X_k\eta\\
 &=q(k)\xi\otimes V^X_k\pi_X(a)\eta
 =q_x(k)(\xi\otimes\pi_X(a)\eta).
\end{align*}
Moreover, for $g,h\in F$ and an elementary tensor $\xi\otimes\eta$,
\begin{align*}
 &(q_x(g)q_x(h)-q_x(gh))(\xi\otimes\eta)\\
 &\qquad=(q(g)q(h)-q(gh))\xi\otimes V^X_{gh}\eta.
\end{align*}
It follows that
\begin{equation}\label{eq:FQR-product-defect}
 \|q_x(g)q_x(h)-q_x(gh)\|
 \leq\|q(g)q(h)-q(gh)\|<\varepsilon
 \qquad(\varepsilon>0).
\end{equation}
For $\varepsilon=0$, both defects in \eqref{eq:FQR-product-defect} are zero by
the convention following \eqref{eq:FQR-cycle-mult}.
The preceding Fredholm and equivariance relations now show that $(q_x,R)$ is
an $X$-Fredholm $(F,\varepsilon)$-quasi-representation over $B$.  If actual
norm-asymptotic equivariance is needed, \Cref{prop:FQR-asymptotic-equivariance}
provides a family $R_t$ satisfying
\[
 \max_{k\in FF}
 \|\operatorname{Ad}_{q_x(k)}(R_t)-R_t\|\longrightarrow0.
\]

We verify that the assignment $[(q,T)]\mapsto[(q_x,R)]$ is well defined.
A homotopy of $(q,T)$ is a Fredholm quasi-representation over
$C([0,1])\otimes A_X$.  Repeating the free-group construction on that module and
taking the Kasparov product over $C([0,1])\otimes A_X$ gives a homotopy of the
resulting cycles.  \Cref{prop:FQR-asymptotic-equivariance}, applied with
parameter space $[0,1]$, makes
the auxiliary families $R_t$ uniform over this homotopy.  Unitary equivalence
is preserved by tensoring the implementing unitary with the identity.
Stabilization and direct sums are preserved because the Kasparov product is
additive.  A degenerate cycle represents zero in
$KK^{\Lambda_F}_0(C_0(X),A_X)$ and hence has zero product; the
degenerate-cycle and operator-homotopy part of the Kasparov product theorem
gives a stable homotopy to a degenerate product cycle.  Thus the assignment
descends to the stable-homotopy group
$\mathrm{FQR}(X;F,\varepsilon;A)$.

Changing the connection or the product operator changes the product cycle in
\eqref{eq:FQR-product-cycle} by an operator homotopy after adding a degenerate
cycle.  This is again a stable homotopy of Fredholm quasi-representations with
the same finitely many free generators.  If two $G$-equivariant Kasparov
cycles represent the same element $x$, a Kasparov homotopy between stabilized
representatives restricts along $\rho_F$ and its product with
$(\mathcal E,m,T)$ gives the required homotopy of product cycles.  Hence
$x_{F,\varepsilon}$ depends only on $x$.

If $F\subseteq F'$ and $\varepsilon'\leq\varepsilon$, carry out the
construction first for $F'$.  Restricting the resulting operators to the
symbols indexed by $FF$ gives exactly \eqref{eq:FQR-product-q} for
$F$.  A product operator chosen independently at the smaller stage gives the
same stable-homotopy class by the preceding independence argument.  This proves
\eqref{eq:FQR-KK-bonding}.  Now let
\(z=(z_c)_{c\in\mathcal C_G}\in\mathrm{FQR}(X;G;A)\).  If
\(c=(F,\varepsilon)\preccurlyeq
c'=(F',\varepsilon')\), then
\begin{align*}
 r^B_{c',c}\big(x_{F',\varepsilon'}(z_{c'})\big)
 &=x_{F,\varepsilon}\big(r^A_{c',c}(z_{c'})\big)\\
 &=x_{F,\varepsilon}(z_c).
\end{align*}
For \(c=(F,\varepsilon)\), put
\(w_c=x_{F,\varepsilon}(z_c)\).  The preceding computation says that
\(r^B_{c',c}(w_{c'})=w_c\), so
\(w=(w_c)_{c\in\mathcal C_G}\) belongs to
\(\mathrm{FQR}(X;G;B)\).  This defines \(x_*\) coordinatewise, with
\(p_c^B x_*=x_{F,\varepsilon}p_c^A\).  Since maps into an inverse limit
are determined by all their coordinates, this characterization also proves
uniqueness.

For the identity class, the unitary
\[
 \mathcal E\widehat\otimes_{A_X}A_X\longrightarrow\mathcal E,
 \qquad \xi\otimes a\longmapsto\xi a,
\]
carries
\eqref{eq:FQR-product-q} back to $q$, and the identity theorem for the
Kasparov product carries the product operator back to $T$ up to stable
homotopy.  This proves
\((1_A)_{F,\varepsilon}=\operatorname{id}\).  For
$x\in KK^G_0(A,B)$ and $y\in KK^G_0(B,C)$, both iterated constructions have
module
\[
 \mathcal E\widehat\otimes_{A_X}\mathcal F_X
 \widehat\otimes_{B_X}\mathcal G_X
\]
under the associativity unitary
$(\xi\otimes\eta)\otimes\zeta\mapsto
\xi\otimes(\eta\otimes\zeta)$, and on a generator $s_k$ both have
action $q(k)\otimes V^X_k\otimes W^X_k$.  Associativity of the Kasparov
product makes their product operators stably homotopic.  This proves the
composition formula for every fixed \((F,\varepsilon)\).  Applying these
equalities to every coordinate of the inverse limit proves
\((1_A)_*=\operatorname{id}\) and
\((x\otimes_B y)_*=y_*\circ x_*\).  Finally, if $y$ is a $KK^G$-inverse for
$x$, the identity and composition formulas give
\[
 y_{F,\varepsilon}\circ x_{F,\varepsilon}=\operatorname{id},\qquad
 x_{F,\varepsilon}\circ y_{F,\varepsilon}=\operatorname{id}
\]
for every fixed $(F,\varepsilon)$; coordinatewise application gives
\(y_*\circ x_*=\operatorname{id}\) and
\(x_*\circ y_*=\operatorname{id}\).
\end{proof}

\begin{corollary}\label{cor:FQR-KK-equivalence}
If $A$ and $B$ are $KK^G$-equivalent by even classes, then
\[
 \mathrm{FQR}(X;G;A)\cong\mathrm{FQR}(X;G;B)
\]
and, more strongly,
\[
 \mathrm{FQR}(X;F,\varepsilon;A)
 \cong\mathrm{FQR}(X;F,\varepsilon;B)
\]
for every finite $F\subseteq G$ and $\varepsilon\geq0$.  These isomorphisms
are natural in the locally compact $G$-space $X$ and commute with all
restriction maps.
\end{corollary}

\begin{proof}
Choose inverse even classes $x\in KK^G_0(A,B)$ and
$y\in KK^G_0(B,A)$.  The last paragraph of the proof of
\Cref{prop:FQR-KK-product} shows at every $(F,\varepsilon)$ that the two
induced homomorphisms are inverse.  Equation~\eqref{eq:FQR-KK-bonding} gives
compatibility with restriction.  Since the construction uses the exterior
product with $C_0(X)$, it is natural for equivariant maps of locally compact
$G$-spaces.
\end{proof}

\subsection{Compactness tests for coefficient products}

\begin{lemma}[compactness in tensor products]\label{lem:creation-compactness}
Let $(\mathcal M_i,\pi_i)$ be Hilbert $B$-modules carrying left
$A$-representations, and let $D:\mathcal M_0\to\mathcal M_1$ be adjointable.
Assume $\pi_1(a)D$ and $D\pi_0(a)$ are compact for every $a\in A$.
For a Hilbert $A$-module $\mathcal E$ and $\xi\in\mathcal E$, the
adjointable operator $T^1_\xi\eta=\xi\otimes\eta$ satisfies
\[
 T^1_\xi D\in\mathcal K(\mathcal M_0,\mathcal E\widehat\otimes_{\pi_1}\mathcal M_1).
\]
The adjoint version holds as well. If the two correspondences are identified
by a unitary intertwining the left actions, adding an $A$-locally compact
operator to the second-factor operator does not change its connection
conditions.
\end{lemma}
\begin{proof}
The identity $(T^1_\xi)^*T^1_\xi=\pi_1(\langle\xi,\xi\rangle)$ gives
\[
 (T^1_\xi D)^*(T^1_\xi D)
       =D^*\pi_1(\langle\xi,\xi\rangle)D\in\mathcal K.
\]
If $S$ is adjointable and $S^*S$ is compact, a compact approximate identity
$(u_\lambda)$ on its domain satisfies
\[
 \|S(1-u_\lambda)\|^2
 =\|(1-u_\lambda)S^*S(1-u_\lambda)\|\longrightarrow0.
\]
Thus $S$ is a norm limit of the compacts $Su_\lambda$. This proves the
claim; apply it to $D^*$ for the adjoint condition. On the same
correspondence, the difference of two connection defects is $T_\xi D$
with the appropriate grading sign, proving the final assertion.
\end{proof}

\begin{lemma}[tensor products with finite-rank vector bundles]\label{lem:bundle-correspondence-compact}
Let $A$ have a nondegenerate central $C_0(Z)$-action and let $\mathcal V$
be a finite-rank Hermitian vector bundle on $Z$. Then the correspondence
\[
 \mathcal M=\Gamma_0(\mathcal V)\widehat\otimes_{C_0(Z)}A,
 \qquad \pi(a)=1\widehat\otimes L_a,
\]
has $\pi(A)\subseteq\mathcal K(\mathcal M)$. For every Hilbert $A$-module $\mathcal E$ and
$K\in\mathcal K_A(\mathcal E)$, the operator $K\widehat\otimes1$ belongs to
$\mathcal K_A(\mathcal E\widehat\otimes_A\mathcal M)$.
\end{lemma}
\begin{proof}
For $a\in A$, choose $f\in C_c(Z)$ with $\|a-fa\|$ arbitrarily small.
Over the compact support of $f$, finitely many bundle charts and a partition
of unity express $\pi(fa)$ as a norm limit of rank-one module operators.
Indeed local frame vectors reduce the assertion to
$L_{xy^*}=\theta_{x,y}$ on the standard $A$-module, and finite sums of
$xy^*$ are dense in $A$. Passing to the limit gives compactness of $\pi(a)$.
Now $T_\xi^*T_\xi=\pi(\langle\xi,\xi\rangle)$ is compact, so $T_\xi$ is
compact by the preceding lemma. Hence
$\theta_{\xi,\eta}\widehat\otimes1=T_\xi T_\eta^*$ is compact. Finite-rank
approximation proves the second assertion. The proof holds over $C(P)$,
with the same parameter algebra throughout.
\end{proof}

\begin{proposition}[the action of equivariant $K$-theory]
\label{prop:central-K-action}
Let $A$ be a $G$-$C_0(Z)$-algebra with nondegenerate central action. A
class $z\in RKK^G(Z;\mathbb C,\mathbb C)$ determines an even coefficient
class $m_A(z)\in KK^G(A,A)$. This construction is additive, compatible
with products and equivariant pullback, and sends the trivial rank-one
class to the identity of $A$.
\end{proposition}
\begin{proof}
Represent $z$ by a $C_0(Z)$-linear equivariant Kasparov cycle
$(\mathcal V,U,S)$ on a graded Hilbert $C_0(Z)$-module. Set
\[
 \mathcal M=\mathcal V\widehat\otimes_{C_0(Z)}A,
 \qquad \pi(a)=1\widehat\otimes L_a,
 \qquad S_A=S\widehat\otimes1.
\]
The action is diagonal, and the operator commutes with the left action with
the appropriate grading convention. If $D$ is a Fredholm or equivariance
defect of $S$, then $fD$ is compact for $f\in C_c(Z)$. For $a\in A$, choose
such $f$ with $\|a-fa\|$ small. Tensoring a rank-one operator on
$\mathcal V$ with $L_{xy^*}=\theta_{x,y}$ gives a rank-one operator on
$\mathcal M$. It follows by finite-rank approximation that
\[
 (fD)\widehat\otimes L_a\in\mathcal K(\mathcal M).
\]
Balancing identifies this with $D\widehat\otimes L_{fa}$. Norm
approximation proves $\pi(a)D_A\in\mathcal K(\mathcal M)$.
Thus $(\mathcal M,\pi,S_A)$ is an equivariant Kasparov $A$--$A$ cycle.
The same construction over a parameter algebra preserves homotopies and
degenerate cycles. It is additive under direct sum. The associativity
unitaries and the connection criterion identify the construction on a
product with the product of the two constructed correspondences. These
also give pullback compatibility. For the rank-one unit class the module
is $A$ with left multiplication and operator zero, which is the identity
Kasparov class.
\end{proof}

\begin{lemma}[square-root continuity]\label{lem:square-root-norm}
For positive elements $S,T$ in any $C^*$-algebra,
\[
 \|S^{1/2}-T^{1/2}\|\leq\|S-T\|^{1/2}.
\]
\end{lemma}
\begin{proof}
Put $h=\|S-T\|$. In the unitization, $S\leq T+h1$. Monotonicity of the
positive square root and scalar functional calculus give
$S^{1/2}\leq (T+h1)^{1/2}\leq T^{1/2}+\sqrt h\,1$.
Interchange $S,T$ to obtain the two-sided order bound, and hence the norm
bound. Square-root monotonicity follows, for example, by integrating the
order-reversing resolvent in
$S^{1/2}=\pi^{-1}\int_0^\infty S(S+t)^{-1}t^{-1/2}\,dt$,
first with positive invertible elements and then by regularization.
\end{proof}

\subsection{Exterior product with a space}

Let $Y$ be a locally compact right $G$-space.  For a cycle
$z=(\mathcal E,q,T)$ over $C_0(X)\otimes A$, set
\begin{equation}\label{eq:times-Y-cycle}
 \tau_Y(z)=
 \big(C_0(Y)\widehat\otimes\mathcal E,\,q^Y,\,1\widehat\otimes T\big),
 \qquad
 q^Y(g)(f\otimes\xi)=\beta_g(f)\otimes q(g)\xi.
\end{equation}
The completed module in \eqref{eq:times-Y-cycle} is a graded Hilbert
$C_0(Y\times X)\otimes A$-module.

\begin{proposition}[exterior product with $Y$]\label{prop:FQR-times-space}
For every control $c=(F,\varepsilon)$, formula
\eqref{eq:times-Y-cycle} induces a homomorphism
\[
 \tau_{Y,c}:\mathrm{FQR}(X;c;A)
 \longrightarrow\mathrm{FQR}(Y\times X;c;A).
\]
The maps $\tau_{Y,c}$ form a natural transformation
\[
 \tau_{Y,\bullet}:\mathbf{FQR}_{X,A}\Longrightarrow
                    \mathbf{FQR}_{Y\times X,A}.
\]
They are natural in $A$ and in equivariant maps of spaces, commute with the
coefficient action of $KK^G$, and satisfy
$\tau_Z\tau_Y=\tau_{Z\times Y}$ under the unitary
\[
 C_0(Z)\widehat\otimes(C_0(Y)\widehat\otimes\mathcal E)
 \longrightarrow C_0(Z\times Y)\widehat\otimes\mathcal E,\qquad
 k\otimes(f\otimes\xi)\longmapsto(k\otimes f)\otimes\xi.
\]
\end{proposition}

\begin{proof}
The formula for $q^Y(g)$ is balanced over $C_0(X)\otimes A$, and the defining
inner-product identity for $q(g)$ shows that $q^Y(g)$ is an
$\alpha_g$-unitary.  For $g,h\in F$,
\[
 q^Y(g)q^Y(h)-q^Y(gh)
 =\beta_{gh}\widehat\otimes\big(q(g)q(h)-q(gh)\big),
\]
so its norm is at most the original multiplication defect.

If $S\in\mathcal K_X(\mathcal E)$, then
$1\widehat\otimes S\in
\mathcal K_{Y\times X}(C_0(Y)\widehat\otimes\mathcal E)$.  To see this
without an implicit density argument, first take
$a=f\otimes h\in C_c(Y)\odot C_c(X)$.  Then
\[
 a(1\widehat\otimes S)=f\widehat\otimes(hS),\qquad
 (1\widehat\otimes S)a=f\widehat\otimes(Sh),
\]
and both operators are compact because $hS$ and $Sh$ are compact on
$\mathcal E$.  Finite sums of such elementary tensors are dense in
$C_0(Y\times X)$, so the assertion follows by norm closure.  Applying this
observation to $T-T^*$, $T^2-1$, and
$\operatorname{Ad}_{q(g)}(T)-T$ proves that \eqref{eq:times-Y-cycle} is a
Fredholm quasi-representation with unchanged control.

Tensoring the explicit grading-reversal homotopy of
\Cref{def:FQR-final} with $C_0(Y)$ shows that the construction preserves
inverses.  The formula also sends a block direct sum, a degenerate cycle, or a
$P$-parametrized homotopy to the corresponding block direct sum, degenerate
cycle, or $P$-parametrized homotopy over $Y\times X$.  It commutes with every
bonding map because it does not alter $F$ or $\varepsilon$, which proves that
the maps form a natural transformation.  Naturality and the product identity
follow from associativity of exterior tensor products; the coefficient
$KK^G$-product commutes with $\tau_Y$ by associativity of the Kasparov product.
\end{proof}

\subsection{The character of a Fredholm quasi-representation}
\label{subsec:FQR-character-coeff}

For the main theorem, we construct the character of a Fredholm
quasi-representation with $X=\{*\}$. Averaging with a cutoff function on a
proper $G$-space $Y$ produces an equivariant operator. Its spectral gap is
checked on a compact set meeting every orbit and then follows throughout
$Y$ by equivariance. The construction uses only the values of the
quasi-representation on a specified finite set.

\begin{lemma}[cutoff construction of the Fredholm character]

\label{lem:point-character-construction}
Let $Y$ be a locally compact, second countable, proper $G$-compact right
$G$-space. There are a finite symmetric set $F_Y\subset G$ and
$\varepsilon_Y>0$, depending only on the chosen cutoff functions on $Y$, with the following
property. If $c=(F,\varepsilon)$ satisfies $F\supseteq F_Y$ and
$0<\varepsilon\leq\varepsilon_Y$, a Fredholm $c$-cycle over any
separable $G$-$C^*$-algebra $A$, with $X=\{*\}$, determines an equivariant
$C_0(Y)$-linear Fredholm cycle over $C_0(Y)\otimes A$.
The construction preserves direct sums, degenerate cycles, and homotopies.
Its control does not depend on the Hilbert module, the Fredholm operator,
or the coefficient algebra.
\end{lemma}

\begin{proof}
Write the given cycle as $(\mathcal E,q,T)$. Replacing $T$ by a self-adjoint
contraction changes it by a compact operator and gives a homotopy at fixed
$q$, so assume this normalization from the outset.

For $g\in G$, let $\mathcal E_g$ be the coefficient twist with
\[
 \xi\cdot_g a=\xi\alpha_{g^{-1}}(a),\qquad
 \langle\xi,\eta\rangle_g=\alpha_g(\langle\xi,\eta\rangle).
\]
When $q(g)$ is defined, $q(g)^{-1}:\mathcal E\to\mathcal E_g$ is an
$A$-linear unitary. Put
\[
 \mathscr H_{\mathcal E}=\bigoplus_{g\in G}\mathcal E_g,
 \qquad \mathscr T=\bigoplus_{g\in G}T.
\]
The identity on the underlying vectors, from the $g$-coordinate to the
$sg$-coordinate, defines an $\alpha_s$-unitary $R_s$. Indeed the twisted
module rules give $R_s(\xi a)=R_s(\xi)\alpha_s(a)$. The $R_s$ form a group action and commute with $\mathscr T$. On
$C_0(Y)\otimes\mathscr H_{\mathcal E}$ use the diagonal action, with
$\beta_s(f)(y)=f(ys)$ in the space variable.

Choose a cutoff $\theta\in C_c(Y)$ and a compact neighborhood $K$ of its
support. Enlarge $F_Y$ so that all group elements and products occurring
when two points of $K$ are translates of one another are tested. This is a
finite enlargement by properness. For $y$ in a neighborhood of
$\supp\theta$, define the adjointable isometry
\begin{equation}\label{eq:point-cutoff-isometry}
 V_q(y)\xi=\big(\theta(yg)^{1/2}q(g)^{-1}\xi\big)_{g\in G},
 \qquad P_q(y)=V_q(y)V_q(y)^*.
\end{equation}
Only finitely many $g$ occur on $K$, and they are included in $F_Y$.
The cutoff identity proves $V_q^*V_q=1$. Extend $\theta P_q$ by zero;
this is continuous even at the boundary of $\supp\theta$ because $P_q$ is
bounded. Define
\begin{equation}\label{eq:point-cutoff-average}
 a_q(y)=\sum_{s\in G}\theta(ys)R_sP_q(ys)R_s^{-1}.
\end{equation}
Every term uses $P_q$ only where it has been defined. The sum is locally
finite, positive, contractive, and equivariant. No arbitrary
extension of $q$ to all of $G$ is used.

For $y\in\supp\theta$ and a nonzero summand, compare
$R_sV_q(ys)$ with $V_q(y)q(s)$. In the $sk$-coordinate their difference is
\[
 \theta(ysk)^{1/2}
 \big(q(k)^{-1}-q(sk)^{-1}q(s)\big).
\]
Its norm is bounded by the multiplication error for $(s,k)$. The
Hilbert-module column estimate, followed by the estimate for the two range
projections, gives a constant $C_Y$, depending only on the chosen finite
overlap sets, such that
\begin{equation}\label{eq:point-character-gap}
 \|a_q(y)-P_q(y)\|\leq C_Y\varepsilon
 \quad(y\in\supp\theta).
\end{equation}
Choose $\varepsilon_Y$ with $C_Y\varepsilon_Y<1/4$. Since every orbit
meets $\supp\theta$, equivariance propagates the spectral gap to all
of $Y$. Functional calculus therefore defines the equivariant
projection
\begin{equation}\label{eq:point-character-projection}
 e_q=1_{(1/2,\infty)}(a_q).
\end{equation}

The resulting cycle is
\begin{equation}\label{eq:point-character-cycle}
 \left(e_q(C_0(Y)\otimes\mathscr H_{\mathcal E}),\ R,\
                         e_q\mathscr T e_q\right).
\end{equation}
Here and below $R$ also denotes the restriction of the diagonal action to
this invariant range. We verify the compactness conditions. On a compact
subset of $Y$, the matrix for $a_q$ has only finitely many nonzero group
rows and columns. The same is true of $e_q$: functional calculus with a
function vanishing at zero leaves that finite corner invariant. Put $C_V=\mathscr T V_q-V_qT$. Its $g$-component is
\[
 \theta(yg)^{1/2}q(g)^{-1}(\Ad_{q(g)}T-T),
\]
which is compact between the correctly twisted modules. On a compact
region only finitely many components occur. Consequently
\[
 [\mathscr T,P_q]=C_VV_q^*-V_qC_V^*
\]
is locally compact. Averaging, and then the resolvent identity for
functional calculus, give
\[
 [\mathscr T,e_q]\in\mathcal K_Y,
 \qquad e_q(\mathscr T^2-1)e_q\in\mathcal K_Y.
\]
For the second assertion, use that $T^2-1\in\mathcal K_A(\mathcal E)$
and only finitely many twisted diagonal copies occur after space
localization. Thus
\[
 (e_q\mathscr T e_q)^2-e_q
 =e_q(\mathscr T^2-1)e_q
       -e_q\mathscr T(1-e_q)\mathscr T e_q\in\mathcal K_Y.
\]
The operator in \eqref{eq:point-character-cycle} is self-adjoint and equivariant. This proves that it is a cycle.

All operations are block diagonal under direct sum. When the original cycle is degenerate,
$\mathscr T$ commutes with every cutoff matrix entry and is a self-adjoint
unitary, so the resulting cycle is degenerate. For a homotopy, perform the same
construction on the module over $C([0,1])\otimes A$. Its norm is the
parameter-algebra norm, and the spectral gap is the same at every parameter.
This gives a homotopy over the entire parameter interval, without separate choices at its individual points.
Unitary equivalence is respected by taking the direct sum of the twisted
implementing unitaries. The construction consequently passes to
stable-homotopy classes.
\end{proof}

\begin{proposition}[restriction and genuine group actions]
\label{prop:point-character-compatibility}
Denote the homomorphism of \Cref{lem:point-character-construction} by
$\mathcal C_{Y,A,c}$. It commutes with restriction to stronger group controls.
Different cutoff choices, and pullback along maps between proper
$G$-compact spaces, give the same homomorphisms after a common strengthening
of the source control. These assertions also hold for one compactly
parameterized homotopy. For $x\in KK^G(\mathbb C,A)$, let $x[c]$ denote the class obtained by restricting an equivariant
Kasparov representative of $x$ to the group elements tested at $c$. Then
\begin{equation}\label{eq:point-character-exact-input}
 \mathcal C_{Y,A,c}(x[c])=p_Y^*x.
\end{equation}
Restricting a genuine group action to the tested elements also defines
the usual map to $\mathrm{FQR}(Y;c;A)$.
\end{proposition}

\begin{proof}
Restriction of group control leaves every entry used in
\eqref{eq:point-cutoff-isometry}--\eqref{eq:point-cutoff-average} unchanged.
For two cutoffs use $\theta_t=(1-t)\theta_0+t\theta_1$ and take a compact
neighborhood of the union of their supports. One finite overlap set and one
positive spectral-gap bound work for all $t$. The parameterized construction
then gives the claimed equality of classes.

An equivariant map between proper $G$-compact spaces is proper. To see this,
choose a compact set meeting every orbit in the source; the inverse image
of a compact target set is contained in finitely many translates of that
source set, by properness of the target action. Pullback of a cutoff is
therefore a cutoff. With that choice the displayed formulas commute with
pullback, and cutoff independence handles any other choice. Apply the same
argument to a map from $[0,1]\times Y$ for a homotopy.

When $q$ is a genuine action, the range projection $P_q$ is already
equivariant, so $a_q=P_q=e_q$. The isometry $V_q$ identifies its range with
$C_0(Y)\otimes\mathcal E$ and intertwines the group actions. Under this
identification the compressed operator is
\[
 \sum_g\theta(yg)q(g)Tq(g)^{-1}.
\]
Its difference from $T$ is $Y$-locally compact, since on each compact subset the difference is a finite sum of equivariance defects, all of which are compact. The straight
line between these two operators has constant image in the local Calkin
algebra. It is an equivariant Fredholm homotopy, proving
\eqref{eq:point-character-exact-input}.
\end{proof}

\begin{definition}[Fredholm character]
\label{def:FQR-character}
For a proper $G$-compact space $Y$ and a control $c$ strong enough for
\Cref{lem:point-character-construction}, define
\begin{equation}\label{eq:FQR-character-as-composite}
 \chi^F_{Y,A,c}:=\mathcal C_{Y,A,c}:
 \mathrm{FQR}(*;c;A)\longrightarrow RKK^G(Y;\mathbb C,A).
\end{equation}
For $Y=P_d(G)$ and $A=\mathbb C$, write $\chi^F_{d,c}$. A source control
is called \emph{admissible for $d$} if it contains the finite set and has
error at most the threshold chosen for this Rips scale. By
\Cref{prop:point-character-compatibility}, the component maps give a homomorphism
\[
 \chi^F:\mathrm{FQR}(*;G;\mathbb C)
    \longrightarrow\varprojlim_d K_G^0(P_d(G)).
\]
For a non-$G$-compact universal space this is a family of classes on its
$G$-compact stages. It is not an assertion that this family is represented
by one global Kasparov cycle, nor that inverse limits preserve exact
sequences.
\end{definition}

\subsection{Support and compatibility with Kasparov products}

\begin{lemma}[localizing an equivariant Kasparov cycle]

\label{lem:fixed-exact-support}
Let $Y$ be proper and $G$-compact, and let $A$ be a separable
$G$-$C^*$-algebra with a nondegenerate central equivariant $C_0(Z)$-action,
where $Z$ is locally compact and proper. Every class in
$RKK^G(Y;\mathbb C,A)$ has an invariant, self-adjoint representative
whose Fredholm defect vanishes outside a closed $G$-compact subset of
$Y\times Z$. Its class is the extension of a class on the coefficient ideal
of an invariant open neighborhood with $G$-compact closure. One such
neighborhood can be chosen for a fixed compactly parameterized family.
\end{lemma}

\begin{proof}
Cutoff averaging on $Y$, followed by normalization, gives a self-adjoint
contractive invariant operator $T$. Averaging changes $T$ by a
$Y$-locally compact operator: after multiplication by a compactly supported
function of $Y$, only finitely many cutoff translates occur, and the original
equivariance defects are locally compact. These changes preserve the class.
Put $D=T^2-1$.

Choose a compact $K\subset Y$ meeting every orbit and $f\in C_c(Y)$ equal
to one on $K$. The operator $fD$ is compact on the Hilbert
$C_0(Y)\otimes A$-module. Nondegeneracy of the central $C_0(Z)$-action
implies that its approximate units converge in norm when applied to any operator
in $\mathcal K(\mathcal E)$.
Choose $0\leq\phi\leq1$ in $C_c(Z)$ with
\[
 \|(1-\phi)fD\|<1/4.
\]
On the fibers with $y\in K$ and $z\notin\supp\phi$, the norm of $D$ is
less than $1/4$. Invariance propagates this estimate to the complement
of $C=G\cdot(K\times\supp\phi)$. The diagonal action is proper; hence $C$
is closed and $G$-compact. Outside $C$, the spectrum of $T$ avoids
$(-\sqrt3/2,\sqrt3/2)$.

Choose an odd continuous $h:[-1,1]\to[-1,1]$ equal to the sign function
for $|t|\geq4/5$, and put $T'=h(T)$. Since $h(\pm1)=\pm1$, $T'-T$ is
$Y$-locally compact. The new operator is invariant, self-adjoint,
and satisfies $(T')^2=1$ outside $C$.

Choose an invariant open $V\supset C$ with $G$-compact closure, and an
invariant central function $k$, $0\leq k\leq1$, equal to one on $C$ and
supported in $V$. Such choices exist on the locally compact Hausdorff orbit
space. Let $I_V$ be the corresponding ideal of $C_0(Y)\otimes A$ and put
$\mathcal E_V=\overline{\mathcal E I_V}$. The restricted defects are compact
after multiplication by any $f\in C_0(Y)$, now as operators on
$\mathcal E_V$: multiplication by $k$ fixes these defects, and finite-rank
approximations can be cut down on both sides by $k$.

Here is the support reduction at the level of cycles. Put
$\mathcal E_{1-k}=\overline{(1-k)\mathcal E}$. On this module $T'$ is
degenerate. The map
\[
 W:\mathcal E\longrightarrow\mathcal E_V\oplus\mathcal E_{1-k},
 \qquad W\xi=(k^{1/2}\xi,(1-k)^{1/2}\xi)
\]
is an adjointable equivariant isometry intertwining $T'$. Its range
projection has entries $k$, $1-k$, and $\sqrt{k(1-k)}$. The complementary
cycle is degenerate because $(1-k)((T')^2-1)=0$. Thus extension of the
restricted cycle represents the original class. No exactness assertion for
an inverse limit or for an unspecified quotient extension is used.

For a compact parameter space, carry out the proof over its function
algebra. Compactness in that algebra gives one $\phi$ for the entire family;
functional calculus and the isometry $W$ then give the parameterized
conclusion.
\end{proof}

\begin{proposition}[compatibility of the character with Kasparov products]
\label{prop:FQR-character-naturality}
Let $x\in KK^G_0(A,B)$, let $Y$ be proper and $G$-compact, and let $c$ be
admissible for the cutoff construction on $Y$. Then
\begin{equation}\label{eq:point-character-coefficient-naturality}
 \chi^F_{Y,B,c}\,x_{*,c}
   =(p_Y^*x)_*\,\chi^F_{Y,A,c}.
\end{equation}
The admissible control depends on $Y$, not on $x$. The equality is
compatible with direct sums, restriction, and compact parameter spaces.
\end{proposition}

\begin{proof}
Use the coefficient product of \Cref{prop:FQR-KK-product}. Represent $x$ by
$(\mathcal F,\pi,S)$, with group action $U$, and let
$(\mathcal E,q,T)$ be a Fredholm cycle over $A$ with $X=\{*\}$. On $\mathcal E\widehat\otimes_A\mathcal F$, a product
operator $R$ has group operators
$q_x(g)=q(g)\widehat\otimes U_g$. The multiplication estimate is exactly
\eqref{eq:FQR-product-defect}.

For each $g$, the map
\begin{equation}\label{eq:point-product-twist-unitary}
 (\mathcal E\widehat\otimes_A\mathcal F)_g
       \longrightarrow\mathcal E_g\widehat\otimes_A\mathcal F,
 \qquad \xi\otimes\eta\longmapsto\xi\otimes U_g\eta
\end{equation}
is a $B$-linear unitary. For example, the balancing relation follows from
$U_g\pi(a)=\pi(\alpha_g(a))U_g$ and the definition of $\mathcal E_g$.
The direct sum of these maps identifies the regular action with
$R^{\mathcal E}\widehat\otimes U$. In the cutoff isometry, the factor
$U_g^{-1}$ from $q_x(g)^{-1}$ is canceled by $U_g$ in
\eqref{eq:point-product-twist-unitary}. Consequently the following are
\emph{equalities of operators}, not estimates:
\begin{equation}\label{eq:point-product-projection-identities}
 P_{q_x}=P_q\widehat\otimes1,\qquad
 a_{q_x}=a_q\widehat\otimes1,\qquad
 e_{q_x}=e_q\widehat\otimes1.
\end{equation}
This proves that no $x$-dependent multiplication threshold is introduced.

We check the product operator on the identified range. Transport the
diagonal regular sum of $R$ by
\eqref{eq:point-product-twist-unitary}, and denote it by $\mathscr R$.
On each row it is the transported original product operator. Let
$e=e_q\widehat\otimes1$. The equivariance defects of the product cycle are compact, which implies
\[
 [\mathscr R,e]\in\mathcal K_Y.
\]
Indeed the finite matrix entries of the cutoff construction involve the
operators $q_x(g)$, and their commutators with the product operator are
compact; use the same finite-corner functional calculus as in
\Cref{lem:point-character-construction}.

The Kasparov connection criterion may now be checked on sections of
$e_q(C_0(Y)\otimes\mathscr H_{\mathcal E})$ with compact space support
and finitely many group coordinates. These form a dense set. On every such
coordinate the criterion is the original $S$-connection criterion for $R$,
transported by \eqref{eq:point-product-twist-unitary}. The additional terms
replace $U_gSU_g^{-1}$ by $S$; they are compact by
$\pi(a)(U_gSU_g^{-1}-S)\in\mathcal K_B(\mathcal F)$ and the identity
$T_\xi^*T_\xi=\pi(\langle\xi,\xi\rangle)$ for the operator $T_\xi\eta=\xi\otimes\eta$. Finite sums, space localization, and density establish the
connection criterion on the full range.

For positivity, the original product criterion on each of the finitely many
rows gives positivity modulo compact operators for the anticommutator of
the transported $T\widehat\otimes1$ with $\mathscr R$. Compress by $e$.
The discrepancy from the anticommutator with
$(e_q\mathscr T e_q)\widehat\otimes1$ consists of terms containing
$(1-e)\mathscr R e$ or $e\mathscr R(1-e)$, and these are locally compact
by the preceding commutator statement. The compressed operator is therefore
a product operator. Notice that this argument does \emph{not} assert that
an arbitrary $A$-compact operator tensored with $1_{\mathcal F}$ is
$B$-compact; the connection and product-equivariance conditions supply the
compactness actually used.

The connection and positivity characterization of the Kasparov product
\cite[Definition~10 and Theorem~12]{skandalis_1984}
identifies the two Kasparov classes, proving
\eqref{eq:point-character-coefficient-naturality}. The same verification over
$C(P)$ proves the parameterized assertion. Independence of product
operators is the operator-homotopy part of the product theorem; the group
operators and their multiplication bounds do not change in that homotopy.
\end{proof}


\section{Amenability, complete positivity and quasidiagonality}\label{sec:qd}

In this section, we show that for amenable groups, finite-dimensional and Fredholm quasi-representations give the same group: the natural map
\[
  \iota\colon \mathrm{QR}(G)\longrightarrow \mathrm{FQR}(G;\mathbb C),
\]
from the group of finite-dimensional quasi-representations to the group of Fredholm quasi-representations is an isomorphism (\Cref{thm:QR-FQR-amenable}), with a quantitative version in \Cref{prop:QR-FQR-amenable-controlled}. 

Amenability enters through two standard properties of the group $C^*$-algebra $C^*(G)$. First, $C^*(G)$ is nuclear. We use nuclearity through \emph{CP-stability} (\Cref{thm:CP-stability}): a unital map on a finite-dimensional subspace of $C^*(G)$ that is nearly completely contractive is close to a unital completely positive (u.c.p.) map. This allows us to impose complete positivity on our almost representations after small perturbations, which makes them easier to manipulate. Second, $C^*(G)$ is quasidiagonal by a theorem of Tikuisis, White, and Winter (\Cref{thm:TWW-QD}): roughly, it can be approximated by finite-dimensional matrix algebras through u.c.p.\ maps that are almost multiplicative and almost isometric on finite sets.
This allows us to compress operators on infinite-dimensional Hilbert spaces to matrices. The absorbing-representation theorem of Dadarlat and Eilers
\cite[Proposition~2.18]{dadarlat_eilers_2002} lets us carry out these
compressions in one quasidiagonal representation. The relative-projection
construction below gives finite-dimensional representatives and their
stable uniqueness, with the finite sets and tolerances chosen before
the cycle.

\subsection{From matrices to Fredholm quasi-representations}

Let us first define the canonical homomorphism $\iota\colon \mathrm{QR}(G)\longrightarrow \mathrm{FQR}(G;\mathbb C)$. 

Given a control parameter $(F, \varepsilon)$ and $n \in \mathbb N$, we first define a map 
\[
    \mathrm{QRep}(F, \varepsilon; n)\longrightarrow \mathrm{FQRep}(F, \varepsilon;\mathbb C) \, , \quad q \longmapsto (\mathbb C^n\oplus0,q,0)\; , 
\]
where  $\mathbb C^n\oplus0$ is regarded as a graded Hilbert space concentrated in even degree. Note that the triple $(\mathbb C^n\oplus0,q,0)$ is indeed an $\mathrm{FQR}(F,\varepsilon; \mathbb C)$-cycle in the sense of
\Cref{def:FQR-cycle} since $0^2-1=-1$ is compact on a finite dimensional Hilbert space. 

Assembling these maps gives a direct-sum- and homotopy-preserving map
$\mathrm{QRep}(F,\varepsilon)\to\mathrm{FQRep}(F,\varepsilon;\mathbb C)$.
If $q_0\oplus q$ is homotopic to $q_1\oplus q$, add the grading
opposite of the Fredholm cycle associated to $q$ to both sides.
The sum of a cycle with its grading opposite is homotopic to a
degenerate cycle, by the rotation argument in \Cref{def:FQR-final}.
Thus the images of $q_0$ and $q_1$ have the same stable Fredholm class.
The map therefore extends to a group homomorphism

\[
 \iota_{F,\varepsilon}:\mathrm{QR}(F,\varepsilon)
 \longrightarrow \mathrm{FQR}(F,\varepsilon;\mathbb C) \; .
\]

As $(F,\varepsilon)$ varies, it is clear that the maps $\iota_{F,\varepsilon}$ commute with the bonding maps, that is 
\[
 r^{\mathrm{FQR}}_{c',c}\,\iota_{c'}
 =\iota_c\,r^{\mathrm{QR}}_{c',c}
 \quad \text{ for any } c, c' \in \mathcal{C}_G \text{ with } c\preccurlyeq c' \; .
\]
Therefore $\iota_\bullet=(\iota_c)_{c\in\mathcal C_G}$ is a natural
transformation of inverse systems, and it induces the homomorphism
\begin{equation}\label{eq:QR-to-FQR}
 \iota := \varprojlim_{c \in \mathcal{C}_G}\iota_c:\mathrm{QR}(G)
 \longrightarrow\mathrm{FQR}(G;\mathbb C),\qquad
 (x_c)_c\longmapsto(\iota_c(x_c))_c.
\end{equation}

The goal of the section is the following theorem.

\begin{theorem}\label{thm:QR-FQR-amenable}
For every countable amenable discrete group $G$, the canonical homomorphism
\[
 \iota:\mathrm{QR}(G)\longrightarrow\mathrm{FQR}(G;\mathbb C)
\]
of \eqref{eq:QR-to-FQR} is an isomorphism.
\end{theorem}

\begin{proposition}[quantitative version of \Cref{thm:QR-FQR-amenable}]
\label{prop:QR-FQR-amenable-controlled}
The isomorphism in \Cref{thm:QR-FQR-amenable} is induced by the following
maps at finite controls. There exist an
order-preserving map $\sigma:\mathcal C_G\to\mathcal C_G$ with
$\sigma(c)\succeq c$ and homomorphisms
\[
 \rho_c:\mathrm{FQR}(*;\sigma(c);\mathbb C)
 \longrightarrow\mathrm{QR}(c)
\]
such that
\begin{align}
 \rho_c\,\iota_{\sigma(c)}
   &=r^{\mathrm{QR}}_{\sigma(c),c},
   \label{eq:QR-FQR-inverse-left}\\
 \iota_c\,\rho_c
   &=r^{\mathrm{FQR}}_{\sigma(c),c}.
   \label{eq:QR-FQR-inverse-right}
\end{align}
The maps can be chosen to commute with the bonding homomorphisms:
\begin{equation}\label{eq:QR-FQR-inverse-coherence}
 r^{\mathrm{QR}}_{c',c}\rho_{c'}
 =\rho_c r^{\mathrm{FQR}}_{\sigma(c'),\sigma(c)}
 \qquad(c\preccurlyeq c').
\end{equation}
The two consequences used below are:
\begin{enumerate}[label=(\arabic*)]
\item for every $c=(E,\eta)$ there is $c'=(F,\varepsilon)\succeq c$ such that
\[
 r^{\mathrm{FQR}}_{c',c}
   \bigl(\mathrm{FQR}(F,\varepsilon;\mathbb C)\bigr)
 \subseteq
 \iota_c\bigl(\mathrm{QR}(E,\eta)\bigr);
\]
\item for every $c=(E,\eta)$ there is $c'=(F,\varepsilon)\succeq c$ such that
\[
 x\in\mathrm{QR}(F,\varepsilon),\quad \iota_{c'}(x)=0
 \quad\Longrightarrow\quad
 r^{\mathrm{QR}}_{c',c}(x)=0.
\]
\end{enumerate}
\end{proposition}

\subsection{Completely positive almost representations}

If $S\subseteq C^*(G)$ is a finite-dimensional self-adjoint linear subspace
containing $1$, we call $S$ an \emph{operator system}.  A map between operator
systems is u.c.p. if it is unital and completely positive.

For an amenable group, the full and reduced group $C^*$-algebras agree and are
nuclear.  We use the following consequence of nuclearity, in the form proved by
Goldbring--Sinclair \cite{goldbring_sinclair_2015}.

\begin{theorem}[CP-stability]\label{thm:CP-stability}
Let $A$ be a unital nuclear $C^*$-algebra.  For every finite-dimensional
operator system $E\subseteq A$ and every $\delta>0$, there are a
finite-dimensional operator system $S\supseteq E$, an integer $m$, and
$\gamma>0$ with the following property.  If $\phi:S\to B$ is unital and
\[
 \|\phi^{(m)}:M_m(S)\to M_m(B)\|<1+\gamma,
\]
then there is a u.c.p. map $\psi:E\to B$ such that
$\|\psi-\phi|_E\|<\delta$.
\end{theorem}

We will use the following variant of this result in the case of group $C^*$-algebras. 

\begin{lemma}\label{lem:CP-stability-group-ring}
Let $G$ be amenable.  Given a finite symmetric set $E\subseteq G$ containing
$e$ and $\delta>0$, there exist a finite symmetric $F\supseteq E$, an integer
$m$, and $\gamma>0$ such that every unital linear map
\[
 \phi:\mathbb C F\longrightarrow B
\]
with $\|\phi^{(m)}\|<1+\gamma$ is $\delta$-close on $\mathbb C E$ to a u.c.p.
map.
\end{lemma}

\begin{proof}
Let $E_0=\operatorname{span}(E)$; symmetry of $E$ makes it an operator
system.  Apply \Cref{thm:CP-stability} to $(E_0,\delta)$ and denote the resulting choices by $(S,m,\gamma_0)$.  Choose numbers $\kappa,\gamma>0$ satisfying
\begin{equation}\label{eq:group-ring-CP-constants}
 (1+\gamma)(1+\kappa)<1+\gamma_0.
\end{equation}
Choose a self-adjoint linear complement $V$ with $S=E_0\oplus V$ and a
self-adjoint basis $v_1,\ldots,v_N$ of $V$.  Let $C_m$ be the norm of the
linear map which sends an $N$-tuple $(w_1,\ldots,w_N)$ to the $m$th
amplification of the map $S\to C^*(G)$ that is zero on $E_0$ and sends
$v_j$ to $w_j$.  Density of $\mathbb CG$ in $C^*(G)$ permits the choice of
self-adjoint $b_j\in\mathbb CG$ with
$\max_j\|b_j-v_j\|<\kappa/C_m$ when $V\ne0$.  If $V=0$, take
$r=\operatorname{id}_{E_0}$ and omit the $b_j$'s.  If $V\ne0$, define the
unital self-adjoint linear map
\[
 r:S\longrightarrow\mathbb CG,\qquad
 r|_{E_0}=\operatorname{id}_{E_0},\quad r(v_j)=b_j.
\]
By the definition of $C_m$, it satisfies
$\|r^{(m)}-\operatorname{id}_S^{(m)}\|<\kappa$.  Let $F$ be a finite
symmetric set containing $E$ and the supports of the finitely many group-ring
elements used to define $r$.  Then $r(S)\subseteq\mathbb CF$ and
$\|r^{(m)}\|<1+\kappa$.

If $\phi:\mathbb CF\to B$ is unital and
$\|\phi^{(m)}\|<1+\gamma$, equation
\eqref{eq:group-ring-CP-constants} gives
\[
 \|(\phi r)^{(m)}\|
 \leq\|\phi^{(m)}\|\,\|r^{(m)}\|<1+\gamma_0.
\]
The CP-stability theorem therefore supplies a u.c.p. map
$\psi:E_0\to B$ with
$\|\psi-(\phi r)|_{E_0}\|<\delta$.  Since $r|_{E_0}$ is the identity,
this is $\|\psi-\phi|_{E_0}\|<\delta$.
\end{proof}

\begin{lemma}[normalization at the identity]\label{lem:qrep-normalize-identity}
Let $F\subseteq G$ be finite, symmetric, and contain $e$, let
$0<\varepsilon<2$, and let $q:FF\to U(\mathcal E)$ be an
$(F,\varepsilon)$-quasi-representation.  Put $u=q(e)$ and
\[
 \overline q(g)=u^*q(g)\qquad(g\in FF).
\]
Then $\overline q(e)=1$, $\overline q$ is
$(F,2\varepsilon)$-multiplicative, and
\[
 \|\overline q(g^{-1})-\overline q(g)^*\|<2\varepsilon
 \qquad(g\in F).
\]
Moreover, $q$ and $\overline q$ are joined by a path of
$(F,2\varepsilon)$-quasi-representations.  If $(q,T)$ is a Fredholm
quasi-representation, the same path, with $T$ fixed, consists of Fredholm
quasi-representations.
\end{lemma}

\begin{proof}
The relation at $(e,e)$ gives
$\|u-1\|=\|u^*(u^2-u)\|<\varepsilon<2$.  Let
$u_t=\exp(-t\log u)$ be the short unitary path from $1$ to $u^*$ and set
$q_t(g)=u_tq(g)$.  The short-geodesic estimate used in
\Cref{lemma:unitary-implies-homotopy} gives $\|u_t-1\|<\varepsilon$.
For $g,h\in F$,
\begin{align*}
 \|q_t(g)q_t(h)-q_t(gh)\|
 &\leq\|q(g)(u_t-1)q(h)\|
      +\|q(g)q(h)-q(gh)\|<2\varepsilon.
\end{align*}
At $t=1$ this proves the multiplicativity assertion for $\overline q$.
Applying it to $(g,g^{-1})$ and using $\overline q(e)=1$ gives the adjoint
estimate.

For a Fredholm cycle,
$[u,T]=(\operatorname{Ad}_u(T)-T)u$ belongs to the relevant compact ideal.
Continuous functional calculus in the image of
$\mathcal L(\mathcal E)$ modulo that ideal therefore gives
$[u_t,T]$ in the same ideal.  Hence every $q_t$ has the required Fredholm
equivariance defect.
\end{proof}

To apply the CP-stability lemma to quasi-representations, we need the following
quantified implication: after enlarging the finite multiplication test and
decreasing its tolerance, the linear extension is almost completely
contractive on a prescribed finite operator system.

\begin{lemma}[matrix norm estimates for approximate representations]\label{lem:finite-system-cb}
Let $E\subseteq G$ be finite and let $m\in\mathbb N$, $\gamma>0$.  There are a
finite symmetric set $F\supseteq E$ and $\varepsilon>0$ such that, for every
Hilbert space $H$ and every $(F,\varepsilon)$-quasi-representation
$q:FF\to U(H)$ with $q(e)=1$ and
$q(g^{-1})$ within $\varepsilon$ of $q(g)^*$, the linear extension
\[
 \phi_q:\mathbb C F\to B(H),\qquad
 \phi_q\left(\sum a_g g\right)=\sum a_gq(g),
\]
satisfies
$\|\phi_q^{(m)}|_{M_m(\mathbb C E)}\|<1+\gamma$ after normalizing the domain by
the $C^*(G)$ norm.
\end{lemma}

\begin{proof}
Suppose the assertion were false.  Then for every $k$ there would be a Hilbert
space $H_k$ and a $(F_k,1/k)$-quasi-representation $q_k$, where
$F_k\nearrow G$, together with a matrix
$x_k\in M_m(\mathbb C E)$ of $C^*(G)$-norm at most $1$ such that
$\|\phi_{q_k}^{(m)}(x_k)\|\ge1+\gamma$.  After passing to a subsequence the
$x_k$ converge in the finite-dimensional space $M_m(\mathbb C E)$ to an
$x$ of norm at most $1$.

Let
\[
 \mathcal Q=\prod_kB(H_k)\Big/\bigoplus_kB(H_k)
\]
be the norm quotient.  The relations
$q_k(g)q_k(h)-q_k(gh)\to0$ and
$q_k(g^{-1})-q_k(g)^*\to0$ imply that
\[
 g\longmapsto[(q_k(g))_k]\in U(\mathcal Q)
\]
is a genuine unitary representation of $G$.  By the universal property of the
full group $C^*$-algebra it extends to a contractive $*$-homomorphism
$\Phi:C^*(G)\to\mathcal Q$.  Its matrix amplification is also contractive.
But
\[
 \|\Phi^{(m)}(x)\|
 =\limsup_k\|\phi_{q_k}^{(m)}(x)\|
 \ge1+\gamma,
\]
contradicting $\|x\|\le1$.  This contradiction gives a finite set $F$ and a tolerance $\varepsilon$.  Enlarging $F$ to contain inverses and taking
$\varepsilon$ smaller also ensures the stated approximate adjoint condition;
for unitary-valued quasi-representations normalized at $e$, that condition in
fact follows from the multiplicative relation for $(g,g^{-1})$.
\end{proof}

Combining the preceding two lemmas gives the promised conversion.

\begin{corollary}\label{cor:qrep-to-ucp}
For amenable $G$, every prescribed finite set $E$ and $\delta>0$ admit a
stronger quasi-representation control $(F,\varepsilon)$ such that an
$(F,\varepsilon)$-quasi-representation has, on $E$, a linear extension within
$\delta$ of a u.c.p. map on the operator system spanned by $E$.
\end{corollary}

\begin{proof}
Make $E$ symmetric and include the identity.  Apply
\Cref{lem:CP-stability-group-ring} with tolerance $\delta$ and obtain a matrix
level $m$, a larger finite group set, and a number $\gamma>0$.  Apply
\Cref{lem:finite-system-cb} for this $m$ and $\gamma$, enlarging the group
control once more and decreasing the multiplicative defect.  The linear
extension of any quasi-representation at that stronger stage then satisfies
the matrix norm hypothesis of \Cref{lem:CP-stability-group-ring}.  The latter
lemma supplies the desired u.c.p. map on the operator system spanned by $E$.
All choices depend only on $E$ and $\delta$, not on the dimension of the
quasi-representation.
\end{proof}

\begin{lemma}[completely positive approximation of a Fredholm quasi-representation]
\label{lem:CP-cycle-correction}\label{lem:CP-pair-correction}
Let $G$ be amenable, let $E\subseteq G$ be a finite symmetric set containing
$e$, and let $\delta>0$.  There exist a finite symmetric set $F\supseteq E$
and $\varepsilon>0$ with the following properties.
\begin{enumerate}[label=(\roman*)]
\item Let $\mathcal E$ be a separable graded Hilbert space and let $(q,T)$ be
a Fredholm $(F,\varepsilon)$-quasi-representation on $\mathcal E$ with
$q(e)=1$.  Define
\[
 \mathcal D_T=
 \{a\in B(\mathcal E)_{\mathrm{even}}:[a,T]\in K(\mathcal E)\}.
\]
There is a u.c.p. map
\[
 \psi:\operatorname{span}(E)\longrightarrow\mathcal D_T
 \quad\text{such that}\quad
 \max_{g\in E}\|\psi(g)-q(g)\|<\delta.
\]

\item Let $P$ be compact metrizable and let $(q,T)$ be a
$P$-parametrized Fredholm $(F,\varepsilon)$-quasi-representation on a
graded Hilbert $C(P)$-module $\mathcal E_P$, again with $q(e)=1$.  Put
\[
 \mathcal D_T^P=
 \{a\in\mathcal L_{C(P)}(\mathcal E_P)_{\mathrm{even}}:
       [a,T]\in\mathcal K_{C(P)}(\mathcal E_P)\}.
\]
There is a u.c.p. map
\[
 \Psi:\operatorname{span}(E)\longrightarrow\mathcal D_T^P,
 \qquad
 \max_{g\in E}\|\Psi(g)-q(g)\|<\delta.
\]
After evaluation at $s\in P$, the same estimate holds simultaneously for
all $s$, because the displayed norm is the Hilbert $C(P)$-module operator
norm.
\end{enumerate}
\end{lemma}

\begin{proof}
First, $\mathcal D_T$ is a unital $C^*$-subalgebra of
$B(\mathcal E)_{\mathrm{even}}$.  It is closed under products by the Leibniz
formula.  If $a\in\mathcal D_T$, then
\[
 [a^*,T]=-[a,T^*]^*
          =-[a,T]^*-[a,T^*-T]^*\in K(\mathcal E),
\]
so it is also closed under adjoints; norm closedness is immediate.

For a tested group element, the cycle relation gives
\[
 [q(g),T]
 =\big(\operatorname{Ad}_{q(g)}(T)-T\big)q(g)\in K(\mathcal E).
\]
Hence the linear extension $\phi_q$ of $q$ to the finite group ring takes
values in $\mathcal D_T$.

We now specify the control $(F,\varepsilon)$.  Apply
\Cref{lem:CP-stability-group-ring} to $(E,\delta)$, and denote the resulting choices by
$(F_0,m,\gamma)$.  Apply \Cref{lem:finite-system-cb} to
$(F_0,m,\gamma)$; let $(F,\varepsilon)$ be the finite set and tolerance
obtained there.  Enlarge $F$ to be symmetric and to contain $F_0F_0$.
For a Fredholm $(F,\varepsilon)$-cycle, the map $\phi_q$ then satisfies the
$m$-level norm hypothesis of
\Cref{lem:CP-stability-group-ring}.  That lemma gives a u.c.p. map
$\psi:\operatorname{span}(E)\to\mathcal D_T$ with the estimate in (i).
Membership in $\mathcal D_T$ is exactly the compact-commutator assertion.

For (ii), repeat the same proof in the unital $C^*$-algebra
$\mathcal L_{C(P)}(\mathcal E_P)$.  The computation above shows that the
linear extension of $q$ takes values in the single $C^*$-subalgebra
$\mathcal D_T^P$.  The two control lemmas apply to this target algebra with
the same $F$ and $\varepsilon$, since their constants do not depend on the
target.  They produce $\Psi$ with the displayed $C(P)$-module norm estimate.
Thus no fiberwise choice of control, completely positive map, or
stabilization is involved.  If polar correction is subsequently applied,
\Cref{lem:polar-correction} keeps the corrected unitaries in
$\mathcal D_T^P$; if asymptotic equivariance is required, apply
\Cref{prop:FQR-asymptotic-equivariance} with this same parameter space $P$.
\end{proof}

\subsection{Quasidiagonal compression and polar correction}

We use the following deep theorem.

\begin{theorem}[quasidiagonality of amenable group $C^*$-algebras]
\label{thm:TWW-QD}
If $G$ is countable and amenable, then $C^*(G)=C_r^*(G)$ is quasidiagonal.
Equivalently, there exist a faithful representation
$\sigma:C^*(G)\to B(H)$ and finite-rank projections $p_j\uparrow1$ strongly
such that
\[
 \|[p_j,\sigma(a)]\|\longrightarrow0
 \qquad(a\in C^*(G)).
\]
One may also require $\|p_j\sigma(a)p_j\|\to\|a\|$.
\end{theorem}

This is the Rosenberg conjecture proved in the generality needed here by
Tikuisis--White--Winter \cite{tikuisis_white_winter_2017}.  The next computation explains why quasidiagonality
produces quasi-representations.

\begin{lemma}\label{lem:QD-compression}
Let $\sigma:A\to B(H)$ be a representation and $p$ a finite-rank projection.
Set $\psi(a)=p\sigma(a)p$ on $pH$.  If
$\|[p,\sigma(a)]\|<\delta$ for $a$ in a finite set $S$, then
\[
 \|\psi(ab)-\psi(a)\psi(b)\|
 \le\delta\,\|b\|
\]
whenever $a,b,ab\in S$ and $a$ is one of the elements on which the commutator
bound holds.  In particular, for group unitaries $u_g$ in $S$,
$\psi(u_g)$ is an almost unitary and the family
$g\mapsto\psi(u_g)$ is approximately multiplicative.
\end{lemma}

\begin{proof}
Since $p^2=p$,
\begin{align*}
 \psi(ab)-\psi(a)\psi(b)
 &=p\sigma(a)\sigma(b)p-p\sigma(a)p\sigma(b)p\\
 &=p\sigma(a)(1-p)\sigma(b)p.
\end{align*}
Now
$p\sigma(a)(1-p)=p[p,\sigma(a)](1-p)$ up to sign, so its norm is at most
$\|[p,\sigma(a)]\|$.  This proves the estimate.  Taking $b=a^*$ and
$a^*a=1$, and similarly $aa^*=1$, gives the almost-unitary assertions.
\end{proof}

\begin{definition}[local u.c.p. pictures]\label{def:local-ucp-pictures}
Let $\mathcal C_A^{\mathrm{loc}}$ be the set of triples
\[
 c=(S,\Omega,\delta),
\]
where $S\subset A$ is a finite-dimensional operator system,
$\Omega\subset S$ is a finite self-adjoint set containing $1$,
$ab\in S$ for all $a,b\in\Omega$, and $\delta>0$.  We order these controls by
\[
 (S,\Omega,\delta)\preccurlyeq(S',\Omega',\delta')
 \quad\Longleftrightarrow\quad
 S\subseteq S',\quad \Omega\subseteq\Omega',\quad
 \delta'\leq\delta.
\]
A \emph{u.c.p. $c$-almost representation} is a u.c.p. map
$\phi:S\to B(H)$ satisfying
\[
 \|\phi(ab)-\phi(a)\phi(b)\|<\delta
\]
for every tested pair $a,b\in\Omega$.  It is finite-dimensional if $H$ is
finite-dimensional.  A \emph{u.c.p. Fredholm cycle} is a pair $(\phi,T)$ on a
graded Hilbert space $\mathcal E$, where $\phi:S\to B(\mathcal E)_{\rm even}$
is a u.c.p. almost representation and $T$ is odd, with
\[
 T-T^*,\ T^2-1\in K(\mathcal E),\qquad
 [\phi(a),T]\in K(\mathcal E)\quad(a\in S).
\]

Stable homotopy, direct sum, and degenerate cycles are defined exactly as in
\Cref{def:FQR-cycle}, with a norm-continuous path of u.c.p. maps in place of a
path of unitary-valued maps.  For $c=(S,\Omega,\delta)$, let
$\mathrm{QR}_{\mathrm{loc}}^{\mathrm{ucp}}(A;c)$ be the Grothendieck group
of stable-homotopy classes of finite-dimensional u.c.p. $c$-almost
representations, and let
$\mathrm{FQR}_{\mathrm{loc}}^{\mathrm{ucp}}(A;c)$ be the abelian group of
stable-homotopy classes of u.c.p. Fredholm cycles at the same control.
Restriction of the domain and tested set, together with relaxation of the
tolerance, gives bonding homomorphisms whenever $c\preccurlyeq c'$.  We
denote the ordinary inverse-limit groups by
$\mathrm{QR}_{\mathrm{loc}}^{\mathrm{ucp}}(A)$ and
$\mathrm{FQR}_{\mathrm{loc}}^{\mathrm{ucp}}(A)$, respectively, where both
limits are taken over $\mathcal C_A^{\mathrm{loc}}$.  The
subscript records that the maps are defined on finite-dimensional operator
systems, while the superscript records that they are unital and completely
positive.  Thus an element of either
group is a compatible family of classes at all finite controls, exactly as in
\Cref{def:FQR-final}.  The notation records an intermediate model.  For an
amenable group, the comparison with the unitary-valued theory is obtained
from the F\o lner replacement and polar correction in
Section~\ref{subsec:fdap-group}.  For a nonamenable group, polar
correction alone gives only the u.c.p.-to-unitary direction; it does
\emph{not} justify a reverse map or an identification of the two groups.
Under the lifting property, the missing direction, its independence of the
chosen u.c.p. lift, and the two composite identities are proved separately in
\Cref{prop:LP-unitary-ucp-comparison}.  When $A=C^*(G)$ and the finite operator systems are spans of
finite subsets of $G$ and their products, we abbreviate these groups by
$\mathrm{QR}_{\mathrm{loc}}^{\mathrm{ucp}}(G)$ and
$\mathrm{FQR}_{\mathrm{loc}}^{\mathrm{ucp}}(G;\mathbb C)$.
The component homomorphism used below sends a finite-dimensional $\phi$ to the cycle
$(\mathcal E^+=H,\mathcal E^-=0,\phi,0)$, as in the finite-dimensional
construction preceding \eqref{eq:QR-to-FQR}.
\end{definition}

\begin{remark}[relation with controlled $K$-homology]
\label{rem:FQRucp-controlled-K-homology}
Fix a balanced, strongly absorbing even representation
$\pi:A\to B(\mathcal H^+\oplus\mathcal H^-)$, in the terminology of
Willett--Yu.  The formulas below have the following precise consequence: for
each finite set and tolerance in either
$\mathrm{FQR}_{\mathrm{loc}}^{\mathrm{ucp}}(A)$ or the controlled
$K$-homology system, there
is a stronger finite set and smaller tolerance on the other side and a
homomorphism between the corresponding groups; the two composites are the
bonding homomorphisms from the chosen stronger stages to the prescribed
stages, as in \Cref{def:inverse-system-equivalence}.  Consequently they give
inverse homomorphisms of the ordinary inverse limits.  The controlled
$K$-homology groups in question are denoted
$KK^\pi_\varepsilon(\Omega;\mathbb C)$
\cite[Sections~3--5]{willett_yu_uct_2024}; see also
\cite[Sections~5--6 and Appendix~A]{willett_yu_controlled_2026}.

Here is the cycle conversion.  Start with a u.c.p. Fredholm cycle
$(\phi,T)$.  Applying the quasicentral approximate-unit construction from
the proof of \Cref{prop:FQR-asymptotic-equivariance}, with the finitely many
operators $\phi(a)$ in place of the group operators, changes $T$ by a
compact operator and makes
\begin{equation}\label{eq:CP-small-commutator}
 \max_{a\in\Omega}\|[T,\phi(a)]\|<\eta
\end{equation}
for any prescribed finite $\Omega\subset A$ and $\eta>0$.  This operation
can be performed on one $P$-parametrized cycle, so it respects homotopies.
Stinespring's theorem gives an even representation
$\rho:A\to B(\mathcal K)$ and an even projection $P\in B(\mathcal K)$ such
that, after identifying the original space with $P\mathcal K$,
\[
 \phi(a)=P\rho(a)P.
\]
Adding a degenerate Stinespring summand makes $\rho$ balanced and absorbing.
For $a,a^*a\in\Omega$, the identity
\begin{equation}\label{eq:Stinespring-control}
 \|(1-P)\rho(a)P\|^2
 =\|\phi(a^*a)-\phi(a)^*\phi(a)\|
\end{equation}
and its adjoint version show that a multiplicativity defect smaller than
$\delta$ gives
$\|[P,\rho(a)]\|<2\sqrt\delta$.  Thus $P$ is a controlled
$K$-homology projection.

The projection-picture conversion is the relative Fredholm-to-projection
construction, not a replacement of an arbitrary odd Fredholm operator by an
invertible one. In particular, a nonzero Fredholm index must be retained.
One useful exact normalization, for an odd self-adjoint contraction $T$, is
\[
 \mathcal J_T=
 \begin{pmatrix}T&(1-T^2)^{1/2}\\(1-T^2)^{1/2}&-T\end{pmatrix}
 \quad\text{on }\mathcal E\oplus\mathcal E^{\mathrm{op}}.
\]
It is a self-adjoint unitary; the off-diagonal operators are compact. The
original class is recorded by the relative projections, including the
unit projections of the two graded maps. It is not the ungraded class of
$\mathcal J_T$ alone.

We use the controlled Fredholm/projection cycle conversion of
\cite[Sections~3--5]{willett_yu_uct_2024} and
\cite[Sections~5--6 and Appendix~A]{willett_yu_controlled_2026}, with the
relative Stinespring construction just described. Its connection and
operator-homotopy formulation keeps these unit projections. The estimates
\eqref{eq:CP-small-commutator} and \eqref{eq:Stinespring-control} provide the required
control: square roots, spectral cuts separated from the
spectrum, and the relative block rotations use a common smaller tolerance.
Compact errors remain compact by functional calculus and \Cref{lem:creation-compactness}. Conversely, compress the fixed
representation by the relative controlled projections and take the
associated Fredholm operator. These two standard cycle conversions give
the component maps stated above; their stable operator homotopies are
performed in one parameter algebra. A choice of matrix size or of a
compact approximation may depend on the particular cycle, but the source
control for each finite target test does not.

The same construction applies to the full-domain u.c.p. cycles defined later
in \Cref{def:global-ucp-controls}.  In the forward direction one merely
restricts a global map to the finite operator system required by a controlled
$K$-homology stage.  In the reverse direction the compression formula
$a\mapsto P\rho(a)P$ is defined on all of $A$, so the resulting map is again
global and u.c.p.  The block-rotation and Stinespring homotopies are also
global.  Consequently the component maps just described give an
inverse-system equivalence between that full-domain Fredholm system and the
controlled $K$-homology system whenever $A=C^*_{\max}(G)$; this is the
version used in \Cref{lem:surface-controlled-KH-detection}.
\end{remark}

The finite-dimensional approximation proof below uses a direct relative
projection construction; it does not depend on the controlled
$K$-homology comparison in \Cref{rem:FQRucp-controlled-K-homology}.

\begin{theorem}[quasidiagonal representatives of Kasparov classes]
\label{thm:QD-Cuntz-representatives}
Let $D$ be a separable unital nuclear quasidiagonal $C^*$-algebra and let
$B$ be a separable unital nuclear $C^*$-algebra.
\begin{enumerate}[label=(\roman*)]
\item Every class $x\in KK(D,B)$ has a normalized graded Kasparov representative
\[
 (\mathcal E,\pi,J),\qquad
 \mathcal E=(H\otimes B)\oplus(H\otimes B),\qquad
 J=\begin{pmatrix}0&1\\1&0\end{pmatrix},
\]
for which there are finite-rank projections $p_n\in K(H)$ such that
$p_n\uparrow1_H$ strongly and
\begin{equation}\label{eq:QD-representative-point}
 \big\|[(p_n\otimes1_B)\oplus(p_n\otimes1_B),\pi(d)]\big\|
 \longrightarrow0
 \qquad(d\in D).
\end{equation}

\item Let $P$ be a compact metrizable space and put
$B_P=C(P)\otimes B$.  For every class $z\in KK(D,B_P)$ there are a
separable Hilbert space $H$, a normalized graded Kasparov representative
\[
 (\mathcal E_P,\Pi,J),\qquad
 \mathcal E_P=(H\otimes B_P)\oplus(H\otimes B_P),
 \qquad J=\begin{pmatrix}0&1\\1&0\end{pmatrix},
\]
and finite-rank projections $p_n\in K(H)$, independent of $s\in P$, such
that $p_n\uparrow1_H$ strongly and
\begin{equation}\label{eq:QD-representative-parametrized}
 \lim_{n\to\infty}\sup_{s\in P}
 \big\|[(p_n\otimes1_B)\oplus(p_n\otimes1_B),\Pi_s(d)]\big\|=0
 \qquad(d\in D).
\end{equation}
Here $\Pi_s(d)$ is the operator induced by $\Pi(d)$ on the interior
tensor-product fiber
$\mathcal E_{P,s}:=\mathcal E_P\widehat\otimes_{\operatorname{ev}_s}B$; in
other words,
\[
 \Pi_s(d)(\xi\otimes b)=(\Pi(d)\xi)\otimes b
 \quad(d\in D,\ \xi\in\mathcal E_P,\ b\in B).
\]
The homomorphism
$\operatorname{ev}_s:C(P)\otimes B\to B$ is evaluation at $s$.
\end{enumerate}
\end{theorem}

The representation in this normalized Kasparov picture need not be unital.
If $p_\pm=\pi_\pm(1)$, their difference is compact and their relative class
retains the index. Passing back to unital u.c.p. maps means restricting each
map to its own unit-support corner, not replacing both $p_\pm$ by the
identity. This convention is also used in finite-dimensional compression
below.

Here is the precise meaning of the parameter in part~(ii).  A
\emph{$P$-parametrized Fredholm cycle from $D$ to $B$} is, by definition, a
Kasparov cycle for $(D,C(P)\otimes B)$.  It is therefore one cycle, rather
than a set of cycles carrying no specified topology.  For example, a
norm-continuous family $(\mathcal E,\pi_s,F_s)_{s\in P}$ on one fixed graded
Hilbert $B$-module determines a cycle on $C(P,\mathcal E)$ by
\[
 (\Pi(d)\xi)(s)=\pi_s(d)\xi(s),\qquad
 (\mathsf F\xi)(s)=F_s\xi(s).
\]
The Kasparov conditions require the functions
\[
 s\longmapsto (F_s-F_s^*)\pi_s(d),\quad
 s\longmapsto (F_s^2-1)\pi_s(d),\quad
 s\longmapsto [F_s,\pi_s(d)]
\]
to belong to $C(P,K(\mathcal E))$ for every $d\in D$.  Evaluation at
$s$ gives the corresponding fiber cycle.  Stabilization of this
$P$-parametrized cycle means replacing it by
\[
 (\mathcal E\oplus\mathcal E_0,
   \Pi\oplus\Pi_0,\mathsf F\oplus\mathsf F_0),
\]
where $(\mathcal E_0,\Pi_0,\mathsf F_0)$ is one degenerate cycle over
$C(P)\otimes B$, and then applying one adjointable
$C(P)\otimes B$-linear unitary.  These operations take place before
evaluation and hence determine, by evaluation, compatible operations on all
fibers.  Formula~\eqref{eq:QD-representative-parametrized} is the uniform
estimate used below.

\begin{proof}
Fix a faithful essential quasidiagonal representation
$\pi_0:D\to B(H)$; ``essential'' means
$\pi_0(D)\cap K(H)=0$.  The normalized-representative theorem
\cite[Remark~4.5(b)]{dadarlat_2021} states that every $x\in KK(D,B)$ has,
after adding a degenerate cycle, a representative of the form
\[
 \pi(d)=\phi(d)\oplus\psi(d),\qquad
 J=\begin{pmatrix}0&1\\1&0\end{pmatrix},
 \qquad \psi=0\oplus(\pi_0\otimes1_B),
\]
where the zero summand is absorbed in the notation $H$ and
$[J,\pi(d)]$ is compact. Direct multiplication of the two-by-two
matrices shows that this condition is equivalent to
$\phi(d)-\psi(d)\in K(H)\otimes B$.  Since $\pi_0$ is quasidiagonal, choose
finite-rank $p_n\uparrow1$ with
$\|[p_n,\pi_0(d)]\|\to0$.

For each $d$, put $k_d=\phi(d)-\psi(d)\in K(H)\otimes B$.  A strongly
convergent increasing approximate unit of finite-rank projections is
quasicentral for every fixed compact operator: first check this for a
finite-rank operator, where it is immediate once the range and adjoint range
are almost contained in $p_nH$, and then approximate an arbitrary compact
operator in norm.  Thus
\[
 \|[p_n\otimes1,k_d]\|\longrightarrow0.
\]
Combining this with the commutator estimate for $\psi(d)$ gives the estimate
for $\phi(d)$ and hence for the diagonal representation $\pi(d)$.  This
proves part~(i).

For part~(ii), apply part~(i), including its normalization argument, with
$B_P=C(P)\otimes B$ in place of $B$.  The hypotheses are preserved because
$P$ is compact metrizable: $B_P$ is again separable, unital, and nuclear.
Use the same faithful essential quasidiagonal representation
$\pi_0:D\to B(H)$ and write the resulting normalized representation as
\[
 \Pi(d)=\Phi(d)\oplus\Psi(d),\qquad
 \Psi(d)=\pi_0(d)\otimes1_{B_P},\qquad
 k_d:=\Phi(d)-\Psi(d)\in K(H)\otimes B_P.
\]
Choose finite-rank $p_n\uparrow1_H$ with
$\|[p_n,\pi_0(d)]\|\to0$ for every $d\in D$, and set
\[
 P_n=(p_n\otimes1_{B_P})\oplus(p_n\otimes1_{B_P}).
\]
Because $(p_n)$ is an approximate identity for $K(H)$, tensor-product
density gives
\[
 \|[p_n\otimes1_{B_P},k_d]\|\longrightarrow0.
\]
Consequently $\|[P_n,\Pi(d)]\|\to0$.  The latter commutator belongs to
\[
 K(H\oplus H)\otimes C(P)\otimes B
 \cong C\big(P,K((H\otimes B)\oplus(H\otimes B))\big),
\]
whose norm is the supremum of the norms of its evaluations.  Therefore the
last convergence is exactly the uniform convergence asserted in
\eqref{eq:QD-representative-parametrized}.
\end{proof}

\subsection{Finite-dimensional approximation of Fredholm cycles}
\label{subsec:finite-dimensional-approximation}

In the amenable proof, nuclearity and quasidiagonality are used
simultaneously for a Fredholm cycle. We first record its Fredholm index
by a compact difference of projections. Nuclearity allows us to place
this relative class in an absorbing representation, and quasidiagonality
then gives finite-dimensional compressions. The common infinite-dimensional
summand is degenerate and cancels. An interval version of the construction
keeps the finite endpoints and gives the required stable uniqueness.
Nuclearity is not asserted to be logically necessary:
Section~\ref{sec:quantitative-KQD} instead uses the approximate $K$-homology
UCT, a separating family of finite-dimensional representations, and
explicit finite-dimensional existence.

Let $A$ be a separable unital $C^*$-algebra. A local control is a triple
$c=(S,\Omega,\varepsilon)$, where $S\subset A$ is a finite-dimensional
operator system, $\Omega\subset S$ is a finite self-adjoint set containing
$1$, and $ab\in S$ for $a,b\in\Omega$. We first take the elements of
$\Omega$ to be contractions. General bounded finite sets are handled by
rescaling the estimates below.

Write $Q_c=QR_{\mathrm{loc}}^{\mathrm{ucp}}(A;c)$ and
$F_c=FQR_{\mathrm{loc}}^{\mathrm{ucp}}(A;c)$ for the groups in \Cref{def:local-ucp-pictures}. Thus a Fredholm cycle consists of a graded Hilbert
space $\mathcal E=\mathcal E^+\oplus\mathcal E^-$, an even u.c.p. map
$\varphi:S\to\mathcal L(\mathcal E)$, and an odd operator $T$, with
\[
 \begin{gathered}
 T-T^*,\quad T^2-1,\quad [T,\varphi(a)]\in\mathcal K(\mathcal E)
 \quad(a\in S),\\
 \|\varphi(ab)-\varphi(a)\varphi(b)\|<\varepsilon
 \quad(a,b\in\Omega).
 \end{gathered}
\]
The homomorphism $\kappa_c:Q_c\to F_c$ sends a finite-dimensional map
to the cycle concentrated in even degree with operator zero. Bonding
homomorphisms are denoted by $r^Q$ and $r^F$.

\begin{theorem}[finite-dimensional and Fredholm local u.c.p. pictures]\label{thm:FD-Cuntz-approx}
Suppose that $A$ is nuclear and quasidiagonal.
\begin{enumerate}[label=\textup{(\roman*)}]
\item\label{item:FD-CP-surjectivity} For every control $c$ there is $d\succeq c$, chosen independently
of the cycle, such that
\[
 r^F_{d,c}(F_d)\subseteq\kappa_c(Q_c).
\]
\item\label{item:FD-CP-injectivity} For every control $c$ there is $d\succeq c$ such that
\[
 x\in Q_d,\quad \kappa_d(x)=0
 \quad\Longrightarrow\quad r^Q_{d,c}(x)=0.
\]
\item After a further strengthening of the source control there is a
homomorphism $\lambda_{d,c}:F_d\to Q_c$ satisfying
\begin{equation}\label{eq:fdap-composites}
 \kappa_c\lambda_{d,c}=r^F_{d,c},
 \qquad \lambda_{d,c}\kappa_d=r^Q_{d,c}.
\end{equation}
These homomorphisms are eventually compatible: for two target controls
$c\preceq c'$ and their chosen sources $d,d'$, there is $e\succeq d,d'$
for which
\begin{equation}\label{eq:fdap-coherence}
 r^Q_{c',c}\lambda_{d',c'}r^F_{e,d'}
 =\lambda_{d,c}r^F_{e,d}.
\end{equation}
Consequently the natural homomorphism
\[
 \kappa_A:\mathrm{QR}_{\mathrm{loc}}^{\mathrm{ucp}}(A)
 \longrightarrow \mathrm{FQR}_{\mathrm{loc}}^{\mathrm{ucp}}(A)
\]
is an isomorphism of the ordinary inverse limits.
\end{enumerate}
In particular, the theorem applies to $A=C^*(G)$ for every countable
amenable discrete group $G$.
\end{theorem}

The proof establishes equality of controlled Fredholm classes. It does
not infer such equality from agreement of finitely many index pairings.
The main construction replaces a cycle by a pair of projections whose
difference is compact. Quasidiagonal compression is then applied to
this relative compact operator. The relevant absorption theorem is
Dadarlat--Eilers \cite[Definition~2.7 and Proposition~2.18]{dadarlat_eilers_2002}, together
with their Remark~2.3. Their Theorem~5.5, which realizes a class on a
prescribed $K$-triple, is not needed here.

\subsubsection{Estimates and relative projection cycles}

For a u.c.p. map $L:A\to\mathcal L(\mathcal H)$ and a finite set
$\Omega$ of contractions, put
\[
 \operatorname{def}_{\Omega}(L)
 =\max_{a,b\in\Omega}\|L(ab)-L(a)L(b)\|.
\]
If $L_0,L_1$ are u.c.p. and $L_t=(1-t)L_0+tL_1$, then
\begin{align}\label{eq:fdap-convex}
 L_t(ab)-L_t(a)L_t(b)
 ={}&(1-t)\bigl(L_0(ab)-L_0(a)L_0(b)\bigr)\notag\\
 &+t\bigl(L_1(ab)-L_1(a)L_1(b)\bigr)\notag\\
 &+t(1-t)(L_0(a)-L_1(a))(L_0(b)-L_1(b)).
\end{align}
Also, changing a u.c.p. map by at most $u$ on
$\Omega\cup\Omega^2$ increases its multiplication defect by at most
$3u$.

If $p$ is a projection, compression $L_p(a)=pL(a)p|_{p\mathcal H}$ is
unital on its own range, and
\begin{equation}\label{eq:fdap-compression}
 \operatorname{def}_{\Omega}(L_p)
 \leq \operatorname{def}_{\Omega}(L)
       +\max_{a\in\Omega}\|[p,L(a)]\|.
\end{equation}
Indeed the difference from the compressed defect of $L$ is
$pL(a)(1-p)L(b)p$.

We shall use these identities also on Hilbert $C(P)$-modules. All the
norms are then module norms, so the estimates are uniform on $P$.

\begin{lemma}[Relative projection cycle]\label{lem:fdap-relative}
Let $p,e\in\mathcal L(\mathcal H)$ be projections with
$p-e\in\mathcal K(\mathcal H)$. Let $L:A\to\mathcal L(\mathcal H)$ be
u.c.p. and satisfy $[e,L(a)]=0$ for all $a\in A$. Then
\begin{equation}\label{eq:fdap-relative-cycle}
 \mathcal R(p,e;L)=
 \left(
 p\mathcal H\oplus e\mathcal H,
 \begin{pmatrix}pL(\cdot)p&0\\0&eL(\cdot)e\end{pmatrix},
 \begin{pmatrix}0&pe\\ep&0\end{pmatrix}
 \right)
\end{equation}
is a Fredholm u.c.p. cycle whenever its two compressed maps satisfy the
required finite multiplication test. The first summand has even degree
and the second has odd degree.

The construction respects norm-continuous paths of $p$ and $L$ for which
$p-e$ remains compact and $L$ commutes with $e$. If $p=e$, the cycle is
degenerate.
\end{lemma}

\begin{proof}
The operator is self-adjoint and its square differs from the unit
projection by compact operators, because $p-e$ is compact. For the lower
left commutator one has the exact identity
\[
 ep\,pL(a)p-eL(a)e\,ep=e(p-e)L(a)p\in\mathcal K.
\]
The upper right commutator follows by taking adjoints and replacing $a$
by $a^*$. This proves the Fredholm conditions. Equation
\eqref{eq:fdap-compression} supplies the multiplication estimate. If $p=e$,
the two maps coincide and the off-diagonal operator is the identity.
The same formulas over a parameter algebra prove the path assertion.
\end{proof}

This lemma is important because a norm-small perturbation of an
arbitrary Fredholm cycle need not preserve compact commutators. In the
relative picture above, compactness is preserved \emph{exactly} during
all of the completely positive interpolations used below.

\subsubsection{From local maps to global maps}

\begin{lemma}[Nuclear globalization at a finite control]\label{lem:fdap-global}
Fix a finite-dimensional operator system $S_0\subset A$, a finite set
$L\subset S_0$, and $u>0$. If $A$ is nuclear, there is a
finite-dimensional operator system $S_1\supset S_0$ such that every
u.c.p. map $\varphi:S_1\to D$ admits a global u.c.p. map
$\Phi:A\to D$ with
\[
 \max_{a\in L}\|\Phi(a)-\varphi(a)\|<u.
\]
Here $D$ is any unital $C^*$-algebra, including a compact-commutator
algebra for one parameterized Fredholm cycle.
\end{lemma}

\begin{proof}
Use the unital completely positive approximation property to choose
u.c.p. maps
\[
 A\xrightarrow{\alpha}M_k\xrightarrow{\beta}A,
 \qquad \|\beta\alpha(a)-a\|<u\quad(a\in L).
\]
Take $S_1$ to contain $S_0$ and the finite-dimensional operator system
$\beta(M_k)$. Put $\Phi=\varphi\beta\alpha$. Both terms used in
$\varphi(\beta\alpha(a)-a)$ lie in $S_1$, so contractivity gives the
estimate. The same maps $\alpha,\beta$ work for every $D$.
\end{proof}

For a Fredholm cycle take
\[
 D=D_T=\{a\in\mathcal L(\mathcal E)_{\mathrm{even}}:
                  [a,T]\in\mathcal K(\mathcal E)\}.
\]
This is a unital $C^*$-algebra, since $T-T^*$ is compact. Thus the global
map still has compact commutators with $T$ on \emph{all} of $A$.
The convex path from $\varphi|_{S_0}$ to $\Phi|_{S_0}$ remains inside
$D_T$. If the original cycle is degenerate, the construction instead
takes values in the exact commutant of $T$ and preserves degeneracy.
For a parameterized cycle use the corresponding algebra in
$\mathcal L_{C(P)}(\mathcal E)$.

\subsubsection{The relative index projection}

Fix the target $c=(S_0,\Omega,\varepsilon)$. Choose
\begin{equation}\label{eq:fdap-budget}
 0<r<\min\{\varepsilon/1000,1/1000\},\qquad
 0<u<r^2/12,\qquad 0<\delta<r^2/4.
\end{equation}
Apply Lemma~\ref{lem:fdap-global} to
$L=\Omega\cup\Omega^2$ and $u$, and take
$d=(S_1,\Omega,\delta)$. These choices precede the variable cycle.
Starting with a cycle at $d$, the global map $\Phi$ satisfies
\begin{equation}\label{eq:fdap-global-defect}
 \operatorname{def}_{\Omega}(\Phi)<\Delta:=\delta+3u<r^2.
\end{equation}
The preceding convex path has defect less than $r^2$.

Normalize $T$ to an odd self-adjoint contraction, changing it only by a
compact operator. This is a Fredholm homotopy with the u.c.p. map fixed.
An even quasicentral positive compact contraction $h$ then gives
\[
 B=(1-h)T(1-h).
\]
One can arrange that $B$ has arbitrarily small commutators with the
finitely many $\Phi(a)$, $a\in L$. Indeed
\[
 \|[B,\Phi(a)]\|
 \leq 2\|T\|\,\|[h,\Phi(a)]\|
     +\|(1-h)[T,\Phi(a)](1-h)\|.
\]
The first term tends to zero by quasicentrality and the second because
$[T,\Phi(a)]$ is compact. The change from $T$ to $B$ is compact and
therefore gives another Fredholm homotopy.

Put $d_B=(1-B^2)^{1/2}$. We choose $h$ far enough out that
\begin{equation}\label{eq:fdap-small-commutators}
 \|[B,\Phi(a)]\|<r,\qquad
 \|[d_B,\Phi(a)]\|<r\quad(a\in L).
\end{equation}
There is no operator-Lipschitz assumption in the second assertion.
For $f(t)=(1-t^2)^{1/2}$ on $[-1,1]$, approximate $f$ by a polynomial
$p(t)=\sum c_jt^j$. For every contraction $a$,
\[
 \|[f(B),a]\|
 \leq2\|f-p\|_\infty
    +\left(\sum_jj|c_j|\right)\|[B,a]\|.
\]
This proves a modulus tending to zero uniformly in the Hilbert module.

Write
\[
 B=\begin{pmatrix}0&S^*\\S&0\end{pmatrix},\quad
 d_+=(1-S^*S)^{1/2},\quad d_-=(1-SS^*)^{1/2},\quad
 e_0=\begin{pmatrix}0&0\\0&1\end{pmatrix}.
\]
The operators $d_+,d_-$ are compact. The matrix
\[
 W=\begin{pmatrix}d_+&-S^*\\S&d_-\end{pmatrix}
\]
is unitary. This follows from $Sd_+=d_-S$, which follows by polynomial
approximation from $S(S^*S)^j=(SS^*)^jS$.
Define
\begin{equation}\label{eq:fdap-index-projection}
 p_0=W(1-e_0)W^*
 =\begin{pmatrix}
 d_+^2&d_+S^*\\Sd_+&SS^*
 \end{pmatrix}.
\end{equation}
Then
\begin{equation}\label{eq:fdap-relative-compact}
 p_0-e_0\in\mathcal K(\mathcal E),\qquad
 \|[p_0,\Phi(a)]\|<4r\quad(a\in L).
\end{equation}
For the estimate write $W=d_B-\Gamma B$, where
$\Gamma=\operatorname{diag}(1,-1)$ commutes with $\Phi$; use
\eqref{eq:fdap-small-commutators} and the fact that $1-e_0$ commutes with
$\Phi$.

The original cycle is stably homotopic, at the target control, to
$\mathcal R(p_0,e_0;\Phi)$. Here is the class verification. The isometry
\[
 V_+:\mathcal E^+\longrightarrow p_0\mathcal E,
 \qquad V_+\xi=(d_+\xi,S\xi)
\]
is onto. Under $V_+$ and the identification
$e_0\mathcal E=\mathcal E^-$, the relative Fredholm operator is the
original $B$, while the positive map becomes
\[
 \widetilde\Phi^+(a)
   =d_+\Phi^+(a)d_++S^*\Phi^-(a)S.
\]
Its difference from $\Phi^+(a)$ is
\begin{equation}\label{eq:fdap-positive-difference}
 d_+[\Phi^+(a),d_+]
       +S^*\bigl(\Phi^-(a)S-S\Phi^+(a)\bigr).
\end{equation}
This is compact for every $a\in A$ and has norm less than $2r$ for
$a\in L$. Hence the convex path between the two positive maps has
compact commutators with the unchanged Fredholm operator. Its
multiplication defect is less than $\Delta+6r<7r<\varepsilon$.

The projection $p_0-e_0$ retains the Fredholm index. For example, when
$B=0$ on a finite-dimensional graded space,
$p_0=\operatorname{diag}(1,0)$ and $e_0=\operatorname{diag}(0,1)$.
Their rank difference is $\dim\mathcal E^+-\dim\mathcal E^-$.

\subsubsection{Stinespring dilation without losing the relative class}

Take an even Stinespring dilation
\[
 \Phi(a)=V^*\rho(a)V,
 \qquad V:\mathcal E\longrightarrow\mathcal H
\]
with $\rho$ a unital representation. This construction also works on
Hilbert $C(P)$-modules: complete $A\odot\mathcal E$ using the inner
product determined by $\Phi(a^*b)$, let $\rho$ act by left multiplication,
and set $V\xi=1\otimes\xi$. Separability gives a countably generated
module. Put $P=VV^*$, and let $e$ be the projection onto the negative
Stinespring summand. Thus $[e,\rho(a)]=0$, $[e,P]=0$, and $eV=Ve_0$.
Set
\begin{equation}\label{eq:fdap-lifted-projection}
 \widehat p=Vp_0V^*+e(1-P).
\end{equation}
This is a projection and
\[
 \widehat p-e=V(p_0-e_0)V^*\in\mathcal K(\mathcal H).
\]

For $a\in\Omega$, Stinespring's identity gives
\begin{equation}\label{eq:fdap-stinespring-estimate}
 \|(1-P)\rho(a)P\|^2
   =\|\Phi(a^*a)-\Phi(a)^*\Phi(a)\|<\Delta<r^2.
\end{equation}
Applying the identity also to $a^*$ controls the other off-diagonal
corner. Here $\Omega$ is self-adjoint, so both required pairs are tested.

Define the block-diagonal u.c.p. map
\[
 \rho^{\mathrm{bd}}(a)=P\rho(a)P+(1-P)\rho(a)(1-P).
\]
It is within $r$ of $\rho$ on $\Omega$, has multiplication defect at
most $r^2$, and commutes with $e$. By \eqref{eq:fdap-lifted-projection},
\[
 \mathcal R(\widehat p,e;\rho^{\mathrm{bd}})
 \cong\mathcal R(p_0,e_0;\Phi)\oplus\text{a degenerate cycle}.
\]
The degenerate cycle is the common summand $e(1-P)\mathcal H$ in both
degrees.

With respect to $P\mathcal H\oplus(1-P)\mathcal H$, the diagonal part
of $[\widehat p,\rho(a)]$ is $[p_0,\Phi(a)]$ and zero. Its off-diagonal
parts contain $p_0-e_0$ multiplied by the off-diagonal parts of $\rho(a)$.
Since the difference of two projections has norm at most one,
\begin{equation}\label{eq:fdap-lifted-commutator}
 \|[\widehat p,\rho(a)]\|<6r\quad(a\in\Omega).
\end{equation}
Interpolate from $\rho^{\mathrm{bd}}$ to $\rho$. Equation
\eqref{eq:fdap-convex} bounds the multiplication defect of the interpolating
map by $2r^2$, and its commutator with $\widehat p$ by $6r$.
Equations \eqref{eq:fdap-compression} and \eqref{eq:fdap-budget} show that the
relative cycles remain within the target control. Their compact
commutators are exact by Lemma~\ref{lem:fdap-relative}. Thus the original
cycle is stably homotopic to
$\mathcal R(\widehat p,e;\rho)$.

\subsubsection{Absorption into a quasidiagonal representation}

Choose a faithful unital quasidiagonal representation
$\sigma_0:A\to\mathcal L(H)$ of infinite multiplicity. It has finite-rank
projections $q_n\to1$ strongly with
\[
 \|[q_n,\sigma_0(a)]\|\longrightarrow0\quad(a\in A).
\]
Infinite multiplicity does not change this property: use the first $n$
copies of a quasidiagonal compression and then a diagonal sequence.

We use the following precise external result. If $C$ is a unital
coefficient algebra, then the admissible scalar representation
$\sigma_0\otimes1_C$ on $H\otimes C$ is nuclearly absorbing
\cite[Proposition~2.18]{dadarlat_eilers_2002}. A representation of a nuclear source is
strictly nuclear \cite[Remark~2.3]{dadarlat_eilers_2002}. Consequently, for every unital
representation $\tau:A\to\mathcal L_C(\mathcal M)$, every finite
$L\subset A$, and every $v>0$, there is a unitary
\[
 U:\mathcal M\oplus(H\otimes C)\longrightarrow H\otimes C
\]
with
\[
 \max_{a\in L}
 \|U(\tau(a)\oplus\sigma_0(a)\otimes1_C)U^*
                   -\sigma_0(a)\otimes1_C\|<v.
\]
We use $C=\mathbb C$ for a single cycle and $C=C([0,1])$ for a
homotopy. This is not an invocation of a finite-dimensional existence
theorem for prescribed $K$-theory pairings.

Add the relative pair with equal projections
$e_{\sigma}=\operatorname{diag}(0,1)$ on two copies of the scalar
representation; its relative cycle is degenerate. Apply absorption
separately in the two ambient grading summands, with accuracy $r$ on
$L=\Omega\cup\Omega^2$. We obtain a projection $p$ on $H\oplus H$, a
reference projection $e=\operatorname{diag}(0,1)$, and a representation
$\rho'$ with
\[
 p-e\in\mathcal K(H\oplus H),\qquad [e,\rho'(a)]=0,
 \qquad \|\rho'(a)-\pi(a)\|<r\quad(a\in L),
\]
where $\pi=\sigma_0\oplus\sigma_0$. The relative cycle for $\rho'$ is
unitarily equivalent to the preceding one plus a degenerate cycle.

Interpolate between $\rho'$ and $\pi$. Both endpoints are
representations, so the multiplication defect of the interpolation is
at most $r^2/4$ on $\Omega$. Moreover
\[
 \|[p,\pi(a)]\|<8r\quad(a\in\Omega),
\]
and the same upper bound holds along the interpolation. The resulting
relative cycles have defect less than $9r$. They have exact compact
commutators by Lemma~\ref{lem:fdap-relative}. Thus we have reduced to
\begin{equation}\label{eq:fdap-fixed-representation}
 \mathcal R(p,e;\pi),\qquad p-e\in\mathcal K,
 \qquad \max_{a\in\Omega}\|[p,\pi(a)]\|<8r.
\end{equation}

\subsubsection{Finite-dimensional compression and exact cancellation of the tail}

Put $Q_n=q_n\oplus q_n$. These projections commute with $e$. Since
$p-e$ is compact,
\[
 \|(1-Q_n)(p-e)\|+\|(p-e)(1-Q_n)\|\longrightarrow0.
\]
The positive contractions
\[
 a_n=e+Q_n(p-e)Q_n=Q_npQ_n+e(1-Q_n)
\]
therefore converge in norm to $p$. For large $n$, put
\begin{equation}\label{eq:fdap-finite-projection}
 p_n=\chi_{(1/2,\infty)}(a_n)=f_n+e(1-Q_n),
 \qquad f_n\leq Q_n.
\end{equation}
Choose $n$ so large that the spectral-projection path
\[
 p_t=\chi_{(1/2,\infty)}\bigl((1-t)p+ta_n\bigr),\qquad0\leq t\leq1,
\]
satisfies $\|p_t-p\|<r$, and also
$\max_{a\in\Omega}\|[Q_n,\pi(a)]\|<r$.
Each $p_t-e$ is compact. The relative cycles
$\mathcal R(p_t,e;\pi)$ give a homotopy, with multiplication defect
less than $10r$.

Now delete the off-diagonal corners of $\pi$ relative to $Q_n$:
\[
 \pi^{(n)}(a)=Q_n\pi(a)Q_n+(1-Q_n)\pi(a)(1-Q_n).
\]
Its distance from $\pi$ on $\Omega$ is less than $r$, and its
multiplication defect is at most $r^2$. Interpolate from $\pi$ to
$\pi^{(n)}$. The interpolation has defect at most $2r^2$ and commutator
with $p_n$ less than $12r$. Hence the relative cycles have defect less
than $13r<\varepsilon$.

At the endpoint, both the projections and the map are block diagonal.
The part on $(1-Q_n)\mathcal H$ has identical projection $e(1-Q_n)$,
identical u.c.p. maps, and off-diagonal identity. It is exactly
degenerate, not a discarded norm-small tail. The remaining cycle is
finite-dimensional, on
\[
 f_n\mathcal H\oplus eQ_n\mathcal H.
\]
Its Fredholm operator can be deformed to zero because every operator
on a finite-dimensional space is compact. The resulting class is
\begin{equation}\label{eq:fdap-finite-output}
\begin{gathered}
 [\varphi_{+,n}]-[\varphi_{-,n}],\\
 \varphi_{+,n}(a)=f_n\pi(a)f_n|_{f_n\mathcal H},\qquad
 \varphi_{-,n}(a)=eQ_n\pi(a)eQ_n|_{eQ_n\mathcal H}.
\end{gathered}
\end{equation}
The maps are unital on their own ranges. Their ranks need not agree.
This proves Theorem~\ref{thm:FD-Cuntz-approx}(i).

All multiplication bounds above are less than $13r$, except for the
initial replacement bound $7r$, and therefore are strictly below
$\varepsilon$. Only the operator system $S_1$ and $\delta$ were needed
as source control. The compact cutoff, the absorption unitary and $n$
may depend on the cycle, but the source control does not.

\subsubsection{Relative homotopies and the vanishing assertion}

We now prove Theorem~\ref{thm:FD-Cuntz-approx}(ii); in particular, homotopy
compatibility is not inferred merely from pointwise existence.

Fix the target $c$ and choose the stronger control as in the preceding
construction, using the same budget \eqref{eq:fdap-budget}. Suppose that a
finite-dimensional graded cycle $v$ at this stronger control has zero
class in the Fredholm group. By stable homotopy, after adding degenerate
cycles there is one homotopy from $v\oplus D_0$ to a degenerate cycle
$D_1$. Regard it as a cycle on a Hilbert $C([0,1])$-module. Its
finite-dimensional summand at the first endpoint has operator zero.

Apply nuclear globalization once, with coefficient algebra
$D_T\subset\mathcal L_{C([0,1])}(\mathcal E)$. At the endpoints it
preserves the direct sums and the degeneracy of $D_0,D_1$. Normalize
the Fredholm operator using functions which fix $0,1,-1$, so its
endpoint operators are unchanged.

The quasicentral cutoff can also be made relative to the endpoints.
After normalization, all the compact commutators $[T,\Phi(a)]$ vanish
at both endpoints. Work in the closed ideal
\[
 J=\{k\in\mathcal K_{C([0,1])}(\mathcal E):k|_0=k|_1=0\}.
\]
It is invariant under multiplication by the adjointable operators in
use. A quasicentral approximate identity of $J$, chosen even, makes
$(1-h)T(1-h)$ satisfy \eqref{eq:fdap-small-commutators} without changing its
endpoint operators. Only the commutators, not the endpoint Fredholm
defects, need to belong to $J$.

Carry out the index projection, Stinespring dilation and absorption
on this one module. Absorption uses
$\sigma_0\otimes1_{C([0,1])}$ and gives one $C([0,1])$-linear unitary
in each ambient grading. We obtain
\[
 p(t)-e\in C([0,1],\mathcal K(H\oplus H)),
 \qquad\sup_t\max_{a\in\Omega}\|[p(t),\pi(a)]\|<8r.
\]
At the degenerate endpoint $p(1)=e$. At the initial endpoint
\begin{equation}\label{eq:fdap-endpoint-form}
 p(0)=e+R_+-R_-,\qquad
 R_+\leq1-e,\quad R_-\leq e,
\end{equation}
where $R_+,R_-$ are the images of the finite-dimensional positive and
negative summands of $v$. The infinite-dimensional degenerate summand
contributes no difference from $e$.

A single finite-rank cutoff can be chosen for the whole homotopy. Indeed
$\{p(t)-e:0\leq t\leq1\}$ is a compact subset of the compact operators,
so its tails are uniformly small. The quasidiagonal bounds for $Q_n$
are already uniform because $\pi$ is fixed.
We may moreover require
\begin{equation}\label{eq:fdap-contain-endpoints}
 R_++R_-\leq Q_n.
\end{equation}
To justify this last requirement, a strongly convergent sequence of
finite-rank projections almost contains any prescribed finite-dimensional
space. In each ambient grading, rotate the range of its compression
onto the prescribed space by a unitary tending to $1$. Conjugating
$Q_n$ by these two small unitaries makes the containment exact, keeps
$[Q_n,e]=0$, and changes its commutators with $\pi(a)$ by a quantity
tending to zero. Monotonicity of the resulting cutoffs is not needed.

Apply \eqref{eq:fdap-finite-projection} uniformly. Since
\eqref{eq:fdap-contain-endpoints} holds,
\[
 f_n(0)=eQ_n+R_+-R_-,\qquad f_n(1)=eQ_n.
\]
The finite projection $f_n(t)$ defines a finite-dimensional vector bundle
over $[0,1]$. It is trivial, and a trivialization turns
\eqref{eq:fdap-finite-output} into a continuous finite-dimensional u.c.p.
homotopy. Its endpoint at $1$ has identical positive and negative maps
and therefore has zero class in $Q_c$.

We verify its endpoint at $0$, rather than presuming that an arbitrary
compression fixes the initial finite-dimensional class. Put
\[
 R=eQ_n-R_-.
\]
The positive space at $0$ is $R_+\mathcal H\oplus R\mathcal H$, and the
negative space is $R_-\mathcal H\oplus R\mathcal H$. The positive map
is already block diagonal in this decomposition because $\pi$ commutes
with $e$. The off-diagonal part of the negative map has norm less than
$3r$ on $\Omega$. To see this, before absorption apply Stinespring's
identity to the original finite-dimensional negative summand. Its
multiplication defect is less than $r^2$, so its Stinespring range has
commutator less than $r$ with the dilated representation. Absorption
changes the commutator bound by at most $2r$.

Delete this finite off-diagonal part through the convex u.c.p. path.
The two occurrences of the map on $R\mathcal H$ are now identical and
cancel in the Grothendieck group. The maps on $R_+\mathcal H$ and
$R_-\mathcal H$ are, after their identifying unitaries, within $r+u$ of
the original finite-dimensional maps on
$\Omega\cup\Omega^2$. This is the sum of the absorption error and the
initial nuclear globalization error. Convex paths connect these pairs.
The common map on $R\mathcal H$ has defect less than $13r$, since it
is a direct summand of the positive endpoint map. The compression to
$R_-\mathcal H$ has defect at most $9r^2$ by its $3r$ commutator
bound with the representation. Thus the block-diagonal negative map
has defect less than $13r$. Formula \eqref{eq:fdap-convex} bounds the
interpolation that deletes its off-diagonal entries by
$13r+9r^2/4<14r$. The last convex interpolations, whose endpoints are
within $r+u$ on the tested elements, also remain below $50r<\varepsilon$.
All of these constants are independent of the dimensions.
Thus the finite endpoint at $0$ represents exactly the restriction of
$v$ in $Q_c$. Since it is joined to zero by a finite-dimensional
homotopy, its class vanishes. This proves (ii).

\begin{remark}
For a general compact metrizable parameter space $P$, the construction
produces finite-dimensional vector bundles over $P$. They need not be
trivial bundles. For $P=[0,1]$ they are trivial, which is precisely the
parameterized assertion required for the stable-homotopy groups in
Section~\ref{sec:qd}. 
\end{remark}

\subsubsection{Homomorphisms, compatibility and inverse limits}

We prove part (iii) with all intermediate controls displayed.
Fix $c$. Choose $t\succeq c$ using part (ii), so that
\begin{equation}\label{eq:fdap-vanish-to-c}
 \kappa_t(x)=0\quad\Longrightarrow\quad r^Q_{t,c}(x)=0.
\end{equation}
Choose $d\succeq t$ using part (i) with target $t$.
For $z\in F_d$, choose $v\in Q_t$ such that
\[
 \kappa_t(v)=r^F_{d,t}(z)
\]
and define
\begin{equation}\label{eq:fdap-lambda-definition}
 \lambda_{d,c}(z)=r^Q_{t,c}(v).
\end{equation}
If $v'$ is another choice, then $\kappa_t(v-v')=0$ and
\eqref{eq:fdap-vanish-to-c} gives the same output. This also proves
independence from representatives of $z$. For $z_1,z_2$ one may choose
$v_1+v_2$ as a lift of $z_1+z_2$, so $\lambda_{d,c}$ is a homomorphism.
Its image under $\kappa_c$ is $r^F_{d,c}(z)$. If $z=\kappa_d(x)$, use
$v=r^Q_{d,t}(x)$; this proves both identities in
\eqref{eq:fdap-composites}.

For compatibility let $c\preceq c'$ and let $t,t'$ and $d,d'$ be the
controls just used. Choose $w\succeq t,t'$ and then a source
$e\succeq w,d,d'$ such that every class in $F_e$ has a finite-dimensional
representative at $w$. For $z\in F_e$ choose that representative
$v\in Q_w$. Its restrictions to $t$ and $t'$ are admissible choices
in the definition of the two homomorphisms. Thus
\[
 \lambda_{d,c}r^F_{e,d}(z)=r^Q_{w,c}(v),\qquad
 \lambda_{d',c'}r^F_{e,d'}(z)=r^Q_{w,c'}(v),
\]
which proves \eqref{eq:fdap-coherence}. The source $e$ is chosen before $z$.

For a compatible family $z=(z_b)_b$ in the inverse limit of the
Fredholm groups define
\[
 (\lambda z)_c=\lambda_{d,c}(z_d).
\]
Eventual compatibility proves that this is a compatible family and
independent of all the choices. The two composite identities show
coordinate by coordinate that $\lambda$ and $\kappa$ are inverse.
No exactness property of inverse limits is used. Nor is a countable
cofinal family of \emph{all} finite-dimensional operator systems
assumed; compatibility was proved directly for their directed set.
This proves all three assertions of \Cref{thm:FD-Cuntz-approx}.

\subsection{Amenable groups and the unitary-valued formulation}\label{subsec:fdap-group}

For a countable amenable group, $C^*(G)=C_r^*(G)$ is separable, unital
and nuclear. Its quasidiagonality is the amenable-group consequence of
Tikuisis--White--Winter \cite{tikuisis_white_winter_2017}. The preceding theorem therefore
proves the local u.c.p.
finite-dimensional approximation statement needed here.

There is also a direct way to pass from the unitary-valued Fredholm
cycles of Section~\ref{sec:qd} to the global u.c.p. maps used in the proof.
It provides explicit finite controls and avoids any extension of a
local u.c.p. map to an unspecified target algebra.

Fix a finite group set $E$ and a small $v>0$. Choose a finite left
F\o lner set $K$ with
\[
 m_K(g):=\frac{|\{h\in K:gh\in K\}|}{|K|}>1-v\quad(g\in E).
\]
Take a symmetric source set $F$ containing $e$, $E$, and $K$. For a
unitary-valued $(F,\delta)$-quasi-representation $q$ on $\mathcal E$,
set
\[
 V_q\xi=|K|^{-1/2}\sum_{h\in K}\delta_h\otimes q(h)^*\xi,
 \qquad
 \Phi_q(a)=V_q^*(\lambda(a)\otimes1)V_q\quad(a\in C^*(G)).
\]
Here $\lambda$ denotes the regular representation of the full group
algebra. The column $V_q$ is an isometry, so $\Phi_q$ is global and
u.c.p. On a tested group element,
\[
 \Phi_q(u_g)=\frac1{|K|}\sum_{\substack{h\in K\\gh\in K}}
                       q(gh)q(h)^*,
 \qquad
 \|\Phi_q(u_g)-q(g)\|<\delta+1-m_K(g)<\delta+v.
\]
Only the relations for $g\in E$ and $h\in K$ have been used.
Furthermore
\[
 \Phi_q(a)=\frac1{|K|}\sum_{s,t\in K}
       \langle\delta_s,\lambda(a)\delta_t\rangle\,q(s)q(t)^*.
\]
Thus, if $[q(h),T]$ is compact for the source test, then
$\Phi_q(A)\subset D_T$. If the cycle is degenerate, the image commutes
with $T$ exactly. The same formula works over $C([0,1])$, with the same
F\o lner set and uniform estimates. By including the products required
by the target test in $E$, one makes the resulting global u.c.p. map
as multiplicative as required before applying the proof above.

Conversely, sufficiently multiplicative u.c.p. maps are turned into
unitary-valued quasi-representations by
\[
 u_g=\varphi(u_g^{A})
       \bigl(\varphi(u_g^{A})^*\varphi(u_g^{A})\bigr)^{-1/2},
\]
where $u_g^A$ is the canonical group unitary. Both almost-unitary
relations are tested, including inverses, and the spectral gap is
uniform. Polar correction preserves direct sums and interval
homotopies, and takes place in $D_T$ in the Fredholm case. Its
composition with F\o lner replacement is close to the original
unitary-valued map on the prescribed finite set; the short logarithmic
paths give a controlled homotopy. For a u.c.p. input, the reverse
composite is connected to the original map by the convex u.c.p. path.

We obtain the following formulation in the notation of Section~\ref{sec:qd}.
\begin{proof}[Proof of \Cref{thm:QR-FQR-amenable,prop:QR-FQR-amenable-controlled}]
Apply the construction to the global F\o lner replacements, using a
smaller target tolerance so that polar correction fits within
$\eta$. For the vanishing assertion use the relative interval argument.
At the finite endpoints, the smoothed maps and their polar corrections
are connected to the original finite-dimensional maps by the controlled
paths just described. Thus both existence and vanishing hold at group
controls chosen independently of the class.

The homomorphisms follow by the two-control construction
\eqref{eq:fdap-lambda-definition}. For clarity, coherent choices can be made
on an increasing cofinal sequence $(c_n)$ of group controls, which
exists because $G$ is countable. Choose increasing controls $t_n$
such that vanishing of a finite-dimensional Fredholm class at $t_n$
implies its vanishing in $QR(c_n)$, and increasing $d_n\succeq t_n$
such that every Fredholm class at $d_n$ has a finite-dimensional
representative at $t_n$. Define $\rho_n$ using such representatives
and restriction to $c_n$.

For a class at $d_{n+1}$, its representative at $t_{n+1}$ restricts
to a representative at $t_n$. Comparing it with any representative
chosen for the restriction to $d_n$, the difference has zero Fredholm
class at $t_n$, and hence is zero at $c_n$. Consequently
\[
 r^Q_{c_{n+1},c_n}\rho_{n+1}=\rho_n r^F_{d_{n+1},d_n}.
\]
For an arbitrary group control $c$, let $n(c)$ be the first index with
$c\preceq c_{n(c)}$, and put
$\sigma(c)=d_{n(c)}$ and
$\rho_c=r^Q_{c_{n(c)},c}\rho_{n(c)}$. This is order preserving in the
source control and has the two composite identities. The preceding
compatibility proves compatibility on the original directed set.
\end{proof}

\section{The character isomorphism theorem for amenable groups}
\label{sec:amenable-proof}

In this section, we construct an inverse to the character map and
complete the proofs of the classification theorem for
quasi-representations of amenable groups
(\Cref{cor:amenable-character}) and of its quantitative version,
\Cref{thm:main-quantitative}. We begin with a roadmap of the proofs.

The proof proceeds through the two comparisons
\[
\mathrm{QR}(G)
\xrightarrow{\;\iota\;}
\mathrm{FQR}(G;\mathbb C)
\xrightarrow{\;\chi^F\;}
\varprojlim_d K_G^0(P_d(G)).
\]
Section~8 establishes the first comparison, including the
homomorphisms needed at prescribed finite sets and error tolerances.
In this section, we construct an inverse to the second map.
The composition $\chi^F\circ\iota$ agrees with the geometric character
$\chi$ defined in Section~4: regarding a finite-dimensional
quasi-representation as a Fredholm cycle with operator zero and then
applying $\chi^F$ gives the same equivariant vector-bundle class as
the original geometric construction.

\smallskip
\noindent\emph{Constructing the inverse.}
Amenability implies the Haagerup property, so the Higson--Kasparov theorem
provides a proper $G$-$C^*$-algebra $\mathcal A$ and classes
\[
 \beta\in KK^G(\mathbb C,\mathcal A),
 \qquad
 \alpha\in KK^G(\mathcal A,\mathbb C),
 \qquad
 \beta\otimes_{\mathcal A}\alpha=1_{\mathbb C}.
\]
For a prescribed control $c$, we choose a localized Fredholm
representative $b_c$ of the Bott class, whose Fredholm and tested
equivariance defects are supported in a compact subset of a proper
$G$-compact space $Z_c$. A cutoff function gives an equivariant map
$j_c:Z_c\to P_{D(c)}(G)$. Given a class in $K_G^0(P_{D(c)}(G))$, we pull
it back to $Z_c$, use it to multiply $b_c$, extend the coefficients to
$\mathcal A$, and apply the Kasparov product with $\alpha$. This defines
\[
 \nu_c:K_G^0(P_{D(c)}(G))
 \longrightarrow \mathrm{FQR}(\ast;c;\mathbb C).
\]
The localization and the Rips scale $D(c)$ are chosen before the variable
$K$-theory class. The construction is given in
\eqref{dw:eq:nu}, using \Cref{dw:lem:prepare}.

\smallskip
\noindent\emph{Recovering quasi-representation classes.}
For a finite-dimensional quasi-representation $q$, we compare its tensor
product with $b_c$ with the product of $b_c$ and the character bundle of
$q$. Lemma~\ref{dw:lem:localW} gives a local unitary identification of
these two modules, with small equivariance error. Multiplication by
$(1-B^2)^{1/2}$ turns this local identification into a globally defined
compact operator, because the defect of the Bott operator is supported
where the identification is available. The block homotopy in
\Cref{dw:lem:cancellation}, followed by a small correction of the tested
group operators, makes the difference of the two cycles degenerate.
Applying $\alpha$ and using $\beta\otimes_{\mathcal A}\alpha=1$
therefore recovers the class of $q$ at the prescribed control. This is
the first composite identity, proved in \Cref{dw:prop:FDidentity}.

\smallskip
\noindent\emph{Recovering equivariant $K$-theory classes.}
For the other composite, we localize the Fredholm character of the fixed
cycle $b_c$ to a subset of $P_d(G)\times Z_c$ with $G$-compact closure.
On this support, the two maps
\[
 (y,z)\longmapsto y,
 \qquad
 (y,z)\longmapsto j_c(z)
\]
are equivariantly homotopic in a sufficiently large Rips complex
$P_D(G)$ by the straight-line homotopy in barycentric coordinates.
Their pullbacks therefore act in the same way on the localized cycle.
Compatibility of the character with Kasparov products, followed by the
same identity $\beta\otimes_{\mathcal A}\alpha=1$, shows that applying
$\nu_c$ and then $\chi^F$ recovers the original $K$-theory class after
restriction to $P_d(G)$. This is \Cref{dw:prop:Kidentity}.
Importantly, $D$ depends on the fixed supporting set, not on the
$K$-theory class being recovered.

\smallskip
\noindent\emph{Compatibility and the conclusion.}
The finite-dimensional comparison from Section~\ref{sec:qd} extends the
first composite identity from finite-dimensional quasi-representations
to all Fredholm classes after strengthening the source control. The two
composite identities then show that the maps $\nu_c$ are compatible
after restriction and induce a two-sided inverse $\nu$ to $\chi^F$ on
ordinary inverse limits. Consequently
\[
 \chi^{-1}=\iota^{-1}\nu.
\]
The same finite-control identities give the quantitative theorem: equality
of characters at a sufficiently large Rips scale implies equality at any
prescribed quasi-representation control, and restriction from a fixed
larger Rips scale factors through a fixed group of quasi-representations.
In both assertions, the finite sets, error tolerances, and Rips scales are
chosen before the variable classes and their matrix sizes. These
conclusions are established in
\Cref{dw:thm:main,prop:main-controlled-equivalence}.

\subsection{Amenability and the Higson--Kasparov proper algebra}

Countable amenable groups constitute a major source of groups with the Haagerup property, i.e., they admits metrically proper affine isometric actions on a real Hilbert space.

\begin{lemma}[{\cite{BekkaCherixValette}}]\label{lem:amenable-Haagerup}
Every countable amenable discrete group has the Haagerup property. 
\end{lemma}

We now invoke the principal external theorem needed in this section.

\begin{theorem}[Higson--Kasparov]\label{thm:HK-proper-algebra}
Let $G$ be a countable group with the Haagerup property. 
Then there is a separable nuclear proper $G$-$C^*$-algebra $\mathcal A(H)$ and classes
\[
 \beta\in KK^G(\mathbb C,\mathcal A(H)),\qquad
 \alpha\in KK^G(\mathcal A(H),\mathbb C)
\]
such that
\[
 \beta\otimes_{\mathcal A(H)}\alpha=1_{\mathbb C} \in KK^G(\mathbb C,\mathbb C),\quad \text{ and } \quad 
 \alpha\otimes_{\mathbb C}\beta=1_{\mathcal A(H)} \in KK^G(\mathcal A(H),\mathcal A(H)) \;.
\]
Thus $\mathcal A(H)$ and $\mathbb C$ are $KK^G$-equivalent.
\end{theorem}

\begin{remark}
The underlying reason for the properness of the $G$-action on $\mathcal{A}(H)$ comes from the fact that there is a $G$-invariant central subalgebra of $\mathcal{A}(H)$ whose Gelfand dual is a locally compact space associated to $H$, in a way such that the metric properness of the affine $G$-action implies the (topological) properness of the induced $G$-action on this space. 
\end{remark}

\subsection{The character comparison and its hypotheses}

\begin{lemma}[compatibility of the two characters]\label{lem:character-compatibility}
For every proper $G$-compact $Y$ and every admissible sufficiently strong
source control $s$,
\[
 \chi^F_{Y,\mathbb C,s}\iota_s=\chi_{Y,s}.
\]
The equality is compatible with restriction and with compact parameter
spaces.
\end{lemma}
\begin{proof}
The Fredholm cycle associated to a finite-dimensional
quasi-representation is concentrated in even degree and has operator zero.
In the Fredholm character construction, the cutoff column, the averaged
projection, and its spectral
correction are exactly those of \Cref{sec:character_map}. The two characters are therefore
represented by the same range bundle. The assertion follows on generators,
then on differences and homotopies. Naturality was proved in
\Cref{prop:point-character-compatibility}.
\end{proof}

Let $G$ be countable and discrete. Write
\[
 \mathsf F_c=\FQR(*;c;\C),\qquad
 \mathsf Q_c=\QR(c),\qquad
 \mathsf K_d=K_G^0(P_d(G)).
\]
Since $X=\{*\}$, the operators occurring in the Fredholm and equivariance
conditions belong to $\mathcal K(\mathcal E)$, rather than merely being
locally compact.
It suffices to work on the cofinal family of symmetric tested sets containing
the identity; other target controls are obtained by restriction.
The bonding maps are denoted $r^F,r^Q,r^K$. All positive-error controls are
strengthened by enlarging their finite sets and decreasing their tolerances.
The maps $\iota_c:\mathsf Q_c\to\mathsf F_c$ send a finite-dimensional
quasi-representation $q$ to $(H_q\oplus0,q,0)$.

The comparison theorem below is formulated for any group with the following
comparison with finite-dimensional cycles: for every
$c$ there are $a(c)\succeq c$ and a homomorphism
\begin{equation}\label{dw:eq:fd-input}
 \rho_c:\mathsf F_{a(c)}\longrightarrow\mathsf Q_c,
 \qquad
 \rho_c\iota_{a(c)}=r^Q_{a(c),c},\quad
 \iota_c\rho_c=r^F_{a(c),c}.
\end{equation}
For amenable groups these are supplied by
\Cref{prop:QR-FQR-amenable-controlled}, proved by the relative-projection
approximation and its interval version in Section~\ref{sec:qd}.
In the general character theorem of
Section~\ref{sec:quantitative-KQD}, these maps are part of
\Cref{def:quantitative-QR-unitary}. For the u.c.p. sufficient
conditions, \Cref{prop:QKQD-unitary-transfer} supplies the passage
to this unitary-valued comparison.

Let $\A$ be a separable proper $G$-$C^*$-algebra, with a nondegenerate central
map $C_0(Z)\to ZM(\A)$ for a locally compact proper $G$-space $Z$. Fix even
classes
\begin{equation}\label{dw:eq:beta-alpha}
 \beta\in KK^G(\C,\A),\qquad
 \alpha\in KK^G(\A,\C),\qquad
 \beta\otimes_\A\alpha=1_\C.
\end{equation}

The Fredholm character construction of \Cref{sec:FQR-cycle} gives,
for sufficiently large finite sets and sufficiently small tolerances,
the homomorphisms
\[
 \chi^F_{d,c}:\mathsf F_c\to\mathsf K_d,\qquad
 \chi_{d,c}:\mathsf Q_c\to\mathsf K_d.
\]
The construction and its coefficient compatibility were proved in
\Cref{lem:point-character-construction,prop:point-character-compatibility,prop:FQR-character-naturality}. In particular,
\begin{equation}\label{dw:eq:character-inputs}
 \chi^F_{d,c}\iota_c=\chi_{d,c},\qquad
 \chi^F_{d,A,c}(x[c])=p_d^*x
 \quad(x\in KK^G(\C,A)).
\end{equation}
For $y\in KK^G(A,B)$ it satisfies
\begin{equation}\label{dw:eq:naturality}
 \chi^F_{d,B,c}\,y_{*,c}
       =(p_d^*y)_*\,\chi^F_{d,A,c}.
\end{equation}
All finite diagrams of character and restriction maps are evaluated after
one common strengthening of the source. This strengthening is chosen before
the variable class.

\begin{theorem}[the Fredholm character and its inverse]\label{dw:thm:main}\label{thm:scalar-character-comparison}
Under the finite-dimensional comparison hypothesis \eqref{dw:eq:fd-input} and \eqref{dw:eq:beta-alpha}, for each target
control $c$ one can construct a scale $D(c)$ and a homomorphism
\[
 \nu_c:\mathsf K_{D(c)}\longrightarrow\mathsf F_c.
\]
For $D\geq D(c)$ put $\nu_{c,D}=\nu_c r^K_{D,D(c)}$. Then:
\begin{enumerate}[label=(\roman*)]
\item For each $c$, there is $s\succeq c$, admissible for $D(c)$, such that
\begin{equation}\label{dw:eq:global-F}
 \nu_c\chi^F_{D(c),s}=r^F_{s,c}
       \quad\text{on all of }\mathsf F_s.
\end{equation}
\item For each admissible pair $(d,c)$, some $D\geq d,D(c)$ satisfies
\begin{equation}\label{dw:eq:global-K}
 \chi^F_{d,c}\nu_{c,D}=r^K_{D,d}
       \quad\text{on all of }\mathsf K_D.
\end{equation}
\end{enumerate}
The source indices are independent of the class to which the maps are
applied and, for finite-dimensional representatives, of their matrix size.
The $\nu_c$ are eventually coherent, and induce a two-sided inverse on the
ordinary inverse limits. Composing with \eqref{dw:eq:fd-input} gives the two
quantitative conclusions of \Cref{thm:main-quantitative} with their stated quantifiers.
\end{theorem}

\begin{lemma}[the Dirac--dual-Dirac identity]\label{lem:alpha-beta-stagewise}
For the classes in \eqref{dw:eq:beta-alpha},
\[
 \alpha_{*,c}\beta_{*,c}=1_{\mathsf F_c}
\]
at each individual control $c$. The maps commute with the bonding maps.
If also $\alpha\otimes_{\mathcal A}\beta=1_{\mathcal A}$, their reverse
composite is the identity on $\mathrm{FQR}(*;c;\mathcal A)$.
\end{lemma}
\begin{proof}
This is the identity and composition assertion of
\Cref{prop:FQR-KK-product}, applied with $X=\{*\}$ and at the same control.
\end{proof}

\subsection{Choosing a representative of the Bott class}
We choose a representative of the Bott class whose Fredholm and
equivariance defects vanish outside a compact subset of $Z$, and whose
equivariance defects have small norm on the tested group elements. This
choice is made before the quasi-representation $q$ is given.

\begin{lemma}[a localized representative of the Bott class]\label{dw:lem:prepare}
Fix a finite symmetric set $F\subset G$ containing $e$, and numbers
$\delta>0$ and $a>0$. There exist an invariant ideal $\A_0\subseteq\A$, a
proper $G$-compact central space $Z_0$ for $\A_0$, and a normalized Fredholm
cycle with $X=\{*\}$
\[
 b=(\E_0,v,B),\qquad B=B^*,\quad\|B\|\leq1,\quad v(e)=1,
\]
with the following properties.
\begin{enumerate}[label=(\alph*)]
\item The multiplication error of $v$ on $F$ is less than $\delta$, and
\[
 \max_{g\in FF}\|\Ad_{v(g)}B-B\|<a.
\]
\item The compact operators $1-B^2$ and $\Ad_{v(g)}B-B$, $g\in FF$, vanish
outside one compact subset $K\subset Z_0$. Consequently
\[
 d_B=(1-B^2)^{1/2}\in\K(\E_0)
\]
vanishes outside $K$.
\item With $i:\A_0\hookrightarrow\A$, the identity $i_*[b]=\beta$
 holds after restriction to the positive-error control $(F,\delta)$. It is
 witnessed by a controlled homotopy, unitary equivalence, and degenerate
 summands. All choices concern only the fixed representative of $\beta$,
 $F,\delta,a$.
\end{enumerate}
\end{lemma}

\begin{proof}
Represent $\beta$ by an equivariant Kasparov cycle $(\E_\beta,U,B_\beta)$
with a unital scalar representation, and normalize its operator to be an odd
self-adjoint contraction. The equivariant quasicentral approximate-unit
construction of \Cref{prop:FQR-asymptotic-equivariance}, applied at the point, gives a
compact perturbation $T$ of $B_\beta$ for which
\[
 \max_{g\in FF}\|\Ad_{U_g}T-T\|
\]
is as small as prescribed. Briefly, choose an even positive compact
contraction $h$ which is nearly invariant for the tested $U_g$ and nearly
acts as the identity on their compact commutator defects. The operator
$T=(1-h)B_\beta(1-h)$ is a compact perturbation, and
\[
 \|\Ad_{U_g}T-T\|
 \leq2\|B_\beta\|\|\Ad_{U_g}h-h\|
       +\|(1-h)(\Ad_{U_g}B_\beta-B_\beta)(1-h)\|.
\]
The two terms tend to zero along the quasicentral approximate unit. The
straight compact-perturbation homotopy keeps the group action $U$.

Choose once and for all an odd continuous function $h_0:[-1,1]\to[-1,1]$
which equals the sign function for $|t|\geq4/5$. Compactness of $T^2-1$ and
nondegeneracy of the central $C_0(Z)$-action give a compact subset $K_0$ such
that $\|T_z^2-1\|<1/4$ off $K_0$. Put $B=h_0(T)$. Then $B^2=1$ off $K_0$
and $B-T$ is compact. Uniform continuity of continuous functional calculus
on self-adjoint contractions shows that the earlier choice of $T$ can be
made so that
\[
 \max_{g\in FF}\|\Ad_{U_g}B-B\|<a/2.
\]
For clarity, the modulus needed here is uniform over the operator algebra:
approximate $h_0$ by one polynomial, and use the telescoping estimate for
the difference of its evaluations. 

It remains to make the tested equivariance defects compactly supported.
Write
\[
 B=\begin{pmatrix}0&S^*\\S&0\end{pmatrix},\qquad
 U_g=U_g^+\oplus U_g^-.
\]
For $g\in FF\setminus\{e\}$, set
\[
 w_g=(U_g^-)^{-1}S U_g^+S^*.
\]
This is coefficient-linear on $\E_\beta^-$ and $w_g-1$ is compact.
It is unitary outside a compact set: there the relevant source and target
fibers of $S$ are unitary. Choose $\phi_g\in C_c(Z)$, $0\leq\phi_g\leq1$,
equal to one on this bad compact set and with
\[
 \|(1-\phi_g)(w_g-1)\|<\epsilon_0.
\]
For sufficiently small $\epsilon_0$, the operator
\[
 z_g=1+(1-\phi_g)(w_g-1),\qquad
 u_g=z_g(z_g^*z_g)^{-1/2}
\]
is a unitary compact perturbation of $1$, arbitrarily close to $1$, and
agrees with $w_g$ off $\supp\phi_g$. Put
\[
 v^+(g)=U_g^+,\qquad v^-(g)=U_g^-u_g,
\]
and keep $v(e)=1$. Choose $\epsilon_0$ so small that
\[
 \max_{g\in FF}\|v(g)-U_g\|
       <\min\{\delta/3,a/4,1/4\}.
\]
The paths $U_g^-\exp(t\log u_g)$ have the same closeness bound.
Write $v_t$ for the resulting tested action. Since $U$ is a genuine action,
for $g,h\in F$ we have
\[
 \begin{aligned}
 \|v_t(g)v_t(h)-v_t(gh)\|
 &\leq \|v_t(g)-U_g\|+\|v_t(h)-U_h\|\nobreak\\
 &\qquad +\|v_t(gh)-U_{gh}\|<\delta.
 \end{aligned}
\]
Also,
\[
 \|\Ad_{v(g)}B-B\|
 \leq 2\|v(g)-U_g\|+\|\Ad_{U_g}B-B\|<a
 \qquad(g\in FF).
\]
Each change is compact, so equivariance modulo compacts is preserved
throughout.

Outside a finite union of the compact sets just chosen and their required
translates, one has $B^2=1$ and
$v^-(g)=S U_g^+S^*$, with $S$ unitary at source and target. Thus $v(g)$
intertwines $B$ exactly there. Denote a compact set containing all these
defect supports by $K$.

Choose an invariant open neighborhood $V$ of $GK$ with $G$-compact
closure, put $Z_0=\overline V$ and $\A_0=C_0(V)\A$, and restrict the module
to $\E_0=\overline{\E_\beta\A_0}$. The original central structure makes $\A_0$ a $C_0(Z_0)$-algebra
supported on $V$. The norm estimates
are unchanged. The restrictions of the defects belong to $\mathcal K(\E_0)$.
Indeed, they vanish outside $K\subset V$, so they can be approximated by
finite sums of rank-one operators whose two defining vectors belong to
$\E_0$.

Here is the class-level support reduction. Choose an invariant continuous
function $k:Z\to[0,1]$ equal to one on $GK$ and supported in $V$. Let
$\E_{1-k}=\overline{(1-k)\E_\beta}$. The cycle on $\E_{1-k}$ is degenerate.
The isometry
\[
 \xi\longmapsto(k^{1/2}\xi,(1-k)^{1/2}\xi)
       \quad\text{into }\E_0\oplus\E_{1-k}
\]
intertwines the operator and every tested group operator. Its range is
complemented. The complementary cycle is degenerate because multiplication
of each defect by $1-k$ is zero. This proves $i_*[b]=\beta$ at the stated
control, including the permitted stabilizations. Finally, functional
calculus gives the assertion about $d_B$.
\end{proof}

\subsection{A local trivialization of the character bundle}
Fix $b$, $Z_0$, and $K$ from Lemma~\ref{dw:lem:prepare}. All tested sets below
refer to the chosen target $F$. Choose a compact neighborhood containing
\[
 K\ \cup\ \bigcup_{g\in FF}(Kg\cup Kg^{-1}).
\]
Choose a cutoff $\theta\in C_c(Z_0)$ which is positive on this neighborhood.
This is possible by adding a nonnegative compactly supported function
positive there to one meeting every orbit, and normalizing the translate sum.
Define
\begin{equation}\label{dw:eq:j}
 j(z)=\sum_{g\in G}\theta(zg)\delta_{g^{-1}}\in P_{D_0}(G),
\end{equation}
where $D_0$ bounds the lengths of the finite overlap set of $\supp\theta$.
Indeed, if $g,k$ are active, then $g^{-1}k$ belongs to that overlap
set and $\dist(g^{-1},k^{-1})=\ell(g^{-1}k)\leq D_0$.
For the right action fixed in Section~\ref{sec:preliminaries},
reindexing with $k=hg$ gives
\[
 j(zh)=\sum_k\theta(zk)\delta_{k^{-1}h}=j(z)\cdot h.
\]
On $P_{D_0}(G)$ use the barycentric cutoff
$\theta_{D_0}(y)=y(e)$. Then
\begin{equation}\label{dw:eq:cutoff-match}
 \theta_{D_0}(j(z)g)=\theta(zg).
\end{equation}

Let $q$ be a normalized $n$-dimensional quasi-representation, with
multiplication error $\epsilon$ on a sufficiently large finite set. The
set is chosen from $\theta$, $K$, and $F$, not from $q$ or $n$. On
$U=\{z:\theta(z)>0\}$ define
\[
 V_q(z)\xi=\big(\theta(zg)^{1/2}q(g)^{-1}\xi\big)_g,\qquad
 P_q=V_qV_q^*.
\]
The finite overlap set contains all indices used on $U$.
Write $R_g=\lambda_g\otimes1_{\mathbb C^n}$ for the regular
fiber action on $\ell^2(G)\otimes\mathbb C^n$. Extend
$\theta P_q$ by zero, and set
\[
 a_q(z)=\sum_{h\in G}\theta(zh)R_hP_q(zh)R_h^{-1},\qquad
 e_q=1_{(1/2,\infty)}(a_q).
\]
For small enough $\epsilon$, $e_q$ is a globally defined equivariant
rank-$n$ projection. Its range bundle $\mathcal V_q$ represents
$j^*\chi_{D_0}(q)$ by \eqref{dw:eq:cutoff-match}.

\begin{lemma}[a local unitary comparison with a uniform bound]\label{dw:lem:localW}
After a strengthening of the source control independent of $n$, there is an
even unitary on the local bundle chart
\[
 W_q(z):\C^n\longrightarrow(\mathcal V_q)_z,\qquad z\in U,
\]
such that, for $g\in FF$ whenever $z$ or $zg$ belongs to $K$,
\begin{equation}\label{dw:eq:Werror}
 \|R_gW_q(zg)-W_q(z)q(g)\|\leq33\epsilon.
\end{equation}
All points occurring here lie in $U$ by its choice. Moreover $W_q$ acts
only on the finite-dimensional bundle factor and central scalar functions,
so it commutes, after tensoring, with the operator $B$ on $\E_0$.
\end{lemma}

\begin{proof}
For relevant local arrows, reindexing the column gives
\[
 \|R_gV_q(zg)-V_q(z)q(g)\|\leq\epsilon.
\]
The square of the column norm is bounded by $\epsilon^2$ times the sum of
the cutoff weights, which is one. Hence $\|a_q-P_q\|\leq2\epsilon$ on
$U$. For $\epsilon<1/16$, the resolvent estimate gives
\[
 \|e_q-P_q\|\leq\frac{4\epsilon}{1-4\epsilon}<\frac13.
\]
Define
\begin{equation}\label{dw:eq:polarW}
 W_q=e_qV_q(V_q^*e_qV_q)^{-1/2}.
\end{equation}
It is onto $e_q$'s range, and $\|W_q-V_q\|\leq16\epsilon$.
Combining the three bounds gives \eqref{dw:eq:Werror}. Source control is
enlarged to include every pair used in the finite overlap computations.
Matrix norm estimates and functional calculus are uniform in $n$ and in
compact parameters. The last assertion follows from $\A_0$-linearity of
$B$, which makes it commute with all central functions, and from the fact
that the matrix entries of $q$ act on the other tensor factor.
\end{proof}

\subsection{The cancellation homotopy}
The following homotopy uses the compact operator $(1-B^2)^{1/2}$ to
compare two representatives of the same class.

\begin{lemma}[a homotopy to a degenerate cycle]\label{dw:lem:cancellation}
Let $b=(\E_0,v,B)$ satisfy Lemma~\ref{dw:lem:prepare}. Let $q$ be as in
Lemma~\ref{dw:lem:localW}, with multiplication error $\epsilon$, and put
$\kappa=33\epsilon$. Form
\[
 \mathcal P_0=\C^n\whot\E_0,\qquad
 \mathcal P_1=\mathcal V_q\whot_{C_0(Z_0)}\E_0.
\]
Their operators and tested actions are
\[
 B_0=1\whot B,\quad Q_0(g)=q(g)\whot v(g),
 \qquad B_1=1\whot B,\quad Q_1(g)=R_g\whot v(g).
\]
Here $Q_1$ is restricted to the invariant range of $e_q\whot1$.
These are Fredholm cycles with $X=\{*\}$. If
\begin{equation}\label{dw:eq:budget}
 \delta+\epsilon+3\bigl(a+\kappa+\sqrt{2a}\bigr)<\eta,
 \qquad a+\kappa+\sqrt{2a}<1,
\end{equation}
then
\begin{equation}\label{dw:eq:two-products-equal}
 [\mathcal P_0,Q_0,B_0]=[\mathcal P_1,Q_1,B_1]
       \quad\text{in }\FQR(*;F,\eta;\A_0).
\end{equation}
The statement and bounds also hold over one compact parameter algebra.
\end{lemma}

\begin{proof}
We first check ordinary Hilbert-module compactness. The correspondence
\[
 \mathcal M_q=\mathcal V_q\whot_{C_0(Z_0)}\A_0
\]
has compact left $\A_0$-action. Indeed, after restriction to a compact subset of $Z_0$, it is the module associated with a finite-rank vector bundle; finite local frames
and approximation of an element of $\A_0$ by its central cutoffs prove
$\pi_q(\A_0)\subseteq\K(\mathcal M_q)$. Thus tensoring a compact on
$\E_0$ with $1_{\mathcal M_q}$ preserves compactness. For example, the operator $T_\xi\eta=\xi\otimes\eta$ satisfies
$T_\xi^*T_\xi=\pi_q(\langle\xi,\xi\rangle)\in\K$, so it is compact,
and $\theta_{\xi,\zeta}\whot1=T_\xi T_\zeta^*$ is compact. Finite-rank
approximation proves the assertion for all compacts. More explicitly, for the second-factor operator zero,
$B_1T_\xi=T_{B\xi}$ is compact, so the connection condition holds.
The positivity condition is $[B\widehat\otimes1,B_1]_+=2B_1^2\geq0$.
Thus $B_1$ is the actual Kasparov product operator for multiplication by
$\mathcal V_q$; see \Cref{lem:bundle-correspondence-compact} and
\cite[Definition~10 and Theorem~12]{skandalis_1984}. The first module has the same compactness property because
$n$ is finite. The equivariance errors of both operators have norm less than
$a$. The actions have multiplication errors at most $\delta+\epsilon$
and $\delta$, respectively.

Put
\[
 d_i=(1-B_i^2)^{1/2}\quad(i=0,1).
\]
The local operator $W_q\whot1$ intertwines $B_0$ and $B_1$ exactly. Choose
$\chi\in C_c(U)$ equal to one on $K$ and its required tested translates.
The map $(W_q\chi)\whot1$ is a globally defined adjointable map from
$\mathcal P_0$ to $\mathcal P_1$. Define
\begin{equation}\label{dw:eq:Cdef}
 C=((W_q\chi)\whot1)d_0.
\end{equation}
Since $d_0$ is an ordinary compact and vanishes off $K$, $C$ is compact,
is independent of $\chi$ with the stated property, and has the local formula
$C=W_qd_0$. Fiberwise calculation, or central restriction, gives
\begin{equation}\label{dw:eq:Cidentities}
 C^*C=d_0^2,\qquad CC^*=d_1^2,\qquad B_1C=CB_0.
\end{equation}
In particular these equalities hold globally, although $W_q$ is only local.

Take the graded difference
$\mathcal P_0\oplus\mathcal P_1^{\mathrm{op}}$. On it, with unchanged
tested action $Q=Q_0\oplus Q_1$, use the path
\begin{equation}\label{dw:eq:Ft}
 F_t=\begin{pmatrix}B_0&tC^*\\tC&-B_1\end{pmatrix},
       \qquad 0\leq t\leq1.
\end{equation}
The off-diagonal entries are odd for this total grading: they are even as
maps between the original graded modules, and the second grading is
reversed. Each $F_t$ is odd and self-adjoint. By \eqref{dw:eq:Cidentities},
\begin{equation}\label{dw:eq:Ft-square}
 F_t^2-1=-(1-t^2)
        \begin{pmatrix}d_0^2&0\\0&d_1^2\end{pmatrix}\in\K.
\end{equation}
Also $F_t-F_0$ is compact, so its equivariance defects remain ordinary
compacts. This is a norm-continuous homotopy of Fredholm cycles with $X=\{*\}$, and
\begin{equation}\label{dw:eq:Fhat-unitary}
 \widehat F:=F_1=\widehat F^*,\qquad \widehat F^2=1.
\end{equation}

We next bound its equivariance error. Functional calculus and the
square-root inequality of \Cref{lem:square-root-norm} give
\begin{equation}\label{dw:eq:d-commutator}
 \|\Ad_{v(g)}d_B-d_B\|
 \leq\|\Ad_{v(g)}(B^2)-B^2\|^{1/2}<\sqrt{2a}.
\end{equation}
On the union of the support of $d_B$ and its tested translate, the local
intertwining estimate \eqref{dw:eq:Werror} applies. Outside that union both
sides of the relevant off-diagonal defect vanish. Therefore
\begin{equation}\label{dw:eq:off-error}
 \|Q_1(g)C Q_0(g)^{-1}-C\|
       \leq\kappa+\sqrt{2a}.
\end{equation}
For example, on the common chart write the difference as the error of the
transported $W_q$, multiplied by the transported $d_B$, plus $W_q$ times
the error of $d_B$. Both $W_q$ and $d_B$ are contractions. The choice of
the buffered chart ensures that no untested value of $q$ enters this formula.
Combining diagonal and off-diagonal bounds yields
\begin{equation}\label{dw:eq:gamma}
 \max_{g\in FF}\|\Ad_{Q(g)}\widehat F-\widehat F\|
       \leq\gamma:=a+\kappa+\sqrt{2a}.
\end{equation}

Use the total grading to write
\[
 \widehat F=\begin{pmatrix}0&J^*\\J&0\end{pmatrix},\qquad
 Q(g)=Q^+(g)\oplus Q^-(g).
\]
The operator $J$ is a coefficient-linear unitary. Put
\[
 u_g=(Q^-(g))^{-1}JQ^+(g)J^*.
\]
Then $u_g$ is coefficient-linear, $u_g-1$ is compact, and
$\|u_g-1\|\leq\gamma<1$. Keep $Q^+(g)$ fixed and replace the negative
part by
\begin{equation}\label{dw:eq:negative-correction}
 Q_t^-(g)=Q^-(g)\exp(t\log u_g).
\end{equation}
This moves each tested group operator by at most $\gamma$. It preserves semilinearity and the equivariance of $\widehat F$ modulo compact operators, because
$\exp(t\log u_g)-1$ is compact. The multiplication error is at most
\[
 \delta+\epsilon+3\gamma<\eta.
\]
At $t=1$ one has $Q_1^-(g)=JQ^+(g)J^*$, so $\widehat F$ intertwines all
of the tested group operators exactly. Since $\widehat F=\widehat F^*$
and $\widehat F^2=1$, all Fredholm and equivariance defects vanish.
The endpoint cycle is therefore degenerate.
Concatenating \eqref{dw:eq:Ft} and \eqref{dw:eq:negative-correction} proves
\eqref{dw:eq:two-products-equal}.

All compact operators and paths can be taken in the parameter-module norm.
For a parameterized family of finite-dimensional quasi-representations,
the matrix size is finite,
and the spectral and logarithmic gaps above are uniform. No spatial tail
is discarded, no non-invariant piece is treated as a separate $G$-space,
and no global extension of $W_q$ as a unitary is used.
\end{proof}

\subsection{Recovering quasi-representation classes}
Fix a target $c=(F,\eta)$, with $F$ finite symmetric and $e\in F$.
Choose $\delta,a>0$ first, and then a multiplication tolerance $\epsilon>0$, so that
\eqref{dw:eq:budget} holds. For instance, put $\eta_0=\min\{\eta,1\}$ and take
\[
 \delta<\eta_0/16,\qquad a<\eta_0^2/2048,\qquad
 33\epsilon<\eta_0/32,
\]
with $\epsilon$ decreased for the spectral gap and identity normalization
when required. These choices are independent of matrix size.
Prepare $b_c$, $\A_c$, $Z_c$, and $K_c$ by Lemma~\ref{dw:lem:prepare}, and
choose $j_c:Z_c\to P_{D(c)}(G)$ as in \eqref{dw:eq:j}.

For $z\in\mathsf K_{D(c)}$, apply \Cref{prop:central-K-action} to
$j_c^*z$ and the $C_0(Z_c)$-action on $\A_c$. This gives
\[
 a_z\in KK^G(\A_c,\A_c).
\]
With $i_c:\A_c\hookrightarrow\A$, define
\begin{equation}\label{dw:eq:nu}
 \nu_c(z)=\alpha_{*,c}(i_c)_{*,c}(a_z)_{*,c}(b_c).
\end{equation}
The Kasparov product construction for Fredholm quasi-representations
shows that this is a homomorphism. Its values are represented by cycles
whose Fredholm and equivariance defects belong to
$\mathcal K(\mathcal E)$ on the resulting Hilbert module $\mathcal E$.
The variable $X$ remains a point.

\begin{proposition}[the identity $\nu_c\chi=\iota_c r^Q$]\label{dw:prop:FDidentity}
For a sufficiently strong source control $s\succeq c$, admissible for
$D(c)$, one has
\begin{equation}\label{dw:eq:FDidentity}
 \nu_c\chi_{D(c),s}=\iota_c r^Q_{s,c}
       \quad\text{on all of }\mathsf Q_s.
\end{equation}
The source control is chosen before the quasi-representation and its dimension.
\end{proposition}

\begin{proof}
First take a finite-dimensional normalized $q$. The coefficient class
$a_{\chi(q)}$ is represented by the Kasparov cycle $(\mathcal M_q,\pi_q,0)$ associated with the character bundle. Thus its product with $b_c$ has module
$\mathcal P_1$ and operator $B_1$ in Lemma~\ref{dw:lem:cancellation}.
The other cycle in that lemma is the explicit tensor product
$\C^n\whot\E_c$ with action $q\whot v$ and operator $1\whot B$.
The lemma identifies the two cycles at the target control $c$.

Tensor the controlled homotopy $i_*b_c=\beta$ from
Lemma~\ref{dw:lem:prepare} with the finite-dimensional $q$. Its multiplication
error is at most the sum of the two errors, hence less than $\eta$.
Finite dimensionality ensures ordinary compactness throughout. Its endpoint
is exactly the representative for $\beta_{*,c}\iota_c(r^Q_{s,c}[q])$.
Consequently
\[
 (i_c)_{*,c}(a_{\chi(q)})_{*,c}(b_c)
       =\beta_{*,c}\iota_c(r^Q_{s,c}[q]).
\]
Apply $\alpha_{*,c}$ and \eqref{dw:eq:beta-alpha}. This proves
\eqref{dw:eq:FDidentity} on generators. Additivity proves it on the whole
Grothendieck group. Normalization of $q(e)$ is performed at a stronger
source and costs at most the usual factor two in the multiplication error.
\end{proof}

\subsection{Recovering equivariant \texorpdfstring{$K$}{K}-theory classes}
Apply \Cref{lem:fixed-exact-support} to the character of the fixed
representative $b_c$. The resulting supporting set is chosen before the
variable $K$-theory class to which $\nu_c$ is applied.

\begin{proposition}[the identity $\chi^F\nu=r^K$]\label{dw:prop:Kidentity}
If $c$ is admissible for $d$, there is $D\geq d,D(c)$, chosen independently
of $z\in\mathsf K_D$, such that
\[
 \chi^F_{d,c}\nu_c r^K_{D,D(c)}=r^K_{D,d}.
\]
\end{proposition}

\begin{proof}
Write
\[
 \mathcal B_{d,c}=\chi^F_{d,\A_c,c}(b_c)
       \in RKK^G(P_d(G);\C,\A_c).
\]
Compatibility with Kasparov products and the character formula for
equivariant Kasparov cycles give
\begin{equation}\label{dw:eq:Bbeta}
 (i_c)_*\mathcal B_{d,c}=p_d^*\beta.
\end{equation}
Apply \Cref{lem:fixed-exact-support} to this fixed class, obtaining a
supporting invariant open $V\subseteq P_d(G)\times Z_c$ with $G$-compact
closure. Compare the two equivariant maps on $\overline V$,
\[
 f_0(y,z)=y,\qquad f_1(y,z)=j_c(z).
\]
A compact subset meeting every orbit of $\overline V$ has compact images
in the two locally finite Rips complexes. Those images use finitely many
vertices. Hence there is $D\geq d,D(c)$ such that the union of the supports
of $y$ and $j_c(z)$ has diameter at most $D$ on all of $\overline V$.
The formula
\[
 (t,y,z)\longmapsto(1-t)y+t j_c(z)
\]
is then an equivariant homotopy into $P_D(G)$. It is chosen before the
variable class $z_D\in\mathsf K_D$.

The two pulled-back classes $f_0^*z_D$ and $f_1^*z_D$ therefore give the
same multiplication on the supporting representative of $\mathcal B_{d,c}$.
Using \eqref{dw:eq:naturality}, \eqref{dw:eq:Bbeta}, and \eqref{dw:eq:beta-alpha},
\begin{align*}
 \chi^F_{d,c}\nu_c r^K_{D,D(c)}(z_D)
 &= (p_d^*\alpha)_*(i_c)_*
       \bigl(f_1^*z_D\cdot\mathcal B_{d,c}\bigr)\\
 &= (p_d^*\alpha)_*(i_c)_*
       \bigl(f_0^*z_D\cdot\mathcal B_{d,c}\bigr)\\
 &= r^K_{D,d}(z_D)\cdot
       p_d^*(\beta\otimes_\A\alpha)
  =r^K_{D,d}(z_D).
\end{align*}
These equalities are on the whole source $\mathsf K_D$. No compact support
for the variable class and no class-dependent choice of $D$ was used.
\end{proof}

\subsection{The comparison for all Fredholm quasi-representations}
\begin{proof}[Proof of Theorem~\ref{dw:thm:main}]
Proposition~\ref{dw:prop:Kidentity} proves \eqref{dw:eq:global-K}. To prove
\eqref{dw:eq:global-F}, fix $c$ and choose a control $t\succeq c$ strong enough
for Proposition~\ref{dw:prop:FDidentity} and all character compatibilities at
$D(c)$. Apply the earlier finite-dimensional comparison at this target $t$.
For $w\in\mathsf F_{a(t)}$, put $x=\rho_t(w)\in\mathsf Q_t$. Then
$\iota_tx=r^F_{a(t),t}w$, and
\begin{align*}
 \nu_c\chi^F_{D(c),a(t)}(w)
 &=\nu_c\chi^F_{D(c),t}(\iota_t x)\\
 &=\nu_c\chi_{D(c),t}(x)\\
 &=\iota_c r^Q_{t,c}(x)
  =r^F_{t,c}\iota_t x
  =r^F_{a(t),c}(w).
\end{align*}
Thus $s=a(t)$ works for every $w$. Crucially, the geometric cancellation
was required only for the finite-dimensional representative $x$, not for
an arbitrary infinite-dimensional local intertwiner.

We spell out eventual coherence. For $c\preccurlyeq c'$, choose $u$
dominating $c'$ and the source controls in \eqref{dw:eq:global-F} for $c,c'$,
and strengthen once for the character compatibilities. Thus
\[
 \nu_c\chi^F_{D(c),u}=r^F_{u,c},\qquad
 \nu_{c'}\chi^F_{D(c'),u}=r^F_{u,c'}.
\]
Apply \eqref{dw:eq:global-K} to $(D(c),u)$ and $(D(c'),u)$, taking one common
large source scale $L$. This gives
\[
 r^F_{u,c}\nu_{u,L}=\nu_{c,L},\qquad
 r^F_{u,c'}\nu_{u,L}=\nu_{c',L},
\]
and hence
\begin{equation}\label{dw:eq:coherence}
 r^F_{c',c}\nu_{c',L}=\nu_{c,L}.
\end{equation}
For a compatible family $y=(y_d)_d$, define
$(\nu y)_c=\nu_c(y_{D(c)})$. Equation~\eqref{dw:eq:coherence} makes it
compatible. The two identities \eqref{dw:eq:global-F}--\eqref{dw:eq:global-K}
show, coordinate by coordinate, that it is the inverse of $\chi^F$.

For the finite-dimensional conclusion define
\begin{equation}\label{dw:eq:psi}
 E(c)=D(a(c)),\qquad
 \psi_c=\rho_c\nu_{a(c)}:\mathsf K_{E(c)}\longrightarrow\mathsf Q_c.
\end{equation}
At a sufficiently strong source $s$, the composite with the character satisfies
\[
 \psi_c\chi_{E(c),s}
 =\rho_c\nu_{a(c)}\chi^F_{E(c),s}\iota_s
 =\rho_c r^F_{s,a(c)}\iota_s
 =r^Q_{s,c}.
\]
For a prescribed Rips scale $d$, choose $c$ sufficiently strong that it and
$a(c)$ are admissible for $d$. Proposition~\ref{dw:prop:Kidentity}, after
increasing $D\geq d,E(c)$, gives
\begin{align*}
 \chi_{d,c}\psi_c r^K_{D,E(c)}
 &=\chi^F_{d,c}\iota_c\rho_c\nu_{a(c),D}\\
 &=\chi^F_{d,c}r^F_{a(c),c}\nu_{a(c),D}
  =\chi^F_{d,a(c)}\nu_{a(c),D}
  =r^K_{D,d}.
\end{align*}
The eventual coherence just proved supplies compatible inverse-limit maps.
The cofinal component formulation is given explicitly below.

Finally fix a target control $c_0$. Set $d_0=E(c_0)$ and choose $s_0$ for
the first identity for $\psi_{c_0}$. For any $d\geq d_0$, choose
$s\succeq s_0$ admissible for $d$ and the restriction compatibility. Equality
of two characters at scale $d$ implies equality at $d_0$; applying
$\psi_{c_0}$ gives equality after restriction to $c_0$. This proves the
quantifier order
\[
 \forall c_0\ \exists d_0\ \forall d\geq d_0\ \exists s\succeq c_0
       \ \forall x,y\in\mathsf Q_s.
\]
For lifting, the second displayed identity factors the bonding map from
$\mathsf K_D$ to $\mathsf K_d$ through the fixed $\mathsf Q_c$. Increase
$D$ once if necessary to arrange $D>d$. Thus
\[
 \forall d\ \exists D>d,c\ \forall z\in\mathsf K_D\ \exists x\in\mathsf Q_c,
 \qquad \chi_{d,c}(x)=r^K_{D,d}(z).
\]
The choices $D,c$ precede $z$. This proves the scalar comparison and its quantitative conclusions.
\end{proof}

\subsection{Compatibility and the quantitative conclusions}

The preceding argument constructs global homomorphisms, not only inverses
on compatible limiting families. We record their control content without
suppressing the source stages.

\begin{proposition}[homomorphisms defining the inverse character]
\label{prop:main-controlled-equivalence}
Under the two hypotheses of \Cref{thm:scalar-character-comparison}, put
$E(c)=D(a(c))$ and $\psi_c=\rho_c\nu_{a(c)}$, as in \eqref{dw:eq:psi}.
For every target control $c$, some source $s\succeq c$, admissible for
$E(c)$, satisfies
\begin{equation}\label{eq:main-inverse-system-A}
 \psi_c\chi_{E(c),s}=r^Q_{s,c}
 \quad\hbox{on all of }\mathsf Q_s.
\end{equation}
For each Rips scale $d$, choose a sufficiently strong admissible $c$. There
is $D\geq d,E(c)$ such that
\begin{equation}\label{eq:main-inverse-system-K}
 \chi_{d,c}\psi_c r^K_{D,E(c)}=r^K_{D,d}
 \quad\hbox{on all of }\mathsf K_D.
\end{equation}
All the source indices precede the variable classes. On countable cofinal
sequences these maps can be chosen as comparison homomorphisms in the
sense of \Cref{def:inverse-system-equivalence}.
\end{proposition}

\begin{proof}
The two equations were computed at the end of the proof of
\Cref{thm:scalar-character-comparison}. Here are the cofinal choices
explicitly. Fix exhausting sequences of group controls and Rips scales.
Choose $c_1$ and $d_1\geq E(c_1)$. Once $c_n,d_n$ have been chosen, choose
$c_{n+1}$ above $c_n$ and the next exhausting group control, admissible for
$d_n$, and strong enough that the first equation and character naturality
give
\[
 \psi_{c_n}r^K_{d_n,E(c_n)}\chi_{d_n,c_{n+1}}
       =r^Q_{c_{n+1},c_n}.
\]
Choose $d_{n+1}$ above $d_n,E(c_{n+1})$ and the next exhausting scale,
large enough for the second equation at $(d_n,c_{n+1})$.
Set
\[
 f_n=\chi_{d_n,c_{n+1}},\qquad
 g_n=\psi_{c_n}r^K_{d_n,E(c_n)}.
\]
Then
\[
 g_nf_n=r^Q_{c_{n+1},c_n},\qquad
 f_ng_{n+1}=r^K_{d_{n+1},d_n}.
\]
The bonding squares follow algebraically. For example,
\[
 r^K_{d_{n+1},d_n}f_{n+1}
   =f_ng_{n+1}f_{n+1}
   =f_nr^Q_{c_{n+2},c_{n+1}},
\]
and
\[
 r^Q_{c_{n+1},c_n}g_{n+1}
   =g_nf_ng_{n+1}
   =g_nr^K_{d_{n+1},d_n}.
\]
Thus $\sigma(n)=n+1$ and $\tau(n)=n$ give homomorphisms defining the inverse
satisfying the stated identities on these cofinal systems. Restriction from a cofinal stage defines the
corresponding inverse on every original inverse-limit coordinate. Every source stage is chosen independently of the class to which the
map is applied.
\end{proof}

\begin{lemma}[equality and lifting after restriction]
\label{lem:inverse-system-finite-stage}
Suppose two inverse systems are related by component maps whose composites
are bonding homomorphisms, as in \Cref{def:inverse-system-equivalence}.
For a prescribed target stage on the first side, equality of the two images
at a suitable stage on the second side implies equality after restriction
to that target. For a prescribed target on the second side, restriction
from one sufficiently strong source stage factors through a single stage
on the first side.
\end{lemma}
\begin{proof}
For a target $i$, the identity $g_if_{\tau(i)}=r^A_{\sigma\tau(i),i}$
shows that the bonding homomorphism factors through $B_{\tau(i)}$.
Applying this identity to the difference of two classes gives the equality
assertion. For a target $j$, the identity
$f_jg_{\sigma(j)}=r^B_{\tau\sigma(j),j}$ is precisely the required lifting
factorization. The same argument applies after restriction to cofinal
systems, followed by a bonding map to any originally prescribed target.
\end{proof}

\begin{proof}[Proof of \Cref{cor:amenable-character}]
By \Cref{lem:amenable-Haagerup,thm:HK-proper-algebra}, an amenable
countable discrete group has a proper coefficient algebra and
Dirac--dual-Dirac classes satisfying \eqref{dw:eq:beta-alpha}.
The comparison maps in \eqref{dw:eq:fd-input} are supplied by
\Cref{prop:QR-FQR-amenable-controlled}.
The scalar comparison theorem and compatibility of the two characters give
\[
 \chi=\chi^F\iota:
 \mathrm{QR}(G)\xrightarrow{\cong}
 \varprojlim_d K_G^0(P_d(G)),
 \qquad \chi^{-1}=\iota^{-1}\nu.
\]
The last group is $K_G^0(\underline EG)$ under the convention
\eqref{eq:RK-inverse-limit}. This proves \eqref{eq:main-nonquantitative}.
\end{proof}

\begin{proof}[Proof of \Cref{thm:main-quantitative}]
Fix $c_0=(F_0,\varepsilon_0)$. Put $d_0=E(c_0)$ and choose $s_0$ for
\eqref{eq:main-inverse-system-A}. For any $d\geq d_0$, choose
$s\succeq s_0$ admissible for the character at $d$ and for its restriction
to $d_0$. If two classes at $s$ have the same character at scale $d$,
their characters at $d_0$ agree. Apply $\psi_{c_0}$ and
\eqref{eq:main-inverse-system-A}; their restrictions to $c_0$ are equal.
Thus the order of choice is
\[
 \forall c_0\quad\exists d_0\quad\forall d\geq d_0
 \quad\exists s\succeq c_0\quad\forall x,y\in\mathsf Q_s.
\]
For a prescribed $d$, choose $c,D$ in
\eqref{eq:main-inverse-system-K}. For $z\in\mathsf K_D$ take
\[
 x=\psi_c r^K_{D,E(c)}(z).
\]
Then $\chi_{d,c}(x)=r^K_{D,d}(z)$. The same $c,D$ work for every $z$;
precompose with restriction from a larger scale if necessary to ensure
$D>d$. These are the two asserted quantitative statements. Their
interpretation as stable homotopy is \Cref{prop:QR-two-descriptions}.
\end{proof}

\begin{remark}[scope of the method]\label{rem:main-proof-scope}
The geometric proof uses the comparison with finite-dimensional cycles
only after establishing cancellation for finite-dimensional
quasi-representations. It therefore applies whenever the comparison maps
in \eqref{dw:eq:fd-input} and the proper coefficient algebra and classes
in \eqref{dw:eq:beta-alpha} are available.
\end{remark}

For the qualitative argument below, abbreviate
\[
 Q_c=\mathrm{QR}(F,\varepsilon),\qquad
 \mathcal F_c=\mathrm{FQR}(*;F,\varepsilon;\mathbb C)
 \quad(c=(F,\varepsilon)),
\]
and write
\[
 Q=\varprojlim_c Q_c=\mathrm{QR}(G),\qquad
 \mathcal F=\varprojlim_c\mathcal F_c,\qquad
 \Klim=\varprojlim_d K_G^0(P_d(G)).
\]
The two characters and the finite-dimensional inclusion satisfy
\begin{equation}\label{eq:factor}
 Q\xrightarrow{\ \iota\ }\mathcal F
  \xrightarrow{\ \chi^F\ }\Klim,
 \qquad \chi=\chi^F\iota.
\end{equation}
The global u.c.p. counterparts and the lifting comparison are developed
in Section~\ref{sec:quantitative-KQD}.

\subsection{Lifting and polar correction}

\begin{proposition}[Unitary--u.c.p. comparison]\label{prop:LP}
If $A=C^*_{\max}(G)$ has LP, there are natural isomorphisms making the
following square commute:
\begin{equation}\label{eq:LPdiagram}
\begin{tikzcd}[column sep=large]
 Q\arrow[r,"\iota"]\arrow[d,"L_Q"',"\cong"]&
 \mathcal F\arrow[d,"L_{\mathcal F}","\cong"']\\
 \mathrm{QR}^{\ucp}(G)\arrow[r,"\iota^{\ucp}"']&
 \mathrm{FQR}^{\ucp}(G).
\end{tikzcd}
\end{equation}
\end{proposition}

This is the inverse-limit form of \Cref{prop:LP-unitary-ucp-comparison,cor:LP-unitary-ucp-limits}.
We recall the mechanism, including why it respects classes.
If increasingly accurate unitary quasi-representations $q_m$ could not
be approximated on a prescribed finite test by global u.c.p. maps,
they would give a genuine representation in
\[
 \prod_mD_m/\bigoplus_mD_m.
\]
The full group-algebra universal property extends it to $A$, and LP
supplies a u.c.p. lift. Its coordinate maps contradict the assumed
failure of approximation. For a Fredholm cycle, take
\[
 D_m=\mathcal D_{T_m}
 =\{b\in B(H_m)_{\mathrm{even}}:[b,T_m]\in\K(H_m)\}.
\]
These are unital $C^*$-algebras. For a degenerate cycle use the exact
commutant instead. A homotopy is handled in one parameter algebra over
$C([0,1])$, so the compactness conditions hold over the whole interval.

Conversely, the elements $\phi(u_g)$ of a sufficiently multiplicative
u.c.p. map are almost unitary on a symmetric test set. Their polar
parts are unitaries. This functional calculus stays in $\mathcal D_T$
and respects direct sums and parameter families. Two choices of a
u.c.p. replacement are joined by the convex u.c.p. homotopy. If
$\phi_t=(1-t)\phi_0+t\phi_1$, then
\begin{align*}
 \phi_t(ab)-\phi_t(a)\phi_t(b)
 ={}&(1-t)\bigl(\phi_0(ab)-\phi_0(a)\phi_0(b)\bigr)\\
 &+t\bigl(\phi_1(ab)-\phi_1(a)\phi_1(b)\bigr)\\
 &+t(1-t)(\phi_0(a)-\phi_1(a))(\phi_0(b)-\phi_1(b)).
\end{align*}
This gives independence, additivity, and the two inverse identities
after relaxing control. \Cref{prop:LP-unitary-ucp-comparison} records the compatible
finite-control maps; passing to inverse limits gives the square.

\subsection{Finite-dimensional Bott cancellation}

\begin{proposition}[Bott cancellation on finite-dimensional classes]
\label{prop:FD}
Suppose that $G$ has the Haagerup property. For every control $c$
there exist a scale $D(c)$, a homomorphism
\[
 \nu_c:K_G^0(P_{D(c)}(G))\longrightarrow\mathcal F_c,
\]
and a stronger source control $s\succeq c$, admissible for the character
at $D(c)$, such that
\begin{equation}\label{eq:FD}
 \nu_c\chi_{D(c),s}=\iota_c r^Q_{s,c}\quad\text{on }Q_s.
\end{equation}
The scale and source control are chosen before the variable class.
No finite-dimensional--Fredholm comparison hypothesis is used.
\end{proposition}

The map is defined in \eqref{dw:eq:nu}, and the identity is
\Cref{dw:prop:FDidentity}. The Higson--Kasparov theorem supplies a
proper coefficient algebra $\mathcal A$ and classes
\[
 \beta\in KK^G(\C,\mathcal A),\qquad
 \alpha\in KK^G(\mathcal A,\C),\qquad
 \beta\otimes_{\mathcal A}\alpha=1.
\]
The construction first prepares a controlled representative of $\beta$ and then
identify its product with the character bundle of a finite-dimensional
quasi-representation. \Cref{dw:lem:prepare,dw:lem:cancellation} give a homotopy to the
tensor-product representative. Finite dimensionality ensures that the
errors remain compact. Taking the product with $\alpha$ gives
\eqref{eq:FD}. This uses the finite-dimensional cancellation identity,
rather than the later comparison on all Fredholm classes.

\subsection{Characters of exact cycles}

\begin{proposition}[Compatibility of characters]\label{prop:exact}
For every countable $G$, exact equivariant Kasparov cycles induce a
homomorphism $\kappa:KK^G(\C,\C)\to\mathcal F$. The characters satisfy
\begin{equation}\label{eq:exact}
 \chi^F\iota=\chi,\qquad
 \chi^F_d(\kappa(\xi))=p_d^*\xi
 \quad(\xi\in KK^G(\C,\C)),
\end{equation}
where $p_d:P_d(G)\to\{*\}$ is the projection and
$\chi_d^F=\operatorname{pr}_d\circ\chi^F$.
\end{proposition}

The first identity is \Cref{lem:character-compatibility}; the second is
\Cref{prop:point-character-compatibility}, specifically \eqref{eq:point-character-exact-input}. For an exact group
action, the character construction gives the external product cycle
on $P_d(G)$. For a finite-dimensional cycle with operator zero, its
Fredholm character is the same range bundle as the geometric character.
Both identifications respect restrictions and homotopies.

For later use, there is a canonical identification
\begin{equation}\label{eq:group-KK}
 KK(C^*_{\max}(G),\C)\cong KK^G(\C,\C).
\end{equation}
See \cite[Section~4, the discussion preceding equation~(3)]{dadarlat_2021}.
A unital representation of the full group algebra is the same as a
unitary representation of $G$. The conditions
$[T,\pi(a)]\in\K$ for $a\in A$ and
$q(g)Tq(g)^{-1}-T\in\K$ for $g\in G$ correspond, by density of the
group algebra. The Fredholm and grading conditions also agree, as do
homotopies. Nonunital cycles are reduced to their unit-support corners;
the compact discrepancies are retained. We use~\eqref{eq:group-KK}
when applying $\kappa$ to a Cuntz-pair class.

\section{Character isomorphism theorems and QR-quasidiagonality}
\label{sec:quantitative-KQD}

In this section, we prove the character isomorphism theorem for groups
satisfying the Haagerup property and QR-quasidiagonality, together with
a quantitative version of this theorem.

We first state both character theorems and prove them using the Bott
cancellation and exact-cycle compatibility established in
Section~\ref{sec:amenable-proof}. We then prove that quasidiagonality
of the maximal group $C^*$-algebra supplies the quantitative comparison
for global u.c.p. maps, and use the lifting property to pass to
unitary-valued quasi-representations. Finally, we verify these sufficient
conditions for amenable, free, and surface groups and give the explicit
classification of quasi-representations for surface groups using  rank and winding number as invariants.

Throughout, $G$ is countable and discrete, $A=C^*_{\max}(G)$ when a
group algebra is in use, and
\[
 Q_c=\mathrm{QR}(F,\varepsilon),\qquad
 \mathcal F_c=\mathrm{FQR}(*;F,\varepsilon;\mathbb C),
 \qquad c=(F,\varepsilon)\in\mathcal C_G.
\]
We write $Q=\varprojlim_cQ_c$, $\mathcal F=\varprojlim_c\mathcal F_c$,
and $\Klim=\varprojlim_d K_G^0(P_d(G))$, as in
Section~\ref{sec:amenable-proof}. Bonding maps are denoted by $r^Q$
and $r^F$, and $\iota_c:Q_c\to\mathcal F_c$ is the inclusion of
finite-dimensional cycles. The maximal completion is used because
every unitary representation of $G$ extends to it.

\subsection{QR-quasidiagonality and the character theorems}
\label{subsec:QR-character-theorems}

We first introduce the notion of QR-quasidiagonality for groups, which plays a key role in the character isomorphism theorem.

\begin{definition}[QR-quasidiagonality]\label{def:QR-unitary}
\label{def:comparison}
A countable discrete group $G$ is \emph{QR-quasidiagonal} if the
canonical comparison of the unitary-valued theories
\[
 \iota:\mathrm{QR}(G)\longrightarrow\mathrm{FQR}(G;\mathbb C)
\]
is an isomorphism. This includes both the existence of
finite-dimensional representatives of limiting Fredholm classes and
their stable uniqueness.
\end{definition}

The following quantitative QR-quasidiagonality  concept is useful for establishing the quantitative version of the  character isomorphism theorem.

\begin{definition}[Quantitative QR-quasidiagonality]
\label{def:quantitative-QR-unitary}
A countable discrete group $G$ is \emph{quantitatively
QR-quasidiagonal} if there are an order-preserving map
$\sigma:\mathcal C_G\to\mathcal C_G$, with $\sigma(c)\succeq c$,
and homomorphisms $\rho_c:\mathcal F_{\sigma(c)}\to Q_c$ such that
\begin{align}
 \rho_c\iota_{\sigma(c)}
   &=r^Q_{\sigma(c),c},\label{eq:QR-unitary-left}\tag{QR1}\\
 \iota_c\rho_c
   &=r^F_{\sigma(c),c},\label{eq:QR-unitary-right}\tag{QR2}\\
 r^Q_{d,c}\rho_d
   &=\rho_c r^F_{\sigma(d),\sigma(c)}
     \quad(c\preccurlyeq d).\label{eq:QR-unitary-coherence}\tag{QR3}
\end{align}
The source control is chosen before the variable class.
\end{definition}

Quantitative QR-quasidiagonality implies QR-quasidiagonality:
the inverse of $\iota$ is
$(y_c)_c\mapsto(\rho_c(y_{\sigma(c)}))_c$.
The third identity makes this family compatible, and the first two
show that it is a two-sided inverse. The converse is not part of the
definition: an isomorphism of inverse limits does not by itself
provide finite-control homomorphisms.

\begin{theorem}[Qualitative character isomorphism]\label{thm:qual}
Let $G$ be a countable discrete group with the Haagerup property.
If $G$ is QR-quasidiagonal, then the geometric character is an
isomorphism of abelian groups:
\[
 \chi:\mathrm{QR}(G)\xrightarrow{\cong}
       \varprojlim_d K_G^0(P_d(G)).
\]
\end{theorem}

The proof in Subsection~\ref{sec:qualitative-proof} first identifies
$\ker\chi=\ker\iota$ by Bott cancellation. The Milnor sequence and
the Higson--Kasparov theorem then provide exact Kasparov classes
realizing the compatible geometric classes.

\begin{theorem}[Quantitative character isomorphism]
\label{thm:QKQD-Haagerup}
Let $G$ be a countable discrete group with the Haagerup property.
Suppose that $G$ is quantitatively QR-quasidiagonal in the sense of
\Cref{def:quantitative-QR-unitary}. Then the character is an isomorphism
\begin{equation}\label{eq:QKQD-Haagerup-character}
 \chi:\mathrm{QR}(G)\xrightarrow{\cong}
 \varprojlim_d K_G^0(P_d(G)).
\end{equation}
In addition, the following two quantified statements hold.
\begin{enumerate}[label=(\roman*)]
\item\label{item:QKQD-Haagerup-equality}
Let $c_0=(F_0,\varepsilon_0)\in\mathcal C_G$.  There is a number $d_0>0$
such that, for every $d\geq d_0$, one can choose a stronger control
$c_1=(F_1,\varepsilon_1)\succeq c_0$ for which the character
\[
 \chi_d:\mathrm{QR}(F_1,\varepsilon_1)
       \longrightarrow K_G^0(P_d(G))
\]
is defined and has the following property: if $q_1$ and $q_2$ are
$(F_1,\varepsilon_1)$-quasi-representations and
\[
 \chi_d([q_1])=\chi_d([q_2])
       \quad\text{in }K_G^0(P_d(G)),
\]
then
\[
 r^{\mathrm{QR}}_{c_1,c_0}([q_1])
 =r^{\mathrm{QR}}_{c_1,c_0}([q_2])
       \quad\text{in }\mathrm{QR}(F_0,\varepsilon_0).
\]
Here $r^{\mathrm{QR}}_{c_1,c_0}$ is the homomorphism obtained by viewing an
$(F_1,\varepsilon_1)$-quasi-representation as an
$(F_0,\varepsilon_0)$-quasi-representation.

\item\label{item:QKQD-Haagerup-lifting}
For every $d>0$, there are $d'>d$ and a control
$c=(F,\varepsilon)\in\mathcal C_G$ for which
\[
 \chi_d:\mathrm{QR}(F,\varepsilon)
       \longrightarrow K_G^0(P_d(G))
\]
is defined and such that, for every
$z\in K_G^0(P_{d'}(G))$, there is
$x\in\mathrm{QR}(F,\varepsilon)$ satisfying
\[
 \chi_d(x)=r^K_{d',d}(z)
       \quad\text{in }K_G^0(P_d(G)).
\]
Here $r^K_{d',d}$ is pullback along the inclusion
$P_d(G)\hookrightarrow P_{d'}(G)$.
\end{enumerate}
\end{theorem}

The two theorems use the unitary-valued comparison. The global u.c.p.
versions are introduced in Subsection~\ref{subsec:ucp-QR-properties};
the lifting property transfers them to the unitary-valued versions by
\Cref{prop:QKQD-unitary-transfer}.

\begin{remark}[Relation with Dadarlat's $K$-quasidiagonality]
\label{rem:Dadarlat-comparison}
Dadarlat calls a separable $C^*$-algebra $A$ \emph{$K$-quasidiagonal}
when every class in $KK(A,\mathbb C)$ admits a quasidiagonal
Cuntz-pair representative \cite[Definition~2.2]{dadarlat_bundles_2014}.
Thus one has representations $\pi_0,\pi_1$ with compact difference and
a common sequence of finite-rank projections converging strongly to
$1$ whose commutators with both representations tend to zero.
For $A=C^*_{\max}(G)$, compression of such pairs concerns
finite-dimensional representatives of \emph{exact} Kasparov classes.
By contrast, QR-quasidiagonality concerns the entire comparison
$\iota:Q\to\mathcal F$ and requires both existence and stable uniqueness;
its quantitative form requires the compatible finite-control maps in
\Cref{def:quantitative-QR-unitary}. These conditions address different comparisons: Dadarlat's
condition is formulated for exact Kasparov classes, whereas the QR
comparison includes stable uniqueness for approximate Fredholm classes.
The character theorems below use the latter comparison directly.
\end{remark}

\subsection{Proofs of the character theorems}\label{sec:qualitative-proof}

\begin{lemma}[Kernel reduction]\label{lem:kernel}
If $G$ has the Haagerup property, then
\[
 \ker\chi=\ker\iota\subset Q.
\]
\end{lemma}
\begin{proof}
One inclusion follows from \eqref{eq:factor}. For the other, let
$x=(x_c)_c\in Q$ satisfy $\chi(x)=0$. Fix an arbitrary target control
$c$ and use Proposition~\ref{prop:FD} to choose $D(c),s$. Compatibility of $x$ gives
$r^Q_{s,c}x_s=x_c$, and the $D(c)$-coordinate of its character is zero.
Consequently
\[
 \iota_c(x_c)=\nu_c\bigl(\chi_{D(c),s}(x_s)\bigr)=0.
\]
This holds at every $c$, so $\iota(x)=0$ in $\mathcal F$.
\end{proof}

\begin{lemma}[Lifting a compatible geometric class]\label{lem:lift}
For a countable Haagerup group, every $z\in\Klim$ is the family
$(p_d^*\xi)_d$ for some $\xi\in KK^G(\C,\C)$.
\end{lemma}
\begin{proof}
Choose cofinal scales $d_1<d_2<\cdots$ and a locally compact mapping
telescope $\mathcal X$ of the Rips inclusions. It is a model for
$\EG$. Its finite initial telescopes are $G$-compact and equivariantly
deformation retract onto the corresponding Rips complexes.
To see the universal proper-space property, first subdivide to remove
simplex inversions. A finite subgroup has a fixed probability vector,
and any compact family of its fixed probability vectors contracts by
linear interpolation at a sufficiently large Rips scale. An infinite
subgroup has no invariant finitely supported probability vector.
Thus the telescope has the required fixed-point homotopy types.
The representable equivariant Milnor sequence gives
\begin{equation}\label{eq:Milnor}
 0\longrightarrow\varprojlim\nolimits^1_n K_G^{-1}(P_{d_n}(G))
 \longrightarrow RK_G^0(\mathcal X)
 \xrightarrow{\mathrm{res}}\Klim\longrightarrow0.
\end{equation}
See Emerson--Meyer~\cite[Equation~(4.6)]{EmersonMeyer}. The
surjectivity of $\mathrm{res}$ gives a lift
$\widetilde z\in RK_G^0(\mathcal X)$ of $z$; its ambiguity is
precisely the image of the $\varprojlim^1$ term.

The Higson--Kasparov theorem gives $\gamma=1$ for $G$~\cite{higson_kasparov_2001}.
Kasparov's dual assembly theorem therefore makes
\[
 p^*:KK^G(\C,\C)\xrightarrow{\cong}
 RKK^G(\mathcal X;\C,\C)=RK_G^0(\mathcal X)
\]
an isomorphism. For the proper-space version, see
\cite[Theorem~4.1]{dadarlat_2021}, which states the result from
\cite[Theorem~6.5]{kasparov_1988}: in general the kernel is
$(1-\gamma)KK^G(\C,\C)$. Take
$\xi=(p^*)^{-1}\widetilde z$. Restriction gives $p_d^*\xi=z_d$ at
every scale.
\end{proof}

\begin{proof}[Proof of Theorem~\ref{thm:qual}]
QR-quasidiagonality, in Definition~\ref{def:comparison}, makes $\iota$
an isomorphism. Lemma~\ref{lem:kernel} therefore
proves injectivity of $\chi$.

Let $z\in\Klim$. Choose $\xi$ as in Lemma~\ref{lem:lift} and put
\[
 x=\iota^{-1}(\kappa(\xi)).
\]
Equations \eqref{eq:factor} and \eqref{eq:exact} give
\[
 \chi_d(x)=\chi^F_d(\kappa(\xi))=p_d^*\xi=z_d
 \qquad\hbox{for every }d.
\]
Thus $\chi(x)=z$, proving surjectivity.

The chosen Milnor lift need not be unique. Nevertheless, the resulting
$x$ is unique because $\chi$ has already been proved injective.
This also proves additivity of the inverse: two lifted sums have the
same character and hence coincide. Thus the inverse is independent
of the choice of Milnor lifts.
\end{proof}

\begin{proof}[Proof of \Cref{thm:QKQD-Haagerup}]
By \Cref{def:quantitative-QR-unitary}, the unitary-valued comparison has
compatible inverse homomorphisms satisfying
\eqref{dw:eq:fd-input}, with controls chosen independently of the class. The Haagerup
hypothesis and \Cref{thm:HK-proper-algebra} supply the proper coefficient
algebra and classes satisfying \eqref{dw:eq:beta-alpha}.
Apply \Cref{thm:scalar-character-comparison}. Its proof gives both global
inverse identities \eqref{dw:eq:global-F}--\eqref{dw:eq:global-K}, and the
finite-dimensional homomorphisms defining the inverse \eqref{dw:eq:psi}. In particular,
\[
 \chi^{-1}=\iota^{-1}\nu
\]
on ordinary inverse limits, proving \eqref{eq:QKQD-Haagerup-character}.
The first component identity detects equality after restriction to a
prescribed group control, and the second factors the restriction from a
fixed larger Rips scale through a fixed quasi-representation stage.
The argument of \Cref{prop:main-controlled-equivalence} gives exactly the
quantifier orders in (i) and (ii). When necessary enlarge the source Rips
scale once to make $d'>d$.
\end{proof}

\subsection{Global u.c.p. maps and the lifting comparison}\label{subsec:ucp-lifting}

\subsubsection{A lifting lemma for the maximal group algebra}

The u.c.p. route to unitary-valued QR-quasidiagonality uses the
lifting property; in particular, it applies to the surface-group examples.  We isolate the required consequence, including the Fredholm and
parameterized cases.

\begin{lemma}[approximation using the lifting property]
\label{lem:LP-CP-replacement}
Let $G$ be countable and suppose that $C^*_{\max}(G)$ has the lifting
property: for every quotient $B\to B/J$, every u.c.p. map
$C^*_{\max}(G)\to B/J$ has a u.c.p. lift to $B$.  Given a finite set
$E\subseteq G$ and $\eta>0$, there is a control
$c=(F,\varepsilon)\in\mathcal C_G$ with $E\subseteq F$ such that the
following holds for every unital $C^*$-algebra $D$.

If $q$ is a unitary $(F,\varepsilon)$-quasi-representation in $D$ and
$q(e)=1$, then there is a u.c.p. map
\[
 \phi:C^*_{\max}(G)\longrightarrow D
\]
such that
\[
 \max_{g\in E}\|\phi(g)-q(g)\|<\eta.
\]
The same control works when $D$ is either of the algebras
\begin{align*}
 \mathcal D_T&=\{a\in B(\mathcal E)_{\mathrm{even}}:[a,T]\in
 K(\mathcal E)\},\\
 \mathcal D_T^P&=\{a\in
 \mathcal L_{C(P)}(\mathcal E_P)_{\mathrm{even}}:[a,T]\in
 \mathcal K_{C(P)}(\mathcal E_P)\},
\end{align*}
where $P$ is compact metrizable.  Thus in a Fredholm cycle, or in one
$P$-parameterized family of such cycles, the approximating u.c.p. map still
has compact commutator with $T$, and the estimate is uniform over $P$.
\end{lemma}

\begin{proof}
Suppose the first assertion is false for some $(E,\eta)$.  Choose a cofinal
sequence $c_n=(F_n,\varepsilon_n)$ with
$F_nF_n\subseteq F_{n+1}$, $\bigcup_nF_n=G$, and
$\varepsilon_n\downarrow0$.  Passing to a subsequence if necessary, there
are unital $C^*$-algebras $D_n$ and normalized unitary
$c_n$-quasi-representations $q_n$ for which no u.c.p. map has the asserted
$E$-estimate.  Extend each $q_n$ arbitrarily to a unitary-valued map on all
of $G$; this does not affect any relation tested at the $n$th stage.

Put
\[
 D_\infty=\prod_{n=1}^\infty D_n\Big/\bigoplus_{n=1}^\infty D_n
\]
and let $[(d_n)]$ denote the class of a bounded sequence.  For $g\in G$ set
\[
 \overline q(g)=[(q_n(g))_n],
\]
where the finitely many coordinates preceding the first occurrence of $g$
in $F_n$ may be chosen arbitrarily.  For fixed $g,h$, the defect
$\|q_n(g)q_n(h)-q_n(gh)\|$ tends to zero.  Hence
$\overline q(g)\overline q(h)=\overline q(gh)$, and
$\overline q(e)=1$.  The universal property of the maximal group algebra
therefore gives a unital $*$-homomorphism
\[
 \overline q:C^*_{\max}(G)\longrightarrow D_\infty.
\]
By the lifting property it has a u.c.p. lift
$\Phi:C^*_{\max}(G)\to\prod_nD_n$.  Write
$\phi_n$ for the composition of $\Phi$ with the $n$th coordinate
projection.  Since $(\phi_n(g)-q_n(g))_n$ represents zero in $D_\infty$,
\[
 \max_{g\in E}\|\phi_n(g)-q_n(g)\|\longrightarrow0.
\]
For sufficiently large $n$ this contradicts the defining property of
$q_n$.  The required control therefore exists.

The algebras $\mathcal D_T$ and $\mathcal D_T^P$ are unital $C^*$-algebras;
the verification for $\mathcal D_T$ was given in the proof of
\Cref{lem:CP-cycle-correction}, and the $C(P)$-module proof is identical.
For a Fredholm quasi-representation,
\[
 [q(g),T]
 =\bigl(\operatorname{Ad}_{q(g)}(T)-T\bigr)q(g)
\]
is compact, so every tested $q(g)$ belongs to the indicated algebra.  Apply
the first part with this algebra as $D$.  In the parameterized case its norm
is the Hilbert $C(P)$-module operator norm, which is already the supremum
norm over the parameter, using a single lift over the parameter algebra.
\end{proof}

\begin{remark}[what lifting hypothesis is actually used]
\label{rem:LLP-versus-LP}
The full lifting property in the lemma is needed only because the conclusion
asks for one u.c.p. map on all of $A$ with values in an arbitrary algebra
$D$, and because the Fredholm application takes
$D=\mathcal D_T$ or $\mathcal D_T^P$.  Two useful weaker statements follow
from the same contradiction argument.
\begin{enumerate}[label=(\roman*)]
\item If the conclusion is required only on the finite-dimensional operator
system spanned by the finitely many $u_g$ that occur in the test, the local
lifting property suffices.  This is enough for the local u.c.p. pictures of
\Cref{def:local-ucp-pictures}.
\item If $D=M_n(\mathbb C)$ and only finite-dimensional cycles are being
considered, the weak matricial local lifting property suffices.  After
lifting to the product of matrix algebras, injectivity of that product
extends the local lift to a u.c.p. map on $A$; this is precisely the argument
of \cite[Proposition~2.7]{willett_conditional_2026}.
\end{enumerate}
Thus the conversion of ordinary finite-dimensional quasi-representations to
u.c.p. ones does not require nuclearity.  Nuclearity in
Section~\ref{sec:qd} is one convenient source of CP-stability and of the
nuclear stable-uniqueness theorem used there.  The nonnuclear stable-uniqueness
theorem used for surface groups instead assumes the approximate
$K$-homology UCT and a separating class of finite-dimensional
representations.  Neither LP nor LLP alone supplies the finite-dimensional
existence and uniqueness statements.
\end{remark}

We now turn the preceding existence assertion into homomorphisms on the
controlled groups.  This step is necessary: a u.c.p. lift is not unique, so
existence of a lift and polar decomposition by themselves do not identify
two groups.

\subsubsection{Global u.c.p. controls and classes}

\begin{definition}[global u.c.p. group controls]
\label{def:global-ucp-controls}
Put $A=C^*_{\max}(G)$, and denote its canonical unitary corresponding to
$g\in G$ by $u_g$.  Let $\mathcal C_G^{\mathrm{sym}}$ be the directed set of
pairs
$a=(E,\delta)$, where $E\subseteq G$ is finite, symmetric, and contains $e$,
and $\delta>0$.  The order is
\[
 (E,\delta)\preccurlyeq(E',\delta')
 \quad\Longleftrightarrow\quad E\subseteq E'
 \text{ and }\delta'\leq\delta.
\]
Thus $\mathcal C_G^{\mathrm{sym}}$ is the cofinal subset of
$\mathcal C_G$ obtained by requiring the finite set to be symmetric and to
contain the identity.  We use this cofinal subset because a global u.c.p. map
automatically respects adjoints, so the tests at $g$ and $g^{-1}$ naturally
belong to the same control.
A \emph{global u.c.p. $a$-quasi-representation} on a Hilbert space $H$ is a
u.c.p. map $\phi:A\to B(H)$ such that
\begin{equation}\label{eq:global-ucp-multiplicativity}
 \|\phi(u_{gh})-\phi(u_g)\phi(u_h)\|<\delta
 \qquad(g,h\in E).
\end{equation}
Let $\QRucpstage{G}{a}$ be the Grothendieck group of the abelian semigroup
of homotopy classes of the finite-dimensional such maps, with addition given
by direct sum, exactly as in \Cref{prop:QR-two-descriptions}.  Here a homotopy from
$\phi_0$ to $\phi_1$ is one u.c.p. map
$A\to C([0,1],M_n)$ whose endpoint evaluations are $\phi_0$ and $\phi_1$
and which satisfies \eqref{eq:global-ucp-multiplicativity} in the supremum
norm.

Let $\FQRucpstage{G}{a}$ be the set of stable-homotopy classes of pairs
$(\phi,T)$ on a graded Hilbert space, where $T$ is odd,
\[
 T-T^*,\ T^2-1\in K(H),
 \qquad
 \phi(A)\subseteq
 \mathcal D_T:=\{b\in B(H)_{\mathrm{even}}:[b,T]\in K(H)\}.
\]
A pair is degenerate when $T=T^*$, $T^2=1$, and
$[\phi(a),T]=0$ for every $a\in A$.  Homotopy, direct sum, and addition of
degenerate cycles are exactly as in \Cref{def:FQR-cycle}, with coefficient
algebra $C([0,1])$.  Reversing the grading gives the inverse, as in the proof
following \Cref{def:FQR-final}, so $\FQRucpstage{G}{a}$ is an abelian group.
Restriction of the tested set and relaxation of the tolerance give bonding
homomorphisms, denoted
\[
 \begin{aligned}
 r^{\mathrm{QR,ucp}}_{a',a}:&\ \QRucpstage{G}{a'}
 \longrightarrow\QRucpstage{G}{a},\\
 r^{\mathrm{FQR,ucp}}_{a',a}:&\ \FQRucpstage{G}{a'}
 \longrightarrow\FQRucpstage{G}{a}.
 \end{aligned}
\]

Sending a finite-dimensional map $\phi:A\to M_n(\mathbb C)$ to the cycle
on $\mathbb C^n\oplus0$ with Fredholm operator zero gives a homomorphism
\begin{equation}\label{eq:ucp-finite-cycle-map}
 \iota^{\mathrm{ucp}}_a:\QRucpstage{G}{a}
       \longrightarrow \FQRucpstage{G}{a}.
\end{equation}
These homomorphisms commute with all bonding maps.  We use the unambiguous
notations
\begin{align}
 \mathrm{QR}^{\mathrm{ucp}}(G)
   &:=\varprojlim_{a\in\mathcal C_G^{\mathrm{sym}}}
      \QRucpstage{G}{a},
   \label{eq:QR-ucp-limit}\\
 \mathrm{FQR}^{\mathrm{ucp}}(G)
   &:=\varprojlim_{a\in\mathcal C_G^{\mathrm{sym}}}
      \FQRucpstage{G}{a}.
   \label{eq:FQR-ucp-limit}
\end{align}
The maps in \eqref{eq:ucp-finite-cycle-map} induce a homomorphism
\begin{equation}\label{eq:ucp-limit-cycle-map}
 \iota^{\mathrm{ucp}}:\mathrm{QR}^{\mathrm{ucp}}(G)
       \longrightarrow\mathrm{FQR}^{\mathrm{ucp}}(G).
\end{equation}

These are the full-domain counterparts of
\Cref{def:local-ucp-pictures}.  The
full-domain condition is used here because both the lifting theorem and the
stable-uniqueness theorem concern maps on $C^*_{\max}(G)$, not merely maps
on an unspecified finite-dimensional operator system.
\end{definition}

\subsubsection{The comparison homomorphisms}

\begin{proposition}[controlled unitary--u.c.p. comparison]
\label{prop:LP-unitary-ucp-comparison}
Suppose that $C^*_{\max}(G)$ has the lifting property.  The unitary systems
$(\mathrm{QR}(c))_{c\in\mathcal C_G}$ and
$(\mathrm{FQR}(*;c;\mathbb C))_{c\in\mathcal C_G}$ are equivalent, in the
precise sense of \Cref{def:inverse-system-equivalence}, to the global u.c.p.
systems
$(\QRucpstage{G}{a})_{a\in\mathcal C_G^{\mathrm{sym}}}$ and
$(\FQRucpstage{G}{a})_{a\in\mathcal C_G^{\mathrm{sym}}}$, respectively.

More explicitly, one can choose order-preserving maps
\[
 \lambda:\mathcal C_G^{\mathrm{sym}}\longrightarrow\mathcal C_G,
 \qquad
 \pi:\mathcal C_G\longrightarrow\mathcal C_G^{\mathrm{sym}}
\]
and homomorphisms, for $a\in\mathcal C_G^{\mathrm{sym}}$ and
$c\in\mathcal C_G$,
\begin{align*}
 L^{\mathrm{QR}}_a &: \mathrm{QR}(\lambda(a))
   \longrightarrow \QRucpstage{G}{a},\\
 P^{\mathrm{QR}}_c &: \QRucpstage{G}{\pi(c)}
   \longrightarrow \mathrm{QR}(c),\\
 L^{\mathrm{FQR}}_a &: \mathrm{FQR}(*;\lambda(a);\mathbb C)
   \longrightarrow \FQRucpstage{G}{a},\\
 P^{\mathrm{FQR}}_c &: \FQRucpstage{G}{\pi(c)}
   \longrightarrow \mathrm{FQR}(*;c;\mathbb C),
\end{align*}
such that $\lambda(\pi(c))\succeq c$,
$\pi(\lambda(a))\succeq a$, and
\begin{align}
 P^{\mathrm{QR}}_cL^{\mathrm{QR}}_{\pi(c)}
   &=r^{\mathrm{QR}}_{\lambda(\pi(c)),c},&
 L^{\mathrm{QR}}_aP^{\mathrm{QR}}_{\lambda(a)}
   &=r^{\mathrm{QR,ucp}}_{\pi(\lambda(a)),a},
 \label{eq:LP-comparison-composites-Q}\\
 P^{\mathrm{FQR}}_cL^{\mathrm{FQR}}_{\pi(c)}
   &=r^{\mathrm{FQR}}_{\lambda(\pi(c)),c},&
 L^{\mathrm{FQR}}_aP^{\mathrm{FQR}}_{\lambda(a)}
   &=r^{\mathrm{FQR,ucp}}_{\pi(\lambda(a)),a}.
 \label{eq:LP-comparison-composites-K}
\end{align}
The component maps commute with all bonding homomorphisms in the two
systems.  They also commute with the inclusions of finite-dimensional cycles:
after applying the indicated bonding maps, the square
\[
\begin{tikzcd}
 \mathrm{QR}(\lambda(a)) \arrow[r,"L^{\mathrm{QR}}_a"]
   \arrow[d,"\iota_{\lambda(a)}"'] &
 \QRucpstage{G}{a} \arrow[d,"\iota^{\mathrm{ucp}}_a"]\\
 \mathrm{FQR}(*;\lambda(a);\mathbb C)
   \arrow[r,"L^{\mathrm{FQR}}_a"'] & \FQRucpstage{G}{a}
\end{tikzcd}
\]
commutes, and the analogous square for polar correction commutes.
\end{proposition}

\begin{proof}
We give the construction and the estimates, including the points needed for
well-definedness on stable-homotopy classes.

First fix a u.c.p. target control $a=(E,\delta)$.  Enlarge $E$ to the finite
symmetric set
\[
 E^\sharp=E\cup E^2\cup(E^2)^{-1}.
\]
Choose $0<\eta<\min\{\delta/12,1/20\}$.  Apply
\Cref{lem:LP-CP-replacement} to $(E^\sharp,\eta)$, and denote the resulting choices by
$(F_{\mathrm{LP}},\varepsilon_{\mathrm{LP}})$.  Choose a stronger control
$\lambda(a)=(F,\varepsilon)$ so that $F$ is symmetric, $e\in F$,
$F_{\mathrm{LP}}\subseteq F$, $2\varepsilon\leq
\varepsilon_{\mathrm{LP}}$, and
\begin{equation}\label{eq:LP-source-smallness}
 2\varepsilon+3\eta<\delta.
\end{equation}
For a unitary cycle, first apply \Cref{lem:qrep-normalize-identity}; this
replaces its defect by at most $2\varepsilon$, gives a homotopy with that
defect, and preserves the Fredholm condition.  The normalized map therefore
satisfies the control returned by the lifting lemma, which gives a u.c.p.
map $\phi:A\to D$ with
\begin{equation}\label{eq:LP-replacement-closeness}
 \max_{k\in E^\sharp}\|\phi(u_k)-q(k)\|<\eta,
\end{equation}
where $D=M_n$ in the finite-dimensional case and $D=\mathcal D_T$ in the
Fredholm case.  Since u.c.p. maps are contractive, for $g,h\in E$,
\begin{align}
 \|\phi(u_{gh})-\phi(u_g)\phi(u_h)\|
 &\leq \|\phi(u_{gh})-q(gh)\|
      +\|q(gh)-q(g)q(h)\|\notag\\
 &\quad+\|q(g)-\phi(u_g)\|+\|q(h)-\phi(u_h)\|\notag\\
 &<2\varepsilon+3\eta<\delta.
 \label{eq:LP-replacement-mult-estimate}
\end{align}
Thus $\phi$ represents a class at the $a$-stage.

The class is independent of the lift.  If $\phi_0$ and $\phi_1$ both
satisfy \eqref{eq:LP-replacement-closeness}, then
$\phi_t=(1-t)\phi_0+t\phi_1$ is u.c.p. and is still within $\eta$ of $q$ on
$E^\sharp$.  Estimate \eqref{eq:LP-replacement-mult-estimate} therefore
holds uniformly in $t$.  In the Fredholm case the convex path takes values
in the $C^*$-algebra $\mathcal D_T$, so it is a Fredholm homotopy with $T$
fixed.  This proves independence without selecting a preferred lift.

The same observation proves additivity: the block sum of replacements of
$q_0$ and $q_1$, and any replacement chosen for $q_0\oplus q_1$, are joined
by the preceding convex homotopy.  If a unitary Fredholm cycle is degenerate,
apply the lifting lemma in the unital $C^*$-algebra
\[
 \mathcal D_T^0=\{b\in B(H)_{\mathrm{even}}:[b,T]=0\}.
\]
The resulting u.c.p. cycle is degenerate, so the construction respects the
zero relation.  Finally, if $(q_s,T_s)$ is a homotopy, apply the parameterized
part of \Cref{lem:LP-CP-replacement} once, with $P=[0,1]$.  This produces one
u.c.p. map into the corresponding $C([0,1])$-algebra and one uniform
estimate, not unrelated lifts at the individual parameters.  Convex
homotopies at the two endpoints connect this replacement to any replacements
chosen there.  We have therefore obtained the homomorphisms
$L^{\mathrm{QR}}_a$ and $L^{\mathrm{FQR}}_a$.

We next construct polar correction.  Fix a unitary target control
$c=(F_0,\varepsilon_0)$.  Choose a u.c.p. control
$\pi(c)=(E_0,\delta_0)$ with
\[
 F_0^2\cup(F_0^2)^{-1}\subseteq E_0,
 \qquad 0<\delta_0<\min\{1/10,\varepsilon_0/7\}.
\]
If $\phi$ is a global u.c.p. $\pi(c)$-quasi-representation, then, for
$k\in F_0^2$, complete positivity gives
$\phi(u_{k^{-1}})=\phi(u_k)^*$, and
\begin{align*}
 \|\phi(u_k)^*\phi(u_k)-1\|&<\delta_0,\\
 \|\phi(u_k)\phi(u_k)^*-1\|&<\delta_0.
\end{align*}
Define
\begin{equation}\label{eq:LP-polar-map}
 p_\phi(k)=\phi(u_k)
   \bigl(\phi(u_k)^*\phi(u_k)\bigr)^{-1/2}
 \qquad(k\in F_0^2).
\end{equation}
By \Cref{lem:polar-correction},
$\|p_\phi(k)-\phi(u_k)\|\leq2\delta_0$ and
$p_\phi$ is an $(F_0,7\delta_0)$-quasi-representation, hence an
$(F_0,\varepsilon_0)$-quasi-representation.  Functional calculus commutes
with direct sums and continuous parameter families.  It takes place inside
$\mathcal D_T$ in the Fredholm case; if the u.c.p. cycle is degenerate, it
takes place inside $\mathcal D_T^0$.  Consequently
\eqref{eq:LP-polar-map} induces homomorphisms $P^{\mathrm{QR}}_c$ and
$P^{\mathrm{FQR}}_c$ on the controlled groups.

It remains to verify the two composite identities.  For the first, choose
the source control in the lifting construction still stronger, if necessary,
and choose $\eta$ so small that
\begin{equation}\label{eq:LP-first-composite-margin}
 2\varepsilon+15\eta<\varepsilon_0.
\end{equation}
If $\phi$ is an $\eta$-replacement of the normalized unitary map $q$, then
\[
 \|\phi(u_k)^*\phi(u_k)-1\|,
 \ \|\phi(u_k)\phi(u_k)^*-1\|<2\eta
\]
on the required finite set.  Applying \Cref{lem:polar-correction} with
$2\eta$ in place of $\delta$ shows
\begin{equation}\label{eq:LP-polar-recovery-distance}
 \|p_\phi(k)-q(k)\|<5\eta.
\end{equation}
The short-geodesic path from $q(k)$ to $p_\phi(k)$ used in
\Cref{lemma:unitary-implies-homotopy} has multiplication defect at most
$2\varepsilon+15\eta$, by the same three-term estimate.  In the Fredholm
case the path lies in $\mathcal D_T$.  Equations
\eqref{eq:LP-first-composite-margin}--
\eqref{eq:LP-polar-recovery-distance} prove the left identities in
\eqref{eq:LP-comparison-composites-Q} and
\eqref{eq:LP-comparison-composites-K}.

For the other composite, start with a u.c.p. map $\phi$ at a sufficiently
strong control, put $q=p_\phi$, and choose a u.c.p. replacement $\psi$ of
$q$.  On the finite set required by the target control,
\[
 \|\phi(u_k)-q(k)\|\leq2\delta_0,
 \qquad \|\psi(u_k)-q(k)\|<\eta.
\]
Every member of the convex path
$(1-t)\phi+t\psi$ is therefore within
$\max\{2\delta_0,\eta\}$ of $q$.  Its multiplication defect is at most
\[
 7\delta_0+3\max\{2\delta_0,\eta\}.
\]
Taking $\delta_0$ and $\eta$ small enough compared with the prescribed
target tolerance proves the right identities in
\eqref{eq:LP-comparison-composites-Q} and
\eqref{eq:LP-comparison-composites-K}.  The same convex path lies in
$\mathcal D_T$, and in $\mathcal D_T^0$ for a degenerate cycle.

For completeness, we explain how the choices are made simultaneously and
compatibly.  Because $G$ is countable, both directed sets have countable
cofinal subsets: exhaust $G$ by finite symmetric sets and use positive
rational tolerances tending to zero.  Choose increasing cofinal sequences
and perform the preceding construction recursively.  At the $n$th step take
the union of all finite sets required in the first $n$ steps and the minimum
of all their positive tolerances.  Also impose
\eqref{eq:LP-first-composite-margin} for the first $n$ target controls.
This gives order-preserving $\lambda$ and $\pi$ on the cofinal sequences.
If two routes around a bonding square are applied to one unitary cycle, the resulting
maps are two u.c.p. replacements of the same restriction of that cycle;
the convex homotopy above proves equality at the target stage.  If the two
routes start with a u.c.p. cycle, polar functional calculus commutes exactly
with restriction.  Hence all bonding squares commute.  For an arbitrary
control, use the first member of the chosen cofinal sequence above it and
then apply the relevant bonding homomorphism.  The same argument shows that
this extension is independent of the intermediate member and gives the
displayed identities on the full directed sets.  These are precisely the maps and identities
required by \Cref{def:inverse-system-equivalence}.

All constructions use the same u.c.p. map before and after regarding a
finite-dimensional map as a Fredholm cycle with operator zero.  This proves
the final compatibility assertion.
\end{proof}

\begin{corollary}[the two inverse-limit groups]
\label{cor:LP-unitary-ucp-limits}
Under the hypothesis of \Cref{prop:LP-unitary-ucp-comparison}, the formulas
\begin{align*}
 (x_c)_{c\in\mathcal C_G}
   &\longmapsto
   \bigl(L^{\mathrm{QR}}_a(x_{\lambda(a)})\bigr)_
      {a\in\mathcal C_G^{\mathrm{sym}}},\\
 (y_a)_{a\in\mathcal C_G^{\mathrm{sym}}}
   &\longmapsto
   \bigl(P^{\mathrm{QR}}_c(y_{\pi(c)})\bigr)_{c\in\mathcal C_G}
\end{align*}
are mutually inverse isomorphisms
\[
 \mathrm{QR}(G)\cong
 \varprojlim_{a\in\mathcal C_G^{\mathrm{sym}}}\QRucpstage{G}{a}.
\]
The same formulas with $\mathrm{QR}$ replaced by $\mathrm{FQR}$ give
\[
 \mathrm{FQR}(*;G;\mathbb C)\cong
 \varprojlim_{a\in\mathcal C_G^{\mathrm{sym}}}\FQRucpstage{G}{a}.
\]
These isomorphisms intertwine the homomorphism from finite-dimensional
cycles to Fredholm cycles.  In particular, the unitary-valued and global
u.c.p. definitions give the same inverse-limit groups for groups whose
maximal group algebra has the lifting property, although there is no
canonical choice of a u.c.p. replacement at the level of individual maps.
\end{corollary}

\begin{proof}
The bonding-square identities show that the two displayed formulas define
compatible families.  Equations
\eqref{eq:LP-comparison-composites-Q} and
\eqref{eq:LP-comparison-composites-K} show coordinate by coordinate that the
two composites are the identity.  Compatibility with finite-dimensional
cycles is the last assertion of
\Cref{prop:LP-unitary-ucp-comparison}.
\end{proof}

\begin{remark}
The proposition clarifies two facts that are easily confused. Polar
correction is a canonical operation once the u.c.p. map is sufficiently
multiplicative.  The reverse operation uses the lifting property and is not
canonical at the level of maps; it becomes canonical only as a class at a
prescribed weaker control, because any two choices are joined by the convex
u.c.p. homotopy in the proof.
\end{remark}

\subsection{The u.c.p. comparison properties and transfer}\label{subsec:ucp-QR-properties}

Dadarlat's definition of $K$-quasidiagonality for a separable
$C^*$-algebra requires every even $K$-homology class to have a
quasidiagonal representative
\cite[Definition~2.2]{dadarlat_bundles_2014}; see also
\cite[Definition~4.4]{dadarlat_2021}. The u.c.p. formulation below is stated
in terms of the comparison homomorphism
\eqref{eq:ucp-finite-cycle-map}. It includes both the existence of
finite-dimensional representatives and their stable uniqueness.
The quantitative condition additionally specifies the finite sets and
tolerances needed for these comparisons. The u.c.p. conditions are distinguished from the unitary-valued
conditions in \Cref{def:QR-unitary,def:quantitative-QR-unitary}; LP will
relate the two formulations. No lifting hypothesis is part of either
u.c.p. definition.

\begin{definition}[U.c.p. QR-quasidiagonality]
\label{def:KQD}
We say that $G$ is \emph{QR-quasidiagonal in the u.c.p. sense} if the canonical homomorphism
\[
 \iota^{\mathrm{ucp}}:\mathrm{QR}^{\mathrm{ucp}}(G)
       \longrightarrow\mathrm{FQR}^{\mathrm{ucp}}(G)
\]
is an isomorphism.  Surjectivity says that every limiting Fredholm class has
a finite-dimensional u.c.p. representative; injectivity is the stable
uniqueness assertion that two such representatives with the same Fredholm
class already determine the same limiting quasi-representation class.  Both
halves are included because the purpose here is classification, rather than
existence of finite-dimensional approximations alone.
\end{definition}

\begin{definition}[Quantitative u.c.p. QR-quasidiagonality]
\label{def:quantitative-KQD}
A countable discrete group $G$ is \emph{quantitatively
QR-quasidiagonal in the u.c.p. sense} if there exist an order-preserving map
\[
 \sigma:\mathcal C_G^{\mathrm{sym}}\longrightarrow
 \mathcal C_G^{\mathrm{sym}},
 \qquad \sigma(a)\succeq a,
\]
and homomorphisms
\[
 \rho_a:\FQRucpstage{G}{\sigma(a)}\longrightarrow
          \QRucpstage{G}{a}
 \qquad(a\in\mathcal C_G^{\mathrm{sym}})
\]
such that
\begin{align}
 \rho_a\iota^{\mathrm{ucp}}_{\sigma(a)}
   &=r^{\mathrm{QR,ucp}}_{\sigma(a),a},
   \label{eq:QKQD-left}\tag{QK1}\\
 \iota^{\mathrm{ucp}}_a\rho_a
   &=r^{\mathrm{FQR,ucp}}_{\sigma(a),a},
   \label{eq:QKQD-right}\tag{QK2}\\
 r^{\mathrm{QR,ucp}}_{b,a}\rho_b
   &=\rho_a r^{\mathrm{FQR,ucp}}_{\sigma(b),\sigma(a)}
   \qquad(a\preccurlyeq b).
   \label{eq:QKQD-coherence}\tag{QK3}
\end{align}
Thus the natural transformation
$\iota^{\mathrm{ucp}}_\bullet$ has homomorphisms defining the inverse after a specified
change of control.  The definition concerns only u.c.p. maps.  Ordinary
unitary-valued quasi-representations enter later through
\Cref{prop:LP-unitary-ucp-comparison}.
\end{definition}

\begin{proposition}[consequences for classes and their restrictions]
\label{prop:QKQD-consequences}
Quantitative u.c.p. QR-quasidiagonality implies u.c.p.
QR-quasidiagonality.  More
precisely, the inverse of \eqref{eq:ucp-limit-cycle-map} is
\begin{equation}\label{eq:QKQD-limit-inverse}
 (z_a)_{a\in\mathcal C_G^{\mathrm{sym}}}\longmapsto
 \bigl(\rho_a(z_{\sigma(a)})\bigr)_
   {a\in\mathcal C_G^{\mathrm{sym}}}.
\end{equation}
At a fixed target control $a$, equations \eqref{eq:QKQD-left} and
\eqref{eq:QKQD-right} say, respectively,
\begin{enumerate}[label=(\roman*)]
\item if $x\in\QRucpstage{G}{\sigma(a)}$ and
$\iota^{\mathrm{ucp}}_{\sigma(a)}(x)=0$, then
$r^{\mathrm{QR,ucp}}_{\sigma(a),a}(x)=0$;
\item every $z\in\FQRucpstage{G}{\sigma(a)}$ becomes the
finite-dimensional class $\iota^{\mathrm{ucp}}_a(\rho_a(z))$ after
restriction to $a$.
\end{enumerate}
\end{proposition}

\begin{proof}
Equation \eqref{eq:QKQD-coherence} makes the family on the right-hand side
of \eqref{eq:QKQD-limit-inverse} compatible.  The two composite identities
show coordinate by coordinate that it is a two-sided inverse.  Assertions
(i) and (ii) are obtained by applying \eqref{eq:QKQD-left} and
\eqref{eq:QKQD-right}, respectively.
\end{proof}

\begin{remark}[why the quantitative assertion is stronger]
An isomorphism between the two inverse limits does not by itself provide the
map $\sigma$, the component homomorphisms $\rho_a$, or any of the identities
\eqref{eq:QKQD-left}--\eqref{eq:QKQD-coherence}.  In particular, a direct
computation
\[
 \mathrm{QR}^{\mathrm{ucp}}(G)\cong\mathbb Z^2
 \quad\hbox{and}\quad
 \mathrm{FQR}^{\mathrm{ucp}}(G)\cong\mathbb Z^2
\]
can prove u.c.p. QR-quasidiagonality, once the canonical map is checked, but it
does not automatically prove quantitative u.c.p. QR-quasidiagonality.  The latter requires vanishing after restriction at uniformly chosen
controls, together with compatible choices of finite-dimensional representatives.  We shall first give the direct
limiting computation for surface groups and only then prove its quantitative
strengthening.
\end{remark}

\begin{remark}[relation with quantitative operator $K$-theory]
The terminology follows the finite-control philosophy of quantitative
operator $K$-theory.  Controlled methods of this kind enter Yu's proof of
the Novikov conjecture for groups with finite asymptotic dimension
\cite{yu_finite_asdim_1998}; quantitative operator $K$-theory for filtered
$C^*$-algebras was developed systematically in
\cite{oyono_oyono_yu_2015}.  Definition
\ref{def:quantitative-KQD} concerns the two specific u.c.p. systems
$\bigl(\QRucpstage{G}{a}\bigr)_a$ and
$\bigl(\FQRucpstage{G}{a}\bigr)_a$.  It is not a
renaming of the controlled almost-projection and almost-unitary groups of a
filtered algebra.
\end{remark}

\begin{proposition}[transfer to the unitary-valued theories]
\label{prop:QKQD-unitary-transfer}
Assume that $C^*_{\max}(G)$ has the lifting property.
\begin{enumerate}[label=(\roman*)]
\item If $G$ is QR-quasidiagonal in the u.c.p. sense, then the canonical homomorphism
\[
 \iota:\mathrm{QR}(G)\longrightarrow
          \mathrm{FQR}(*;G;\mathbb C)
\]
is an isomorphism.
\item If $G$ is quantitatively QR-quasidiagonal in the u.c.p. sense, then the natural
transformation $(\iota_c)_c$ between the two unitary-valued systems is an
inverse-system equivalence in the explicit sense of
\Cref{def:inverse-system-equivalence}. After a common strengthening
of controls, these maps give quantitative QR-quasidiagonality in
\Cref{def:quantitative-QR-unitary}.
\end{enumerate}
Thus the first conclusion gives \Cref{def:QR-unitary}, and the second
gives its quantitative version.
\end{proposition}

\begin{proof}
The two vertical comparisons in the commutative diagram
\[
\begin{tikzcd}[column sep=large]
 \mathrm{QR}(G) \arrow[r,"\iota"] \arrow[d,"\cong"'] &
 \mathrm{FQR}(*;G;\mathbb C) \arrow[d,"\cong"]\\
 \mathrm{QR}^{\mathrm{ucp}}(G)
   \arrow[r,"\iota^{\mathrm{ucp}}"'] &
 \mathrm{FQR}^{\mathrm{ucp}}(G)
\end{tikzcd}
\]
are the isomorphisms of \Cref{cor:LP-unitary-ucp-limits}; commutativity is
the compatibility assertion in
\Cref{prop:LP-unitary-ucp-comparison}.  This proves (i).  For (ii), compose
the two inverse-system equivalences in
\Cref{prop:LP-unitary-ucp-comparison} with the component maps in
\Cref{def:quantitative-KQD}, using
\Cref{lem:inverse-system-equivalence-composition}.  The resulting two
composites are bonding homomorphisms, exactly as required by
\Cref{def:inverse-system-equivalence}.
\end{proof}


\subsection{Quasidiagonal algebras and quantitative comparison}
\label{subsec:algebraic-QR-comparison}

We now give a sufficient condition for the u.c.p. comparison, independent
of LP and the Haagerup property. In this subsection and its proof, $A$
is any separable unital $C^*$-algebra. A finite control is
$a=(X,\varepsilon)$, where $X=X^*\subset A_1$ is finite and contains
$1$, and $\varepsilon>0$. Self-adjointness here means that $X$ is
closed under adjoints, not that each element of $X$ is self-adjoint.
For a global u.c.p. map $\phi$, put
\[
 \defect_X(\phi)=
 \max_{x,y\in X}\|\phi(xy)-\phi(x)\phi(y)\|.
\]
We abbreviate by $\QQ_a$ the Grothendieck group of homotopy classes of
global u.c.p. maps $A\to M_n(\mathbb C)$ with
$\defect_X(\phi)<\varepsilon$. Homotopies are global u.c.p. maps
$A\to C([0,1],M_n)$ with the same strict bound in the supremum norm;
zero-dimensional maps are allowed. The group $\FF_a$ consists of
stable-homotopy classes of even global u.c.p. maps
$\phi:A\to B(\HH^+\oplus\HH^-)$ with the same multiplication bound,
together with an odd operator $T$ satisfying
\[
 T-T^*,\quad T^2-1,\quad[T,\phi(x)]\in\mathcal K(\HH)
 \qquad(x\in A).
\]
Degenerate cycles, unitary equivalence, and Hilbert
$C([0,1])$-module homotopies use the conventions of
\Cref{def:global-ucp-controls}. The map
$\iota_a:\QQ_a\to\FF_a$ places a finite-dimensional map in even
degree with operator zero. The symbols $r^Q,r^F$ denote the bonding
maps for these algebra controls within this proof.

Ordinary quasidiagonality provides a fixed faithful quasidiagonal
representation for the whole homotopy. This is stronger than merely
choosing a quasidiagonal representative for each individual
$KK(A,\mathbb C)$-class. The proof keeps the relative index and the
prescribed finite-dimensional endpoints throughout the compression.

Define the increasing function
\begin{equation}\label{eq:theta}
 \theta(\varepsilon)=10^{-6}\min\{\varepsilon,1\}^{2},
 \qquad \tau(X,\varepsilon)=(X,\theta(\varepsilon)).
\end{equation}
Thus $\tau(a)\succeq a$ and $\tau$ is order preserving.

\begin{theorem}[Quantitative finite-dimensional comparison]
\label{thm:main}
Let $A$ be separable, unital, and quasidiagonal. With
\begin{equation}\label{eq:sigma}
 \sigma=\tau\circ\tau,\qquad
 \sigma(X,\varepsilon)
   =\left(X,10^{-18}\min\{\varepsilon,1\}^{4}\right),
\end{equation}
there are homomorphisms
\[
 \rho_a:\FF_{\sigma(a)}\longrightarrow\QQ_a
\]
such that
\begin{align}
 \rho_a\iota_{\sigma(a)}
    &=r^Q_{\sigma(a),a},\label{eq:left}\\
 \iota_a\rho_a
    &=r^F_{\sigma(a),a},\label{eq:right}\\
 r^Q_{b,a}\rho_b
    &=\rho_a r^F_{\sigma(b),\sigma(a)}
       \qquad(a\preccurlyeq b).\label{eq:coherence}
\end{align}
The test set is unchanged. The input tolerance is uniform over all
cycles and homotopies; the ranks of the finite-dimensional
representatives need not be uniform.
\end{theorem}

\begin{corollary}[Maximal group algebras]\label{cor:group}
Let $G$ be a countable discrete group. If $C^*_{\max}(G)$ is
quasidiagonal, then $G$ is quantitatively QR-quasidiagonal in the
u.c.p. sense of \Cref{def:quantitative-KQD}. For
$a=(E,\varepsilon)$ one may use
\[
 \sigma(a)=\left(E,10^{-18}\min\{\varepsilon,1\}^{4}\right).
\]
\end{corollary}
\begin{proof}
Apply Theorem~\ref{thm:main} to $A=C^*_{\max}(G)$, which is separable.
For a finite symmetric $E\subseteq G$ containing the identity, take
$X=\{u_g:g\in E\}\subset C^*_{\max}(G)$, where $u_g$ is the canonical
unitary. The groups and maps above are exactly the global u.c.p.
groups and the comparison in the definition. Equations
\eqref{eq:left}--\eqref{eq:coherence} are its three required identities.
\end{proof}

The two uses of $\tau$ have distinct purposes: the first gives a
finite-dimensional representative at output control $\tau(a)$, and
the second makes its restriction to $a$ independent of the choice.
This explains the composition $\sigma=\tau\circ\tau$.

\begin{corollary}[Character theorem for quasidiagonal maximal algebras]
\label{cor:character}
Let $G$ be a countable discrete group with the Haagerup property. If
$C^*_{\max}(G)$ is quasidiagonal and has the lifting property, then
the geometric character is an isomorphism of abelian groups:
\[
 \chi:\mathrm{QR}(G)\xrightarrow{\cong}
       \varprojlim_d K_G^0(P_d(G)).
\]
The finite-control detection and realization statements of
\Cref{thm:QKQD-Haagerup} also hold.
\end{corollary}
Here $P_d(G)$ uses the proper right-invariant metric fixed in
Section~\ref{sec:preliminaries}. The target is the ordinary inverse
limit of the finite-scale equivariant $K$-theory groups.

\begin{proof}[Proof of Corollary~\ref{cor:character}]
Corollary~\ref{cor:group} supplies the three quantitative u.c.p.
comparison identities. Under LP, the controlled lifting and polar
correction maps of \Cref{prop:LP-unitary-ucp-comparison,cor:LP-unitary-ucp-limits}
transfer this comparison to unitary-valued quasi-representations,
compatibly with the inclusion of finite-dimensional cycles.
The Haagerup hypothesis supplies the proper Dirac--dual-Dirac classes
used by the character comparison. Thus all hypotheses of
\Cref{thm:QKQD-Haagerup} hold, giving the stated character isomorphism
and the finite-control conclusions below.
\end{proof}

\paragraph{Road map.}
We first record the operator estimates and the absorption theorem used
in the proof. Next we turn a Fredholm cycle into a relative projection
in a fixed quasidiagonal representation and compress it. The same
construction over $C([0,1])$, with an endpoint-preserving normalization,
proves a finite-dimensional vanishing assertion. Finally, existence and
vanishing together define the homomorphisms $\rho_a$ without making
canonical choices of compression projections or absorption unitaries.
This use of absorption and stable uniqueness is related to
Dadarlat--Eilers's asymptotic-unitary-equivalence methods
\cite{dadarlat_eilers_2001}. The relative-index construction uses the method of
Section~\ref{sec:qd}; the finite-control and endpoint arguments for the global u.c.p.
comparison are included below. Figure~\ref{fig:cycles} records the successive
models of the same controlled class.

\begin{figure}[htbp]
\centering
\begin{tikzcd}[column sep=5.2em,row sep=3.2em]
 (\Phi,T) \arrow[r,"\text{normalize}"] &
 (\Phi,B) \arrow[r,"\text{relative index}"] &
 \RR(p_0,e_0;\Phi) \arrow[d,"\text{dilate}"] \\
 {[\Phi_+^{\mathrm{fd}}]-[\Phi_-^{\mathrm{fd}}]} &
 \RR(p,e;\Pi) \arrow[l,"\text{compress}"'] &
 \RR(\widehat p,\widehat e;\rho) \arrow[l,"\text{absorb}"']
\end{tikzcd}
\caption{The relative projection retains the index. Compression leaves
 a finite-dimensional representative and an exactly degenerate tail.
 The same construction on an interval retains the endpoint classes.}
\label{fig:cycles}
\end{figure}

\subsection{Proof of the quantitative comparison}\label{subsec:QR-comparison-proof}

\subsubsection{Completely positive interpolation and compression}
For u.c.p. maps $L_0,L_1$ and $L_t=(1-t)L_0+tL_1$, direct multiplication
gives
\begin{align}\label{eq:convex}
 L_t(xy)-L_t(x)L_t(y)
 ={}&(1-t)\bigl(L_0(xy)-L_0(x)L_0(y)\bigr)\notag\\
 &+t\bigl(L_1(xy)-L_1(x)L_1(y)\bigr)\notag\\
 &+t(1-t)(L_0(x)-L_1(x))(L_0(y)-L_1(y)).
\end{align}
In particular, if the maps are within $v$ on $X$, then
\begin{equation}\label{eq:convex-bound}
 \defect_X(L_t)\leq
 \max\{\defect_X(L_0),\defect_X(L_1)\}+v^2/4.
\end{equation}
If instead a u.c.p. map changes by at most $v$ on $X\cup X^2$, its
defect increases by at most $3v$.

For a projection $p$, compression $L_p(x)=pL(x)p|_{p\HH}$ is unital
on its range. The identity
\[
 L_p(xy)-L_p(x)L_p(y)
 =p\bigl(L(xy)-L(x)L(y)\bigr)p+pL(x)(1-p)L(y)p
\]
gives, when $\max_{x\in X}\|[p,L(x)]\|\leq\eta$,
\begin{equation}\label{eq:compression-bound}
 \defect_X(L_p)\leq\defect_X(L)+\eta^2.
\end{equation}
We will sometimes use the weaker bound $\defect_X(L)+\eta$ when
$\eta\leq1$. Every estimate remains valid on Hilbert modules.

\subsubsection{Spectral corrections and compact families}\label{subsec:QR-spectral}
We use the following elementary spectral estimate. If $p$ is a
projection, $a=a^*$, and $\|a-p\|<\eta<1/4$, then
\[
 p_s=\mathbf1_{(1/2,\infty)}((1-s)p+sa),\qquad 0\leq s\leq1,
\]
is a continuous projection path with $\|p_s-p\|\leq4\eta$.
Indeed, apply the resolvent formula on the circle $|z-1|=1/2$;
the two resolvent norms are at most $4$ and $2$. If $a-p$ is compact,
then $p_s-p$ is compact, by the same functional calculus in the
quotient. The argument is uniform in a compact parameter.

Choose a faithful unital quasidiagonal representation
$\pi_0:A\to B(H_0)$. Replacing it by
\[
 \pi=\bigoplus_{k=1}^{\infty}\pi_0
\]
gives a faithful essential representation on $H$, meaning
$\pi(A)\cap\K(H)=0$. It is still quasidiagonal: if $r_n$ are
quasidiagonal finite-rank projections for $\pi_0$, use $r_n$ on each
of the first $n$ copies and zero on the others. Denote the resulting
projections by $q_n$. They satisfy
\begin{equation}\label{eq:QD}
 q_n\longrightarrow1\ \text{strongly},\qquad
 \|[q_n,\pi(x)]\|\longrightarrow0\quad(x\in A).
\end{equation}
Put
\[
 \Pi=\pi\oplus\pi,\qquad
 e=\diag(0,1),\qquad Q_n=q_n\oplus q_n.
\]
If $\mathcal C\subset\K(H\oplus H)$ is norm compact, then
\begin{equation}\label{eq:uniform-tail}
 \sup_{k\in\mathcal C}
  \bigl(\|(1-Q_n)k\|+\|k(1-Q_n)\|\bigr)\longrightarrow0.
\end{equation}
For one compact operator this follows by finite-rank approximation;
a finite norm net proves the assertion for $\mathcal C$.

\begin{lemma}[Including a finite endpoint]\label{lem:cutoff}
Let $R$ be a finite-rank projection commuting with $e$. The projections
$Q_n$ above can be replaced by finite-rank projections $Q'_n$ such that
\[
 R\leq Q'_n,\qquad [Q'_n,e]=0,\qquad \|Q'_n-Q_n\|\longrightarrow0.
\]
They retain \eqref{eq:QD} and \eqref{eq:uniform-tail}.
\end{lemma}
\begin{proof}
For large $n$, $RQ_nR$ is invertible on $R(H\oplus H)$. The isometry
\[
 v_n=Q_nR(RQ_nR)^{-1/2}:R(H\oplus H)\longrightarrow Q_n(H\oplus H)
\]
converges in norm to the inclusion. Thus $R_n=v_nv_n^*\leq Q_n$
converges to $R$. These projections commute with $e$.
For close projections, the polar unitary of
$RR_n+(1-R)(1-R_n)$ sends $R_n$ to $R$, commutes with $e$, and
converges to $1$. Conjugating $Q_n$ by that unitary proves the claim.
The perturbation tends to zero, which preserves both asserted limits.
Monotonicity of the perturbed cutoffs is not needed.
\end{proof}

\subsubsection{A stable relation as one homotopy}

\begin{lemma}[Single-homotopy form]\label{lem:single-homotopy}
If cycles $z_0,z_1$ at control $a$ are stably homotopic, there are
degenerate cycles $D_0,D_1$ at the same control and one Hilbert
$C([0,1])$-module homotopy from $z_0\oplus D_0$ to $z_1\oplus D_1$.
In particular, a zero class has such a homotopy to a degenerate cycle.
The tolerance need not be relaxed.
\end{lemma}
\begin{proof}
An even unitary equivalence can be included in a homotopy. After
identifying its source and target graded Hilbert spaces, join the
implementing unitary $U=U^+\oplus U^-$ to $1$ in norm. The unitary
group is connected in finite dimension, and Kuiper's theorem gives
contractibility in each infinite-dimensional summand \cite{Kuiper}.
Conjugating both $\phi$ and $T$ along this path preserves their defect
norms and compactness exactly. One can instead use the intertwining
unitary directly as the identification at a concatenation endpoint.

Two module homotopies with identified endpoint fibres can be
concatenated. After rescaling to $[0,1/2]$ and $[1/2,1]$, take the
Hilbert module of pairs of sections whose values at $1/2$ agree under
the endpoint unitary. Inner products glue to a continuous function,
and the given operators and maps act componentwise. The compact
operator fields glue as well: finite-rank approximation on the two
closed subintervals, followed by a partition of unity near the joint,
gives compact module operators on the glued module. Equivalently,
compact operators form the corresponding pullback over the compact
operators on the common fibre. The module is separable and hence
countably generated. The norm of each glued defect is the maximum
of its two norms and is still less than the prescribed tolerance.

Define a relation by the existence of one homotopy after adding
possibly different degenerate cycles at the two ends. It is reflexive
and symmetric. It is transitive: if the first homotopy ends at
$z_1\oplus D_1$ and the second starts at $z_1\oplus D'_1$, add
$D'_1$ to the first homotopy and $D_1$ to the second, permute the two
degenerate summands, and concatenate. This relation contains unitary
equivalence, homotopy, and addition of a degenerate cycle, so it
contains the stable-homotopy relation. The reverse inclusion follows
from its definition. Taking $z_1=0$ gives the last assertion.
\end{proof}

\subsubsection{Absorption over the homotopy parameter algebra}
The following classical form of absorption is the only representation
theorem used in our finite-dimensional reduction.

\begin{proposition}[Kasparov absorption]\label{prop:absorption}
Let $A$ be separable and unital, let $D$ be a nonzero $\sigma$-unital
nuclear $C^*$-algebra, and let $\pi:A\to B(H)$ be a faithful unital
essential representation of infinite multiplicity. Put
$H_D=H\otimes D$ and $\pi_D=\pi\otimes1_{M(D)}$. Then
\[
 \pi_D(A)\cap\K_D(H_D)=\{0\}.
\]
If $\rho:A\to\LL_D(\mathcal E)$ is a unital representation on a
countably generated Hilbert $D$-module, then, for every finite
$Y\subset A$ and $v>0$, there is a $D$-linear unitary
\[
 U:\mathcal E\oplus H_D\longrightarrow H_D
\]
such that
\begin{equation}\label{eq:absorb}
 \max_{x\in Y}
 \|U(\rho(x)\oplus\pi_D(x))U^*-\pi_D(x)\|<v.
\end{equation}
For a graded representation, applying this assertion in each grading
summand gives an even unitary into $H_D\oplus H_D$.
\end{proposition}
\begin{proof}
We use Kasparov's scalar-amplification theorem
\cite[Theorem~6, p.~149]{Kas80}. Its hypotheses are that $A$ is
separable and unital, that $D$ has a countable approximate identity,
and that completely positive maps
$A/(A\cap\K(H))\to M_n(D)$ are nuclear. Here $A\cap\K(H)=0$ by
essentiality, and every completely positive map into $M_n(D)$ is
nuclear because $D$ is nuclear. Thus the theorem applies without
nuclearity of $A$. It gives approximate unitary absorption on the
standard module.

For the stated compact-intersection property, if
$\pi(x)\otimes1_{M(D)}$ is compact, multiplication on the left by $1\otimes d^*$ and on the right by
$1\otimes d$, for a nonzero $d\in D$, gives
$\pi(x)\otimes d^*d\in\K(H)\otimes D$. A slice functional that is
nonzero on $d^*d$ shows that $\pi(x)$ is compact. Essentiality and
faithfulness force $x=0$.

For a general $\mathcal E$, stabilization
\cite[Theorem~2]{Kas80} gives
$W:\mathcal E\oplus H_D\to H_D$. Apply standard-module absorption
to the unital representation
$\widetilde\rho=W(\rho\oplus\pi_D)W^*$, and then compose with
$W\oplus1$. The resulting unitary absorbs
$\rho\oplus\pi_D\oplus\pi_D$ into $\pi_D$. Infinite multiplicity
identifies the two copies of $\pi_D$ with one, giving
\eqref{eq:absorb}. Applying the same construction independently to
the positive and negative summands gives the graded assertion.
\end{proof}

Only $D=\C$ and $D=C([0,1])$ are used below. The fixed representation
in Section~\ref{subsec:QR-spectral} satisfies all the hypotheses. The scalar nature of its
amplification is important; the proposition makes no absorption claim
for an arbitrary representation into $\LL_D(H_D)$.

We also use Stinespring dilation for Hilbert modules
\cite[Theorem~3]{Kas80} and quasicentral approximate units for ideals
\cite{Arv}. We always take the explicit minimal dilation: complete
$A\odot\mathcal E$, after quotienting zero-length vectors, for
\[
 \langle a\otimes\xi,b\otimes\eta\rangle
   =\langle\xi,\Phi(a^*b)\eta\rangle.
\]
Left multiplication gives $\rho$, and $V\xi=1\otimes\xi$ is an
adjointable isometry with $V^*\rho(x)V=\Phi(x)$; the span of
$\rho(A)V\mathcal E$ is dense. For $D=C([0,1])$, evaluation at $t$
commutes with this construction: the map on algebraic tensors sends
$(a\otimes\xi)|_t$ to $a\otimes\xi(t)$ and extends to a unitary onto
the minimal dilation of $\Phi_t$. A block-diagonal endpoint map
therefore has the direct sum of the dilations of its summands.
The original endpoint spaces are retained through their Stinespring
isometries. This observation will identify the projections $R_\pm$
in Section~\ref{subsec:QR-vanishing} and justify their commutator bounds.

\subsubsection{A relative index projection}

We work over $D=\C$ or $D=C([0,1])$. A compact operator in this
section means a compact Hilbert $D$-module operator.

\begin{lemma}[Relative projection cycles]\label{lem:relative}
Let $p,e\in\LL_D(\HH)$ be projections with $p-e\in\K_D(\HH)$.
Let $L:A\to\LL_D(\HH)$ be u.c.p. and commute with $e$. Then
\begin{equation}\label{eq:relative-cycle}
 \RR(p,e;L)=
 \left(p\HH\oplus e\HH,
 \begin{pmatrix}pL(\cdot)p&0\\0&eL(\cdot)e\end{pmatrix},
 \begin{pmatrix}0&pe\\ep&0\end{pmatrix}\right)
\end{equation}
has compact Fredholm and commutator defects. It is a cycle at any
control satisfied by its two compressed maps. The construction
respects norm-continuous relative projection paths and u.c.p.
interpolations commuting with $e$. If $p=e$, it is degenerate.
\end{lemma}
\begin{proof}
The square of the displayed odd operator differs from the identity on
$p\HH\oplus e\HH$ by compact operators. For the lower left commutator,
\[
 ep\,pL(x)p-eL(x)e\,ep=e(p-e)L(x)p\in\K_D(\HH).
\]
The other corner follows by adjoints and replacement of $x$ by $x^*$.
This works for every $x\in A$. Compression is u.c.p. on its own range,
and \eqref{eq:compression-bound} gives the multiplicativity bounds.
The same identities prove the assertions about families and degeneracy.
\end{proof}

The lemma explains why norm-small changes of $L$ will be legitimate:
compactness of the cycle defects follows from the exact relation
$p-e\in\K$ and is not inferred from their being small in norm.

Fix an output control $a=(X,\varepsilon)$ and put
\begin{equation}\label{eq:budget}
 r=10^{-3}\min\{\varepsilon,1\},\qquad
 \delta=r^2=\theta(\varepsilon),\qquad Y=X\cup X^2.
\end{equation}
Take a cycle $(\Phi,T)$ with $\defect_X(\Phi)<\delta$.
The auxiliary set $Y=X\cup X^2$ is used for small commutators and
absorption accuracy, which can be arranged after the cycle is chosen.
The only source multiplicativity tests remain those on $X$. In
particular, the Stinespring estimate uses the pairs $(x^*,x)$ already
tested because $X^*=X$. This is why the control changes leave $X$
unchanged.

\subsubsection{Making commutators small without changing the index}
Replace $T$ first by $(T+T^*)/2$, then apply the clipping function to
$[-1,1]$. Both changes are compact and give Fredholm homotopies.
We have an odd self-adjoint contraction $F$ with $F^2-1$ compact.
These normalizations fix $0$ and self-adjoint unitaries.

Choose an even quasicentral positive compact contraction $h$ and set
\[
 B=(1-h)F(1-h).
\]
The change $B-F$ is compact. For $x\in Y$,
\begin{equation}\label{eq:qca-estimate}
 \|[B,\Phi(x)]\|
 \leq2\|F\|\,\|[h,\Phi(x)]\|
       +\|(1-h)[F,\Phi(x)](1-h)\|.
\end{equation}
Both terms tend to zero along a quasicentral approximate unit.
We may therefore arrange
\begin{equation}\label{eq:smallcomm}
 \|[B,\Phi(x)]\|<r,\qquad
 \|[(1-B^2)^{1/2},\Phi(x)]\|<r\quad(x\in Y).
\end{equation}
For the second inequality, approximate $f(t)=\sqrt{1-t^2}$ uniformly
on $[-1,1]$ by a polynomial $p(t)=\sum c_jt^j$ and use
\[
 \|[f(B),b]\|\leq2\|f-p\|_\infty
        +\Bigl(\sum j|c_j|\Bigr)\|[B,b]\|\quad(\|b\|\leq1).
\]
First choose $p$, then make the first commutator sufficiently small.
No operator-Lipschitz property of the square root is needed.

\subsubsection{The projection and its class}
With respect to the grading, write
\[
 B=\begin{pmatrix}0&S^*\\S&0\end{pmatrix},\quad
 d_+=(1-S^*S)^{1/2},\quad d_-=(1-SS^*)^{1/2},\quad
 e_0=\begin{pmatrix}0&0\\0&1\end{pmatrix}.
\]
Both $d_+$ and $d_-$ are compact. Since $Sd_+=d_-S$, the matrix
\[
 W=\begin{pmatrix}d_+&-S^*\\S&d_-\end{pmatrix}
\]
is unitary. Define
\begin{equation}\label{eq:indexprojection}
 p_0=W(1-e_0)W^*
    =\begin{pmatrix}d_+^2&d_+S^*\\Sd_+&SS^*\end{pmatrix}.
\end{equation}
Then $p_0-e_0$ is compact. Writing $W=(1-B^2)^{1/2}-\Gamma B$,
where $\Gamma$ is the grading, gives
\begin{equation}\label{eq:p0comm}
 \|[p_0,\Phi(x)]\|<4r\quad(x\in Y).
\end{equation}

The relative cycle $\RR(p_0,e_0;\Phi)$ has the same class as
$(\Phi,B)$ at the output control. To check this, the isometry
\[
 V_+:\HH^+\longrightarrow p_0\HH,\qquad
 V_+\xi=(d_+\xi,S\xi)
\]
is onto. Under $V_+$ and $e_0\HH=\HH^-$, the relative Fredholm
operator becomes exactly $B$. Its positive u.c.p. map becomes
\[
 \widetilde\Phi^+(x)
   =d_+\Phi^+(x)d_++S^*\Phi^-(x)S.
\]
The difference from $\Phi^+(x)$ is
\begin{equation}\label{eq:positive-difference}
 d_+[\Phi^+(x),d_+]
   +S^*(\Phi^-(x)S-S\Phi^+(x)).
\end{equation}
It is compact for every $x\in A$ and has norm less than $2r$ on $Y$.
The convex interpolation of the positive maps preserves the compact
commutators with $B$. Its defect is less than $\delta+6r<7r$,
which is below $\varepsilon$.

This construction retains the Fredholm index. For a finite-dimensional
cycle with $B=0$, the projections are $p_0=\diag(1,0)$ and
$e_0=\diag(0,1)$, and their rank difference is
$\dim\HH^+-\dim\HH^-$. For a degenerate cycle, $p_0=e_0$.

\subsubsection{Finite-dimensional existence}

\begin{proposition}[Existence at a prescribed control]\label{prop:existence}
For every control $a$,
\begin{equation}\label{eq:existence}
 r^F_{\tau(a),a}(\FF_{\tau(a)})\subseteq\iota_a(\QQ_a).
\end{equation}
\end{proposition}
\begin{proof}
Start with a cycle at $\tau(a)$ and apply the relative index
construction using \eqref{eq:budget}. There are three further steps.

\smallskip\noindent\emph{Step 1: dilation.}
Take the even Stinespring dilation $\Phi(x)=V^*\rho(x)V$ on
$\mathcal H$. Let $P=VV^*$, and let $\widehat e$ be the projection
onto the negative Stinespring summand. Then
\[
 [\widehat e,\rho(x)]=0,\qquad
 [\widehat e,P]=0,\qquad \widehat eV=Ve_0.
\]
Define
\begin{equation}\label{eq:liftedprojection}
 \widehat p=Vp_0V^*+\widehat e(1-P).
\end{equation}
This is a projection with
$\widehat p-\widehat e=V(p_0-e_0)V^*$ compact.
For $x\in X$, Stinespring's identity reads
\begin{equation}\label{eq:stine}
 \|(1-P)\rho(x)P\|^2
   =\|\Phi(x^*x)-\Phi(x)^*\Phi(x)\|<\delta=r^2.
\end{equation}
As $X=X^*$, both off-diagonal corners of $\rho(x)$ are bounded by $r$.

The block-diagonal map
\[
 \rho^{\mathrm{bd}}(x)=P\rho(x)P+(1-P)\rho(x)(1-P)
\]
is u.c.p., commutes with $\widehat e$, differs from $\rho$ by less
than $r$ on $X$, and has defect at most $r^2$. Moreover,
\[
 \RR(\widehat p,\widehat e;\rho^{\mathrm{bd}})
  \cong\RR(p_0,e_0;\Phi)\oplus\text{a degenerate cycle}.
\]
In the $P\oplus(1-P)$ decomposition, the diagonal block of
$[\widehat p,\rho(x)]$ is $[p_0,\Phi(x)]$ and zero; its off-diagonal
blocks are the off-diagonal blocks of $\rho(x)$ multiplied by
$p_0-e_0$. Thus
\begin{equation}\label{eq:liftcomm}
 \|[\widehat p,\rho(x)]\|<6r\quad(x\in X).
\end{equation}
Interpolate from $\rho^{\mathrm{bd}}$ to $\rho$. Equations
\eqref{eq:convex-bound} and \eqref{eq:compression-bound} bound the
defect of the relative cycles by $2r^2+6r<7r$.
Lemma~\ref{lem:relative} preserves their compact defects exactly.

\smallskip\noindent\emph{Step 2: absorption.}
Add the relative pair with equal reference projections on two copies
of $\pi$. It is degenerate. Apply Proposition~\ref{prop:absorption}
in each ambient grading, with accuracy $r$ on $Y=X\cup X^2$.
After the identifying unitaries we have $p,e=\diag(0,1)$ on
$H\oplus H$ and a representation $\rho'$ with
\[
 p-e\in\K,\qquad [e,\rho'(x)]=0,\qquad
 \|\rho'(x)-\Pi(x)\|<r\quad(x\in Y).
\]
The original relative cycle, with the added degenerate cycle, is
unitarily equivalent to $\RR(p,e;\rho')$.
Interpolation between $\rho'$ and $\Pi$ has defect at most $r^2/4$.
Also, \eqref{eq:liftcomm} gives
\begin{equation}\label{eq:fixedcomm}
 \|[p,\Pi(x)]\|<8r\quad(x\in X),
\end{equation}
and the same bound holds along this interpolation. Its relative cycles
have defect less than $9r$. We have reached $\RR(p,e;\Pi)$.

\smallskip\noindent\emph{Step 3: finite-rank compression.}
Choose $Q=Q_n$ so large that
\[
 a_Q=e+Q(p-e)Q=QpQ+e(1-Q)
\]
satisfies $\|a_Q-p\|<r/8$ and
$\max_{x\in X}\|[Q,\Pi(x)]\|<r$. Spectral correction gives a path
\[
 p_s=\mathbf1_{(1/2,\infty)}((1-s)p+sa_Q),\qquad
 p_1=f+e(1-Q),\quad f\leq Q,
\]
with $\|p_s-p\|<r$. The cycles $\RR(p_s,e;\Pi)$ have defect less
than $10r$.
Next put
\[
 \Pi^{Q}(x)=Q\Pi(x)Q+(1-Q)\Pi(x)(1-Q).
\]
This u.c.p. map has defect at most $r^2$, commutes with $e$, and is
within $r$ of $\Pi$ on $X$. The interpolation from $\Pi$ to
$\Pi^Q$ has defect at most $2r^2$ and commutator with $p_1$ less
than $12r$. Its relative cycles have defect less than $13r$.

At the endpoint, all operators split over $Q\HH\oplus(1-Q)\HH$.
The second summand has identical projection $e(1-Q)$ in both degrees,
identical compressed maps, and off-diagonal identity. It is exactly
degenerate. The first summand is finite dimensional. Homotoping its
Fredholm operator to zero leaves the class
\begin{equation}\label{eq:finiteoutput}
 \left[x\longmapsto f\Pi(x)f\big|_{f\HH}\right]
 -\left[x\longmapsto eQ\Pi(x)eQ\big|_{eQ\HH}\right]\in\QQ_a.
\end{equation}
Its image under $\iota_a$ is the original cycle restricted to $a$.
All intervening defects were less than $13r<\varepsilon$.
The cutoffs and unitaries were chosen after the cycle, but the source
control $\tau(a)$ was fixed in \eqref{eq:theta}.
\end{proof}

\subsubsection{Relative homotopies and finite-dimensional vanishing}\label{subsec:QR-vanishing}

\begin{proposition}[Vanishing at a prescribed control]\label{prop:vanishing}
For every control $a$ and every $x\in\QQ_{\tau(a)}$,
\begin{equation}\label{eq:vanishing}
 \iota_{\tau(a)}(x)=0\quad\Longrightarrow\quad
 r^Q_{\tau(a),a}(x)=0.
\end{equation}
\end{proposition}
\begin{proof}
Write $x=[\phi_+]-[\phi_-]$ and let $v$ be its finite-dimensional
graded Fredholm cycle, with operator zero. By
Lemma~\ref{lem:single-homotopy}, its zero class is represented by one
homotopy from $v\oplus D_0$ to a degenerate cycle $D_1$, at the
unchanged control $\tau(a)$. Work on its graded Hilbert
$C([0,1])$-module.

\paragraph{Keeping the endpoint operators fixed.}
The self-adjoint contraction normalization fixes both endpoint
operators: they are zero on the finite-dimensional summand and
self-adjoint unitaries on the degenerate summands. For the quasicentral
normalization use the ideal
\begin{equation}\label{eq:endpointideal}
 J=\{k\in\K_{C([0,1])}(\HH):k|_0=k|_1=0\}.
\end{equation}
All $[F,\Phi(y)]$, $y\in A$, belong to $J$: they are compact module
operators and their endpoint values vanish. The ideal is invariant
under all the adjointable operators used here. An even quasicentral
approximate unit of $J$ therefore gives \eqref{eq:smallcomm} while
leaving the endpoint operators unchanged. The polynomial argument for
$(1-B^2)^{1/2}$ works in the module norm as before.

Carry out the index projection, the minimal Stinespring dilation,
and absorption on this one module. At $t=0$, the dilation is the
direct sum of those of $\phi_+$, $\phi_-$, and the degenerate
summands, by the endpoint property of the explicit dilation described above. Proposition~\ref{prop:absorption} applies with
$D=C([0,1])$. After the final interpolation we obtain
\begin{equation}\label{eq:relativepath}
 p(t)-e\in C([0,1],\K(H\oplus H)),\qquad
 \sup_t\max_{z\in X}\|[p(t),\Pi(z)]\|<8r.
\end{equation}
The endpoints have the exact form
\begin{equation}\label{eq:endpoints}
 p(1)=e,\qquad p(0)=e+R_+-R_-,\qquad
 R_+\leq1-e,\quad R_-\leq e,
\end{equation}
where $R_+,R_-$ are the images of the two original finite-dimensional
summands. Indeed, the index projection is equal to its reference
projection on a degenerate cycle; on a finite-dimensional zero-operator
cycle its difference is the positive support minus the negative support.
Dilation and the grading-preserving absorption unitaries retain this
form exactly.

The compactness in \eqref{eq:relativepath} is norm-continuity with values
in the compact operators, not merely pointwise compactness. This follows
from $\K(H_{C([0,1])})=C([0,1],\K(H))$ and is what permits a common
finite-rank cutoff.

\paragraph{Compressing the whole homotopy.}
By \eqref{eq:uniform-tail} and Lemma~\ref{lem:cutoff}, choose one
finite-rank projection $Q$ commuting with $e$, containing $R_++R_-$,
and satisfying
\[
 \sup_t\|Q(p(t)-e)Q-(p(t)-e)\|<r/8,\qquad
 \max_{z\in X}\|[Q,\Pi(z)]\|<r.
\]
Spectral correction gives
\[
 p_Q(t)=\mathbf1_{(1/2,\infty)}\bigl(e+Q(p(t)-e)Q\bigr)
       =f(t)+e(1-Q).
\]
It is norm-continuous, $f(t)\leq Q$, and the exact endpoint identities
are
\begin{equation}\label{eq:finite-endpoints}
 f(0)=eQ+R_+-R_-,\qquad f(1)=eQ.
\end{equation}
Applying Step~3 of Proposition~\ref{prop:existence} uniformly in $t$
produces the finite-dimensional classes
\begin{equation}\label{eq:finitehomotopy}
 [\,z\mapsto f(t)\Pi(z)f(t)|_{f(t)\HH}\,]
 -[\,z\mapsto eQ\Pi(z)eQ|_{eQ\HH}\,]
\end{equation}
with defect less than $13r$.
A continuous finite-rank projection path over the interval defines a
trivial vector bundle, so \eqref{eq:finitehomotopy} is a homotopy of
matrix maps after a choice of trivialization. For completeness, subdivide
the interval until consecutive projections are less than one apart;
the polar unitaries for close projections give local trivializations,
which can be joined successively. Endpoint unitary conjugations do not
change the classes. At $t=1$ the class in \eqref{eq:finitehomotopy} is zero.

\paragraph{Recovering the prescribed finite-dimensional endpoint.}
It remains to identify the class at $t=0$ with $x$, rather than with an
unspecified preimage of its Fredholm class. Set
\[
 R=eQ-R_-.
\]
The two finite-dimensional spaces at $t=0$ are
\[
 (R_+\HH)\oplus(R\HH),\qquad
 (R_-\HH)\oplus(R\HH).
\]
The positive compressed map is already block diagonal in this
splitting because $\Pi$ commutes with $e$. For the negative map,
\begin{equation}\label{eq:Rminuscomm}
 \max_{z\in X}\|[R_-,\Pi(z)]\|<3r.
\end{equation}
To see this, before absorption $R_-$ projects onto the image of the
original finite-dimensional space under the Stinespring isometry for
$\phi_-$, inside its reducing minimal-dilation summand. It is not
the projection onto the whole dilating space. Equation~\eqref{eq:stine} bounds its commutator with the
representation by $r$. Absorption changes the representation by less
than $r$ on $X$, so the commutator increases by less than $2r$.
The same statement holds for $R_+$.

Delete the off-diagonal blocks between $R_-\HH$ and $R\HH$ by a
convex u.c.p. homotopy. Their norm on $X$ is less than $3r$.
Furthermore,
\[
 \|[R,\Pi(z)]\|\leq\|[eQ,\Pi(z)]\|+
                         \|[R_-,\Pi(z)]\|<4r.
\]
The compressed maps on $R_-$ and $R$ have defects at most $9r^2$
and $16r^2$, respectively, by \eqref{eq:compression-bound}.
The negative map before splitting has defect at most $r^2$.
Equation~\eqref{eq:convex-bound} bounds the splitting homotopy by
$16r^2+9r^2/4<19r^2<\varepsilon$.
The two copies of the compressed map on $R\HH$ are now identical
and cancel in $\QQ_a$.

Under their identifying unitaries, the remaining maps on $R_+$ and
$R_-$ are within $r$ of the original $\phi_+$ and $\phi_-$ on
$Y=X\cup X^2$. Before absorption their compressions are exactly the
original finite-dimensional maps; the only change is the absorption
error. Convex homotopies connect them to the original maps. Their
defects are bounded by $10r^2<\varepsilon$, using
\eqref{eq:convex-bound} and \eqref{eq:Rminuscomm}.
Thus the initial class in \eqref{eq:finitehomotopy} is precisely
$r^Q_{\tau(a),a}(x)$. Its terminal class is zero, proving
\eqref{eq:vanishing}.
\end{proof}

\FloatBarrier
\subsubsection{The inverse homomorphisms and their coherence}

We now complete the final, algebraic step to prove Theorem~\ref{thm:main}. 

Fix $a$ and set $t=\tau(a)$ and $s=\tau(t)=\sigma(a)$.
Given $z\in\FF_s$, Proposition~\ref{prop:existence}, with output
control $t$, supplies $v\in\QQ_t$ such that
\[
 \iota_t(v)=r^F_{s,t}(z).
\]
Define
\begin{equation}\label{eq:rho}
 \rho_a(z)=r^Q_{t,a}(v).
\end{equation}
If $v'$ is a second choice, then $\iota_t(v-v')=0$.
Proposition~\ref{prop:vanishing}, with output control $a$, gives
$r^Q_{t,a}(v-v')=0$. Thus \eqref{eq:rho} is well defined.
For $z_1,z_2$, the sum of two choices is a choice for $z_1+z_2$;
therefore $\rho_a$ is a homomorphism.

Applying $\iota_a$ to \eqref{eq:rho} gives
\[
 \iota_a\rho_a(z)=r^F_{s,a}(z),
\]
which is \eqref{eq:right}. For $x\in\QQ_s$ one may choose
$v=r^Q_{s,t}(x)$ as the lift of $z=\iota_s(x)$. Consequently
\[
 \rho_a\iota_s(x)=r^Q_{s,a}(x),
\]
which is \eqref{eq:left}.

Finally let $a\preccurlyeq b$. Since $\tau$ is order preserving,
\[
 \tau(a)\preccurlyeq\tau(b),\qquad
 \sigma(a)\preccurlyeq\sigma(b).
\]
Given $z\in\FF_{\sigma(b)}$, choose a finite-dimensional lift
$v_b\in\QQ_{\tau(b)}$ for the definition of $\rho_b(z)$.
Its restriction $r^Q_{\tau(b),\tau(a)}v_b$ is an admissible lift for
$r^F_{\sigma(b),\sigma(a)}z$ in the definition of $\rho_a$.
Independence of the choice of lift therefore gives
\begin{align*}
 \rho_a r^F_{\sigma(b),\sigma(a)}(z)
  &=r^Q_{\tau(a),a}r^Q_{\tau(b),\tau(a)}(v_b)\\
  &=r^Q_{b,a}r^Q_{\tau(b),b}(v_b)
   =r^Q_{b,a}\rho_b(z).
\end{align*}
This is \eqref{eq:coherence}, for every pair of controls, without a
cofinal-sequence argument. Figure~\ref{fig:identities} summarizes the
three identities. Formula~\eqref{eq:sigma} follows directly
from \eqref{eq:theta}. The theorem is proved.

\begin{figure}[htbp]
\centering
\begin{minipage}{.47\linewidth}\centering
\begin{tikzcd}[column sep=2.0em,row sep=2.3em]
\QQ_{\sigma(a)} \arrow[r,"\iota_{\sigma(a)}"]
 \arrow[dr,"r^Q_{\sigma(a),a}"'] &
\FF_{\sigma(a)} \arrow[d,"\rho_a"]\\
& \QQ_a
\end{tikzcd}
\end{minipage}\hfill
\begin{minipage}{.47\linewidth}\centering
\begin{tikzcd}[column sep=2.0em,row sep=2.3em]
\FF_{\sigma(a)} \arrow[r,"\rho_a"]
 \arrow[dr,"r^F_{\sigma(a),a}"'] &
\QQ_a \arrow[d,"\iota_a"]\\
& \FF_a
\end{tikzcd}
\end{minipage}
\par\medskip
\begin{tikzcd}[column sep=4em,row sep=2.5em]
\FF_{\sigma(b)} \arrow[r,"\rho_b"]
 \arrow[d,"r^F_{\sigma(b),\sigma(a)}"'] &
\QQ_b \arrow[d,"r^Q_{b,a}"]\\
\FF_{\sigma(a)} \arrow[r,"\rho_a"'] & \QQ_a
\end{tikzcd}
\caption{The two inverse triangles and the restriction square commute
 as homomorphisms on their full source groups.}
\label{fig:identities}
\end{figure}
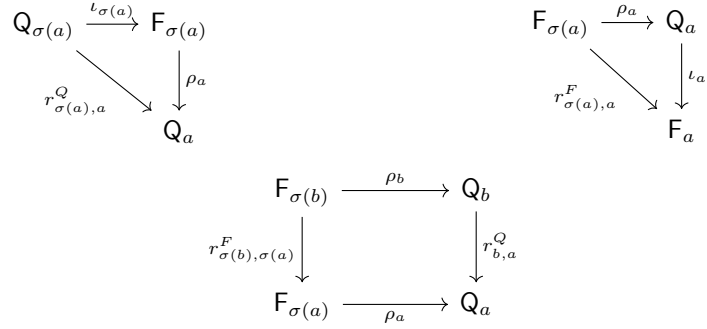
\FloatBarrier

\begin{corollary}[Inverse limits]\label{cor:limits}
For $A$ as in Theorem~\ref{thm:main}, the canonical map
\[
 \varprojlim_a\QQ_a\longrightarrow\varprojlim_a\FF_a
\]
is an isomorphism. If each finite-control group is discrete and the
limits have their inverse-limit topologies, it is a homeomorphism.
Its inverse is
\[
 (z_a)_a\longmapsto\bigl(\rho_a(z_{\sigma(a)})\bigr)_a.
\]
\end{corollary}
\begin{proof}
Coherence makes the displayed family compatible. The two composite
identities prove that it is a two-sided inverse, coordinate by
coordinate. Both maps are continuous for the stated topologies:
each output coordinate depends on one input coordinate through a
homomorphism between discrete groups.
\end{proof}

\subsection{Examples}
\label{sec:applications}

\begin{example}[Amenable and free groups]
If $G$ is countable and amenable, then
$C^*_{\max}(G)=C_r^*(G)$ is separable and nuclear. Its quasidiagonality
is the amenable-group consequence of Tikuisis--White--Winter
\cite[Corollary~C]{tikuisis_white_winter_2017}, and LP follows from the Choi--Effros lifting
theorem \cite{CE}. Countable amenable groups have the Haagerup
property \cite{CCJJV}. Hence both the algebraic comparison and
Corollary~\ref{cor:character} apply.

For a countable free group, the full group algebra is residually
finite-dimensional by Choi's theorem \cite{Choi} and its subgroup
consequence, and has LP by Kirchberg \cite{Kir94}. Choi treats the
free group on two generators; every countable free group embeds in it,
and full subgroup algebras embed isometrically in the full group
algebra. The family of restricted finite-dimensional representations
therefore still separates points. Free groups have the Haagerup
property \cite{CCJJV}.

For clarity, residual finite-dimensionality implies quasidiagonality
here by a direct argument. A separable residually finite-dimensional
algebra has a countable separating family of finite-dimensional
representations $\pi_j$. The faithful representation
$\bigoplus_j\pi_j$ has finite-rank projections onto the first $m$
summands which commute exactly with the representation and converge
strongly to $1$.
\end{example}

\begin{example}[{Closed orientable surface groups}]
Let $\Gamma_g=\pi_1(\Sigma_g)$, where $\Sigma_g$ is a closed
orientable surface of genus $g\geq2$. Lubotzky--Shalom's property FD
implies that $C^*_{\max}(\Gamma_g)$ is residually finite-dimensional
\cite[Theorem~2.8]{lubotzky_shalom_2004}. Thus Corollary~\ref{cor:group} already proves
quantitative QR-quasidiagonality in the u.c.p. sense, without any
lifting hypothesis.
We record an independent route to the LP hypothesis needed for the
unitary-valued character.

Map $a_1$ to $1\in\Z$ and the other standard generators to zero in
\[
 \Gamma_g=\left\langle a_1,b_1,\ldots,a_g,b_g
     \ \middle|\ \prod_{i=1}^g[a_i,b_i]=1\right\rangle .
\]
The resulting epimorphism splits, with section $1\mapsto a_1$.
Its kernel $K$ is a countable free group. One can see this from the
corresponding noncompact surface cover. More explicitly, put $t=a_1$,
$b_k=t^kb_1t^{-k}$, $a_{i,k}=t^ka_it^{-k}$, and
$b_{i,k}=t^kb_it^{-k}$ for $i\geq2$ and $k\in\Z$.
The Reidemeister--Schreier presentation has relations
\[
 b_{k+1}b_k^{-1}\prod_{i=2}^g[a_{i,k},b_{i,k}]=1.
\]
They eliminate all $b_k$ except $b_0$, leaving a free basis consisting
of $b_0$ and the $a_{i,k},b_{i,k}$. Consequently
\[
 \Gamma_g=K\rtimes\Z,\qquad
 C^*_{\max}(\Gamma_g)
   \cong C^*_{\max}(K)\rtimes\Z.
\]
The kernel algebra has LP by the free-group case. LP is preserved by
crossed products with amenable groups; we use the direct proof in
Enders--Shulman \cite[Theorem~8.3 and Corollary~8.4]{ES}.
This proves LP for $C^*_{\max}(\Gamma_g)$.
Finally, surface groups have the Haagerup property
\cite{CCJJV}, so Corollary~\ref{cor:character} applies. The genus-one
case is the amenable group $\Z^2$.

Property FD supplies residual finite-dimensionality, while the
free-by-$\mathbb Z$ argument verifies the full lifting property.  
\end{example}


\subsection{Surface groups: direct classification and finite control}
\label{subsec:surface-classification}

Corollary~\ref{cor:character} and the preceding examples already give
the character isomorphism and its finite-control consequences for
surface groups. The purpose of this subsection is to identify the
two coordinates explicitly---rank and winding number---and to give
a direct finite-control classification in these coordinates. This
also provides a second description of the comparison, using the
index pairing and stable uniqueness.

Let $\Sigma_g$ be a closed connected orientable surface of genus $g\geq2$
and write
\begin{equation}\label{eq:surface-presentation}
 \Gamma_g=\pi_1(\Sigma_g)
 =\left\langle a_1,b_1,\ldots,a_g,b_g\ \middle|\
       \prod_{j=1}^g[a_j,b_j]=e\right\rangle,
 \qquad [u,v]=uvu^{-1}v^{-1}.
\end{equation}
In this subsection $A=C^*_{\max}(\Gamma_g)$.

\begin{lemma}[operator-algebraic and $K$-theoretic properties]
\label{lem:surface-input}
For $g\geq2$, the group $\Gamma_g$ has the following properties.
\begin{enumerate}[label=(\roman*)]
\item $\Gamma_g$ has the Haagerup property.
\item $A$ has the lifting property, and $\Gamma_g$ has property FD: the
finite-dimensional representations that factor through finite quotients are
dense, in the Fell topology, in the unitary dual of $\Gamma_g$.  In
particular, these representations separate the points of $A$ and form the
good class used in the stable-uniqueness theorem.
\item $A$ satisfies the approximate $K$-homology UCT used in
\cite[Definition~3.25 and Theorem~3.28]{willett_conditional_2026}, and
\begin{equation}\label{eq:surface-K-groups}
 K_0(A)\cong\mathbb Z[1]\oplus\mathbb Z\beta,
 \qquad K_1(A)\cong\mathbb Z^{2g}.
\end{equation}
Here $\beta=\mu_{\max}([D_{\mathrm{spin}}])$, with the orientation
chosen so that
\[
 \operatorname{ch}_{\mathrm{hom}}([D_{\mathrm{spin}}])=[\Sigma_g].
\]
Equivalently, the normalized class is
$[D_{\mathrm{spin}}]=[\bar\partial]+(g-1)[\mathrm{pt}]$ in
$K_0(\Sigma_g)$. Thus the trivial representation pairs with $\beta$
as zero, fixing the normalization used in the winding-number formula
\cite{carrion_dadarlat_surface_2013}. Evaluation on $[1]$ and $\beta$
gives an isomorphism
\begin{equation}\label{eq:surface-Khom-pairing}
 KL^0(A)\longrightarrow\mathbb Z^2,\qquad
 z\longmapsto\bigl(\langle[1],z\rangle,
                    \langle\beta,z\rangle\bigr).
\end{equation}
\end{enumerate}
\end{lemma}

\begin{proof}
The group $\Gamma_g$ is a cocompact lattice in
$\mathrm{PSL}(2,\mathbb R)$, which has the Haagerup property.  Restriction
of a proper affine isometric action to a closed discrete subgroup remains
proper, proving (i).

Fournier-Facio and Willett prove that the maximal group $C^*$-algebra of the
fundamental group of every connected manifold of dimension at most three has
the lifting property and that the group has property FD
\cite[Theorem~1.2]{fournier_facio_willett_2026}.  Property FD for surface
groups was established earlier by Lubotzky and Shalom
\cite[Theorem~2.8]{lubotzky_shalom_2004}.  This proves (ii).

The Higson--Kasparov theorem, $K$-amenability, and the UCT consequence used
in \cite[Theorem~3.26 and Remark~5.6]{willett_conditional_2026} give the
approximate $K$-homology UCT for $A$.  Maximal assembly identifies the
$K$-theory of $A$ with the $K$-homology of $\Sigma_g$.  The
Atiyah--Hirzebruch spectral sequence collapses for the two-dimensional
complex $\Sigma_g$, giving
\[
 K_0(\Sigma_g)\cong H_0(\Sigma_g;\mathbb Z)
      \oplus H_2(\Sigma_g;\mathbb Z)\cong\mathbb Z^2,
 \qquad
 K_1(\Sigma_g)\cong H_1(\Sigma_g;\mathbb Z)
      \cong\mathbb Z^{2g}.
\]
This proves \eqref{eq:surface-K-groups}.  Both groups are free, so the
$\operatorname{Ext}$ term in the UCT vanishes.  Since
$K_0(\mathbb C)=\mathbb Z$ and $K_1(\mathbb C)=0$, the resulting pairing
identifies $KL^0(A)$ with
$\operatorname{Hom}(K_0(A),\mathbb Z)\cong\mathbb Z^2$, proving (iii).
\end{proof}

\subsubsection{The two surface coordinates}

We first recall the winding number in a form that applies uniformly at
finite control.  Enlarge a finite group set so that it contains $e$, the
generators and their inverses in \eqref{eq:surface-presentation}, and every
prefix of the relator.  For a normalized unitary
$(F,\varepsilon)$-quasi-representation $q$, repeated use of the
multiplication estimate gives
\begin{equation}\label{eq:surface-relator-estimate}
 \left\|\prod_{j=1}^g[q(a_j),q(b_j)]-1\right\|
 <(6g-1)\varepsilon.
\end{equation}
Indeed, multiplying the $4g$ letters costs at most
$(4g-1)\varepsilon$, and replacing the $2g$ matrices assigned to inverse
letters by the adjoints of the corresponding generator matrices costs at
most $2g\varepsilon$.  When the right-hand side is less than one, define
\begin{equation}\label{eq:surface-winding}
 R(q)=\prod_{j=1}^g[q(a_j),q(b_j)],
 \qquad
 w(q)=\frac{1}{2\pi i}\operatorname{Tr}\log R(q),
\end{equation}
where $\log$ is the principal logarithm.  Since $R(q)$ is a product of
commutators, $\det R(q)=1$.  Hence
$\exp(\operatorname{Tr}\log R(q))=1$ and $w(q)\in\mathbb Z$.

For a sufficiently multiplicative u.c.p. map
$\phi:A\to M_n(\mathbb C)$, the matrices $\phi(u_s)$ associated to the
generators and their inverses are almost unitary.  We write
\begin{equation}\label{eq:surface-ucp-polar}
 p_\phi(s)=\phi(u_s)
   \bigl(\phi(u_s)^*\phi(u_s)\bigr)^{-1/2}.
\end{equation}
After strengthening the control, \Cref{lem:polar-correction} shows that
$p_\phi$ is a unitary quasi-representation for which
\eqref{eq:surface-winding} is defined.  Put
\begin{equation}\label{eq:surface-ucp-winding}
 w(\phi):=w(p_\phi).
\end{equation}
Functional calculus and the spectral gap show that this integer is constant
on every u.c.p. controlled homotopy, and it is additive under direct sum.
Consequently, on a cofinal set of u.c.p. controls $a$, there is a
homomorphism
\begin{equation}\label{eq:surface-Q-coordinates}
 \begin{aligned}
 \Theta^{\mathrm{QR,ucp}}_a&:\QRucpstage{\Gamma_g}{a}
     \longrightarrow\mathbb Z^2,\\
 \Theta^{\mathrm{QR,ucp}}_a([\phi_+]-[\phi_-])
   &=\bigl(\dim\phi_+-\dim\phi_-,
           w(\phi_+)-w(\phi_-)\bigr).
 \end{aligned}
\end{equation}
The coordinate maps commute with the bonding homomorphisms whenever both
sides are defined.

\begin{lemma}[the generalized Exel--Loring formula in the u.c.p. picture]
\label{lem:surface-ucp-EL}
There is a u.c.p. control $a_{\mathrm{EL}}$ such that, at every stronger
control, the partial $K$-theory map induced by a finite-dimensional u.c.p.
quasi-representation $\phi:A\to M_n(\mathbb C)$ is defined on $[1]$ and
$\beta$, and
\begin{equation}\label{eq:surface-EL-formula}
 \phi_*([1])=n,
 \qquad
 \phi_*(\beta)=w(\phi).
\end{equation}
Equivalently, these are the pairings of the finite-dimensional Fredholm
cycle $\iota^{\mathrm{ucp}}[\phi]$ with $[1]$ and $\beta$.
\end{lemma}

\begin{proof}
Choose projection representatives for the two classes and strengthen the
control so that applying any u.c.p. map at that control produces almost
projections whose spectra avoid $1/2$.  These projection representatives and estimates are supplied by
\cite[Lemmas~3.19 and~3.24]{willett_conditional_2026}.
The class $[1]$ is sent to the identity of $M_n$, giving the first formula.

For the second, make the multiplication tolerance small enough that
\eqref{eq:surface-ucp-polar} is defined and that
$\|p_\phi(s)-\phi(u_s)\|$ is smaller than the spectral-gap tolerance of the
chosen projection representing $\beta$.  Applying $\phi$ and applying the
unitary map $p_\phi$ therefore produce almost projections connected by the
straight-line path without crossing $1/2$; their spectral projections have
the same $K_0$-class.  The generalized Exel--Loring formula of Carri\'on and
Dadarlat \cite[Theorem~2.6]{carrion_dadarlat_surface_2013} identifies the
class obtained from $p_\phi$ with
$w(p_\phi)=w(\phi)$.  This proves the second formula and also its
compatibility with the inclusion into the group of Fredholm quasi-representations.
\end{proof}

For a surface group, the classes with finite coefficients introduce no
additional values that need to be specified. To make the notation precise,
we use a finite collection $P$ of projections representing classes in
$K_0(A)$ and in $K_0(A;\mathbb Z/m)$ for finitely many integers $m\geq2$,
as in \cite[Definition~3.23]{willett_conditional_2026}. The notation
$\langle P\rangle$ denotes the submodule of total $K$-theory generated by
their classes under the coefficient operations specified there. Saying
that $P$ contains a class means that the projections needed to represent
that class, as a difference when necessary, have been included.

\begin{lemma}[the unit and the surface class determine the pairings]
\label{lem:surface-K-data-reduction}
Let $P$ be a finite collection of projection representatives as just
described, containing representatives of $[1]$ and $\beta$.  If two partially defined
homomorphisms
\[
 \alpha_0,\alpha_1:\langle P\rangle
   \longrightarrow\underline K_0(\mathbb C)
\]
commute with the coefficient and Bockstein operations and have the same
values on $[1]$ and $\beta$, then they agree on all of $\langle P\rangle$.
\end{lemma}

\begin{proof}
For every $m\geq2$, the coefficient exact sequence contains
\[
 0\longrightarrow K_0(A)\otimes\mathbb Z/m
 \longrightarrow K_0(A;\mathbb Z/m)
 \longrightarrow\operatorname{Tor}(K_1(A),\mathbb Z/m)
 \longrightarrow0.
\]
By \eqref{eq:surface-K-groups}, $K_1(A)$ is free, so the last term is zero.
Thus $K_0(A;\mathbb Z/m)$ is generated by the reductions of $[1]$ and
$\beta$.  On the target side
$K_0(\mathbb C;\mathbb Z/m)=\mathbb Z/m$ and
$K_1(\mathbb C;\mathbb Z/m)=0$.  Compatibility with coefficient reduction
therefore determines every coefficient component from the two integral
values.  Compatibility with the Bockstein maps adds no further choice.
Hence $\alpha_0$ and $\alpha_1$ agree on the submodule generated by $P$.
\end{proof}

\subsubsection{Stable uniqueness for matrix-valued maps}

We shall use the following quantitative homotopy observation.  It is simpler
for u.c.p. maps than for unitary-valued maps because convex combinations
remain u.c.p.

\begin{lemma}[close u.c.p. maps give a controlled homotopy]
\label{lem:close-ucp-homotopy}
Let $E\subseteq G$ be finite.  Suppose
$\phi_0,\phi_1:A\to M_n(\mathbb C)$ are u.c.p. maps with
\[
 \|\phi_i(u_{gh})-\phi_i(u_g)\phi_i(u_h)\|<\delta
 \qquad(g,h\in E,\ i=0,1)
\]
and
\[
 \max_{g\in E}\|\phi_0(u_g)-\phi_1(u_g)\|<\alpha.
\]
Then the convex path $\phi_t=(1-t)\phi_0+t\phi_1$ consists of u.c.p. maps
whose multiplication defect on $E$ is less than
$\delta+\alpha^2/4$.

More generally, let $\theta:A\to M_k(\mathbb C)$ be a representation and
suppose a unitary $v\in U(n+k)$ satisfies
\[
 \max_{g\in E}
 \|v(\phi_0(u_g)\oplus\theta(u_g))v^*
       -(\phi_1(u_g)\oplus\theta(u_g))\|<\alpha.
\]
If $\delta+\alpha^2/4<\varepsilon$, then $\phi_0$ and $\phi_1$ determine
the same element of $\QRucpstage{\Gamma_g}{(E,\varepsilon)}$ after adding the
common summand $\theta$.
\end{lemma}

\begin{proof}
Write $d_i(g,h)=\phi_i(u_{gh})-\phi_i(u_g)\phi_i(u_h)$.  Direct expansion
gives
\begin{align*}
 \phi_t(u_{gh})-\phi_t(u_g)\phi_t(u_h)
 &=(1-t)d_0(g,h)+t d_1(g,h)\\
 &\quad+t(1-t)
   (\phi_0(u_g)-\phi_1(u_g))
   (\phi_0(u_h)-\phi_1(u_h)).
\end{align*}
Since $t(1-t)\leq1/4$, the required bound follows.  For the second
assertion, choose a path $v_s$ in $U(n+k)$ from $1$ to $v$.  Conjugating
$\phi_0\oplus\theta$ by $v_s$ gives a u.c.p. homotopy with defect at most
$\delta$.  Apply the first part to the endpoint of this path and
$\phi_1\oplus\theta$.  Concatenating the two paths proves stable homotopy
at the stated control.
\end{proof}

\begin{proposition}[stable uniqueness for matrix-valued maps]
\label{prop:surface-finite-detection}
For every target u.c.p. control
$a_0\in\mathcal C_{\Gamma_g}^{\mathrm{sym}}$ there is a
stronger control $a_Q\succeq a_0$ such that the coordinate map
$\Theta^{\mathrm{QR,ucp}}_a$ is defined for every $a\succeq a_Q$ and
\begin{equation}\label{eq:surface-Q-detection}
 \Theta^{\mathrm{QR,ucp}}_a(x)=0
 \quad\Longrightarrow\quad
 r^{\mathrm{QR,ucp}}_{a,a_0}(x)=0
 \qquad(x\in\QRucpstage{\Gamma_g}{a}).
\end{equation}
Equivalently, two classes at a sufficiently strong common control that have
the same rank and winding number become equal at the prescribed target
control.
\end{proposition}

\begin{proof}
Write $a_0=(E_0,\varepsilon_0)$.  Choose a finite symmetric set
\(\bar E_0\supseteq E_0\) containing the surface generators, their
inverses, and the products needed in \eqref{eq:surface-winding}, and put
\(\bar a_0=(\bar E_0,\varepsilon_0)\).  It is enough to obtain equality at
the stronger target \(\bar a_0\), because the bonding homomorphism from
\(\bar a_0\) to \(a_0\) may then be applied.  Choose numbers
$\alpha,\mu>0$ such that
\begin{equation}\label{eq:surface-ucp-homotopy-margin}
 \mu+\alpha^2/4<\varepsilon_0.
\end{equation}

Apply the stable-uniqueness theorem
\cite[Theorem~5.3]{willett_conditional_2026} with target set $\bar E_0$ and
target error $\alpha$.  The hypotheses are supplied by
\Cref{lem:surface-input}: the approximate $K$-homology UCT gives the
$K$-theoretic uniqueness theorem, and the finite-quotient representations
form a good class by property FD.  We obtain a finite collection $P$ of such projection representatives, a
finite set $E_1\subseteq\Gamma_g$, and $\delta_1>0$ with the following
property.  If two same-dimensional u.c.p. maps
$\phi_0,\phi_1:A\to M_n(\mathbb C)$ are
$(E_1,\delta_1)$-multiplicative and induce the same map on $P$, then there
are a finite-quotient representation
$\theta:A\to M_k(\mathbb C)$ and a unitary $v\in U(n+k)$ such that
\begin{equation}\label{eq:surface-ucp-stable-uniqueness}
 \max_{g\in \bar E_0}
 \|v(\phi_0(u_g)\oplus\theta(u_g))v^*
       -(\phi_1(u_g)\oplus\theta(u_g))\|<\alpha.
\end{equation}
Enlarge $P$ to contain $[1]$ and $\beta$.  The implication just stated
remains valid, because equality of the induced maps on the enlarged collection implies
equality on the original collection supplied by stable uniqueness.
Choose $a_Q=(E_Q,\delta_Q)$ so that
\[
 \bar E_0\cup E_1\subseteq E_Q,
 \qquad
 0<\delta_Q<\min\{\delta_1,\mu\},
\]
and so that the induced $K$-theory maps on $\langle P\rangle$, as well as
the winding number in
\Cref{lem:surface-ucp-EL}, are defined.  The same statements remain true at
every stronger control.

Let $a\succeq a_Q$ and let
$x\in\QRucpstage{\Gamma_g}{a}$ have zero coordinates.  Write
$x=[\phi_+]-[\phi_-]$.  The rank coordinate is zero, so the two maps have
the same matrix size.  The winding coordinates agree, and
\Cref{lem:surface-ucp-EL} says that the induced maps agree on $\beta$; they
also agree on $[1]$ because their dimensions agree.
\Cref{lem:surface-K-data-reduction} now shows that the maps induced by
$\phi_+$ and $\phi_-$ agree on all of $P$.  Therefore
\eqref{eq:surface-ucp-stable-uniqueness} applies.

After restriction to $\bar E_0$, the multiplication defects of both stabilized
maps are smaller than $\mu$; the representation $\theta$ contributes no
defect.  The margin in \eqref{eq:surface-ucp-homotopy-margin}, together with
\Cref{lem:close-ucp-homotopy}, gives a homotopy through u.c.p.
$(\bar E_0,\varepsilon_0)$-quasi-representations.  Thus their difference
vanishes in $\QRucpstage{\Gamma_g}{\bar a_0}$, and hence also in
$\QRucpstage{\Gamma_g}{a_0}$.  This is \eqref{eq:surface-Q-detection}.
\end{proof}

\subsubsection{Fredholm classes and index pairings}

For Fredholm quasi-representations, the vanishing criterion comes from
the inverse to the controlled index-pairing map. Its proof does not use
stable uniqueness of finite-dimensional maps.

\begin{lemma}[a vanishing criterion from the approximate $K$-homology UCT]
\label{lem:surface-controlled-KH-detection}
Fix a balanced absorbing representation $\pi$ of $A$.  For a controlled
$K$-homology control $b=(\Omega,\eta)$, put
\[
 \mathrm{KH}^{\pi}(A;b):=KK^\pi_\eta(\Omega;\mathbb C).
\]
For every target control
$b_0$ there exist a stronger control $b_1\succeq b_0$ and a finite collection $P$ of projection representatives containing
$[1]$ and $\beta$ such that the controlled index map
\[
 \operatorname{Ind}_{P,b_1}:\mathrm{KH}^{\pi}(A;b_1)\longrightarrow
 \operatorname{Hom}_{\Lambda_0^{(N)}}
    (\langle P\rangle,\underline K_0(\mathbb C))
\]
satisfies
\begin{equation}\label{eq:surface-H-detection}
 \operatorname{Ind}_{P,b_1}(y)=0
 \quad\Longrightarrow\quad
 r^{\mathrm{KH}}_{b_1,b_0}(y)=0.
\end{equation}
\end{lemma}

\begin{proof}
We use the finite-stage detection statement in the proof of
\cite[Theorem~4.4]{willett_conditional_2026}, with the component maps
specified in \cite[Remark~3.30]{willett_conditional_2026}.
Under the approximate $K$-homology UCT verified in
\Cref{lem:surface-input}, it gives, for the prescribed $b_0$, a
finite collection $P$ and a stronger control $b_1$ such that
\[
 \operatorname{Ind}_{P,b_1}(y)=\operatorname{Ind}_{P,b_1}(y')
 \quad\Longrightarrow\quad
 r^{\mathrm{KH}}_{b_1,b_0}(y)=r^{\mathrm{KH}}_{b_1,b_0}(y')
\]
for all classes $y,y'$ at $b_1$. In the notation of that proof,
the component index map is
\[
 \operatorname{Ind}_{P,b_1}
   =\kappa^{(Y,2\sqrt\delta)}_P\circ\Phi^{(Y,2\sqrt\delta)},
 \qquad b_1=(Y,2\sqrt\delta),
\]
after a cofinal change of control. Its kernel is killed by the
forgetful map to the prescribed target control, as stated there.
Consequently the well-defined component factorization is
\[
 \begin{aligned}
 s_{P,b_1,b_0}:\operatorname{im}(\operatorname{Ind}_{P,b_1})
   &\longrightarrow\mathrm{KH}^{\pi}(A;b_0),\\
 s_{P,b_1,b_0}(\operatorname{Ind}_{P,b_1}(y))
   &=r^{\mathrm{KH}}_{b_1,b_0}(y).
 \end{aligned}
\]
This is a homomorphism because both the index map and the bonding
map are homomorphisms. Enlarge $P$ to include $[1]$ and $\beta$,
and strengthen $b_1$ so that the enlarged pairings are defined.
Taking $y'=0$ gives \eqref{eq:surface-H-detection}. This uses the
cited finite-stage statement, rather than a factorization inferred
solely from an isomorphism of inverse limits.
\end{proof}

\begin{proposition}[vanishing of Fredholm classes after restriction]
\label{prop:surface-Fredholm-detection}
For every target u.c.p. control
$a_0\in\mathcal C_{\Gamma_g}^{\mathrm{sym}}$ there is a
stronger control $a_K\succeq a_0$ and homomorphisms
\[
 \Theta^{\mathrm{FQR,ucp}}_a:\FQRucpstage{\Gamma_g}{a}
 \longrightarrow\mathbb Z^2
 \qquad(a\succeq a_K)
\]
which commute with bonding maps and satisfy
\begin{equation}\label{eq:surface-K-detection}
 \Theta^{\mathrm{FQR,ucp}}_a(z)=0
 \quad\Longrightarrow\quad
 r^{\mathrm{FQR,ucp}}_{a,a_0}(z)=0.
\end{equation}
\end{proposition}

\begin{proof}
Use the component maps in
\Cref{rem:FQRucp-controlled-K-homology} for the global u.c.p. Fredholm
system.  For the prescribed target $\FQRucpstage{\Gamma_g}{a_0}$, choose a
controlled $K$-homology stage $b_0$ and a comparison homomorphism
\begin{equation}\label{eq:surface-H-to-K-recovery}
 \varrho_0:\mathrm{KH}^{\pi}(A;b_0)
 \longrightarrow \FQRucpstage{\Gamma_g}{a_0}.
\end{equation}
Apply \Cref{lem:surface-controlled-KH-detection} to obtain $b_1$ and $P$.
Choose a sufficiently strong u.c.p. group control $a_K$ and a component
homomorphism
\begin{equation}\label{eq:surface-K-to-H-component}
 \jmath_1:\FQRucpstage{\Gamma_g}{a_K}
 \longrightarrow \mathrm{KH}^{\pi}(A;b_1)
\end{equation}
so that the two component constructions satisfy the exact composite identity
\begin{equation}\label{eq:surface-K-H-composite}
 \varrho_0r^{\mathrm{KH}}_{b_1,b_0}\jmath_1
   =r^{\mathrm{FQR,ucp}}_{a_K,a_0}.
\end{equation}
This is one of the two bonding-map identities in
\Cref{def:inverse-system-equivalence}; no limiting class is being chosen.

By \Cref{lem:surface-K-data-reduction}, the restriction of the controlled
index map to $\langle P\rangle$ is determined by its values on $[1]$ and
$\beta$.  Define
\begin{equation}\label{eq:surface-DK-defined}
 \Theta^{\mathrm{FQR,ucp}}_{a_K}(z)=
 \bigl(\langle[1],\jmath_1z\rangle,
       \langle\beta,\jmath_1z\rangle\bigr).
\end{equation}
If this pair is zero, the index map of $\jmath_1z$ vanishes on all of $\langle P\rangle$.
Equations \eqref{eq:surface-H-detection} and
\eqref{eq:surface-K-H-composite} then give
$r^{\mathrm{FQR,ucp}}_{a_K,a_0}(z)=0$.

For $a\succeq a_K$, define
$\Theta^{\mathrm{FQR,ucp}}_a=
\Theta^{\mathrm{FQR,ucp}}_{a_K}
r^{\mathrm{FQR,ucp}}_{a,a_K}$.  This gives compatibility with bonding maps,
and the same argument followed by restriction to $a_0$ proves
\eqref{eq:surface-K-detection} at every stronger stage.
\end{proof}

\begin{lemma}[compatibility of index pairings with rank and winding number]
\label{lem:surface-coordinate-compatibility}
After one common strengthening of the controls in
\Cref{prop:surface-finite-detection,prop:surface-Fredholm-detection},
\begin{equation}\label{eq:surface-coordinate-compatibility}
 \Theta^{\mathrm{FQR,ucp}}_a\iota^{\mathrm{ucp}}_a
   =\Theta^{\mathrm{QR,ucp}}_a
\end{equation}
for every stronger control $a$.
\end{lemma}

\begin{proof}
The component maps between u.c.p. Fredholm cycles and controlled
$K$-homology preserve the index pairing.  For a finite-dimensional cycle,
the pairing is exactly the partial $K$-theory map induced by its u.c.p. map.
The two values are therefore its dimension and its pairing with $\beta$.
By \Cref{lem:surface-ucp-EL}, the latter is its winding number.  This proves
the displayed identity on generators, hence on the Grothendieck group.
Naturality of the component maps gives the identity at every stronger
control.
\end{proof}

\subsubsection{Direct computation of the two limiting groups}

We next construct the two generators geometrically.

\begin{lemma}[u.c.p. realization of the surface coordinate]
\label{lem:surface-clock-shift}
At every sufficiently strong u.c.p. control $a$ there is a homomorphism
\[
 \gamma_a:\mathbb Z^2\longrightarrow\QRucpstage{\Gamma_g}{a}
 \quad\text{such that}\quad
 \Theta^{\mathrm{QR,ucp}}_a\gamma_a=1_{\mathbb Z^2}.
\]
\end{lemma}

\begin{proof}
Consider the central extension
\begin{equation}\label{eq:surface-central-extension}
 \widetilde\Gamma_g=
 \left\langle a_1,b_1,\ldots,a_g,b_g,z\ \middle|\
 z\text{ central},\ \prod_{j=1}^g[a_j,b_j]=z\right\rangle.
\end{equation}
Choose a section $s:\Gamma_g\to\widetilde\Gamma_g$ which takes the surface
generators to the displayed lifts.  There is an integer-valued cocycle
$\kappa$ such that
\begin{equation}\label{eq:surface-cocycle}
 s(x)s(y)=z^{\kappa(x,y)}s(xy).
\end{equation}
For $N\geq3$, let $\omega_N=e^{2\pi i/N}$ and define clock and shift
unitaries by
\[
 U_Ne_k=\omega_N^ke_k,
 \qquad V_Ne_k=e_{k+1\pmod N}.
\]
Then $[U_N,V_N]=\omega_N1_N$.  The assignments
\[
 a_1\mapsto U_N,\quad b_1\mapsto V_N,\quad
 a_j,b_j\mapsto1_N\ (j\geq2),\quad z\mapsto\omega_N1_N
\]
define a representation $\pi_N:\widetilde\Gamma_g\to U(N)$.  Put
$q_N(x)=\pi_N(s(x))$.  Equation \eqref{eq:surface-cocycle} gives
\begin{equation}\label{eq:surface-qN-defect}
 \|q_N(x)q_N(y)-q_N(xy)\|
 =|\omega_N^{\kappa(x,y)}-1|
 \leq\frac{2\pi}{N}|\kappa(x,y)|.
\end{equation}
Thus $q_N$ satisfies any prescribed unitary control when $N$ is large.
Moreover $w(q_N)=1$.

Fix a u.c.p. control $a$.  Apply
\Cref{lem:LP-CP-replacement} with approximation tolerance small enough that
polar correction preserves the spectral gap defining the winding number.
For all sufficiently large $N$, it supplies a global u.c.p. map
$\phi_N:A\to M_N(\mathbb C)$ representing a class at the $a$-stage and
satisfying $w(\phi_N)=w(q_N)=1$ by
\Cref{lem:surface-ucp-EL}.  If $\mathbf1_N$ denotes the $N$-dimensional
trivial representation, then
\[
 b_a=[\phi_N]-[\mathbf1_N],
 \qquad \Theta^{\mathrm{QR,ucp}}_a(b_a)=(0,1).
\]
Let $u_a$ be the class of the one-dimensional trivial representation, so
$\Theta^{\mathrm{QR,ucp}}_a(u_a)=(1,0)$.  Define
\[
 \gamma_a(r,k)=r u_a+k b_a.
\]
Negative coefficients are interpreted in the Grothendieck group.  The
displayed coordinate formulas give
$\Theta^{\mathrm{QR,ucp}}_a\gamma_a(r,k)=(r,k)$.
\end{proof}

The following elementary lemma gives the inverse-limit argument needed
for the direct computation.

\begin{lemma}[an inverse-limit isomorphism criterion]
\label{lem:eventual-detection-limit}
Let $(B_a,r_{b,a})$ be an inverse system of abelian groups over a directed
set, and suppose that on a cofinal set of controls there are compatible
homomorphisms $\Theta_a:B_a\to L$ to one abelian group $L$.  Assume:
\begin{enumerate}[label=(\roman*)]
\item for every target $a_0$ there is $a_1\succeq a_0$ such that
$\Theta_a(x)=0$ implies $r_{a,a_0}(x)=0$ for every $a\succeq a_1$;
\item for every sufficiently strong $a$, the map $\Theta_a$ is surjective.
\end{enumerate}
Then the induced homomorphism
$\Theta:\varprojlim_aB_a\to L$ is an isomorphism.
\end{lemma}

\begin{proof}
If $\Theta((x_a)_a)=0$, choose $a_1$ as in (i).  Then
$x_{a_0}=r_{a_1,a_0}(x_{a_1})=0$.  Since $a_0$ was arbitrary, $\Theta$ is
injective.

Let $\ell\in L$.  For a target $a_0$, choose $a_1$ as in (i), sufficiently
strong that $\Theta_{a_1}$ is onto, and choose $y_{a_1}\in B_{a_1}$ with
$\Theta_{a_1}(y_{a_1})=\ell$.  Set
$x_{a_0}=r_{a_1,a_0}(y_{a_1})$.  This does not depend on the choices.  Indeed,
given a second choice $(a_1',y_{a_1'})$, choose $a_2\succeq a_1,a_1'$ and
$y_{a_2}$ with $\Theta_{a_2}(y_{a_2})=\ell$.  Then
\[
 \Theta_{a_1}(y_{a_1}-r_{a_2,a_1}y_{a_2})=0,
\]
so (i) shows that the two elements have the same image in $B_{a_0}$; the
same argument applies to $a_1'$.  If $a_0\preccurlyeq a_0'$, a choice valid
for $a_0'$ is also valid for $a_0$ after further strengthening.  Hence
$r_{a_0',a_0}(x_{a_0'})=x_{a_0}$.  The family $(x_{a_0})$ lies in the
inverse limit and maps to $\ell$, proving surjectivity.
\end{proof}

\begin{theorem}[direct $\mathbb Z^2$ classification for surface groups]
\label{thm:surface-direct-Z2}
The rank and winding-number maps give canonical isomorphisms
\begin{equation}\label{eq:surface-two-Z2-isomorphisms}
 \mathrm{QR}^{\mathrm{ucp}}(\Gamma_g)
   \xrightarrow{\cong}\mathbb Z^2,
 \qquad
 \mathrm{FQR}^{\mathrm{ucp}}(\Gamma_g)
   \xrightarrow{\cong}\mathbb Z^2.
\end{equation}
Under these identifications the homomorphism
$\iota^{\mathrm{ucp}}$ is the identity of $\mathbb Z^2$.  In particular,
$\Gamma_g$ is QR-quasidiagonal in the u.c.p. sense of \Cref{def:KQD}.

Using the lifting-property comparison, the same coordinates give
\begin{equation}\label{eq:surface-unitary-two-Z2}
 \mathrm{QR}(\Gamma_g)\cong\mathbb Z^2,
 \qquad
 \mathrm{FQR}(*;\Gamma_g;\mathbb C)\cong\mathbb Z^2.
\end{equation}
The canonical inclusion is the identity in these coordinates, so
$\Gamma_g$ also satisfies \Cref{def:QR-unitary}.
\end{theorem}

\begin{proof}
By \Cref{rem:FQRucp-controlled-K-homology}, the second inverse limit in
\eqref{eq:surface-two-Z2-isomorphisms} is canonically $KL^0(A)$.  The
isomorphism \eqref{eq:surface-Khom-pairing} identifies it with
$\mathbb Z^2$ by the pairings with $[1]$ and $\beta$.

For the first inverse limit, the coordinate maps
\eqref{eq:surface-Q-coordinates} are compatible on a cofinal family of
controls.  Condition (i) of
\Cref{lem:eventual-detection-limit} is
\Cref{prop:surface-finite-detection}, and condition (ii) is
\Cref{lem:surface-clock-shift}.  That lemma therefore gives
$\mathrm{QR}^{\mathrm{ucp}}(\Gamma_g)\cong\mathbb Z^2$.
Compatibility \eqref{eq:surface-coordinate-compatibility} shows that
$\iota^{\mathrm{ucp}}$ becomes the identity under these identifications with $\mathbb Z^2$.
Finally, \Cref{cor:LP-unitary-ucp-limits} transfers the two isomorphisms to
the unitary-valued theories, proving \eqref{eq:surface-unitary-two-Z2}.
\end{proof}

\begin{corollary}[the character classification for a surface group]
\label{cor:surface-character-classification}
The character is an isomorphism
\begin{equation}\label{eq:surface-character-classification}
 \chi:\mathrm{QR}(\Gamma_g)\xrightarrow{\cong}
 RK^0_{\Gamma_g}(E\Gamma_g)
 \cong K^0(\Sigma_g)\cong\mathbb Z^2.
\end{equation}
For a sufficiently controlled unitary quasi-representation $q$, the two
coordinates are
\begin{equation}\label{eq:surface-final-coordinates}
 \left(\dim q,
 \frac{1}{2\pi i}\operatorname{Tr}\log
       \prod_{j=1}^g[q(a_j),q(b_j)]\right),
\end{equation}
with subtraction for formal differences.
\end{corollary}

\begin{proof}
\Cref{thm:surface-direct-Z2} computes the two limiting groups and identifies
the map from finite-dimensional quasi-representations with the identity on rank and winding
number. The character bundle descends to $\Sigma_g$ because the action is
free. Its rank is the dimension of the quasi-representation, while
\Cref{lem:surface-ucp-EL}, the character compatibility in
\Cref{lem:character-compatibility}, and the oriented surface index pairing
identify its first Chern number with the winding coordinate in
\eqref{eq:surface-final-coordinates}. Rank and first Chern number identify
$K^0(\Sigma_g)$ with $\mathbb Z^2$, by the two-dimensional
Atiyah--Hirzebruch calculation. Thus the character is the identity on these
two coordinates and is an isomorphism. This argument uses the direct
coordinate computation, not the quantitative QR-quasidiagonality theorem
proved in the next subsection.
\end{proof}

\subsubsection{The quantitative strengthening}

The direct computation above answers the qualitative classification
question.  We now keep track of the finite sets and tolerances to obtain the
stronger property of \Cref{def:quantitative-KQD}.

\begin{theorem}[surface groups are quantitatively
QR-quasidiagonal]
\label{thm:surface-QKQD}
For every $g\geq2$, the surface group $\Gamma_g$ is quantitatively
QR-quasidiagonal in the u.c.p. sense of
\Cref{def:quantitative-KQD}.  Consequently the character is not only the
isomorphism of \Cref{cor:surface-character-classification}; the following
quantified statements also hold.
\begin{enumerate}[label=(\roman*)]
\item\label{item:surface-character-equality}
Let $c_0=(F_0,\varepsilon_0)\in\mathcal C_{\Gamma_g}$.  There is a number
$d_0>0$ such that, for every $d\geq d_0$, one can choose a stronger control
$c_1=(F_1,\varepsilon_1)\succeq c_0$ for which
\[
 \chi_d:\mathrm{QR}(F_1,\varepsilon_1)
       \longrightarrow K_{\Gamma_g}^0(P_d(\Gamma_g))
\]
is defined.  If $q_1$ and $q_2$ are
$(F_1,\varepsilon_1)$-quasi-representations and
\[
 \chi_d([q_1])=\chi_d([q_2])
       \quad\text{in }K_{\Gamma_g}^0(P_d(\Gamma_g)),
\]
then
\[
 r^{\mathrm{QR}}_{c_1,c_0}([q_1])
 =r^{\mathrm{QR}}_{c_1,c_0}([q_2])
       \quad\text{in }\mathrm{QR}(F_0,\varepsilon_0).
\]
Here $r^{\mathrm{QR}}_{c_1,c_0}$ is the homomorphism obtained by viewing an
$(F_1,\varepsilon_1)$-quasi-representation as an
$(F_0,\varepsilon_0)$-quasi-representation.

\item\label{item:surface-character-lifting}
For every $d>0$, there are $d'>d$ and a control
$c=(F,\varepsilon)\in\mathcal C_{\Gamma_g}$ for which
\[
 \chi_d:\mathrm{QR}(F,\varepsilon)
       \longrightarrow K_{\Gamma_g}^0(P_d(\Gamma_g))
\]
is defined and such that, for every
$z\in K_{\Gamma_g}^0(P_{d'}(\Gamma_g))$, there is
$x\in\mathrm{QR}(F,\varepsilon)$ satisfying
\[
 \chi_d(x)=r^K_{d',d}(z)
       \quad\text{in }K_{\Gamma_g}^0(P_d(\Gamma_g)).
\]
Here $r^K_{d',d}$ is pullback along
$P_d(\Gamma_g)\hookrightarrow P_{d'}(\Gamma_g)$.
\end{enumerate}
\end{theorem}

\begin{proof}
Fix a target u.c.p. control $a$.  Choose a common stronger control $b$ at
which the following four conclusions hold simultaneously:
\begin{enumerate}[label=(\roman*)]
\item the vanishing criterion for finite-dimensional classes from $b$ to $a$, as in
\eqref{eq:surface-Q-detection};
\item the vanishing criterion for Fredholm classes from $b$ to $a$, as in
\eqref{eq:surface-K-detection};
\item coordinate compatibility
$\Theta^{\mathrm{FQR,ucp}}_b\iota^{\mathrm{ucp}}_b
=\Theta^{\mathrm{QR,ucp}}_b$;
\item a homomorphism
$\gamma_b:\mathbb Z^2\to\QRucpstage{\Gamma_g}{b}$ with
$\Theta^{\mathrm{QR,ucp}}_b\gamma_b=1$.
\end{enumerate}
Define
\begin{equation}\label{eq:surface-component-rho}
 \rho_{b,a}:\FQRucpstage{\Gamma_g}{b}
 \longrightarrow\QRucpstage{\Gamma_g}{a},
 \qquad
 \rho_{b,a}=r^{\mathrm{QR,ucp}}_{b,a}
             \gamma_b\Theta^{\mathrm{FQR,ucp}}_b.
\end{equation}
For $z\in\FQRucpstage{\Gamma_g}{b}$, the element
\[
 z-\iota^{\mathrm{ucp}}_b\gamma_b
   \Theta^{\mathrm{FQR,ucp}}_b(z)
\]
has zero Fredholm coordinate.  The vanishing criterion for Fredholm classes gives
\begin{equation}\label{eq:surface-component-right}
 \iota^{\mathrm{ucp}}_a\rho_{b,a}(z)
   =r^{\mathrm{FQR,ucp}}_{b,a}(z).
\end{equation}
For $x\in\QRucpstage{\Gamma_g}{b}$, the element
$x-\gamma_b\Theta^{\mathrm{QR,ucp}}_b(x)$ has zero finite-dimensional
coordinate.
The vanishing criterion for finite-dimensional classes gives
\begin{equation}\label{eq:surface-component-left}
 \rho_{b,a}\iota^{\mathrm{ucp}}_b(x)
   =r^{\mathrm{QR,ucp}}_{b,a}(x).
\end{equation}
Thus both inverse identities hold at the prescribed target control.

It remains to choose these maps coherently.  We give the details because
coherence is additional information not contained in
\eqref{eq:surface-two-Z2-isomorphisms}.  Choose an increasing cofinal
sequence
\[
 a_1\preccurlyeq a_2\preccurlyeq\cdots
\]
in $\mathcal C_{\Gamma_g}^{\mathrm{sym}}$.  Recursively choose increasing
controls $b_n\succeq a_n$ and homomorphisms
\[
 \lambda_n:\FQRucpstage{\Gamma_g}{b_n}
 \longrightarrow\QRucpstage{\Gamma_g}{a_n}
\]
of the form
\eqref{eq:surface-component-rho}, satisfying
\eqref{eq:surface-component-left}--
\eqref{eq:surface-component-right}.

Suppose $b_n$ and $\lambda_n$ have been chosen.  Choose
$b_{n+1}\succeq b_n,a_{n+1}$ so that the two vanishing criteria with
target $a_{n+1}$, coordinate compatibility, and
$\Theta_{b_{n+1}}^{\mathrm{QR,ucp}}\gamma_{b_{n+1}}=1$ all hold at the
$b_{n+1}$-stage.  Define
\[
 \lambda_{n+1}
   =r^{\mathrm{QR,ucp}}_{b_{n+1},a_{n+1}}
      \gamma_{b_{n+1}}\Theta^{\mathrm{FQR,ucp}}_{b_{n+1}}.
\]
This map has the two componentwise inverse properties at $a_{n+1}$ by the
first part of the proof.

We verify compatibility with $\lambda_n$.  Fix
$z\in\FQRucpstage{\Gamma_g}{b_{n+1}}$.  In
$\QRucpstage{\Gamma_g}{b_n}$, compare the two elements
\[
 r^{\mathrm{QR,ucp}}_{b_{n+1},b_n}
     \gamma_{b_{n+1}}\Theta^{\mathrm{FQR,ucp}}_{b_{n+1}}(z)
 \quad\hbox{and}\quad
 \gamma_{b_n}\Theta^{\mathrm{FQR,ucp}}_{b_n}
     r^{\mathrm{FQR,ucp}}_{b_{n+1},b_n}(z).
\]
Their $\Theta^{\mathrm{QR,ucp}}_{b_n}$-coordinates are equal, by compatibility
of the $\Theta$ maps and the identities
$\Theta^{\mathrm{QR,ucp}}\gamma=1$.  Our choice of
$b_n$ included the vanishing criterion for finite-dimensional classes from $b_n$ to $a_n$; hence
their images in $\QRucpstage{\Gamma_g}{a_n}$ are equal.  Since
$\lambda_n=r^{\mathrm{QR,ucp}}_{b_n,a_n}
\gamma_{b_n}\Theta^{\mathrm{FQR,ucp}}_{b_n}$, this equality is precisely
\begin{equation}\label{eq:surface-rho-coherence-sequence}
 r^{\mathrm{QR,ucp}}_{a_{n+1},a_n}\lambda_{n+1}
 =\lambda_n r^{\mathrm{FQR,ucp}}_{b_{n+1},b_n}.
\end{equation}
Thus the recursion continues with exact equality of homomorphisms, not only
with equality on inverse-limit elements.

For an arbitrary control $a$, let $n(a)$ be the least integer with
$a\preccurlyeq a_{n(a)}$, and set
\begin{equation}\label{eq:surface-all-controls}
 \sigma(a)=b_{n(a)},
 \qquad
 \rho_a=r^{\mathrm{QR,ucp}}_{a_{n(a)},a}\lambda_{n(a)}.
\end{equation}
Then $\sigma$ is order preserving and $\sigma(a)\succeq a$.
Equation \eqref{eq:surface-rho-coherence-sequence} gives
\eqref{eq:QKQD-coherence}; the two identities
\eqref{eq:surface-component-left} and
\eqref{eq:surface-component-right}, followed by restriction to $a$, give
\eqref{eq:QKQD-left} and \eqref{eq:QKQD-right}.  This proves quantitative
QR-quasidiagonality.  By \Cref{lem:surface-input}, the group $\Gamma_g$ has
the Haagerup property and $C^*_{\max}(\Gamma_g)$ has the lifting property.
By \Cref{prop:QKQD-unitary-transfer}, it is quantitatively
QR-quasidiagonal in the unitary-valued sense as well. Thus
\Cref{thm:QKQD-Haagerup} applies.  Its assertion
\ref{item:QKQD-Haagerup-equality}, with $G=\Gamma_g$, is exactly
\ref{item:surface-character-equality}, and its assertion
\ref{item:QKQD-Haagerup-lifting} is exactly
\ref{item:surface-character-lifting}.
\end{proof}


\bibliographystyle{alpha}

\end{document}